%% file: phd-thesis.tex
\documentclass[11pt, english, twoside, babel, reqno]{smfbook}
\usepackage{etoolbox}
\usepackage{enumitem}
\usepackage{graphicx}
\RequirePackage[all,ps,cmtip]{xy}
\usepackage[english]{babel}
\usepackage{amssymb,url,xspace, smfthm}
\usepackage{mathrsfs}
\usepackage{epstopdf}
\usepackage[margin=4cm]{geometry}
\newcommand{\mc}[1]{\mathscr{#1}}
\patchcmd{\thebibliography}{*}{}{}{}
\pretocmd\thebibliography{\csname c@secnumdepth\endcsname=-2}{}{}
\input{custom.tex}

\usepackage{eucal}
\usepackage{verbatim}
\SwapTheoremNumbers
\NumberTheoremsAs{subsection}
\usepackage{yfonts}
\numberwithin{equation}{subsection}
\newtheorem*{theorem*}{Theorem}
\begin{document}
\frontmatter
\setcounter{page}{-7}
\vspace*{-2cm}\enlargethispage{2cm}
\bgroup
\begin{center}\vspace*{5cm}
\huge\bfseries\MakeUppercase{Elliptic curves with complex multiplication and $\Lambda$-structures}
\vfil
{\LARGE Lance Rory Gurney\par\vspace*{.5cm}
\par\vspace*{.5cm}}
{\large July 2015 \par\vspace*{.5cm}
\par\vspace*{.5cm}}
\vspace*{1.5cm} \normalsize{\fontfamily{ptm} Thesis submitted for the degree Doctor of Philosophy \\ of the Australian National University}
\vfil\vfil
         \includegraphics[width=5cm]{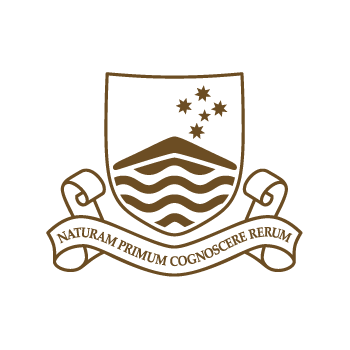}
\mbox{ }
\end{center}
\cleardoublepage

\newpage\thispagestyle{empty}\vspace*{3cm}
\begin{flushleft}\textbf{\MakeUppercase{Declaration}}\\\vspace{3cm}The work in this thesis is my own except where otherwise stated.\end{flushleft} \vspace*{3cm} \begin{flushright}\includegraphics[width=4.5cm]{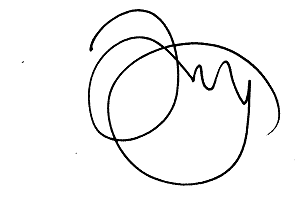} 

Lance Rory Gurney \end{flushright}
\tableofcontents\newpage\null\thispagestyle{empty}\newpage

\chapter*{Introduction}
This thesis began as an attempt to answer the question: \begin{quotation} Q: What do elliptic curves with complex multiplication and $\Lambda$-structures have to do with one another?\end{quotation}  The reader is probably somewhat familiar with first term, less likely so with the second. Even if he or she is familiar with both, why such a question might have an answer worth finding is probably not clear at all. So we will begin by explaining how both of these terms are related to a third: class field theory.

Let us remind the reader of the relationship between elliptic curves with complex multiplication and class field theory. For all that follows we fix an imaginary quadratic field $K$, with ring of integers $O_K$. If $L/K$ is a finite extension, an elliptic curve with complex multiplication by $O_K$ over $L$, here on called a CM elliptic curve over $L$, is an elliptic curve $E/L$ with the property that its ring of endomorphisms $\End_L(E)$ is isomorphic to $O_K$. For a general elliptic curve $E/L$, the Tate module $T(E)=\lim_n E[n](L^\sep)$ is a rank two $\widehat{\Z}$-module equipped with an action of $G(L^\sep/L)$. However, when $E/L$ is a CM elliptic curve, this rank two $\widehat{\Z}$-module becomes a rank one $\widehat{\Z}\otimes_{\Z}O_K$-module and so the action of $G(L^\sep/L)$ is defined by a character \[\rho_{E/L}: G(L^\ab/L)=G(L^\sep/L)^\ab\to (\widehat{\Z}\otimes_{\Z}O_K)^\times.\] In particular, it follows that extensions of $L$ generated by torsion points of $E$ are abelian over $L$. It is then known that if one can find a CM elliptic curve $E$ defined over $K$ itself, the resulting character \[\rho_{E/K}: G(K^\ab/K)\to (\widehat{\Z}\otimes_{\Z} O_K)^\times\] is injective from which it follows that every abelian extension of $K$ is a sub-extension of one generated by the torsion points of $E$ --- thus realising explicitly the class field theory of $K$.

In general there do not exist CM elliptic curves defined over $K$. The smallest field of definition of a CM elliptic curve is the Hilbert class field $H/K$ (the maximal abelian, everywhere unramified extension of $K$). If $E/H$ is such a curve then the extensions of $H$ generated by its torsion points are abelian over $H$, but they are not necessarily abelian over $K$. If one is still interested in the class field theory of $K$, this problem can be overcome by considering certain $O_K^\times=\Aut_H(E)$ invariant maps \[w: E\to \mathbf{P}^1_H,\] called Weber functions. The extensions of $K$ generated by the co-ordinates of the images of torsion points of $E$ under a fixed Weber function are abelian, and every abelian extension of $K$ is a sub-extension of one of these --- again realising the class field theory of $K$. Therefore, if one is interested in the class field theory of $K$ one need go no further than CM elliptic curves.

Now let us explain the relationship between $\Lambda$-structures and class field theory. First, a $\Lambda$-structure on a flat $\Spec(O_K)$-scheme $X$ (we will say more about the non-flat case later) is nothing more than a commuting family endomorphisms \[\psi^\a: X\to X\] indexed by the non-zero ideals $\a$ of $O_K$ such that for any two ideals $\a, \b$ we have $\psi^\a\circ \psi^\b=\psi^{\a\b}$ and with the property that for each prime ideal $\p$, the restriction of $\psi^\p$ to the fibre $X_\p:=X\times_{\Spec(O_K)}\Spec(O_K/\p)$ is the $N\p$-power Frobenius endomorphism \[\Fr^{N\p}: X_\p\to X_\p.\] We call the endomorphisms $\psi^\a$ (for all $\a$) a commuting family of Frobenius lifts. The resulting notion of a $\Lambda$-morphism of $\Lambda$-schemes $f:X\to Y$ being one that commutes with the Frobenius lifts. The $O_K$-scheme $\Spec(O_K)$ has a unique $\Lambda$-structure with Frobenius lifts all equal to the identity and if $L/K$ is an abelian extension with ring of integers $O_L$ then (ignoring the ramified primes) the finite locally free $O_K$-scheme $\Spec(O_L)$ admits a unique $\Lambda$-structure as well. More generally, if $S$ is any finite locally free $\Spec(O_K)$-scheme equipped with a $\Lambda$-structure then the extension of $K$ generated by the co-ordinates of $S$ (in any affine embedding) is an abelian extension of $K$. The link between $\Lambda$-structures and class field theory appears. These observations also have the following implication. If $X$ a $\Lambda$-scheme over $\Spec(O_K)$ and $0_X: \Spec(O_K)\to X$ a $\Lambda$-morphism then (under certain hypotheses) for each ideal $\a$ the scheme $X[\a]:=\psi^{\a*}(0_X)\subset X$ is a finite locally free $\Lambda$-scheme over $\Spec(O_K)$. Therefore, the extension of $K$ generated by the coordinates of $X[\a]$ will be abelian.

With this in hand, let us return to CM elliptic curves. It now turns out that if $E/K$ is a CM elliptic curve, writing $\mc{E}\to \Spec(O_K)$ for the N\'eron model of $E/K$, the flat $O_K$-scheme $\mc{E}$ admits a unique $\Lambda$-structure (ignoring the primes of bad reduction for $E/K$) and the morphism \[0_{\mc{E}}: \Spec(O_K)\to \mc{E}\] is a $\Lambda$-morphism. Moreover, for each integer $n\geq 0$, viewing $(n)\subset O_K$ as an ideal, we have that $\psi^{(n)*}(0_{\mc{E}})=\mc{E}[n]$ is the $n$-torsion of $\mc{E}$ which is now a finite locally free $\Lambda$-scheme. It is now also natural to ask whether in general there exist CM elliptic curves $E/H$ defined over the Hilbert class field whose N\'eron models admit $\Lambda$-structures and this turns out to be a subtle question.

At this point, we should note that everything we have said so far is the work of others. Indeed, the theory of complex multiplication and its relationship with class field theory is classical and has a very long history and to give a list of names would be very difficult. The theory of $\Lambda$-schemes and $\Lambda$-structures is due to Borger (\cite{Borger11}, \cite{Borger11bis}), and the relationship between $\Lambda$-structures and class field theory is due to Borger-de Smit (\cite{BorgerdeSmit04}, \cite{BorgerdeSmit08}) and indeed it was my advisor James Borger who originally posed the question at the beginning of this introduction and the more specific one asking for the existence of CM elliptic curves $E/H$ over the Hilbert class field with $\Lambda$-structures. This second question we can now answer:

\begin{theorem*}
\begin{enumerate}[label=\textup{(\roman*)}]
\item There always exists a \textup{CM} elliptic curve $E/H$ over the Hilbert class field whose N\'eron model admits a $\Lambda$-structure.
\item The N\'eron model of a \textup{CM} elliptic curve $E/H$ admits a $\Lambda$-structure if and only if the extension of $K$ generated by its torsion is abelian over $K$.
\end{enumerate}
\end{theorem*}

CM elliptic curves (over arbitrary abelian extensions of $K$) with the property that the extensions generated by their torsion is abelian over $K$ were introduced originally by Shimura and are now called CM elliptic curves of Shimura type. Indeed, it is using results of Shimura that, after proving (ii) we are able to prove (i). It worth pointing out that CM elliptic curves of Shimura type have been studied by several authors, with particular reference to their $L$-functions and the Birch-Swinnerton-Dyer Conjecture. Indeed, the papers of Coates-Wiles \cite{CoatesWiles1977} and Rubin \cite{Rubin1981} concern curves of this type.

Let us now also answer the question posed in the first paragraph: \begin{quotation} Q: What do elliptic curves with complex multiplication and \mbox{$\Lambda$-structures} have to do with one another?\ \\
A: Everything!
\end{quotation}

The central theorem we aim to prove is the following and explains our answer to the question above.

\begin{theorem*}\label{theo:intro-main} Let $\mc{M}_\CM$ denote the moduli stack of \textup{CM} elliptic curves and let $\mc{E}\to \mc{M}_\CM$ denote the universal \textup{CM} elliptic curve. Then both $\mc{E}$ and $\mc{M}_\CM$ admit canonical $\Lambda$-structures and the morphism $\mc{E}\to \mc{M}_\CM$ is a $\Lambda$-morphism.
\end{theorem*}

The reader will probably have noticed that we have not even defined what it means for a general (non-flat) scheme, let alone a stack, to have a $\Lambda$-structure and in fact we do not propose to define $\Lambda$-structures on stacks in this thesis (not because it isn't possible to do so, only because doing so would lead us into the nightmarish realm of 2-monads on 2-categories). In any case, the definition of $\Lambda$-structure we have given for a flat $O_K$-scheme admits an obvious naive generalisation --- that of a commuting family of Frobenius lifts --- though it is not the correct one. Before we explain the correct definition, and the actual meaning of (\ref{theo:intro-main}), let us describe the naive $\Lambda$-structure on $\mc{M}_\CM$.

If $S$ is an $O_K$-scheme then $\mc{M_\CM}(S)$ is the category of CM elliptic curves $E$ over $S$. In particular, the objects $E/S\in \mc{M}_\CM(S)$ are $O_K$-modules and for each non-zero ideal $\a\subset O_K$ it is possible to make sense of the $O_K$-module $E\otimes_{O_K}\a^{-1}$ which is again a CM elliptic curve over $S$. This defines for each ideal $\a$ an endofunctor \begin{equation*}-\otimes_{O_K}\a^{-1}:\mc{M}_\CM\to \mc{M}_\CM: E\mto E\otimes_{O_K}\a^{-1},\end{equation*} and for different ideals $\a$ these endo-functors all `commute' in the obvious sense. More important is the following fact: if $\p\subset O_K$ is prime and $S$ is an $O_K$-scheme of characteristic $\p$, i.e.\ the morphism $S\to \Spec(O_K)$ factors through $\Spec(O_K/\p)$, then for all CM elliptic curves $E/S$ there is a canonical isomorphism \[E\otimes_{O_K}\p^{-1}\isomto \Fr^{N\p*}(E)\] between $E\otimes_{O_K}\p^{-1}$ and the pull-back of $E$ along the $N\p$-power Frobenius $\Fr^{N\p}: S\to S$ (this construction in its simplest form is due to Serre). In other words, the functor $-\otimes_{O_K}\p^{-1}$ is a lift of the $N\p$-power Frobenius and we have defined a naive $\Lambda$-structure on $\mc{M}_\CM$.

However, as we have already noted, the naive notion of a $\Lambda$-structure as a commuting family of Frobenius lifts is not the correct one. Let us now at least say enough about the non-naive notion of $\Lambda$-structure so that we may explain to the reader some of the actual meaning of (\ref{theo:intro-main}).

There is a functor \[W^*: \Sch_{O_K}\to \Sch_{O_K}: X\mto W^*(X)\] sending an $O_K$-scheme $X$ to its scheme $W^*(X)$ of (big $O_K$-typical) Witt vectors (technically speaking $W^*(X)$ is actually an ind-scheme, but for the purposes of this introduction we will ignore this fact). We will not say much here about the geometry of the Witt vectors $W^*(X)$ themselves, other than to remark that they have many miraculous properties, chief among which is that if $X$ is an $O_K$-scheme of characteristic $\p$ for some prime of $O_K$, then (very roughly speaking) the Witt vectors $W^*(X)$ of $X$ have characteristic $0$. The properties of the Witt vectors $W^*(X)$ which are of interest to us presently are the following:
\begin{enumerate}[label=(\roman*)]
\item $W^*(X)$ possesses a family of commuting Frobenius lifts \[\psi^{\a}: W^*(X)\to W^*(X)\] indexed by the ideals $\a$ of $O_K$ (i.e.\ a naive $\Lambda$-structure),
\item there is a morphism $g_{(1)}: X\to W^*(X)$, and
\item for all flat $\Lambda$-schemes $S$ (the definition of which we have given) and all morphisms $X\to S$ there is a unique morphism $W^*(X)\to S$ compatible with the Frobenius lifts on $W^*(X)$ and $S$ such that the following diagram commutes \begin{equation}\label{eqn:def-lambda-intro}\begin{gathered}\xymatrix{& W^*(X)\ar@{.>}[d]\\
X\ar[ur]^{g_{(1)}} \ar[r] & S.}\end{gathered}\end{equation}
\end{enumerate}

If we now let $S$ be an arbitrary $O_K$-scheme, the diagram (\ref{eqn:def-lambda-intro}) can be taken as the \textit{definition} of a $\Lambda$-structure on $S$: a $\Lambda$-structure on an $O_K$-scheme $S$ is, for each morphism $X\to S$, a \textit{canonical lifting} $W^*(X)\to S$ of that morphism making the diagram (\ref{eqn:def-lambda-intro}) commute (together with certain iterated compatibilities which we will not give here). It is not unreasonable to make the comparison between the construction of Serre-Tate of the canonical lift to the ($p$-typical) Witt vectors of ordinary elliptic curve over a finite field.

We can now give the reader a better sense of the meaning of (\ref{theo:intro-main}). If $S$ is an $O_K$-scheme and $E/S$ is a CM elliptic curve, i.e.\ if one is given a morphism $S \stackrel{E}{\to} \mc{M}_\CM$, then there is a functorially defined CM elliptic curve $W^*_\CM(E)$ over the (big $O_K$-typical) Witt vectors $W^*(S)$ of $S$, i.e.\ there is a morphism $W^*(S)\stackrel{W_\CM^*(E)}{\longrightarrow} \mc{M}_\CM$, together with a canonical isomorphism $g_{(1)}^*(W_\CM^*(E))\isomto E$. This to say we have a `commutative' diagram: \[\xymatrix{& W^*(S)\ar@{.>}[d]^{W^*_\CM(E)}\\
S\ar[ur]^{g_{(1)}} \ar[r]^-{E} & \mc{M}_\CM}\] and this is what we mean when we say that $\mc{M}_\CM$ admits a $\Lambda$-structure. It follows that CM elliptic curves can be lifted canonically to the Witt vectors of the base. We would like to point out that the base $S$ here is arbitrary.

We shall not anything more about (\ref{theo:intro-main}) here, nor the $\Lambda$-structure on $\mc{E}$. However, what (\ref{theo:intro-main}) does do is equip $\mc{E}\to \mc{M}_\CM$ with an incredible amount of very rich structure --- structure which is of interest in and of itself in the world of $\Lambda$-geometry, but which can also be exploited to prove new results in both the theory of CM elliptic curves and the arithmetic of imaginary quadratic fields. The following result is an example of this phenomenon:

\begin{theorem*} Let $K(\f)$ be the ray class field of conductor $\f$ and let $E/K(\f)$ be a \textup{CM} elliptic curve of Shimura type. If the $\f$-torsion $E[\f]$ is constant then $E/K(\f)$ admits a global minimal model away from $\f$. In particular, if $\f=O_K$ so that $K(\f)=H$ then \textup{every} \textup{CM} elliptic curve $E/H$ of Shimura admits a \textup{global} minimal model.
\end{theorem*}

In the special case when $\mathrm{disc}(K/\Q)$ is prime and $\f=O_K$ this result was proven by Gross (Corollary 4.4 \cite{Gross82}),\\

We now give an overview of the chapters:

Chapter 1: We recall the local and global reciprocity maps, define Lubin--Tate modules and study their moduli stack $\mc{M}_\LT$. We show that $\mc{M}_\LT$ admits a certain torsor structure and using this explain how to derive the local reciprocity map directly from $\mc{M}_\LT$ using only its formal properties. We then give an overview of a certain special case of global reciprocity, and present some basic constructions regarding local systems of rank one.

Chapter 2: We undertake a quite detailed study of CM elliptic curves over arbitrary bases and their moduli stack $\mc{M}_\CM$. We show that $\mc{M}_\CM$ admits a certain torsor structure analogous to that of $\mc{M}_\LT$. We also give CM analogues of some classical theorems for general elliptic curves and, in a similar vein to Lubin-Tate modules, we explain how to derive a certain global reciprocity map directly from $\mc{M}_\CM$ using only its formal properties. Using this we classify all CM elliptic curves over fields (of arbitrary characteristic) in terms of their associated Galois representations. Finally, we consider the moduli stacks with level structure and consider their representability in the fine and coarse setting.

Chapter 3: We give a short overview of the general theory of $\Lambda$-schemes, Witt vectors and arithmetic jets. We then prove a handful of (slightly technical) new results for later use.

Chapter 4: Here we prove the main theorem regarding lifting CM elliptic curves over arbitrary bases to their big $O_K$-typical Witt vectors. We then prove that a CM elliptic curve is of Shimura type if and only if its N\'eron model admits a $\Lambda$-structure. We then consider the minimal models of CM elliptic curves of Shimura type and prove their global existence under certain hypotheses. Next we consider quotients of CM elliptic curves by their groups of automorphisms and prove that (suitable reinterpreted) they always exist and we show that the quotient of the universal CM elliptic curve descends to the coarse space of the moduli stack of CM elliptic curves. We show how this descended curve admits a canonical $\Lambda$-structure and allows one to construct the ray class fields of $K$ in a choice free, integral and coherent manner and we also show show that this descended curve is nothing but a (global) projective line. We then construct a flat, affine and pro-smooth cover of the moduli stack of CM elliptic curves, which comes equipped with a $\Lambda$-structure compatible with that on $\mc{M}_\CM$. Finally, we exhibit an interesting relationship between certain deformations of CM elliptic curves with $\Lambda$-structures and their Tate modules.

Appendix: We give an abstract and formal definition of smooth formal groups, we consider the general properties of `Serre's tensor product' and we give a short overview of Faltings' generalised Cartier duality, needed to prove certain results in Chapter 1. Finally, we prove a strengthening of an old principal ideal theorem for arbitrary number fields.

\mainmatter
\chapter*{Foundations, conventions and terminology}
\section{Sheaves}
\subsection{} In order to have a nice common ground for all the objects we would like to work with, we shall here define the basic categories of (pre)sheaves in which all objects we consider will live. We have decided to not to bother ourselves with set theoretic issues of `size' (which really only come up when one tries to sheafify wild presheaves for large topologies \cite{Waterhouse75} --- something we will have no need to do), however the concerned reader may add the word `universe' whenever he or she sees fit.

Let $\Aff$ denote the category of affine schemes and $\PSh$ the category of presheaves of sets on $\Aff$, i.e.\ the category of functors \[X:\Aff^\circ\to \Set.\] We write $\Sch$ for the category of schemes and as usual we embed $\Aff$ and $\Sch$ in $\PSh$ by sending a scheme to the functor it represents.

We shall be working with sheaves for the fpqc and \'etale topologies on $\Aff$ which we now recall. The covers of an affine scheme $S$ for the fpqc (resp. \'etale) topology are given by flat (resp. \'etale) families $(S_i\to S)_{i\in I}$ indexed by a finite set $I$ which are covers in the usual sense. A sheaf for the fpqc topology will just be called a sheaf and the category of such sheaves will be denoted $\Sh\subset \PSh$. A sheaf for the \'etale topology will be called an $\et$-sheaf and we write $\Sh^{\et}\subset \PSh$ for the corresponding category. We have the inclusions \[\Aff\subset \Sch\subset \Sh\subset \Sh^\et\subset \PSh.\]

\subsection{} If $f: S'\to S$ is a morphism of presheaves and $X\to S$ is an $S$-presheaf then we denote the fibre product by $X\times_S S'$, $f^*(X)$, or when $f$ is clear by $X_{S'}$. Viewed as a functor $f^*: \PSh_{S}\to \PSh_{S'}$, the right adjoint to $f^*$ is denoted by $f_*$ and the left adjoint by $f_!$. Recall that if $X'\to S'$ is an $S'$-presheaf then $f_!(X')$ is the $S$-presheaf $X'\to S'\stackrel{f}{\to} S$ and $f_*$ is the $S$-presheaf defined by \[\Spec(A)\mto \Hom_{S'}(\Spec(A), f_*(X'))=\Hom_{S'}(f^*(\Spec(A)), X').\] The functor $f_!$ will be used rarely, and only in the case when $f$ is an isomorphism, in which case it is isomorphic to $g^*$ where $g=f^{-1}$.

\begin{exem} An important family of sheaves for the fpqc topology are the ind-affine schemes, the category of which we denote by $\IndAff$. Recall an ind-affine scheme $X$ is a pre-sheaf $X$ which can be written as filtered a colimit $X=\colim_i \Spec(A_i)$ of affine schemes (it follows automatically  from this that it is a sheaf). One of the main examples we work with is the following: if $S=\Spec(A)$ is an affine scheme and $I$ is an ideal we write $\Spf_I(A)$ (or just $\Spf(A)$ when $I$ is clear from the context) for the ind-affine scheme $\colim_n \Spec(A/I^{n+1})$. If $X\to \Spec(A)$ is a presheaf over $\Spec(A)$ we say that $X$ is $I$-adic or that $I$ is locally nilpotent on $X$ if the morphism $X\to \Spec(A)$ factors through $\Spf(A)\subset \Spec(A)$. If $\Spec(B)\to \Spec(A)$ is an affine $\Spec(A)$-scheme then $\Spec(B)$ is $I$-adic if and only if the ideal $I B\subset B$ is nilpotent.
\end{exem}

\subsection{} If $A$ is a set and $S$ is an $\et$-sheaf then we write \[\underline{A}_S=\coprod_{a\in A}S\] for the constant $\et$-sheaf over $S$ associated to $A$. If $S$ is actually a sheaf, so is $\underline{A}_S$. When $S$ is a fixed base (usually the spectrum of some Dedekind domain) we will drop the sub-script and just write $\underline{A}$.

\subsection{} By a cover of a sheaf $S$ (resp. $\et$-sheaf) we just mean a family $(S_i\to S)_{i\in I}$ of morphisms of sheaves (resp. $\et$-sheaves) such that $\coprod_{i\in I}S_i\to S$ is an epimorphism. When referring to properties or making claims which are compatible with base change we will use the word local to mean after base change along a cover.

\subsection{} We say that a morphism of pre-sheaves $f: X\to S$ is representable, or that $X$ is representable over $S$, if for each affine scheme $S'\to S$ the pre-sheaf $X\times_S S'$ is (representable by) a scheme. In general, for a morphism of sheaves (or $\et$-sheaves) to be representable is not local. However, it is the case for $f$ which are representable by open immersions, or when $f$ is representable by affine morphisms. In both of these cases we will just say that $f$ is open immersion, or that $f$ is affine. Similarly for any other condition of $f$ that includes affine in its definition: finite, finite locally free, a closed immersion and so on.

\begin{prop} If $f: X\to Y$ is a finite locally free \'etale morphism of $\et$-sheaves then the inclusion of the image $f(X)\to Y$ in $\Sh^\et$ is an open and closed immersion, the inclusion of the complement $Y-f(X)\to Y$ is also an open and closed immersion and $Y\amalg (Y-f(X))=Y$. Moreover, if $X$ and $Y$ are sheaves, so are the $\et$-sheaves $f(X)$ and $Y-f(X)$.
\end{prop}

The following result will be used often (it follows from Th\'eor\`eme 2.1 of Expos\'e VIII in \cite{SGA1}) cited as `by descent':

\begin{prop}\label{prop:affine-fpqc-descent} Let $X\to S$ be a morphism of ($\et$-)sheaves and let $S'\to S$ be an epimorphism of ($\et$-)sheaves. Then $X\to S$ is affine if and only if $X\times_S S'\to S'$ is affine. 
\end{prop}

\subsection{} We now say a little about what we mean by a quasi-coherent $\mc{O}_S$-module on a sheaf $S$. We have the relatively representable sheaf of rings \begin{equation}\label{def:structure-sheaf}\mc{O}_S:=\mathbf{A}^1_{S}=\Spec(\Z[T])\times_{\Spec(\Z)} S\end{equation} and the abelian category of $\mc{O}_S$-modules $\Mod(\mc{O}_S)$. For any map of sheaves $f: S' \to S$, as $\mc{O}_S\times_S S'=\mc{O}_{S'}$, we obtain the functor \[f^*: \Mod(\mc{O}_S)\to \Mod(\mc{O}_{S'}): \mc{M}\mto \mc{M}\times_S S'=\mc{M}_{S'}.\] Note that this functor is exact for any map $f$.
 
The category of quasi-coherent sheaves $\QCoh(\mc{O}_S)\subset \Mod(\mc{O}_S)$ is the full sub-category of $\mc{O}_S$-modules $\mc{M}$ such that there exists a cover $(S_i\to S)_{i\in I}$ and for each $i\in I$ an exact sequence \[\mc{O}_{S_i}^{M_i}\to \mc{O}_{S_i}^{N_i}\to\mc{M}_{S_i}\to 0\] where $M_i$ and $N_i$ are sets. When $S$ is a scheme this category coincides with the usual category of quasi-coherent sheaves over $S$ and the functor $f^*$ defined above has its usual meaning. However, the inclusion $\QCoh(\mc{O}_S)\subset \Mod(\mc{O}_S)$ is not exact (to be precise it does not preserve kernels). This explains why $f^*: \Mod(\mc{O}_S)\to \Mod(\mc{O}_{S'})$ is exact for any $f$ while the same is not true for $\QCoh(\mc{O}_S)\to \QCoh(\mc{O}_S)$.

We shall be almost exclusively concerned with locally free finite rank $\mc{O}_S$-modules --- or what is the same vector bundles --- the equivalence between which is a tautology with our definition $\mc{O}_S:=\mathbf{A}^1_S$.

\section{Group actions}

\subsection{} Let $A\to B$ be a homomorphism of rings and let $G$ be some group of $A$-automorphisms of $B$. The example to have in mind here is a Galois extension of fields $K\to L$ and $G=G(L/K)$.

Given $\sigma\in G$, in order to avoid the cumbersome notation $\Spec(\sigma): \Spec(B)\to \Spec(B)$, we will just write $\sigma: \Spec(B)\to \Spec(B)$. However, associating an affine scheme to a ring is contravariant, so that this notation becomes confusing when considering compositions $\sigma\circ \tau$ of elements in $G$. In order to avoid problems here, we make the convention that the product of two elements of $\sigma, \tau\in G$, will be denoted by $\sigma\tau$ so that $\sigma\tau\in G$ is the automorphism \[B\stackrel{\tau}{\to} B\stackrel{\sigma}{\to} B.\] We will only use the composition symbol $\circ$ when viewing $\sigma$ and $\tau$ as automorphisms of $\Spec(B)$ so that $\sigma\circ \tau$ will denote the $\Spec(A)$-automorphism of $\Spec(B)$ \[\Spec(B)\stackrel{\tau}{\longrightarrow} \Spec(B)\stackrel{\sigma}{\longrightarrow} \Spec(B).\] With this convention the following three symbols denote the same $\Spec(A)$-automorphism of $\Spec(B)$ \[\sigma\circ \tau=\Spec(\tau\sigma)=\tau\sigma: \Spec(B)\to \Spec(B).\]

\section{Stacks} 

\subsection{} Let $C$ be a site (the examples we have in mind are $C=\Aff$, $\IndAff$, $\Sh$). A fibred category over $C$ is a category $\mc{X}$ equipped with a functor \[p: \mc{X}\to C\] together with, for each morphism $f: S'\to S$ of $C$, a pull-back functor $f^*: \mc{X}(S)\to \mc{X}(S')$ where $\mc{X}(S)$ denotes the fibre of $p$ over $S$ (objects of $\mc{X}$ mapping to $S$ and morphisms those mapping to $\id_S$) together with various natural transformations between their compositions satisfying certain identities. There is also the notion of a morphism of fibred categories $f: \mc{X}'\to \mc{X}$ being a functor (strictly) compatible with the functors from $\mc{X}'$ and $\mc{X}$ to $C$ together with certain compatibility relations between the pull-back functors of $\mc{X}'$ and $\mc{X}$. Finally, a stack over $C$ is a fibred category whose objects and morphisms satisfy descent with respect to the topology on $C$.

To keep things (relatively) concrete we shall say the following for $C=\Aff$ with the fpqc topology --- it will also apply when working with other sites $C$.

The main point we want to make is that when defining stacks $\mc{X}\to \Aff$ we shall often skip the details and define only the fibres $\mc{X}(S)$ for $S\in \Aff$ and the pull-back functors $f^*:\mc{X}(S)\to \mc{X}(S')$ for $f: S'\to S$ in $\Aff$ (and sometimes not even these when they are clear). Similarly a morphism of stacks will be defined only on the fibres, the various compatibilities which must be satisfied will be obvious from the context.

\subsection{} To each sheaf $X$ there is an associated stack $\Aff_{/X}\to \Aff$ and for each morphism $f: X\to Y$ a morphism \[\Aff_{/X}\to \Aff_{/Y}: (S\to X)\mto (S\to X\stackrel{f}{\to} Y).\] Moreover, every morphism of stacks \[\Aff_{/X}\to \Aff_{/Y}\] is uniquely isomorphic to one of this form. In other words, when considering sheaves as the more general objects stacks, via $X\mto \Aff_{/S}$, we do not lose or gain anything especially important and so when considering a sheaf as a stack we shall continue to denote it by $X$.

\subsection{} \label{def:coarse-sheaf} Finally, if $\mc{X}\to \Aff$ is a stack and $S\in \Aff$ we write $\mc{X}(S)/\sim$ for the set of isomorphism classes of objects of $\mc{X}(S)$. The pull-back maps $f^*: \mc{X}(S)\to \mc{X}(S')$ for $f: S'\to S$ define maps $\mc{X}(S)/\sim\to \mc{X}(S')/\sim$ and this defines a (separated) pre-sheaf \[\Aff^{\circ}\to \Set: S\mto \mc{X}(S)/\sim\] whose sheafification $C(\mc{X})=X$ (which will be denoted by the print form of the cursive letter denoting the stack) is called the coarse sheaf associated to $\mc{X}$. There is an induced morphism $c_{\mc{X}}: \mc{X}\to X$ and for each sheaf $Y$ and each morphism $f:\mc{X}\to Y$ there is a unique morphism $f': X\to Y$ such that $f'\circ c_{\mc{X}}=f$.

\section{Dedekind domains} Here we fix some general notation for Dedekind domains and various related objects.

\subsection{} Let $O$ be a Dedekind domain with field of fractions $K$ and finite residue fields. An integral ideal of $O$ is any non-zero ideal and a prime ideal will mean a non-trivial prime ideal. We write $\Id_{O}$ (resp. $\Prin_O$) for the monoid of integral (resp. and principal) ideals of $O$ and $\Id_K$ (resp. $\Prin_K$) for the group of fractional (resp. principal) ideals of $O$. If $\a, \b\in \Id_{O}$ we write $(\a, \b)=\a+\b$ so that $\a$ and $\b$ are relatively prime if and only if $(\a, \b)=O$. We write $\CL_{O}$ for the class group of $O$, i.e.\ the group of isomorphism classes of rank one projective $O$-modules.

Let $\f$ be an integral ideal. We write $\Id_O^{(\f)}\subset \Id_O$ (resp.\ $\Id_{K}^{(\f)}\subset \Id_K$) for the sub-monoid (resp.\ sub-group generated) by the prime ideals prime to $\f$. If $a\in K^\times$ then we say $a=1\bmod \f$ to mean that the ideal $(a-1)$ can be written as $\f\a\b^{-1}$ where $\a, \b\subset O$ are ideals (if $a=1$ we allow $\a=(0)$) with $\b\neq (0)$ and $(\b, \f)=O$. If $\f|\g$ is another ideal then we write $\Prin_{1\bmod \f}^{(\g)}$ to denote the group of principal fractional ideals $\a=(a)$ prime to $\g$ with $a=1\bmod \f$.

We write $N\f$ for the cardinality of $O/\f$ and if $\f=\p$ is prime we write $\F_\p=O/\p$. If $A$ is an $\F_\p$-algebra then we write $\Fr^{N\p}: A\to A$ for the $N\p$-power Frobenius endomorphism.

We write $O^{\times, \f}$ for the kernel of the homomorphism $O^{\times}\to (O/\f)^\times$ and we say that $\f$ separates units if this homomorphism is injective (this is not often the case, but will be used constantly in the text).

We write $O[\f^{-1}]$ for the sub-$O$-algebra of $K$ generated by the elements of $\f^{-1}\subset K$. If $X\to\Spec(O)$ is any $\Spec(O)$-sheaf then we write $X[\f^{-1}]=X\times_{\Spec(O)}\Spec(O[\f^{-1}])$. We say that $\f$ is invertible on $X$ if $X[\f^{-1}]=X$. If $\p$ is a prime ideal we say that $X$ has characteristic $\p$ if the structure map $X\to \Spec(O)$ factors through $\Spec(\F_\p)\to \Spec(O)$.

If $\p$ is a prime we write \[\Spf_\p(O)=\colim_{n\geq 0} \Spec(O/\p^{n+1})\subset \Spec(O).\] We say that $X$ is $\p$-adic, or that $\p$ is nilpotent on $S$, if the structure morphism $X\to \Spec(O)$ factors through $\Spf_\p(O)\subset \Spec(O)$. For an affine $\Spec(O)$-scheme $\Spec(A)$ to be $\p$-adic is equivalent to the ideal $\p A$ being nilpotent.
\newpage\null\thispagestyle{empty}\newpage

\chapter{Local and global reciprocity}
In \S\S 1.1 and 1.3 of this chapter we recall the local and global reciprocity maps of class field theory. In \S 1.2 we define Lubin--Tate $O$-modules (in families) for a non-archimedian local field $K$ with ring of integers $O$ and show that the moduli stack $\mc{M}_\LT$ of Lubin--Tate $O$-modules is a torsor under the stack $\mc{CL}_O$ of rank one $O$-local systems (\ref{prop:lubin-tate-torsor}). We show how using this structure one can derive the local reciprocity map directly from the stack $\mc{M}_\LT$ using only the formal properties of Lubin--Tate $O$-modules (\ref{theo:main-theorem-lubin-tate}). The construction we give is an analogue, for non-archimedian local fields, of the derivation of the global reciprocity map for imaginary quadratic fields using CM elliptic curves (which we address in Chapter 2). In \S 1.4 we recall how, for global fields $K$ equipped with a certain special place $\infty$ and associated ring of integers $O_K$, the reciprocity map associated to the maximal abelian extension of $K$ which is totally split at $\infty$ can be reinterpreted in terms of certain class groups associated to $O_K$ (such pairs $(K, \infty)$ were first considered by Drinfel'd \cite{Drinfeld74}). Finally, in \S 1.5 we define and give the basic properties of rank one $O_K$-local systems, their level-$\f$ structures and moduli for use in the later chapters.

The only new results (or perhaps observations) in this chapter are the $\mc{CL}_O$-torsor structure on $\mc{M}_\LT$ (which in any case we prove using results of Faltings) and the derivation of the local reciprocity map directly from $\mc{M}_\LT$.

\section{The local reciprocity map} The purpose of this section is to recall the basic properties of local fields and the local reciprocity map. Everything here is contained in Chapters I and VI of \cite{CasselsFrohlich67}).

\subsection{} Let $K$ be a local field. There are two cases and for each we fix the following notation:
\begin{enumerate}[label=\textup{(\roman*)}]
\item $K$ is archimedian and is isomorphic to $\R$ or $\C$. We write $|-|_K$ for the usual absolute value if $K\isomto \R$ and the square of the usual absolute value if $K\isomto \C$.
\item $K$ is non-archimedian and is the field of fractions of a discrete valuation ring $O\subset K$ with maximal ideal $\p$ and finite residue field $\F_\p$. We equip it with the absolute value $|a|_K=N \p^{-v_\p(a)}$ where $a\cdot O=\p^{v_\p(a)}\subset K$).
\end{enumerate} We also fix maximal abelian and separable extensions $K\subset K^\ab\subset K^\sep$. If $K$ is non-archimedian these extensions are not complete but from the point of view of Galois theory nothing is lost.
\subsection{} Let $K$ be a non-archimedian local field and let $K\subset L\subset K^\sep$ be a (not necessarily finite) Galois extension of $K$. Write $O_L\subset L$ for the ring of integers of $L$, $\p_L\subset K$ for its unique maximal ideal and $\F_{\p_L}=O_L/\p_L$ for its residue field. The reduction homomorphism $G(L/K)\to G(\F_{\p_L}/\F_\p)$ is surjective and is bijective whenever $L/K$ is unramified. In this case we write $\sigma_{L/K}\in G(L/K)$ for the unique element lifting the $N\p$-power Frobenius automorphism of $\F_{\p_L}$.
We denote by $K\subset K^\ur\subset K^\sep$ the maximal unramified extension of $K$. The map \[\widehat{\Z}\to G(K^\ur/K): n\mto \sigma_{K^\ur/K}^{n}\] is a continuous isomorphism whose inverse is denoted by $v_K: G(K^\ur/K)\to \widehat{\Z}$. If $L/K^\ur$ is any extension we also write $v_K: G(L/K)\to \widehat{\Z}$ for the map $\sigma\mto v_K(\sigma|_{K^\ur})$ and define $W(L/K):=v_K^{-1}(\Z)\subset G(L/K)$ (this is the `Weil group').
\begin{theo}\label{theo:archimedian-local-reciprocity} There is a unique isomorphism \begin{equation} K^\times \to W(K^\ab/K): a\mto (a, K^\ab/K)\label{eqn:local-rec-map}\end{equation} such that
\begin{enumerate}[label=\textup{(\roman*)}]
\item for all finite extensions $K\subset L\subset K^\ab$ , the kernel of the composition $a\mto (a, K^\ab/K)|_{L}$ is $N_{L/K}(L^\times)$ and the induced map \[K^\times/N_{L/K}(L^\times)\to G(L/K)\] is an isomorphism, and
\item the diagram \[\xymatrix{K^\times\ar[d]_{v_\p}\ar[rr]^-{(-, K^\ab/K)} && \ar[d] W(K^\ab/K)\\
\Z \ar[rr]^-{n\mto \sigma_{K^\ur/K}^n} && W(K^\ur/K)}\] commutes.
\end{enumerate}
\end{theo}
\begin{proof} Uniqueness and existence are Proposition 6, \S 2.8 of Chapter VI and Theorem 2, \S 2.2 Chapter VI of \cite{CasselsFrohlich67} respectively.
\end{proof}

\subsection{}\label{rema:local-rec-properties} We list the following further property of the reciprocity homomorphism (see \S 2 Chapter VI of \cite{CasselsFrohlich67}): If $K'/K$ is any finite separable extension and $K'^\ab/K'$ a maximal abelian extension of $K'$, then the diagram \[\xymatrix{K'^\times\ar[rr]^-{(-, K'^\ab/K')}\ar[d]_{N_{K'/K}} && W(K'^\ab/K')\ar[d]\\
K^\times\ar[rr]^-{(-, K^\ab/K)} && W(K^\ab/K)}\] commutes.

\subsection{} We also recall the reciprocity map $(-, K^\ab/K)$ associated to an archimedian local field $K$. It is the unique continuous homomorphism \[(-, K^\ab/K):K^\times\to G(K^\ab/K)\] sending $-1$ to the unique generator of $G(K^\ab/K)$ which is either trivial (if $K\isomto \C$) or cyclic of order two (if $K\isomto \R$).
\section{The local reciprocity map and Lubin--Tate $O$-modules}\label{sec:lubin-tate-modules} We now define rank one $O$-local systems over $\p$-adic sheaves and their moduli stack $\mc{CL}_O$. We then define and study families of Lubin--Tate $O$-modules over $\p$-adic sheaves and show that $\mc{CL}_O$ acts on the moduli stack $\mc{M}_\LT$ of Lubin--Tate $O$-modules, making it a torsor under $\mc{CL}_O$ (see (\ref{prop:lubin-tate-torsor}) and (\ref{exem:lubin-tate-modules-exist})). Finally, we explain in (\ref{subsec:lubin-tate-reciprocity}) how this action, combined with the basic properties of Lubin--Tate $O$-modules, allows one to construct the reciprocity map of (\ref{theo:archimedian-local-reciprocity}).

\subsection{} We shall be working with the category of $\p$-adic sheaves, i.e.\ sheaves $S\to \Spf(O)=\colim_n \Spec(O/\p^{n+1})$. If $S\to \Spf(O)$ is a $\p$-adic sheaf then we write $S_n=S\times_{\Spf(O)}\Spec(O/\p^{n+1})$ so that $S=\colim_{n\geq 0} S_n$. Unless otherwise stated $S$ denotes a $\p$-adic sheaf.

\subsection{}\label{sub-sec-p-adic-formal} For each $\p$-adic sheaf $S$ we write $\widehat{O}_S$ for the pro-constant sheaf of rings \[\widehat{O}_S:=\lim_n \underline{O/\p^{n+1}}_{S}.\] If $L$ is any finitely generated $O$-module then we write $\widehat{L}_S$ for the pro-constant sheaf of $\widehat{O}_S$-modules \[\widehat{L}_S:=\lim_n \underline{L/\p^{n+1} L}_{S}.\] If $F$ and $G$ are two $\widehat{O}_S$-modules over $S$ we write $F\otimes_{O} G$ for the $\widehat{O}_S$-module $F\otimes_{\widehat{O}_S} G$ and $\underline{\Hom}_S^{O}(F, G)$ for the sheaf of $\widehat{O}_S$-homomorphisms $F\to G$. Moreover, if $G=\widehat{L}_S$ for some finite rank projective $O$-module $L$ we shall just write $F\otimes_{O} L$ for $F\otimes_{\widehat{O}_S} \widehat{L}_S$.

\subsection{} A rank one $O$-local system over a $\p$-adic sheaf $S$ is a sheaf $\mc{L}$ of $\widehat{O}_{S}$-modules with the property that there exists a cover $(S_i\to S)_{i\in I}$ and rank one projective $O$-modules $(L_i)_{i\in I}$ such that $\mc{L}\times_{S}S_i\isomto \underline{\widehat{L}_i}_{S_i}$. We denote by $\mc{CL}_{O}$ the moduli stack of rank one $O$-local systems over $\Sh_{\Spf(O)}$. We list the following (usual) constructions and properties of $O$-local systems:
\begin{enumerate}[label=(\roman*)]
\item The tensor product $\mc{L}\otimes_{O}\mc{L}'$ of two rank one $O$-local systems $\mc{L}$ and $\mc{L}'$ (in the category of $\widehat{O}_S$-modules) is again a rank one $O$-local system.
\item The sheaf of $\widehat{O}_S$-homomorphisms $\underline{\Hom}_{S}^{O}(\mc{L}, \mc{L}')$ is again a rank one $O$-local system and defining $\mc{L}^\vee:=\underline{\Hom}_S^{O}(\mc{L},\widehat{O}_S)$ we have $\underline{\Hom}_{S}^{O}(\mc{L}, \mc{L}')\isomto \mc{L}'\otimes_{O}\mc{L}^\vee$.
\item The sheaf of automorphisms $\underline{\Aut}_S^{O}(\mc{L})$ of a rank one $\widehat{O}_S$-local system $\mc{L}$ is isomorphic to $\widehat{O}_S^\times:=\lim_{n}\underline{(O/\p^{n+1})^\times}_S$.
\item The sheaf of $\widehat{O}_S$-isomorphisms $\underline{\Isom}_S(\mc{L}, \mc{L}')$ is pro-finite and \'etale over $S$ and is an $\widehat{O}^\times_S$-torsor over $S$. We denote by $\rho_{\mc{L}/S}\in H^1(S, \widehat{O}^\times_S)$ the class of the torsor $\underline{\Isom}_S(\mc{L}, \widehat{O}_S)$ so that the map \[\mc{L}\mto \rho_{\mc{L}/S}\in H^1(S, \widehat{O}^\times_S)\] defines a bijection between isomorphism classes of rank one $O$-local systems over $S$ and $H^1(S, \widehat{O}^\times_S)$.
\end{enumerate}

\subsection{} Let $\widehat{O}_S\to \mc{O}_{S}$ be the unique homomorphism whose pull-back to $S_n=S\times_{\Spf(O)}\Spec(O/\p^{n+1})$ is \[\widehat{O}_{S_n}\to \underline{O/\p^{n+1}}_{S_n}\to \mc{O}_{S_n}\] where the second map is induced by the structure map $S_n\to \Spec(O/\p^{n+1})$.

A strict formal $O$-module over $S$ is a formal group $F$ over $S$ equipped with the structure of a $\widehat{O}_S$-module which is strict with respect to the homomorphism $\widehat{O}_S\to \mc{O}_S$ (cf.\ (\ref{subsec:strict-formal-modules})). Recall that this means that the two actions of $\widehat{O}_S$ on the $\mc{O}_S$-module $\underline{\Lie}_{F/S}$, coming from the action of $\widehat{O}_S$ on $F$ and the homomorphism $\widehat{O}_S\to \mc{O}_S$, coincide.

If $\mc{L}$ is a rank one $O$-local system over $S$ then $\mc{L}$ (as an $\widehat{O}_S$-module) satisfies condition (P) of (\ref{subsec:sheaves-of-rings}) as locally it is isomorphic to $\widehat{O}_S$. So we may apply (\ref{coro:strict-tensor}) to see that if $F$ is a strict formal $O$-module over $S$ and $\mc{L}$ is a rank one $O$-local system over $S$ then $F\otimes_{O} \mc{L}$ is again a strict formal $O$-module over $S$ (of the same dimension as $F$).

Finally, we define the $\p^n$-torsion of a strict formal $O$-module $F$ over $S$ to be the kernel of the homomorphism \[i_{\p^n}: F\to F\otimes_{O}\p^{-n}\] induced by the inclusion $O\to \p^{-n}$.

\begin{prop}\label{prop:formal-o-mod-frob} Let $F$ be a strict formal $O$-module of dimension one over $S$. Then
\begin{enumerate}[label=\textup{(\roman*)}]
\item locally on $S$ there exists an isomorphism $\rho: F\isomto \widehat{\mathbf{A}}^1_S$ with the property that for all $\zeta\in \mu_{N\p-1}\subset O^\times$ we have $\rho\circ [\zeta]_F=\zeta\cdot \rho$,
\item $\colim_n F[\p^n]=F$, and
\item \label{item:frob-is-p-torsion} $\ker(\Fr_{F_0/S_0}^{N\p^n})\subset F_0[\p^n]$.
\end{enumerate}
\end{prop}
\begin{proof} The claims are local on $S$ so that we assume that $S=\Spec(A)$, $\p^r A=(0)$ for some $r\geq 1$ as $S$ is $\p$-adic, and by (\ref{prop:formal-explicit-intrinsic}) we may also assume that $F=\widehat{\mathbf{A}}^1_S$.

(i) This is Lemma 4.1.2 of \cite{Lubin64}.

Before we show (ii) and (iii) let us fix some notation and make some reductions. If $a\in O$ let us write $[a](T)\in A[[T]]$ for the power series defining the multiplication by $a$ map \[a:F= \widehat{\mathbf{A}}^1_S\to F=\widehat{\mathbf{A}}^1_S.\] We also choose a generator $(\pi)=\p$ so that $\ker(\pi^n)=F[\p^n]$ for all $n\geq 0$. Considering the coefficients of the series \[[\pi](T)=c_1 T+c_2 T^2+\cdots \in A[[T]],\] the strictness of the action of $\widehat{O}_S$ shows that $c_1=\pi$ and by (i) we may assume that $[\zeta](T)=\zeta T$ for $\zeta\in \mu_{N\p-1}\subset O^\times$. The relation \[\zeta [\pi](T)=[\zeta]([\pi](T))=[\pi]([\zeta](T))=[\pi](\zeta T)\] for all $\zeta\in \mu_{N\p-1}$ then shows that $c_2=\cdots =c_{N\p-1}=0$. Thus we may assume that \[[\pi](T)=\pi T \bmod T^{N\p}.\] We note that as $\p^r A=(0)$ we have \[[\pi^r](T)=0\bmod T^{N\p}.\]

(ii) It is enough to show that for all $A$-algebras $B$ and $b\in F(\Spec(B))=\widehat{\mathbf{A}}^1_S(\Spec(B))=\mathrm{nil}(B)$ there is an $n$ such that $[\pi^n](b)=0$. However, as $[\pi^r](T)=0\bmod T^{N\p}$ we see that $[\pi^{rm}](b)=0$ whenever $b^{m N\p}=0$ which shows the claim.

(iii) We may assume that $S=S\times_{\Spf(O)} \Spec(O/\p)=S_0$ so that $\p A=0$. Then \[[\pi](T)=c_{N\p} T^{N\p}\bmod T^{N\p+1}\] and hence \[[\pi^n](T)=c_{N\p}^{\frac{N\p^{n}-1}{N\p-1}}T^{N\p^n}\bmod T^{N\p^{n}+1}.\] Thus if $\Spec(B)$ is an affine $\p$-adic $\Spec(A)$-scheme and if $b\in \mathrm{nil}(B)$ satisfies $\Fr_B^{N\p^n}(b)=b^{N\p^n}=0$ we have \[[\pi^n](b)=0\bmod b^{N\p^n}=0\] and therefore $\ker(\Fr_{F/S}^{N\p^n})\subset F[\p^n]$ for all $n\geq 1$.
\end{proof}

\subsection{}\label{subsec:definition-lubin-tate-module} A Lubin--Tate $O$-module over $S$ is a strict formal $O$-module $F$ over $S$ of dimension one (cf.\ (\ref{subsec:dimension-formal-groups})) with the property that the homomorphism $i_\p: F\to F\otimes_{O}\p^{-1}$ is finite locally free of degree $N\p$ (in particular it is affine and faithfully flat). A morphism of Lubin--Tate $O$-modules is just a homomorphism of the underlying $\widehat{O}_S$-modules. We denote by $\mc{M}_{LT}$ the moduli stack of Lubin--Tate $O$-modules over $\Sh_{\Spf(O)}$.

\begin{prop} If $F/S$ is a Lubin--Tate $O$-module and $\mc{L}/S$ is an $O$-local system then $F\otimes_{O}\mc{L}/S$ is a Lubin--Tate $O$-module.
\end{prop}
\begin{proof} We have that $F\otimes_{O}\mc{L}$ is a strict $O$-module of dimension one by (\ref{coro:strict-tensor}). The homomorphism \[i_\p: F\otimes_{O}\mc{L}\to (F\otimes_{O}\mc{L})\otimes_{O}\p^{-1}\] is finite locally free of degree $N\p$ as this can be checked locally, e.g.\ when $\mc{L}\isomto \widehat{O}_S$, and in this case it is obvious.
\end{proof}

\begin{coro}\label{prop:cm-tensor-homs-lt} For each pair $\mc{L}, \mc{L}'$ of rank one $O$-local systems over $S$ and each pair $F, F'$ of Lubin--Tate $O$-modules over $S$ the natural map \[\underline{\Hom}_S^{O}(F, F')\otimes_{O}\underline{\Hom}_S^{O_K}(\mc{L}, \mc{L}')\to \underline{\Hom}_S^{O}(F\otimes_{O_K}\mc{L}, F'\otimes_{O_K}\mc{L}')\] is an isomorphism.
\end{coro}
\begin{proof} This follows from (\ref{prop:tensor-properties-lie-hom}).
\end{proof}

\begin{lemm}\label{lemm:p-torsion-lifts-frob-in-lubin-tate} Let $S$ be a sheaf of characteristic $\p$, $n\geq 0$ and $F$ be a Lubin--Tate $O$-module over $S$. Then $F[\p^n]=\ker(\Fr^{N\p^n}_{F/S})$.
\end{lemm}
\begin{proof} We may work locally on $S$ and so assume that $S=\Spec(A)$ and that $F =\widehat{\mathbf{A}}^1_S$. By (iii) of (\ref{prop:formal-o-mod-frob}) we have $\ker(\Fr^{N\p^n}_{F/S})\subset F[\p^n]$. As $F = \widehat{\mathbf{A}}^1_S$, it follows that $\ker(\Fr^{N\p^n}_{F/S})$ is finite locally free of rank $N\p^n$ over $S$. As $F[\p^n]$ is also finite locally free of rank $N\p^n$ the closed immersion $\ker(\Fr^{N\p^n}_{F/S})\subset F[\p^n]$ must be an isomorphism.
\end{proof}

\begin{coro}\label{coro:tensor-p-equals-frobenius-lubin-tate} For each $n\geq 0$ there is a unique isomorphism of functors \[\nu_{\p^n}: -\otimes_{O}\p^{-n} \isomto \Fr^{N\p^n*}(-)\] on $\mc{M}_{LT}\times_{\Spf(O)}\Spec(\F_\p)$ such that for all Lubin--Tate $O$-modules $F$ over characteristic $\p$-sheaves $S$ the diagram \[\xymatrix{&F\ar[dr]^{\Fr_{F/S}^{N\p^n}}\ar[dl]_-{i_{\p^n}}&\\
F\otimes_{O}\p^{-n} \ar[rr]^{\nu_{\p^n}}_\sim && \Fr^{N\p*}(F)}\] commutes.
\end{coro}
\begin{proof} For any Lubin--Tate $O$-module $F/S$ the two homomorphisms $i_{\p^n}$ and $\Fr_{F/S}^{N\p}$ are epimorphisms with the same kernel (\ref{lemm:p-torsion-lifts-frob-in-lubin-tate}) and the claim follows.
\end{proof}

\begin{rema} The result (\ref{coro:tensor-p-equals-frobenius-lubin-tate}) can be read as saying that the moduli stack $\mc{M}_\LT$ admits an endomorphism $-\otimes_{O}\p^{-1}: \mc{M}_\LT\to \mc{M}_\LT$ which (upto canonical isomorphism) lifts the $N\p$-power Frobenius endomorphism. This kind of structure is very closely related to the notions of $\Lambda$-structures, Witt vectors and arithmetic jets (due to Borger and Buium) which we define and study in Chapter 3. It is essentially the topic of Chapter 4 to study and exploit this relationship in the context of CM elliptic curves (for which we prove an analogue of (\ref{coro:tensor-p-equals-frobenius-lubin-tate}) in Chapter 2, see (\ref{coro:tensor-p-equals-frobenius})).
\end{rema}

\begin{prop}\label{prop:lubin-tate-faltings} \begin{enumerate}[label=\textup{(\roman*)}] 
\item If $F$ is a Lubin--Tate $O$-module over $S$ the natural homomorphism \[\widehat{O}_S\to \underline{\End}_S^O(F)\] is an isomorphism.
\item If $F, F'$ are a pair of Lubin--Tate $O$-modules over $S$ then $\underline{\Hom}_S^{O}(F, F')$ is an $O$-local system over $S$ and the evaluation homomorphism \[F\otimes_{O}\underline{\Hom}_S^{O}(F, F')\to F'\] is an isomorphism.
\end{enumerate}
\end{prop}
\begin{proof} The proof of these statements is an application of Faltings' generalised Cartier duality \cite{Faltings02}. It would take us too far afield to give the proof here and so we defer it to the appendix (see (i) and (ii) of (\ref{coro:lubin-tate-faltings})).
\end{proof}

\begin{prop}\label{prop:lubin-tate-torsor} The functor \[\mc{M}_{LT}\times \mc{CL}_{O}\to \mc{M}_{LT}\times \mc{M}_{LT}: (F,\mc{L})\mto (F, F\otimes_{O}\mc{L})\] is an equivalence of stacks.
\end{prop}
\begin{proof} As with (\ref{prop:lubin-tate-faltings}), the proof of this statement will be given in the appendix (see (iii) of (\ref{coro:lubin-tate-faltings})).
\end{proof}

\begin{rema} It would be preferable to have elementary proofs of  (\ref{prop:lubin-tate-faltings}) and (\ref{prop:lubin-tate-torsor}) which do not rely on the machinery of Faltings' generalised Cartier duality.
\end{rema}

\begin{exem}\label{exem:lubin-tate-modules-exist} Let us now at least tell the reader that there do exist Lubin--Tate $O$-modules. First, if $O=\Z_p$ then the $p$-power torsion in $\mathbf{G}_{\mathrm{m}/\Spf(O)}$: \[\mu_{p^\infty}=\colim_n \mu_{p^n}\subset \mathbf{G}_{\mathrm{m}/\Spf(O)}\] is a Lubin-Tate $\Z_p$-module and it also has the property that the multiplication by $p$ map $p: \mu_{p^\infty}\to \mu_{p^\infty}$ reduces to the $p$-power Frobenius after base change along $\Spec(\F_p)\to \Spf(\Z_p)$.

In fact, for any $O$, given a generator $\pi$ of the prime ideal $\p$, there is a unique (upto isomorphism) Lubin--Tate $O$-module $F_\pi$ over $\Spf(O)$ with the property that the endomorphism $\pi: F_\pi\to F_\pi$ reduces to the $N\p$-power Frobenius map after base change along $\Spec(\F_\p)\to \Spf(O)$ (see \S 3.5 Chapter VI \cite{CasselsFrohlich67}). Moreover, there exists a unique isomorphism $F_\pi\isomto \widehat{\mathbf{A}}^1_{\Spf(O)}$ under which the multiplication by $\pi$ is represented by the power series $[\pi](T)=\pi T+T^{N\p}$.

A consequence of this is that the morphism $\mc{M}_\LT\to \Spf(O)$ admits a section. The statement of (\ref{prop:lubin-tate-torsor}) above can now be interpreted as saying that $\mc{M}_\LT$ is (in a stack theoretic sense) a torsor under the (stack theoretic) group $\mc{CL}_{O}$ albeit a trivial one:

\end{exem}
\begin{coro}\label{coro:lubin-tate-moduli-trivial} Fixing a Lubin--Tate $O$-module $F$ over $\Spf(O)$ the functor \[\mc{CL}_O\to \mc{M}_\LT: \mc{L}/S\mto F_S\otimes_{O}\mc{L}/S\] is an equivalence of stacks.
\end{coro}
\begin{coro}\label{coro:lift-lubin-tate} Let $S_0\to S$ be a nilpotent immersion of $\p$-adic sheaves. The functor \[\mc{M}_{LT}(S)\to \mc{M}_{LT}(S_0): F\mto F_{S_0}\] is an equivalence of categories.
\end{coro}
\begin{proof} This follows from (\ref{coro:lubin-tate-moduli-trivial}) and the corresponding obvious claim for $\mc{CL}_{O}$ (which is a moduli space of pro-\'etale objects).
\end{proof}

\begin{coro}\label{prop:lubin-tate-hom} Let $f: F\to F'$ be a homomorphism of Lubin--Tate $O$-modules over $S$. Then there is a unique decomposition $S=S_{(0)}\amalg_{0\leq n<\infty} S_{\p^n}$ such that $f_{S_{(0)}}$ is the zero map and such that $\ker(f_{S_{\p^n}})=F_{S_{\p^n}}[\p^n]$ for $0\leq n <\infty$.
\end{coro}
\begin{proof} We have $F'\isomto F\otimes_{O}\mc{L}$ for some rank one $O$-local system $\mc{L}$ by (\ref{prop:lubin-tate-torsor}) and the homomorphism \[f: F\to F\otimes_{O} \mc{L}\] is of the form $\id_F\otimes_{O} h$ for some homomorphism \[h: \widehat{O}_S\to \mc{L}\] by (i) of (\ref{prop:tensor-properties-lie-hom}).

Define the sub-sheaves $S_{(0)}\subset S$ (resp. $S_{\p^n}\subset S$ for $0\leq n<\infty$) by the property that $h_{S_{(0)}}$ is the zero map (resp. $h_{S_{\p^n}}$ factors as \[\widehat{O}_{S_{\p^n}} \isomto \p^n\otimes _{O}\mc{L}_{S_{\p^n}}\to \mc{L}_{S_{\p^n}}\] where the second map is multiplication). These definitions combined with $f=\id_F\otimes_{O} h$ show that $f_{S_{(0)}}$ is the zero map that $f_{S_{\p^n}}=\id_{F_{S_{\p^n}}}\otimes_{O} h_{S_{\p^n}}$ factors as \[F_{S_{\p^n}}\isomto F'_{S_{\p^n}}\otimes_{O}\p^n\to F'_{S_{\p^n}}\] where the second map is multiplication and hence $\ker(f_{S_n})=F_{S_n}[\p^n]$. Moreover, it is clear that $S_{(0)}$ and all the $S_{\p^n}$ for $0\leq n<\infty$ are disjoint and so to prove our claim we need to show that $S_{(0)} \amalg_{0\leq n< \infty} S_{\p^n}\to S$ is an epimorphism.

For this we may localise $S$ and assume that $\mc{L}=\widehat{O}_S$ and $h=a\in O\subset \widehat{O}_S(S)$. Then either $a=0$, in which case $h$ is the zero map so that $S_{(0)}\isomto S$, or $a\neq 0$, in which case $a\cdot O=\p^n$ for some integer $n\geq 0$ and $h=a$ factors as \[\widehat{O}_S\isomto \widehat{O}_S\otimes_{O}\p^n\to \widehat{O}_S,\] so that $S_{\p^n}\isomto S$. It follows that $S_{(0)}\amalg_{0\leq i<\infty}S_{\p^n}=S$.
\end{proof}

\subsection{}\label{subsec:lubin-tate-reciprocity} Classically one relates the reciprocity map of the local field $K$ to Lubin--Tate $O$-modules as follows. Write $S=\Spf(O_{K^\sep})$ and let $F$ be the unique Lubin--Tate $O$-module over $\Spf(O)$ such that $\pi: F\to F$ reduces to the $N\p$-power Frobenius after base change to $\Spec(\F_\p)$. Then for each $r\geq 0$, the $O/\p^r$-module $F[\p^r](S)$ is free of rank one and $\colim_r (F[\p^r](S))\isomto K/O$ (non-canonically). We then obtain the character \[\rho_\pi: G(K^\sep/K)\to \lim_r \Aut_{O}(F[\p^r](S))=\lim_r (O/\p^r)^\times=O^\times\] defining the action of $G(K^\sep/K)$ on $\colim_r (F[\p^r](S))$. The relationship between the character $\rho_\pi$ and the reciprocity map (\ref{eqn:local-rec-map}) is that for all $\sigma\in W(K^\sep/K)$ we have \begin{equation} \sigma|_{K^\ab}=(\pi^{v_K(\sigma)}\rho_\pi(\sigma)^{-1}, K^\ab/K)\label{eqn:lub-rec-classical}\end{equation} (see \S 3.7 Chapter VI of \cite{CasselsFrohlich67}).

We would now like to show how one can construct the reciprocity map of local class field theory in a slightly more abstract but direct way using only the Frobenius lift property (\ref{coro:tensor-p-equals-frobenius-lubin-tate}) of Lubin--Tate $O$-modules and the $\mc{CL}_{O_K}$-torsor structure of $\mc{M}_\LT$. Write $\overline{S}=\Spec(\F_\p^\sep)\subset S=\Spf(O_{K^\sep})$ and for a Lubin--Tate $O$-module $F\to S$ write $\overline{F}=F\times_{S} \overline{S}$ and $F_r=F[\p^r]$ for each $r\geq 0$.

\begin{prop}\label{theo:main-theorem-lubin-tate} Let $\sigma\in W(K^\sep/K)$ satisfy $v_K(\sigma)=n\geq 0$. Then \begin{enumerate}[label=\textup{(\roman*)}]
\item for each Lubin--Tate $O$-module $F$ over $S$ there is a unique isomorphism $\nu_\sigma:F\otimes_{O}\p^{-n}\isomto \sigma^*(F)$ whose pull-back along $\overline{S}\to S$ is the isomorphism $\nu_{\p^{n}}:\overline{F}\otimes_{O}\p^{-n}\isomto \Fr^{N\p^n*}(\overline{F})$ of \textup{(\ref{coro:tensor-p-equals-frobenius-lubin-tate})},
\item there is a generator $\chi_K(\sigma)\in K^\times$ of $\p^n$, independent of $F$, such that for all $r\geq 0$, the isomorphism induced by $\nu_\sigma$ on the $S$-points of the $\p^r$-torsion of $F$ \[F_r(S)\otimes_{O}\p^{-n}\isomto \sigma^*(F_r)(S)\] is equal to \[F_r(S)\otimes_{O}\p^{-n}\stackrel{F_r(\sigma)\otimes \chi_K(\sigma)} {\longrightarrow}F_r(\sigma_!(S))=\sigma^*(F_r)(S),\]
\item if $\tau\in W(K^\sep/K)$ also satisfies $v_K(\tau)\geq 0$ then in the notation of \textup{(ii)} we have $\chi_K(\sigma\tau)=\chi_K(\sigma)\chi_K(\tau)$, and
\item $\sigma|_{K^\ab}=(\chi_K(\sigma), K^\ab/K)$.
\end{enumerate}
\end{prop}
\begin{proof} (i) The existence of $\nu_{\p^n}$ follows from (\ref{coro:lift-lubin-tate}) applied to the nilpotent immersion $S_\p\to S$ and the fact that the restriction of $\sigma$ to $\overline{S}$ is equal to $\Fr^{N\p^n}$.

(ii) For all $r\geq 0$ we have that $F_r(S)$ is a free rank one $O_K/\p^r$-module and so writing $\nu_{\sigma, r}$ for the restriction of $\nu_{\sigma}$ to the $S$-valued points of the $\p^r$-torsion, we set \begin{eqnarray*} \chi_{K, r}(\sigma)=F_r(\sigma)^{-1}\circ \nu_{\sigma, r} &\in& \Isom_O(F_r(S)\otimes_O \p^{-n}, F_r(S))\\&=&\Isom_O(F_r(S), F_r(S)\otimes_O \p^n)\\
&\subset& \p^n\otimes_{O} O/\p^r.\end{eqnarray*} Then $\chi_K(\sigma)$ is given by the limit $\lim_{r} \chi_{K, r}(\sigma)\in \lim_r \p^n \otimes_{O}O/\p^r=\p^n$. It is a generator of $\p^n$ as $\chi_{K, r}(\sigma)$ is a generator of the free rank one $O/\p^r$-module $O/\p^r\otimes_O \p^n$ for all $r\geq 0$.

First, it is clear that $\chi_K(\sigma)$ depends only on the isomorphism class of $F$. However, by (\ref{prop:lubin-tate-torsor}), every other Lubin--Tate $O$-module over $S$ is of the form $F\otimes_{O}\mc{L}$ for some rank one $O$-local system $\mc{L}$. But every rank one $O$-local system $\mc{L}$ over $\Spf(O_{K^\sep})$ is pro-constant and isomorphic to $\widehat{O}_S$. Therefore, all Lubin--Tate $O$-modules over $S$ are isomorphic and $\chi_K(\sigma)$ is independent of the Lubin--Tate $O$-module $F$ over $S$ (admittedly there is only one!).

(iii) Write $m=v_K(\tau)$. Then the two isomorphisms \[\nu_{\sigma\tau} \quad \text{ and } \quad \tau^*(\nu_\sigma)\circ (\nu_\tau\otimes_{O} \p^{-m})\] between \[F\otimes_{O}\p^{-n-m} \isomto (\sigma\circ \tau)^*(F)=\tau^*(\sigma^*(F)))\] both pull-back to the isomorphism $\nu_{\p^{n+m}}$ of (\ref{coro:tensor-p-equals-frobenius-lubin-tate}) along $\overline{S}\to S$. By the uniqueness in (i) we get \[\nu_{\sigma\tau} = \tau^*(\nu_\sigma)\circ (\nu_\tau\otimes_{O} \p^{-m})\] from which we find the relation $\chi_K(\sigma\tau)=\chi_K(\sigma)\chi_K(\tau)$.

(iv) With notation as in (\ref{subsec:lubin-tate-reciprocity}) take $F= F_{\pi}\times_{\Spf(O)}S$. As $\pi^n: F_\pi\to F_\pi$ lifts the $N\p^n$-power Frobenius endomorphism of $F_\pi\times_{\Spf(O)}\Spec(\F_\p)$ the (unique) isomorphism \[\nu_\sigma: F\otimes_{O}\p^{-n}\isomto \sigma^*(F)\] of (i) is given by \begin{equation}\label{eqn:def-nu-descend}F\otimes_{O}\p^{-n}\stackrel{\pi^n}{\longrightarrow} F\stackrel{d_\sigma}{\longrightarrow} \sigma^*(F)\end{equation} where $d_\sigma: F\isomto \sigma^*(F)$ is the descent isomorphism (coming from the fact that $F=F_\pi\times_{\Spf(O)} S$ is defined over $\Spf(O)$). The isomorphism $d_\sigma$ on the $S$-points of the $\p^r$-torsion is given by \[F_r(S) \stackrel{\rho_\pi(\sigma)^{-1} \cdot F_r(\sigma)}{\longrightarrow} F_{r}(\sigma_!(S))=\sigma^*(F_r)(S)\] so that $\nu_\sigma$ on the $S$-points of the $\p^r$ torsion is given by (cf.\ (\ref{eqn:def-nu-descend})) \[F_r(S)\otimes_{O}\p^{-n}\stackrel{1\otimes \pi^n}{\longrightarrow} F_r(S)\stackrel{\rho_\pi(\sigma)^{-1}\cdot F_r(\sigma)}{\longrightarrow} F_r(\sigma_!(S))=\sigma^*(F_r)(S).\] Therefore $\chi_K(\sigma)=\pi^n\rho_\pi(\sigma)^{-1}$ and by (\ref{eqn:lub-rec-classical}) we get \[\sigma|_{K^\ab}=(\pi^n\rho_\pi(\sigma)^{-1}, K^\ab/K)=(\chi_K(\sigma), K^\ab/K).\]
\end{proof}

\begin{rema}\label{rema:loca-reciprocity} From (i), (ii) and (iii) of (\ref{theo:main-theorem-lubin-tate}) we see that we can associate to any element of $\sigma\in v_K^{-1}(\N_{\geq 0})\subset W(K^\sep/K)$ an element $\chi_K(\sigma)\in K^\times$ and that this association is multiplicative. It therefore extends to a homomorphism \begin{equation}\label{eqn:local-rec}W(K^\sep/K)\to K^\times: \sigma\mto \chi_K(\sigma)\end{equation} and (iv) of (\ref{theo:main-theorem-lubin-tate}) states that this map satisfies \[\sigma|_{K^\ab}=(\chi_K(\sigma), K^\ab/K)\] for all $\sigma\in W(K^\sep/K).$ Thus we have derived the local reciprocity map (\ref{eqn:local-rec-map}) using nothing more than the basic properties of Lubin--Tate $O$-modules (in particular, the Frobenius lifting property (\ref{coro:tensor-p-equals-frobenius-lubin-tate}) and the $\mc{CL}_O$-torsor structure of $\mc{M}_\LT$ (\ref{prop:lubin-tate-torsor})).
\end{rema}

\begin{rema}\label{rema:classification-lubin} If $K\subset L\subset K^\sep$ is a finite extension and $F/\Spf(O_L)$ is a Lubin--Tate $O$-module let us write \[\rho_{F/O_L}: G(K^\sep/L)\to O^\times\] for the (continuous) character defining the action of $G$ on \[\colim_{r} (F[\p^r](\Spf(O_{K^\sep})))\isomto K/O.\] Then a continuous character $\rho: G(K^\sep/L)\to O^\times$ is of the form $\rho_{F/O_L}$ if and only if the diagram \[\xymatrix{W(K^\sep/L)\ar[dr]_{\chi_K}\ar[r]^-{\rho_{F/O_L}^{-1}} & O^\times\ar[d]\\
& K^\times}\] commutes where the right vertical map is the inclusion. We mention this mainly as it is analogous to the classification of elliptic curves with complex multiplication over fields in terms of their associated characters we will give in Chapter 2 (see (iii) of  (\ref{subsec:reinterpret-character-idelic})).
\end{rema}

\begin{rema} Of course, the theory of Lubin--Tate $O$-modules and the local reciprocity map are themselves not particularly complicated and one can derive the reciprocity map in any number of ways. We believe the this approach above has some advantages over the classical one, first and foremost it is choice free, and secondly one gets the whole of the local reciprocity map right of the bat, rather than first finding a character \[\rho_\pi: G(K^\sep/K)\to O^\times\] which one then restricts to $W(K^\sep/K)\subset G(K^\sep/K)$, takes the reciprocal of and then multiplies by the character \[W(K^\sep/K)\to K^\times: \sigma\mto \pi^{v_K(\sigma)}.\] The derivation of the local reciprocity map we have given is also analogous to the derivation (\ref{prop:compute-the-homomorphisms-h}) of the global reciprocity map for imaginary quadratic fields which we will give using CM elliptic curves in Chapter 2. Moreover, in the case of CM elliptic curves and imaginary quadratic fields, the situation is somewhat more delicate --- one cannot just find CM elliptic curves with suitable properties from which one can construct the global reciprocity map in the same way one can with Lubin--Tate $O$-modules.
\end{rema}

\section{The global reciprocity map} In this section we recall the basic objects required to define, and then we recall, the global reciprocity map (\ref{prop:global-reciprocity}) associated to a global field $K$.

\subsection{}\label{subsec:finite-extensions-frobenius} Let $K$ be a global field and fix maximal abelian and separable extensions $K\subset K^\ab\subset K^\sep$. Write $\mc{P}_K$ for the set of places of $K$, and as usual, if $K$ is a number field we identify the non-archimedian places $v\in \mc{P}_K$ with the prime ideals $\p \subset O_K$ of the ring of integers of $K$. For $v\in K$ we write $K_v$ for the completion of $K$ with respect to $v$ and if $v$ is archimedian $O_{K_v}\subset K_v$ for the ring of local integers. We write $\mc{P}_K^{\arch}\subset \mc{P}_K$ for the subset of archimedian places (which is of course empty if $K$ is a function field).

Let $K\subset L\subset K^\sep$ be a finite Galois extension. If $w\in \mc{P}_L$ is a non-archimedian place and the extension $L/K$ is unramified at $w$ then we write $\sigma_{L/K, w}\in G(L/K)$, or just $\sigma_w$, for the Frobenius element associated to $w$. If $L/K$ is abelian (so that $\sigma_{L/K, w}$ depends only on $w|_K=v\in \mc{P}_K$) then we write $\sigma_{L/K, v}$, or again just $\sigma_v$, for $\sigma_{L/K, w}$.

\subsection{} For each finite set $S\subset \mc{P}_K$ containing all archimedian places of $K$ we write $I_{K, S}$ for the topological group \[I_{K, S}= \prod_{v\in \mc{P}_K-S} O_{K_v}^\times\times \prod_{v\in S} K_v^\times\] (the topology being the product topology). For $S\subset S'\subset \mc{P}_K$ the inclusions $I_{K, S}\subset I_{K, S'}$ are open and the group of id\`eles of $K$ is the topological group \[I_K=\colim_{S} I_{K, S}\] (the topology being the colimit topology).  The diagonal embedding $K^\times\to I_K$ makes $K^\times$ a discrete subgroup of $I_K$ and the id\`ele class group of $K$ is the quotient $C_K=I_K/K^\times$. The embeddings $K_v^\times\to I_K\to C_K$ for each place $v$ of $K$ make $K_v^\times$ a closed subgroup of $I_K$ and $C_K$.

\begin{theo}\label{prop:global-reciprocity} There is a unique continuous homomorphism \[C_K\to G(K^\ab/K): s\mto (s, K^\ab/K)\] such that for each place $v$ of $K$ and each $K$-linear embedding $K^\ab\to K_v^\ab$ the following diagram commutes \[\xymatrix{C_K\ar[rr]^-{(-, K^\ab/K)} && G(K^\ab/K)\\
K_v^\times\ar[u]\ar[rr]^-{(-, K_v^\ab/K_v)} && G(K_v^\ab/K_v).\ar[u]}\]
\end{theo}
\begin{proof} See \S\S 4--6 Chapter VII of \cite{CasselsFrohlich67}.
\end{proof}
\subsection{}\label{subsec:global-reciprocity-properties} We list the following further properties of the reciprocity map (see \S\S 4--6 Chapter VII of \cite{CasselsFrohlich67}).
\begin{enumerate}[label=\textup{(\alph*)}]
\item If $K'/K$ is any finite Galois extension, with maximal abelian extension $K'^\ab$, the diagram \[\xymatrix{C_{K'}\ar[rr]^-{(-, K'^\ab/K')}\ar[d]_{N_{K'/K}} && G(K'^\ab/K')\ar[d]\\
C_K\ar[rr]^-{(-, K^\ab/K)} && G(K^\ab/K)}\] commutes (where $N_{K'/K}$ denotes the map induced by the norm $I_{K'}\to I_{K}$).
\item The kernel of the reciprocity map is $C_K^\circ\subset C_K$ (a superscript $\circ$ denotes the connected component of the identity of a topological group). If $K$ is a function field then $C_K^\circ$ is trivial and $(-, K^\ab/K)$ is injective. If $K$ is a number field then $C_K^\circ$ is the closure of the sub-group \[\prod_{v\in \mc{P}_K^\arch}K_v^{\times,\circ}\subset C_K\] and $(-, K^\ab/K)$ is surjective. We note that if $\mc{P}_K^\arch=\{\infty\}$ contains only one place then $K_\infty^{\times, \circ}=C_K^\circ$ and we obtain a topological isomorphism \[C_K/K_\infty^{\times,\circ}\isomto G(K^\ab/K).\]
\end{enumerate}
\section{Global fields with a single `infinite' place and class groups}\label{section:special-place} In this section we describe a variant of the reciprocity map associated to the maximal abelian extension $K^\infty$ of a global field $K$ which is totally split at a fixed place $\infty$ satisfying $\mc{P}_K^\arch\subset \{\infty\}$. If $K$ is a function field then any place $\infty$ satisfies the above property and if $K$ is a number field then the only possibilities are $K=\Q$ or $K$ an imaginary quadratic and in each case $\infty$ equal to the unique archimedian place of $K$.

As mentioned in the introduction to this chapter, the fact that such pairs $(K, \infty)$ should be considered along similar lines is due to Drinfel'd \cite{Drinfeld74} and essentially all we are doing here is collecting the relevant facts for use in the later chapters. It is worth noting here that the abelian extensions of $K$ contained in $K^\infty$ have a long history mainly due to the fact that they are amenable to explicit computation via the use of Drinfel'd modules in case $K$ is a function field, tori if $K=\Q$ and CM elliptic curves if $K$ is an imaginary quadratic field (what Drinfel'd originally called `elliptic modules of rank one'). While we are only concerned with the final case in this thesis, the abstract theory we describe below is valid for arbitrary pairs $(K, \infty)$.

\subsection{}\label{subsec:global-field-with-special-place} Let $K$ be a global field, and fix a place $\infty$ of $K$ such that $\mc{P}^\arch_K\subset \{\infty\}$. We call $\infty$ the infinite place, and the places in $\mc{P}_K-\infty$ the finite places and denote them by $\mc{P}_K^\fin$. As every place $v\neq \infty$ is non-archimedian the subset $O_K=\{a\in K: |a|_v\leq 1 \text{ for all } v\neq \infty\}\subset K$ is a Dedekind domain, and its prime ideals are in bijection with the set $\mc{P}_K^\fin$.

The group of units $O_K^\times$ is finite and we denote its order by $w$. We want to point out that for what follows in this section, and in the following chapters, this fact is quite crucial. It implies in particular that given any ideal $\f$ there is an ideal $\f|\f'$ with the property that $\f'$ separates units, i.e.\ the homomorphism \[O_K^\times\to (O_K/\f')^\times\] is injective or what is the same $O_{K}^{\times, \f'}=\{1\}$ (of course, if $\f\neq O_K$ any high power of $\f$ will do).

\subsection{} We write $A_{O_K}$ for the topological ring \[\lim_{\a}O_K/\a=\prod_{\p} O_{K_\p},\] the topology being the product topology or the inverse limit topology induced by the discrete topologies on the finite sets $O_K/\a$ (they are the same). We also view \[A_{O_K}^\times=\lim_{\a} (O_K/\a)^\times=\prod O_{K_\p}^\times\] as a topological group via the topology induced from $A_{O_K}$, the product topology or the inverse limit topology (again they are one and the same). For each integral ideal $\f$ we denote by $A_{O_K}^{\times, \f}$ the open subgroup \[\ker(A_{O_K}^\times\to (O_K/\f)^\times).\] For each integral ideal $\a$ we equip \[A_{O_K}[\a^{-1}]^\times := \prod_{\p|\a}K_\p^\times\times \prod_{\p\nmid\a} O_{K_\p}^\times\] with the product topology. If $\a|\b$ the inclusion $A_{O_K}[\a^{-1}]^\times\subset A_{O_K}[\b^{-1}]^\times$ is open and we equip \[(A_{O_K}\otimes_{O_K}K)^\times=\colim_{\a} A_{O_K}[\a^{-1}]^\times\] with the colimit topology. The natural map \[I_K\to (A_{O_K}\otimes_{O_K} K)^\times,\] induced by forgetting the component at $\infty$, is continuous and surjective and induces topological isomorphism \[I_K/K_\infty^\times\isomto (A_{O_K}\otimes_{O_K} K)^\times.\]

\subsection{} We will now relate the group $(A_{O_K}\otimes_{O_K}K)^\times$ to certain class groups associated to $O_K$. So let $\f\in \Id_{O_K}$ and let $L$ be a projective rank one $O_K$-module. A level-$\f$ structure on $L$ is a surjective homomorphism $a: L\to O_K/\f$. An $\f$-isomorphism $f: (L, a)\isomto (L', a')$ between a pair of projective rank one $O_K$-modules with level-$\f$ structures is an $O_K$-isomorphism $f: L\to L'$ such that $h\circ a'=a$. We denote by $\CL_{O_K}^{(\f)}$ the set of $\f$-isomorphism classes of rank one projective $O_K$-modules with level-$\f$ structure. Equipping it with the product \[(L, a)\cdot (L', a')=(L\otimes_{O_K} L', a\otimes_{O_K} a'),\] $\CL_{O_K}^{(\f)}$ becomes a group, which we call the ray class group of conductor $\f$.

If $\a$ is any fractional ideal prime to $\f$, the multiplication map \[\a\otimes_{O_K}O_K/\f\isomto O_K/\f\] is well defined, and an isomorphism, so that setting $f: \a\to O_K/\f$ to be the composition \begin{equation}\a\to \a\otimes_{O_K}O_K/\f\isomto O_K/\f\label{eqn:can-level-ideal}\end{equation} equips $\a$ with a level-$\f$ structure. We write $[\a]_\f=(\a, f)\in \CL_{O_K}^{(\f)}$ for the corresponding class. This defines a surjective homomorphism \[\Id_{K}^{(\f)}\to \CL_{O_K}^{(\f)}: \a\mto [\a]_\f\] whose kernel is the group $\Prin_{1\bmod \f}^{(\f)}$ of principal fractional ideals $\a=O_K\bmod \f$ admitting a generator $a\in K^\times$ with $a=1\bmod \f$. If $\f|\f'$ then \[\CL_{O_K}^{(\f')}\to \CL_{O_K}^{(\f)}: (L, a)\mto (L, a \bmod \f)\] defines a surjective homomorphism and we define the topological group \[\CL_{O_K, \infty}=\lim_\f \CL_{O_K}^{(\f)}\] (the topology being the inverse limit of the discrete topologies on the $\CL_{O_K}^{(\f)}$). Finally, if $\f=O_K$ then we identify $\CL_{O_K}^{O_K}$ with the class group $\CL_{O_K}$ of $O_K$.

\subsection{} Given an element $s\in (A_{O_K}\otimes_{O_K}K)^\times$, we write $(s)\in \Id_K$ for the fractional ideal \[(s)=\prod_{\p}\p^{v_\p(s)}.\] For each integral ideal $\f$, we equip $(s)^{-1}$ with the level-$\f$ structure \[(s)^{-1}\stackrel{s}{\to} A_{O_K}\to O_K/\f\] and write $[s]_\f\in \CL_{O_K}^{(\f)}$ for the corresponding class. This defines a continuous surjective homomorphism \[(A_{O_K}\otimes_{O_K}K)^\times\to \CL_{O_K}^{(\f)}: s\mto [s]_\f\] with kernel $K^\times\cdot A_{O_K}^{\times, \f}$. Finally, if $\f|\f'$ the image of $[s]_{\f'}$ under $\CL_{O_K}^{(\f')}\to \CL_{O_K}^{(\f)}$ is $[s]_\f$ and so taking the limit over $\f$ we obtain a homomorphism \begin{equation}[-]:(A_{O_K}\otimes_{O_K}K)^\times\to \CL_{O_K, \infty}=\lim_\f \CL_{O_K}{(\f)}: s\mto [s]=\lim_\f [s]_\f.\label{def:bracket-homo}\end{equation}

\begin{prop}\label{prop:ideles-and-classes} The map $[-]$ is continuous and the sequence \[0\to K^\times\to (A_{O_K}\otimes_{O_K}K)^\times\stackrel{[-]}{\to} \CL_{O_K, \infty}\to 0\] is exact.
\end{prop}
\begin{proof} It is clear that $s\mto [s]$ is continuous (as $s\mto [s]_\f$ is continuous for each $\f$) and if we can show that $\ker(s\mto [s])=K^\times$ then the surjectivity of $s\mto [s]$ follows as $s\mto [s]_\f$ is surjective for each $\f$ and $(A_{O_K}\otimes_{O_K}K)^\times/K^\times=C_K/K_\infty^\times$ is compact.

The kernel of $s\mto [s]$ is equal to \[\bigcap_\f \ker(s\mto [s]_\f)=\bigcap_\f (K^\times \cdot A_{O_K}^{\times, \f}).\] If $s$ is an element of this kernel then for all integral ideals $\f$ we can write $s=a_\f s_\f$ where $a_\f\in K^\times$ and $s_\f\in A_{O_K}^{\times, \f}$. The elements $a_\f$ and $s_\f$ are unique upto scaling by an element of $O_{K}^{\times, \f}$ so that if $\f$ separates units both $a_\f$ and $s_\f$ are unique, and moreover equal to $a_{\f'}$ and $s_{\f'}$ for any integral ideal $\f'$ divisible by $\f.$ Fixing such an $\f$ it follows that \[s_\f\in \bigcap_{\f|\f'} A_{O_K}^{\times, \f'}=\{1\}\] so that $s=a_\f s_\f=a_\f \in K^\times$ and we are done.
\end{proof}

\begin{rema} The exactness of the sequence (\ref{prop:ideles-and-classes}) is the first result of many that will rely crucially on the fact that the unit group $O_K^\times$ is finite.
\end{rema}

\subsection{}\label{subsec:ray-class-fields} We write $K^{\infty}/K$ for the maximal abelian extension of $K$ which is totally split at $\infty$. The kernel of the surjective map \[C_K\to G(K^{\infty}/K)\] is precisely $K_\infty^\times$ (cf.\ (b) of (\ref{subsec:global-reciprocity-properties})) so that we obtain continuous isomorphisms \[(A_{O_K}\otimes_{O_K} K)^\times/K^\times=I_{K}/(K^\times K^\times_\infty)=C_K/K_\infty^\times \to G(K^{\infty}/K).\] Let us write \begin{equation}\label{def:idele-rec} (-, K^\infty/K): (A_{O_K}\otimes_{O_K}K)^\times/K^\times\isomto G(K^\infty/K).\end{equation} for this isomorphism and also \begin{equation}\label{eqn:define-theta}\theta_K:\CL_{O_K, \infty}\isomto G(K^\infty/K)\end{equation} for the isomorphism \[\CL_{O_K, \infty}\stackrel{\sim}{\longleftarrow}(A_{O_K}\otimes_{O_K}K)^\times/K^\times \stackrel{
(-, K^\infty/K)}{\longrightarrow} G(K^\infty/K)\] so that for all $s\in (A_{O_K}\otimes_{O_K}K^\times)/K^\times$ we have \[\theta_K([s])=(s, K^\infty/K).\] If $\f$ is an integral ideal of $O_K$ then, under $\theta_K$, the kernel of the homomorphism \[\CL_{O_K, \infty}\to \CL_{O_K}^{(\f)}\] corresponds to the subgroup $G(K^{\infty}/K(\f))\subset G(K^\infty/K)$ of automorphisms fixing a certain finite abelian extension $K\subset K(\f)\subset K^{\infty}$ which we call the ray class field of conductor $\f$. By definition the map (\ref{eqn:define-theta}) induces an isomorphism \begin{equation}\label{eqn:define-theta-finite}\theta_{K, \f}: CL_{O_K}^{(\f)}\isomto G(K(\f)/K).\end{equation} The extension $K(\f)/K$ is unramified away from $\f$ and if $\p$ is prime to $\f$ then \[\theta_{K, \f}([\p]_\f^{-1})=\theta_{K, \f}([\pi]_\f)=\sigma_{K(\f)/K, \p}\] where $\pi\in K_\p^\times\subset (A_{O_K}\otimes_{O_K}K)^\times$ is any local uniformiser. In particular, when $\f=O_K$ the field $H:=K(O_K)$ is called the Hilbert class field. It is unramified everywhere and the isomorphism \[\theta_{K, O_K}: \CL_{O_K}\isomto G(H/K)\] maps the class of the inverse of each prime ideal $[\p^{-1}]\in \CL_{O_K}$ to the Frobenius element $\sigma_{H/K, \p}$.

\begin{rema} For future reference we make the following observations.
\begin{enumerate}[label=(\alph*)]
\item The composition \[A_{O_K}^{\times, \f}/O_K^{\times, \f} \to (A\otimes_{O_K}K^\times)/K^\times \stackrel{[-]}{\to} CL_{O_K, \infty}\stackrel{\theta_K}{\to} G(K^\infty/K)\] induces an isomorphism \begin{equation} A_{O_K}^{\times, \f}/O_K^{\times, \f}\isomto G(K^\infty/K(\f))\subset G(K^\infty/K). \label{ray-class-isom}\end{equation} In particular, when $\f$ separates units we have $O_{K}^{\times, \f}=\{1\}$ so that (\ref{ray-class-isom}) becomes \[A_{O_K}^{\times, \f}\isomto G(K^\infty/K(\f)).\]
\item For each prime $\p$ of $O_K$ and each $K$-linear embedding $K^\sep\to K_\p^\sep$, the map \[K^\times_\p\stackrel{(-, K_\p^\ab/K_\p)}{\longrightarrow} G(K_\p^\sep/K_\p)\to G(K^\sep/K)\stackrel{-|_{K^\infty}}{\to} G(K^\infty/K)\stackrel{\theta_K^{-1}}{\to} \CL_{O_K, \infty}\] is given by \begin{equation} \label{eqn:local-global-class-compatibility}a\mto [a]\end{equation} where we view $a\in K_\p^\times\subset (A_{O_K}\otimes_{O_K}K)^\times$.
\end{enumerate}
\end{rema}
\section{Class stacks} We now extend the definition of the level-$\f$ structures on $O_K$-modules to level-$\f$ structures on $O_K$-local systems over sheaves and give several basic results concerning them and their moduli stacks.

\subsection{} So let $S$ be a sheaf over $\Spec(O_K)$ (this is technically not important for what follows) and consider the constant sheaf of rings $\underline{O_K}_S$ on $S$ associated to $O_K$. If $L$ is any $O_K$-module we write $\underline{L}_S$ for the corresponding constant $\underline{O_K}_S$-module. If $F$ and $G$ are two $\underline{O_K}_S$-modules over $S$ we write $F\otimes_{O_K} G$ for the tensor product $F\otimes_{\underline{O_K}_S}G$ and if $G=\underline{L}_S$ for some $O_K$-module $L$ we just write $F\otimes_{O_K}L$. We also write $\underline{\Hom}_S^{O_K}(F, G)$ for the sheaf of $\underline{O_K}_S$-homomorphisms $F\to G.$

\subsection{} A rank one $O_K$-local system on $S$ is a sheaf of $\underline{O_K}_S$-modules $\mc{L}$ over $S$ such that there exists a cover $(S_i\to S)_{i\in I}$, rank one projective $O_K$-modules $(L_i)_{i\in I}$ and $\underline{O_K}_S$-isomorphisms $\mc{L}\times_S S_i\isomto \underline{L_i}_{S_i}$. The moduli stack of rank one $O_K$-local systems over $\Sh_{O_K}$ is denoted by $\mc{CL}_{O_K}$. We list the following (usual) constructions and properties of $O_K$-local systems.

\begin{enumerate}[label=(\roman*)]
\item The tensor product $\mc{L}\otimes_{O_K}\mc{L}'$ of two rank one $O_K$-local systems $\mc{L}$ and $\mc{L}'$ is again a rank one $O_K$-local system.
\item The sheaf of $\underline{O_K}_S$-homomorphisms $\underline{\Hom}_{S}^{O_K}(\mc{L}, \mc{L}')$ is again a rank one $O_K$-local system and defining $\mc{L}^\vee:=\underline{\Hom}_S^{O_K}(\mc{L},\underline{O_K}_S)$ we have $\mc{L}'\otimes_{O_K}\mc{L}^\vee\isomto \underline{\Hom}_{S}^{O_K}(\mc{L}, \mc{L}')$.
\item The sheaf of automorphisms $\underline{\Aut}_S^{O_K}(\mc{L})$ of a rank one $O_K$-local system $\mc{L}$ is isomorphic to $\underline{O_K}_S^\times$.
\item The sheaf of $O_K$-isomorphisms $\underline{\Isom}_S^{O_K}(\mc{L}, \mc{L}')$ is finite and \'etale over $S$ and $\mc{L}$ and $\mc{L}'$ are locally isomorphic on $S$ if and only if $\underline{\Isom}_S^{O_K}(\mc{L}, \mc{L}')\to S$ is an epimorphism if and only if the action of $\underline{O_K^\times}_S$ on $\underline{\Isom}_S^{O_K}(\mc{L}, \mc{L}')\to S$ makes it a torsor.
\item Every rank one $O_K$-local system is, locally on $S$, isomorphic $\underline{L}_S$ for some rank one projective $O_K$-module $L$ whose corresponding class in $\CL_{O_K}$ is independent of the choice of $L$. Thus given an $O_K$-local system $\mc{L}$ on $S$ one obtains a section $c_{\mc{L}/S}\in \underline{CL_{O_K}}(S)$, or what is the same a map \[c_{\mc{L}/S}: S\to \underline{CL_{O_K}}\] from $S$ to the constant sheaf over $\Spec(O_K)$ associated to the group $CL_{O_K}$. Moreover, if one chooses representatives $L$ of each class $[L]\in \CL_{O_K}$ and considers the rank one $O_K$-local system over $\underline{CL_{O_K}}$: \[[L]^\univ:=\coprod_{[L]\in \CL_{O_K}}\underline{L}\to \coprod_{[L]\in CL_{O_K}}\Spec(O_K)=\underline{\CL_{O_K}}\] then $\underline{\Isom}_S^{O_K}(\mc{L}, c_{\mc{L}/S}^*([L]^\univ))$ is an $\underline{O_K^\times}_S$-torsor whose corresponding class in $H^1(S, \underline{O_K^\times}_S)$ we denote by $\rho_{\mc{L}/S}$. The resulting map \[\mc{L}\mto (c_{\mc{L}/S}, \rho_{\mc{L}/S})\in \underline{CL_{O_K}}(S)\times H^1(S, \underline{O_K^\times}_S)\] defines a bijection between isomorphism classes of rank one $O_K$-local systems over $S$ and the set $\underline{CL_{O_K}}(S)\times H^1(S, \underline{O_K^\times}_S)$.
\end{enumerate}

\subsection{}\label{rema:level-f-local-system} If $\f$ is an integral ideal of $O_K$ a level-$\f$ structure on a rank one $O_K$-local system $\mc{L}$ over $S$ is an epimorphism of $\underline{O_K}_S$-modules $\alpha: \mc{L}\isomto \underline{O_K/\f}_S$. An $\f$-isomorphism $f: (\mc{L}, \alpha)\isomto (\mc{L}', \alpha')$ of rank one $O_K$-local systems of over $S$ equipped level-$\f$ structures is an $O_K$-linear isomorphism $f: \mc{L}\isomto \mc{L}'$ such that $\alpha'\circ f=\alpha$. With this definition every $(\mc{L}, \alpha)$ is, locally on $S$, of the form $(\underline{L}_S, \underline{a}_S)$ for some $(L, a)\in \CL_{O_K}^{(\f)}$. When working with rank one $O_K$-local systems we will often drop explicit reference to the level-$\f$ structure when it is clear from context. We list the following (usual) constructions and properties of $O_K$-local systems with level-$\f$ structure.

\begin{enumerate}[label=(\roman*)]
\item The tensor product $\mc{L}\otimes_{O_K}\mc{L}'$ of two rank one $O$-local systems $(\mc{L}, \alpha)$ and $(\mc{L}', \alpha')$ equipped with level-$\f$ structures is again a rank one $O_K$-local system with level-$\f$ structure given by $\alpha\otimes_{O_K}\alpha'$.
\item The sheaf of $\underline{O_K}_S$-homomorphisms $\underline{\Hom}_{S}^{O_K}(\mc{L}, \mc{L}')$ is again a rank one $O_K$-local system with level-$\f$ structure given by \[\underline{\Hom}_S^{O_K}(\mc{L}\otimes_{O_K}O_K/\f, \mc{L}'\otimes_{O_K} O_K/\f)\\\isomto \underline{\Hom}_S^{O_K}(\underline{O_K/\f}_S, \underline{O_K/\f}_S)=\underline{O_K/\f}_S,\] and equipping $\mc{L}^\vee:=\underline{\Hom}_S^{O_K}(\mc{L},\underline{O_K}_S)$ with this level-$\f$ structure makes $\mc{L}'\otimes_{O_K}\mc{L}^\vee\isomto \underline{\Hom}_{S}^{O_K}(\mc{L}, \mc{L}')$ is an $\f$-isomorphism.
\item The sheaf of $\f$-automorphisms $\underline{\Aut}_S^{(\f)}(\mc{L})$ of a rank one $O_K$-local system $(\mc{L}, \alpha)$ with level-$\f$ structure is equal to $\underline{(O_K/\f)^\times}_S$ and in particular is trivial if $\f$ separates units.
\item The sheaf of $\f$-isomorphisms $\underline{\Isom}_S^{(\f)}(\mc{L}, \mc{L}')$ between two rank one $O_K$-local systems with level-$\f$ structure is finite and \'etale over $S$ and $\mc{L}$ and $\mc{L}'$ are locally $\f$-isomorphic if and only if $\underline{\Isom}_S^{(\f)}(\mc{L}, \mc{L}')$ is an $\underline{O_K^{\times, \f}}_S$-torsor.
\item Every rank one $O_K$-local system with level-$\f$ structure is, locally on $S$, $\f$-isomorphic $(\underline{L}_S, \underline{a}_S)$ for some rank one projective $O_K$-module with level-$\f$ structure $(L, a)$ whose corresponding class in $\CL_{O_K}^{(\f)}$ is independent of the choice of $L$. Thus given an $O_K$-local system $\mc{L}$ on $S$ one obtains a section $c_{\mc{L}/S, \f}\in \underline{CL_{O_K}^{(\f)}}(S)$, or what is the same a map \[c_{\mc{L}/S, \f}: S\to \underline{CL_{O_K}^{(\f)}}\] from $S$ to the constant sheaf over $\Spec(O_K)$ associated to the group $CL_{O_K}^{(\f)}$. Moreover, if one chooses representatives $(L, a)$ of each class $[L, a]\in \CL_{O_K}^{(\f)}$ and considers the rank one $O_K$-local system with level-$\f$ structure over $\underline{CL_{O_K}^{(\f)}}$: \[[L, a]^\univ:=\coprod_{[L, a]\in \CL_{O_K}}(\underline{L}, \underline{a})\to \coprod_{[L, a]\in CL_{O_K}^{(\f)}}\Spec(O_K)=\underline{\CL_{O_K}^{(\f)}}\] then $\underline{\Isom}_S^{(\f)}(\mc{L}, c_{\mc{L}/S}^*([L, a]^\univ))$ is an $\underline{O_K^{\times, \f}}_S$-torsor whose corresponding class in $H^1(S, \underline{O_K^{\times, \f}}_S)$ we denote by $\rho_{\mc{L}/S, \f}$. The resulting map \[\mc{L}\mto (c_{\mc{L}/S, \f}, \rho_{\mc{L}/S, \f})\in \underline{CL_{O_K}^{(\f)}}(S)\times H^1(S, \underline{O_K^\times}_S)\] defines a bijection between isomorphisms classes of rank one $O_K$-local systems over $S$ and the set $\underline{CL_{O_K}}(S)\times H^1(S, \underline{O_K^{\times, \f}}_S)$. In particular, if $\f$ separates units then $H^1(S, \underline{O_K^{\times, \f}}_S)=0$ and the map \[\mc{L}\mto c_{\mc{L}/S, \f}\in \underline{\CL_{O_K}^{(\f)}}(S)\] defines a bijection between $\f$-isomorphism classes of rank one $O_K$-local systems with level-$\f$ structure over $S$ and elements of $\underline{\CL_{O_K}^{(\f)}}$.
\item If $\f$ separates units, $(\mc{L}, \alpha)$ is an $O_K$-local system equipped with a level-$\f$ structure and $S$ is connected, then $(\mc{L}, \alpha)\isomto (\underline{L}_S, \underline{a}_S)$ for some rank one projective $O_K$-module with level-$\f$ structure and in particular, $\mc{L}$ is constant.
\end{enumerate}

\begin{coro}\label{prop:class-stacks-triv-admis} The map \[\mc{CL}_{O_K}^{(\f)}\to\underline{CL_{O_K}^{(\f)}}: \mc{L}/S\mto c_{\mc{L}/S}\in \underline{CL_{O_K}^{(\f)}}(S)\] identifies $\underline{CL_{O_K}^{(\f)}}$ with the coarse sheaf of $\mc{CL}_{O_K}^{(\f)}$ and is an equivalence whenever $\f$-separates units.
\end{coro}
\begin{proof} This follows from the remarks (\ref{rema:level-f-local-system}).
\end{proof}

\chapter{Elliptic curves with complex multiplication} In this chapter we develop the general theory of (families of) elliptic curves with complex multiplication by the ring of integers $O_K$ of a fixed imaginary quadratic field $K$, here on called just CM elliptic curves. In \S 1 we recall several standard results from the theory of (families of general) elliptic curves. In \S 2 we define the notion of a family $E\to S$ of CM elliptic curves, the corresponding moduli stack $\mc{M}_\CM$, and we show that (just as with Lubin--Tate $O$-modules) the moduli stack $\mc{CL}_{O_K}$ of rank one $O_K$-local systems acts in a natural way on the moduli stack $\mc{M}_\CM$ of CM elliptic curves. We then describe, for a prime $\p\subset O_K$, the properties of the $\p$-power torsion subgroups $E[\p^\infty]\subset E$ of a family of CM elliptic curves and show that when $S$ is a $\p$-adic sheaf $E[\p^\infty]$ is a Lubin--Tate $O_{K_\p}$-module. In \S 3 we consider CM elliptic curves over complex and $\p$-adic bases and give CM analogues of the classification of elliptic curves over complex bases, and theorem of Serre-Tate describing deformations of elliptic curves over $\p$-adic bases. In \S 4 we show that any two CM elliptic curves over the same base are locally isogenous and from this deduce that the action of $\mc{CL}_{O_K}$ on $\mc{M}_{\CM}$ gives $\mc{M}_{\CM}$ the structure of a torsor. In \S 5 we derive the global reciprocity map associated to the maximal abelian extension of $K$ (totally split at $\infty$ -- but this is a vacuous condition) directly from the stack $\mc{M}_\CM$ in a manner quite analogous to the derivation of the local reciprocity map via the moduli stack of Lubin--Tate $O$-modules. We then classify all CM elliptic curves over fields (both of characteristic zero and finite characteristic) and prove some results regarding good reduction. In \S 5 we define level-$\f$ structures for CM elliptic curves and consider the corresponding moduli stacks $\mc{M}_\CM^{(\f)}$. As with $\mc{M}_\CM$ and $\mc{CL}_{O_K}$, we show that $\mc{M}_\CM^{(\f)}$ is a torsor under $\mc{CL}_{O_K}^{(\f)}$ (at least after inverting $\f$). Using this we show that the coarse sheaf of $\mc{M}_\CM^{(\f)}$ is isomorphic to $\Spec(O_{K(\f)}[\f^{-1}])$ where $K(\f)$ is the ray class field of conductor $\f$.

We should point out that, aside from the $\mc{CL}_{O_K}$-torsor structure of $\mc{M}_\CM$, consistently working over a general base (instead of a field) and our derivation of the reciprocity map, almost everything in this chapter is probably more or less already known. This combined with the fact that $\mc{M}_\CM$ is zero dimensional over $\Spec(O_K)$, the advantages of our general approach may be somewhat unclear. However, while $\mc{M}_\CM$ is geometrically rather simple, it is arithmetically quite complicated and the general approach we take in this chapter will allow for great deal of flexibility later when we wish to study some of its finer arithmetic properties.

\section{General elliptic curves} We now recall the definition of a family of elliptic curves over a sheaf $S$ and recall several standard results. In particular, the fact that the moduli stack of elliptic curves is indeed a stack, the rigidity principal for homomorphisms, the representability of Isom sheaves, Grothendieck's formal GAGA, the classification of elliptic curves over complex schemes $S$ in terms of rank two $\Z$-local systems over $S^\an$, the Serre-Tate theorem, and the criterion of good reduction.

\subsection{}\label{subsec:elliptic-curve-def} Let $S$ be a sheaf. An elliptic curve over $S$ is a sheaf of groups $E\to S$ which is relatively representable, smooth of relative dimension one, proper and geometrically connected. A morphism is of course a homomorphism of the underlying (sheaves of) groups over $S$. For many more general properties and constructions related to families of elliptic curves $E\to S$ we refer the reader to the wonderful book of Katz-Mazur \cite{KatzMazur85}.

If $S$ is a sheaf we write $\Ell(S)$ for the category of elliptic curves over $S$ and we denote by $\mc{M}_\Ell$ the fibred category over $\Sh$ whose fibre over a sheaf $S$ is the category elliptic curves $E/S$ together with their isomorphisms.
\begin{prop}\label{prop:moduli-ell-stack} The fibred category $\mc{M}_\Ell$ is a stack over $\Sh$.
\end{prop}
\begin{proof}[Sketch] If $S$ is an affine scheme and $f: E\to S$ is a family of elliptic curves let $\mc{I}_{E/S}\subset \mc{O}_E$ denote the ideal sheaf defining the zero section $S\to E$ (it is a locally free rank one $\mc{O}_E$-module). The quasi-coherent $\mc{O}_S$-module $\mc{W}_{E/S}=f_*(\mc{I}^{-3}_{E/S})$ is a vector bundle of rank three, the morphism \[f^*f_*(\mc{I}^{-3}_{E/S})\to \mc{I}^{-3}_{E/S}\] is an epimorphism and defines a closed immersion $w_{E/S}: E\to \mathbf{P}(\mc{W}_{E/S})$. Both the vector bundle $\mc{W}_{E/S}$ and the morphism $w_{E/S}$ are functorial in $S$ so that by descent if $E\to S$ is any family of elliptic curves over a \text{sheaf} $S$ then there is a unique vector bundle $\mc{W}_{E/S}$ of rank three over $S$ together with a closed immersion $w_{E/S}: E\to \mathbf{P}_{S}(\mc{W}_{E/S})$ compatible with those defined when $S$ is affine.

Now if $S$ is an sheaf and $(f_i: E_i\to S_i)_{i\in I}$ is a family of elliptic curves equipped with descent data relative to a cover $(S_i)_{i\in I}$ of $S$ then the $E_i$ descend to a \textit{sheaf} of groups $E\to S$, the vector bundles $\mc{W}_{E_i/S_i}$ to a vector bundle $\mc{W}_{E/S}$ and the closed immersions $w_{E_i}: E_i\to \mathbf{P}(\mc{W}_{E_i/S_i})$ to a closed immersion \[w_{E/S}: E\to \mathbf{P}(\mc{W}_{E/S}).\] This shows that the morphism $p: E\to S$ is representable (in fact projective) and by descent it follows that $E\to S$ is also smooth of relative dimension one, proper, and geometrically connected so that $E\to S$ is an elliptic curve over $S$.
\end{proof}
\subsection{}\label{subsec:rigidity} Here we recall several useful properties enjoyed by homomorphisms of elliptic curves.
\begin{prop} Let $S$ be a sheaf and let $E\to S$ be a family of elliptic curves. For each $n\geq 1$ multiplication map $n: E\to E$ is finite locally free of degree $n^2$.
\end{prop}
\begin{proof} This is Theorem 2.3.1 of \cite{KatzMazur85}.
\end{proof}
\begin{prop}[Rigidity]\label{prop:rigidity} If $f: E\to E'$ is a homomorphism of elliptic curves over a sheaf $S$ there is a unique decomposition $S=\amalg_{n\geq 0} S_{(n)}$ with the property that $f\times_S S_{(0)}$ is the zero map and such that $f\times_S S_{(n)}$ is finite locally free of degree $n$ for $n\geq 1$.
In particular, if $f, g: E\to E'$ are a pair of homomorphisms of elliptic curves and $S'\to S$ is a morphism of sheaves which is surjective on geometric points then $f=g$ if and only if $f\times_S S'=g\times_S S'$.
\end{prop}
\begin{proof} The first statement is Theorem 2.4.2 of \cite{KatzMazur85}. For the second statement the only if direction is clear, so assume that $f\times_S S'=g\times_S S'$. The claim is local on $S$ and $S'$ and so we may assume that they are affine schemes. If $S_{(0)}\subset S$ and $S_{{(0)}}'\subset S'$ denote the open and closed sub-schemes where $f-g$ and $f\times_S S'-g\times_S S'$ are equal to the zero map respectively, then $S_{(0)}'\to S$ factors through $S_{(0)}$. However, $S_{(0)}'=S'$, so that as $S'\to S$ is surjective on geometric points, it follows that $S_{(0)}=S$.
\end{proof}
\begin{rema} We will make much use of (\ref{prop:rigidity}) and when doing so just say `by rigidity'.
\end{rema}
\begin{prop}\label{prop:isom-are-fin-unram} For each pair of elliptic curves $E$ and $E'$ over a sheaf $S$ the sheaf $\underline{\Isom}_S(E, E')$ is finite and unramified over $S$.
\end{prop}
\begin{proof} The claim is local on $S$ so we may assume that $S$ is an affine scheme and this is Proposition 5.3 (i) of \cite{Deligne75}.
\end{proof}

\subsection{}\label{prop:formal-gaga} Let $A$ be a noetherian ring complete with respect to the $I$-adic topology for $I\subset A$ an ideal and write $\Spf(A)=\colim_{n\geq 0} \Spec(A/I^{n+1})\subset \Spec(A)$.
\begin{theo}[Formal GAGA] The functor \[\Ell(\Spec(A))\to \Ell(\Spf(A)): E/\Spec(A)\mto E\times_{\Spec(A)} \Spf(A)\] induced by base change is an equivalence of categories.
\end{theo}
\begin{proof} This is an easy application Grothendieck's formal GAGA (in particular Corollaire 2 and Th\'eor\`eme 4 of \cite{Groth95}).
\end{proof}

\subsection{}\label{subsec:complex-class} The following is taken from N\textsuperscript{\underline{o}} 2 of \cite{Deligne71}. Let $S$ be a locally finitely presented $\Spec(\C)$-scheme and let $f: E\to S$ be an elliptic curve. The analytification $f^\an:E^\an \to S^\an$ is an analytic space over $S^\an$ which is smooth of relative dimension one and proper with connected fibres. There is a canonical exact sequence, the exponential sequence, of sheaves on the big analytic site of $X^\an$ \[0\to T_\Z(E)\to \underline{\Lie}_{E^\an/S^\an} \to E^\an\to 0\] (we view $\underline{\Lie}_{S^\an/E^\an}$ as a rank one locally free sheaf of $\mc{O}_{S^\an}$-modules on the \textit{big} analytic site of $S$) with $T_\Z(E)$ a rank two $\Z$-local system on $S^\an$. Denote by $\mathrm{Lat}(S^\an)$ the category of pairs $(T \subset \mc{V})$ where $\mc{V}$ is a locally free rank one $\mc{O}_{S^\an}$-module and $T\subset \mc{V}$ is a rank two $\Z$-local system which is fibre-wise over $S^\an$ discrete in $\mc{V}$.

\begin{prop}\label{prop:class-complex-elliptic} The functor \[\Ell(S)\to \mathrm{Lat}(S^\an): E/S\mto (T_\Z(E/S)\subset \underline{\Lie}_{E^\an/S^\an})\] is an equivalence of categories.
\end{prop}

\subsection{}\label{subsec:classical-serre-tate} Let $p$ be a rational prime and $S$ a $p$-adic sheaf (i.e.\ a sheaf over $\Spf(\Z_p)=\colim_{n} \Spec(\Z/p^{n+1}))$. A $p$-divisible group over $S$ is an sheaf of groups $F\to S$ such that the multiplication map $p: F\to F$ is representable, finite locally free and faithfully flat and such that $\colim_n \ker(p^n)=F$. If $E$ is an elliptic curve over $S$ we write $E[p^\infty]$ for the $p$-divisible group $\colim_n E[p^n]$.

Let $S_0\to S$ be a nilpotent closed immersion of $p$-adic sheaves and consider the category $D(S, S_0)$ whose objects are triples $(E/S_0, F/S, \rho)$ with
\begin{enumerate}[label=\textup{(\roman*)}]
\item $E_0/S_0$ an elliptic curve,
\item $F/S$ is a $p$-divisible group, and
\item $\rho: E_0[p^\infty]\isomto F\times_S S_0$ an isomorphism of $p$-divisible groups,
\end{enumerate} \noindent and whose morphisms $(E_0/S_0, F/S, \rho) \to (E'/S_0, F', \rho')$ are pairs $(f, g)$ where $f: E\to E'$ is a morphism of elliptic curves, $g: F\to F'$ is a morphism of $p$-divisible groups such that \[(g\times_{S} S_0)\circ \rho=\rho'\circ f|_{E[p^\infty]}.\]

\begin{theo}[Serre-Tate]\label{theo:serre-tate} The functor \[\Ell(S)\to D_p(S, S_0): E/S\mto (E\times_S S_0/S_0, E[p^\infty]/S, \id_{E[p^\infty]}|_{S_0})\] is an equivalence of categories.
\end{theo}
\begin{proof} A short argument using (\ref{prop:moduli-ell-stack}) reduces us to the case where $S$ is an affine scheme and this case is the content of the Appendix of \cite{Drinfeld76}.
\end{proof}

\subsection{} Let $S$ be a Dedekind scheme, i.e.\ a one dimensional, regular and irreducible scheme. Let $f:\Spec(K)\to S$ be its generic point and let $E/\Spec(K)$.

\begin{theo}[N\'eron models] The functor $X/S\mto E(X\times_{S} \Spec(K))$ on \textup{smooth} schemes over $S$ is representable by a smooth, one dimensional group scheme $\Ner_S(E)/S$. Moreover, if $S'\to S$ is an \'etale map of Dedekind schemes and $\Spec(K')\to S'$ is the generic point of $S'$ then \[\Ner_S(E)\times_S S'=\Ner_{S'}(E\times_{\Spec(K)}\Spec(K')).\]
\end{theo}
\begin{proof} Representability of the functor is Theorem 3 \S 1.4 Chapter 1 of \cite{Raynaud90} and compatibility with \'etale base change is Proposition 2 (b) \S 1.2 of \cite{Raynaud90}.
\end{proof}

\begin{theo}\label{theo:criterion-of-good-reduction} Let $L/\Q$ be a finite extension, $v$ a place over $\Q^\sep$ lying over the prime $\p$ of $O_L$ with residue characteristic $p>0$ and let $E/L$ be an elliptic curve. Then $\Ner_{O_{L, \p}}(E)\to \Spec(O_{L, \p})$ is an elliptic curve if, for some prime $l\neq p$, the action of the inertia group $I_v\subset G(\Q^\sep/L)$ on $E[l^\infty](\Q^\sep)$ is trivial.
\end{theo}
\begin{proof} This follows from Theorem 1 of \cite{SerreTate68}.
\end{proof}

\section{Elliptic curves with complex multiplication} In this section we define families of elliptic curves with complex multiplication by the ring of integers $O_K$ of an imaginary quadratic field $K$ or, for short, CM elliptic curves. We give analogues for CM elliptic curves of several of the results of (\ref{sec:lubin-tate-modules}) given for Lubin--Tate $O$-modules. In particular, we consider the moduli stack $\mc{M}_\CM$ of CM elliptic curves and show that the stack $\mc{CL}_{O_K}$ of rank one $O_K$-local system acts on $\mc{M}_\CM$.

\subsection{}\label{subsec:imag-quad-set-up} For the remainder of this chapter we fix an imaginary quadratic field $K=\Q(\sqrt{-d})$ for $d\in \N$ square free and with ring of integers $O_K$. We note that $K$ has only one archimedian place $\infty$ and so $(K, \infty)$ satisfies the conditions of (\ref{subsec:global-field-with-special-place}) and we shall make use and notation (and in later sections theory) set up in \S \ref{section:special-place} of Chapter 1. As $K_\infty\isomto \C$ is algebraically closed every finite extension of $K$ is totally split at infinity so that $K^\infty=K^\ab$, although we shall continue to use the notation $K^\infty$.
The unit group $O_K^\times$ is finite and is equal to $\{\pm 1\}$ unless $K=\Q(\mu_n)$ for $n=4, 6$ in which case $O_K^\times=\mu_n$. The unique non-trivial automorphism of $K/\Q$ is denoted by $a\mto \overline{a}$. We shall be working solely in the category $\Sh_{O_K}$ and so by a sheaf $S$ we will mean a sheaf over $\Spec(O_K)$. In particular, to simplify some of the notation we will write $X\times Y$ for the product $X\times_{\Spec(O_K)} Y$ in $\Sh_{O_K}$.

\subsection{} An elliptic curve with complex multiplication by $O_K$ over $S$, or for short a CM elliptic curve over $S$, is an elliptic curve $E\to S$ (\ref{subsec:elliptic-curve-def}) equipped an $\underline{O_K}_S$-module structure which is strict with respect to the morphism $\underline{O_K}_S\to \mc{O}_S$ coming from the structure map $S\to \Spec(O_K)$ (cf.\ (\ref{subsec:strict-formal-modules})). A morphism of CM elliptic curves over $S$ is just a homomorphism of $\underline{O_K}_S$-modules. For $a\in \underline{O_K}_S(S)$ we write $[a]_E: E\to E$, or just $a: E\to E$, for the corresponding endomorphism (in particular for $a\in O_K\subset \underline{O_K}_S(S)$).

Finally, we write $\CM(S)$ for the category of CM elliptic curves over a sheaf $S$ and we write $\mc{M}_\CM$ for the stack over $\Sh_{O_K}$ whose fibre over $S$ is the category of CM elliptic curves over $S$ together with their isomorphisms.

\begin{prop}\label{prop:cm-endo} For any \textup{CM} elliptic curve $E/S$ the morphism \[\underline{O_K}_S\to \underline{\End}_S^{O_K}(E)\] is an isomorphism.
\end{prop}
\begin{proof} The claim is local on $S$ so may assume that $S$ is an affine scheme. That $\underline{O_K}_S\to \underline{\End}_S^{O_K}(E)$ is injective follows from the fact that both $\underline{\Z}_S\subset \underline{O_K}_S$ and $\underline{\Z}_S\to\underline{\End}_S(E)$ are monomorphisms. By Corollary 1, \S 4 of \cite{SerreTate68} the map $\underline{O_K}_S\to \underline{\End}_S^{O_K}(E)$ is an isomorphism when evaluated on any closed point $\Spec(k)\to S$. Therefore, if $f: E\to E$ is any morphism, for each closed point $h:\Spec(k)\to S$ there is an element $\alpha_h\in O_K$ such that $(f-[\alpha_h]_E)|_{\Spec(k)}=0$ and by rigidity (\ref{subsec:rigidity}) there is an open (and closed) neighbourhood $\Spec(k)\to U\subset S$ such that $(f-[\alpha_h]_E)|_U=0$. It follows that \[\underline{O_K}_S\to \underline{\End}_S^{O_K}(E)\] is an epimorphism and so is an isomorphism.
\end{proof}

\begin{prop}\label{prop:cm-tensoring-is-cm} Let $E/S$ be a \textup{CM} elliptic curve and $\mc{L}/S$ a rank one $O_K$-local system. Then $E\otimes_{O_K}\mc{L}$ is a \textup{CM} elliptic curve over $S$.
\end{prop}
\begin{proof} The claims are all local on $S$ so we may assume that $S$ is an affine scheme. It follows from (\ref{prop:tensor-properties}) that, at least locally on $S$, $E\otimes_{O_K}\mc{L}$ is representable by a proper, smooth, geometrically connected group scheme. It follows from (\ref{prop:tensor-properties-lie-hom}) that, when this is the case, the Lie algebra of $E\otimes_{O_K}\mc{L}$ is locally free of rank one so that the relative dimension of $E\otimes_{O_K}\mc{L}\to S$ is also one. Therefore $E\otimes_{O_K}\mc{L}$ is, locally on $S$, an elliptic curve so that it is in fact an elliptic curve over $S$ as the moduli stack of elliptic curves is a stack. Finally, $E\otimes_{O_K}\mc{L}$ has an obvious structure of an $\underline{O_K}_S$-module and it follows from (\ref{prop:tensor-properties-lie-hom}) that it is strict.
\end{proof}

\begin{rema} By (\ref{prop:cm-tensoring-is-cm}) above we find a functor \[\mc{M}_{CM}\times \mc{CL}_{O_K}\to \mc{M}_\CM: (E/S, \mc{L}/S)\mto E\otimes_{O_K}\mc{L}/S\] which we may view as defining an action of $\mc{CL}_{O_K}$ on $\mc{M}_{CM}$. We will show later (see (\ref{theo:mcm-is-a-torsor})) that, as with Lubin--Tate $O$-modules, this makes $\mc{M}_\CM$ a torsor under $\mc{CL}_{O_K}$.
\end{rema}

\begin{coro}\label{prop:cm-tensor-homs} For each pair $\mc{L}, \mc{L}'$ of rank one $O_K$-local systems over $S$ and each pair $E, E'$ of \textup{CM} elliptic curves over $S$ the natural map \[\underline{\Hom}_S^{O_K}(E, E')\otimes_{O_K}\underline{\Hom}_S^{O_K}(\mc{L}, \mc{L}')\to \underline{\Hom}_S^{O_K}(E\otimes_{O_K}\mc{L}, E'\otimes_{O_K}\mc{L}')\] is an isomorphism.
\end{coro}
\begin{proof} This follows from (\ref{prop:tensor-properties-lie-hom}).
\end{proof}

\begin{rema}\label{rema:cm-tensor-homs} The isomorphism of (\ref{prop:cm-tensor-homs}) together with (\ref{prop:cm-endo}) gives the particularly simple formula when $E=E'$ and $\mc{L}=\underline{O_K}_S$ and $\mc{L}'=\underline{L}_S$: \[\underline{L}_S\isomto \underline{\Hom}_S^{O_K}(E, E\otimes_{O_K}L): l\mto \id_{E}\otimes_{O_K} l.\]
\end{rema}

\subsection{} For each integral ideal $\a\subset O_K$ we write \[i_\a: E\to E\otimes_{O_K}\a^{-1}\] for the homomorphism induced by the inclusion $O_K\to \a^{-1}$. We define the $\a$-torsion in $E$ to be $E[\a]=\ker(i_\a)$. We have $E[\a]=\cap_{a\in \a} \ker(a)$ so that $E[\a]=\ker(a)$ if $\a=(a)$ is principal.

\begin{prop} The homomorphism $i_\a: E\to E\otimes_{O_K}\a^{-1}$ is finite locally free of degree $N\a$.
\end{prop}
\begin{proof} If $S$ is empty or if $\a=O_K$ the claim is obvious so we assume that $S\neq \emptyset$ and that $\a\neq O_K$. Using (\ref{rema:cm-tensor-homs}), the morphism $i_\a$ is equal to the zero map only if $S=\emptyset$, and it is an isomorphism if and only if $\a=O_K$ or $S=\emptyset$. Therefore, as $\a\neq O_K$ and $S\neq \emptyset$, the morphism $i_\a$ is finite locally free of degree greater than one.

As tensoring with rank one $O_K$-local systems is exact (\ref{prop:tensor-exact}) the kernel of $i_\a\otimes_{O_K}\b^{-1}$ is $E[\a]\otimes_{O_K}\b^{-1}$ and as $O_K/\a\otimes_{O_K}\b^{-1}\isomto O_K/\a$ (non-canonically) for all pairs of integral ideals $\a, \b$ we have \[E[\a]\otimes_{O_K}\b^{-1}\isomto E[\a].\] Therefore $\deg(i_\a\otimes_{O_K}\b^{-1})=\deg(i_\a)$ and \[\deg(i_{\a\b})=\deg((i_\b\otimes_{O_K}\b^{-1})\circ i_\b)=\deg(i_\a)\deg(i_\b).\] As $N{\a\b}=N\a N\b$ and $\deg(i_{\a\b})=\deg(i_\a)\deg(i_\b)$ we may assume that $\a=\p$ is a prime ideal in which case we find \[\deg(i_\p)\deg(i_{\overline{\p}})=\deg(i_{\p\overline{\p}})=\deg([N\p]_E)=N\p^2.\] If $\p=\overline{\p}$ then $\deg(i_\p)^2=N\p^2$ and so $\deg(i_\p)=N\p$ and if $\p\neq \overline{\p}$ then, as $N\p$ is prime and as both $\deg(i_\p),\deg(i_{\overline{\p}})\neq 1$, it also follows that $\deg(i_\p)=N\p$.
\end{proof}

\subsection{} We associate the following sheaves of groups to $E/S$:
\begin{enumerate}[label=\textup{(\roman*)}]
\item For each maximal ideal $\p\subset O_K$ the $\p$-divisible group of $E$ is the ind-finite locally free group scheme $E[\p^{\infty}]:=\colim_n E[\p^n]$. It is a $\lim_n \underline{O_K/\p^n}_{S}$-module.
\item The torsion subgroup of $E/S$ is \[E[\tors]:=\colim_{\a} E[\a]=\bigoplus_{\p} E[\p^\infty].\] It is a $\lim_\a \underline{O_K/\a}_{S}$-module.
\item The formal group of $E/S$ is $\widehat{E}:=\colim_k \Inf_{S}^{(k)}(E)$ (cf.\ \ref{subsec:formal-completion}). It is a strict formal $\underline{O_K}_S$-module of dimension one.
\end{enumerate}

\begin{prop}\label{prop:p-torsion-classification} Let $E/S$ be a \textup{CM} elliptic curve and $\p\subset O_K$ a prime ideal. Then
\begin{enumerate}[label=\textup{(\roman*)}]
\item\label{item:etale} $\p$ is invertible on $S$ if and only if $E[\p]$ is finite and \'etale over $S$. In this case, $E[\p^\infty]$ is \'etale over $S$ and is locally isomorphic to the constant $\lim_n \underline{O_K/\p^n}_S$-module \[\colim_n\underline{\p^{-n}/O_K}_S=\underline{K_\p/O_{K_\p}}_S,\] and
\item $\p$ is locally nilpotent on $S$ if and only if $E[\p^\infty]=\widehat{E}$. In this case, $E[\p^\infty]$ is a Lubin--Tate $O_{K_\p}$-module over $S$.
\end{enumerate}
\end{prop}
\begin{proof} The claims are true if and only if they are locally on $S$ and so we may assume that $S$ is an affine scheme.

(i) As $i_{\p^n}$ is finite locally free of degree $N\p^n$ (in particular, faithfully flat) its kernel $E[\p^n]$ is \'etale over $S$ if and only if the morphism $E\to E\otimes_{O_K}\p^{-n}$ is \'etale, or equivalently, induces an isomorphism $\underline{\Lie}_{E/S}\to \underline{\Lie}_{E\otimes_{O_K}\p^{-1}/S}=\underline{\Lie}_{E/S}\otimes_{O_K}\p^{-1}$ and this is true if and only if $\p$ is invertible on $S$.

For the second claim we may, after localising $S$, assume that $E[\p^\infty]$ is constant and hence that $E[\p^n]$ is constant for all $n\geq 0$. For each $n\geq 0$ let $E_n$ be a finite $O_K$-module with $E[\p^n]\isomto \underline{E_n}_S$ for $n\geq 0$. For $n\geq 0$, $E_{n}$ is an $O_K/\p^n$-module and the $\p^n$-torsion of $\colim_n E_n$ is equal to $E_n$. As $\# E_{n+1}=N\p^{n+1}$ if $E_{n+1}$ is not a free rank one $O_K/\p^{n+1}$-module it must consist of $\p^n$-torsion and therefore $E_{n+1}\subset E_n$ but $\# E_n=N\p^n$ so that this is impossible. It follows that for each $n\geq 0$ the $O_K/\p^n$-module $E_n$ is free of rank one and the inclusions $E_{n}\to E_{n+1}$ identify $E_n$ with the $\p^n$-torsion of $E_{n+1}$. Therefore, we may fix isomorphisms $E_n\isomto \p^{-n}/O_K$ for all $n\geq 0$ with the property that the inclusions $E_{n}\subset E_{n+1}$ become the natural inclusions $\p^{-n}/O_K\subset \p^{-n-1}/O_K$. Thus \[E[\p^\infty]\isomto \colim_n \underline{\p^{-n}/O_K}_S.\]

(ii) First assume that $\widehat{E}=E[\p^\infty]$. Then $S\to \Spec(O_K)$ factors through \[\Spf(O_{K_\p})=\colim_{n} \Spec(O_K/\p^{n+1})\to \Spec(O_K)\] if and only if $S\times\Spec(O_K[\p^{-1}])=\emptyset$. Thus we may assume that $S=S\times\Spec(O_K[\p^{-1}])$ and show that it is empty. 
By (i) the sheaf $\widehat{E}=E[\p^\infty]$ is \'etale over $S$. In this case it follows that, for all $k\geq 0$, the scheme $\Inf_{S}^{(k)}(E)$ is unramified over $S$ as it is a sub-scheme of the \'etale scheme $\widehat{E}=E[\p^\infty]$. Therefore, the morphism defining the zero section $S\to \Inf_{S}^{(k)}(E)$ is both a nilpotent immersion and an open immersion which is possible if and only if $S=\Inf_S^{(k)}(E)$ and this is possible if and only if $S=\emptyset$.

Conversely, assume that $S\to \Spec(O_K)$ factors through $\Spf(O_{K_\p})\subset \Spec(O_K)$ and write $\widehat{E}[\p^n]=\ker(\widehat{E}\to \widehat{E}\otimes_{O_K}\p^{-n})$. As $S$ is affine it follows that $\p$ is nilpotent on $S$ and as the action of $\underline{O_K}_S$ on $\widehat{E}$ is strict it follows that \[\colim_n \widehat{E}[\p^n]=\widehat{E}\] (just as in the proof of (ii) of (\ref{prop:formal-o-mod-frob})). In particular, $\widehat{E}\subset E[\p^\infty]$ and the strict $\underline{O_K}_S$-module structure on $\widehat{E}$ extends uniquely to a strict $\lim_{n} \underline{O_K/\p^{n}}_S$-module structure.

We are now reduced to showing that for each $r\geq 0$ the zero section $S\to E[\p^r]$ is a nilpotent immersion as this will give $E[\p^r]\subset \widehat{E}$. As $S\times\Spec(\F_\p)\to S$ is a nilpotent immersion, we may replace $S$ by $S\times\Spec(\F_\p)$ and assume that $S$ has characteristic $\p$. From (iii) of (\ref{prop:formal-o-mod-frob}) we find closed immersions \[\ker(\Fr_{\widehat{E}/S}^{N\p^r})\subset \widehat{E}[\p^r]\subset E[\p^r].\] But, as $\widehat{E}$ is a smooth formal group of dimension one, $\ker(\Fr_{\widehat{E}/S}^{N\p^r})$ is finite locally free of rank $N\p^r$, as is $E[\p^r]$. Therefore the closed immersion \[\ker(\Fr_{\widehat{E}/S}^{N\p^r})\subset E[\p^r]\] is an isomorphism. It follows that $S\to E[\p^r]$ is a nilpotent immersion, so that $E[\p^r]\subset \widehat{E}$, and therefore \[\widehat{E}=E[\p^\infty].\] This proves the first claim. For the second, the sheaf of groups $\widehat{E}=E[\p^\infty]$ is a strict formal $\widehat{O_{K_\p}}_S$-module of dimension one. Moreover, as $\widehat{E}=E[\p^\infty]$ we have \[\ker(i_\p: \widehat{E}\to \widehat{E}\otimes_{O_K}\p^{-1})=E[\p]\] so that \[i_\p:\widehat{E}\to \widehat{E}\otimes_{O_{K_\p}}\p^{-1}\] is finite locally free of degree $N\p$. This is precisely the definition of a Lubin--Tate $O_{K_\p}$-module (\ref{subsec:definition-lubin-tate-module}).
\end{proof}

\begin{coro}\label{coro:level-isom-cover} Let $E/S$ be a \textup{CM} elliptic curve.
\begin{enumerate}[label=\textup{(\roman*)}]
\item The morphism $\underline{\Isom}_S^{O_K}(E[\a], \underline{O_K/\a}_S)\to S$ is affine and \'etale, factors through $S[\a^{-1}]=S\times\Spec(O_K[\a^{-1}])\subset S$ and defines an affine \'etale cover \[\underline{\Isom}_S^{O_K}(E[\a], \underline{O_K/\a}_S)\to S[\a^{-1}].\]
\item If $P\subset \Id_{O_K}$ is any set of ideals which do not admit a common divisor the family \[(\underline{\Isom}_S^{O_K}(E[\a], \underline{O_K/\a}_S)\to S)_{\a\in P}\] is an affine \'etale cover of $S$.
\end{enumerate}
\end{coro}
\begin{proof} (i) If there exists an isomorphism $E_T[\a]\isomto \underline{O_K/\a}_T$ over some affine $S$-scheme $T$ then $E_T[\a]$ is \'etale, so that by (i) of (\ref{prop:p-torsion-classification}) the map $T\to S$ factors through the affine \'etale sub-scheme $S[\a^{-1}]\subset S$. It follows that $\underline{\Isom}_S^{O_K}(E[\a], \underline{O_K/\a}_S)\to S$ factors through $S[\a^{-1}]$ and so base changing along $S[\a^{-1}]\to S$ we may assume $\a$ is invertible on $S$.

Applying (i) of (\ref{prop:p-torsion-classification}) to the prime power divisors of $\a$ we find that $E[\a]$ is locally isomorphic to $\underline{\a^{-1}/O_K}_S$ which is isomorphic to $\underline{O_K/\a}_S$. This shows that \[\underline{\Isom}_S^{O_K}(E[\a], \underline{O_K/\a}_S)\to S\] is an epimorphism so that by descent, to show that it is finite and \'etale, we may assume that $E[\a]= \underline{O_K/\a}_S$. In this case we have \[\underline{\Isom}_S^{O_K}(\underline{O_K/\a}_S, \underline{O_K/\a}_S)=\underline{(O_K/\a)^\times}_S\] and the claim is clear.

(ii) The hypothesis on $P$ implies that $(S[\a^{-1}]\to S)_{\a\in P}$ is a cover of $S$ so that this follows from (i).
\end{proof}

\begin{coro}\label{coro:tensor-p-equals-frobenius} For each $n\geq 0$ there is a unique isomorphism of functors \[\nu_{\p^n}: -\otimes_{O_K}\p^{-n} \isomto \Fr^{N\p^n*}(-)\] on $\mc{M}_{CM}\times\Spec(\F_\p)$ such that for all \textup{CM} elliptic curves $E$ over characteristic $\p$-sheaves $S$ the diagram \[\xymatrix{&E\ar[dr]^{\Fr_{E/S}^{N\p^n}}\ar[dl]_-{i_{\p^n}}&\\
E\otimes_{O_K}\p^{-n} \ar[rr]^{\nu_{\p^n}}_\sim && \Fr^{N\p*}(E)}\] commutes.
\end{coro}
\begin{proof} The two homomorphisms $i_{\p^n}$ and $\Fr_{E/S}^{N\p^n}$ are finite locally free of degree $N\p^n$ so we need only show that their kernels are equal. As $\p$ is nilpotent on $S$, we have $\widehat{E}=E[\p^\infty]$ and $\widehat{E}$ is a Lubin--Tate $O_{K_\p}$-module by (ii) of (\ref{prop:p-torsion-classification}) so that by (\ref{lemm:p-torsion-lifts-frob-in-lubin-tate}) we find \[\ker(\Fr_{E/S}^{N\p^n})=\ker(\Fr_{\widehat{E}/S}^{N\p^n})=\widehat{E}[\p^n]=E[\p^n]=\ker(i_{\p^n}).\]
\end{proof}

\begin{rema} The above (\ref{coro:tensor-p-equals-frobenius}) is the first indication that the moduli stack $\mc{M}_\CM$ should admit a $\Lambda$-structure.
\end{rema}

\begin{coro}\label{prop:hom-ell-curves} Let $f: E\to E'$ be a homomorphism of \textup{CM} elliptic curves over $S$. Then there is a unique decomposition $S=\amalg_{\a\subset O_K}S_{(\a)}$ such $f_{S_{(0)}}$ is the zero map and such that $\ker(f_{S_{(\a)}})=E_{S_{(\a)}}[\a]$ for all $(0)\neq \a\subset O_K$.
\end{coro}
\begin{proof} By rigidity we may decompose $S$ as $S_{(0)}\amalg S_{\mathrm{isog}}$ where $f_{S_{(0)}}$ is the zero map and where $f_{S_\mathrm{isog}}$ is finite locally free of positive degree. Thus we may replace $S$ by $S_{\mathrm{isog}}$ and assume that $f$ is finite locally free. Then $\ker(f)\subset E[\tors]$ and \[\ker(f)=\bigoplus_{\p}\ker(f_\p)\] where $f_\p: E[\p^\infty]\to E'[\p^\infty]$ is the restriction of $f$ to the $\p$-divisible groups. As $\ker(f)$ is finite locally free so is $\ker(f_\p)$ for all prime ideals $\p$.

The claim is local so we may reduce to the case where $S=\Spec(A)$ is an affine scheme and, by passage to the limit, to when $S$ finitely presented over $\Spec(O_K)$. As $S$ is noetherian it admits a finite cover by connected affine schemes and so we may assume that $S$ is connected. In this case, the finite locally free group schemes $\ker(f_\p)$ have constant degree and we are reduced to showing that, locally on $S$, there exists an integer $n\geq 0$ such that $\ker(f_\p)=E[\p^n]$.

We continue to localise $S$ and so assume that $S=\Spec(A)$ for $A$ a local noetherian ring with maximal ideal $I$ and by descent that $A$ is $I$-adically complete. As $E$ and $E[\p^n]$ and $\ker(\f_\p)$ are all proper over $S$ to check that $\ker(f_\p)=E[\p^n]$ for some $n$ we may, by Grothendieck's formal GAGA, replace $S$ by $\Spec(A/I^r)$ for each $r\geq 0$ and assume that $S=\Spec(A)$ where $A$ is an artinian local ring. If $\p$ is invertible in $A$ then $E[\p^\infty]$ is locally isomorphic to the constant group scheme \[\underline{K_\p/O_{K_\p}}_S\] by (\ref{prop:p-torsion-classification}), from this and the corresponding fact for $K_\p/O_{K_\p}$, we see that any finite locally free $\underline{O_K}_S$-sub-module of $E[\p^\infty]$ is of the form $E[\p^n]$ for some $n\geq 0$. On the other hand, if $\p$ is nilpotent in $A$, then $\ker(f_\p)$ is the kernel of the homomorphism $f_\p : E[\p^\infty]\to E'[\p^\infty]$ of Lubin--Tate $O_{K_\p}$-modules and so the claim follows from (\ref{prop:lubin-tate-hom}).
\end{proof}

\section{Complex and $\p$-adic bases} In this section we give CM analogues of the classification of elliptic curves over complex schemes in terms of lattices, and of the Serre-Tate theorem. We use this to show there exists a CM elliptic curve $E\to \Spec(O_{K^\sep})$ (\ref{lemm:cm-over-separably-closed-fields}), that deformations of CM elliptic curves over $\p$-adic bases always exist and are unique (\ref{coro:cm-etale}).

\subsection{} Let us first consider complex bases. We fix a homomorphism $O_K\to \C$ so that we may view $\Spec(\C)$ as a $\Spec(O_K)$-scheme. Let $S$ be a locally finitely presented $\Spec(\C)$-scheme and $E/S$ a CM elliptic curve. Consider the exponential sequence associated to the analytification $E^\an/S^\an$ (cf.\ (\ref{subsec:complex-class})) \[0\to T_\Z(E/S)\to \underline{\Lie}_{E^\an/S^\an} \to E^\mathrm{an}\to 0.\] The strictness of the $\underline{O_K}_S$-action on $E$ implies that the homomorphism \[\underline{\Lie}_{E^\an/S^\an}\to E^\an\] is a homomorphism of $\underline{O_K}_{S^\an}$-modules. This makes the rank two $\underline{\Z}_{S^\an}$-local system $T_\Z(E/S)$ a rank one $O_K$-local system, which we shall denote by $T_{O_K}(E/S)$, and the homomorphism \[T_{O_K}(E/S)\otimes_{\underline{O_K}_{S^\an}}\mc{O}_{S^\an}\to \underline{\Lie}_{E^\an/S^\an}\] is now an isomorphism. As the automorphism sheaf of any rank one $O_K$-local system over $S$ or $S^\an$ is just the constant sheaf associated to the finite group $O_K^\times$ (from which one sees that every $O_K$-local system over $S$ or $S^\an$ is just a sum of finite \'etale $S$-schemes, or finite covering spaces of $S^\an$) by GAGA the functor sending a rank one $O_K$-local system on $S$ to the corresponding rank one $O_K$-local system on $S^\an$ is an equivalence. Therefore:

\begin{prop}\label{prop:cm-curves-complex} The functor \[E/S\mto T_{O_K}(E/S)\] is an equivalence between the category of \textup{CM} elliptic curves $E/S$ and the rank one $O_K$-local systems over $S$.
\end{prop}
\begin{proof} This follows from (\ref{prop:class-complex-elliptic}) and the remarks above.
\end{proof}
\begin{coro}\label{lemm:cm-over-separably-closed-fields} There exists a \textup{CM} elliptic curve $E$ over $\Spec(O_{K^\sep})$.
\end{coro}
\begin{proof} By (\ref{prop:cm-curves-complex}) we see that there exists a CM elliptic curve $E/\Spec(\C)$. Writing $\Spec(\C)=\lim_\lambda \Spec(A_\lambda)$ as a filtered inverse limit of finite type affine $\Spec(K)$-schemes by passage to the limit we find a CM elliptic curve $E_\lambda/\Spec(A_\lambda)$ for some $\lambda$. As $\Spec(A_\lambda)$ is of finite type over $\Spec(K)$ it admits a closed point $\Spec(L)\to \Spec(A_\lambda)$ with $L/K$ finite and we find a CM elliptic curve $E_\lambda\times_{\Spec(A_\lambda)}\Spec(L)$ over $\Spec(L)$.

We now show (essentially following the arguments of \cite{SerreTate68} only for the sake of completeness) that for any CM elliptic curve $E\to \Spec(L)$ over a finite extension $L/K$ there is a finite extension $L'/L$ such that $E\times_{\Spec(L)}\Spec(L')$ has good reduction everywhere. As $E$ has bad reduction at only finitely many primes, it is enough to show that for each prime $\p$ of $L$ there is a finite extension $L'/L$ such that $E\times_{\Spec(L)} \Spec(L')$ admits good reduction at all primes of $L'$ over $\p$, as then $E$ will obtain good reduction over the compositum of these finitely many fields.

So fix a prime $\p$ lying over the rational prime $p$ and let $\ell\neq p$ be another rational prime. Let $\rho_\ell:G(L^\sep/L)^\ab=G(L^\ab/L)\to (O_K\otimes_{\Z} \Z_\ell)^\times$ be the character defining the $O_K\otimes_{\Z}\Z_\ell$-linear action of $G(L^\sep/L)$ on \[E[\ell^\infty](\Spec(L^\sep))\isomto O_K\otimes_{\Z}\Q_\ell/\Z_\ell.\] First we claim that for each place $v$ of $L^\sep$ lying over $\p$ the restriction of $\rho_\ell$ to the inertia group $I_v\subset G(L^\sep/L)$ at $v$ has finite image and that this image is independent of $v$. As the target of $\rho_\ell$ is commutative and as the image of $I_v(L^\sep/L)$ in $G(L^\sep/L)=G(L^\ab/L)$ is the inertia subgroup $I_\p\subset G(L^\ab/L)$ at $\p$, which is independent of $v$, it is enough to show that $\rho_\ell(I_\p)$ is finite. But it is well known that $I_\p$ admits an open subgroup of finite index which is pro-$p$ and that $(O_K\otimes_{\Z}\Z_\ell)^\times$ admits an open subgroup of finite index which is pro-$\ell$. Therefore, as $\rho_\ell$ is continuous and $\ell\neq p$ the image of the inertia group $\rho_\ell(I_\p)$ must be finite. It now follows that there is an integer $n\geq 0$ with the property that \[\rho_\ell(I_\p)\to (O_K\otimes_{\Z}\Z_\ell)^\times \to (O_K\otimes_{\Z}\Z/\ell^n)^\times\] is injective. Taking $L'=L(E[\ell^n])$, the action of the inertia group $I_v(L^\sep/L')\subset G(L^\sep/L')$ at any place $v$ lying over $\p$ on $E'[\ell^\infty](\Spec(L^\sep))=E'[\ell^\infty](\Spec(L^\sep))$ is trivial and we now apply (\ref{theo:criterion-of-good-reduction}) to deduce that $E'/\Spec(L')$ has good reduction at all places of $L'$ lying over $\p$.
\end{proof}

\subsection{} Now let us consider $\p$-adic bases. Fix a prime ideal $\p$ of $O_K$ and let $S_0\to S$ be a nilpotent thickening of $\p$-adic sheaves and denote by $D_\p(S, S_0)$ the category whose objects are triples $(E_0/S_0, F/S, \rho)$ where
\begin{enumerate}[label=\textup{(\roman*)}]
\item $E/S_0$ is a CM elliptic curve,
\item $F/S$ is a Lubin--Tate $O_{K_\p}$-module, and
\item $\rho: F\times_S S_0\isomto E[\p^\infty]$ is an isomorphism of Lubin--Tate $O_{K_\p}$-modules,
\end{enumerate} \noindent and whose morphisms $(E_0/S_0, F/S, \rho)\to (E_0'/S_0, F'/S, \rho')$ are given by pairs $(g_0, g_\p)$ where $g_0: E_0\to E'_0$ is a homomorphism of CM elliptic curves, $g_\p: F\to F'$ is a homomorphism of Lubin--Tate $O_{K_\p}$-modules such that $\rho'\circ (g_\p\times_{S} S_0)=(g_0|_{E[\p^\infty]})\circ \rho$. There is an obvious functor \[\CM(S)\to D_\p(S_0, S): E/S\mto (E\times_S S_0, E[\p^\infty], \id_{E[\p^\infty]\times_S S_0}).\] The following is the theorem of Serre-Tate (\ref{theo:serre-tate}) adapted for CM elliptic curves.
\begin{prop}\label{cm-serre-tate} The functor $CM(S)\to D_\p(S_0, S)$ is an equivalence of categories.
\end{prop}
\begin{proof} Let $p$ be the rational prime lying under $\p$. We will show that to each object $(E_0/S_0, F/S, \rho)$ of $D_\p(S_0, S)$ (resp. morphism) one can functorially define a element of $D_p(S_0, S)$ (resp. morphism) (cf.\ (\ref{subsec:classical-serre-tate})). If $\overline{\p}=\p$ then this is clear as $E_0[\p^\infty]=E[p^\infty]$. On the other hand if $\overline{\p}\neq \p$ then $E_0[p^\infty]=E_0[\p^\infty]\times_{S_0} E_0[\overline{\p}^\infty]$ and $E_0[\overline{\p}^\infty]$ is \'etale over $S_0$ so that there is a unique deformation of $E_0[\overline{\p}^\infty]$ along $S_0\to S$ and the product over $S$ of this deformation with $F$ will give the desired deformation of $E_0[p^\infty]$. Regarding morphisms $(g_0, g_\p)$ the restriction $g_0|_{E[\overline{\p}^\infty]}$ lifts uniquely and the product of this with $g_\p$ defines the map on morphisms.

In both cases this defines a functor $D_\p(S_0, S)\to D_p(S_0, S)$ and by the classical Serre-Tate theorem a functor $D_\p(S_0, S)\to \mathrm{Ell}(S)$. By functoriality, if $E/S$ is the image of $(E_0/S_0, F/S, \rho)$ then $E/S$ admits the structure of an $\underline{O_K}_S$-module (deforming the corresponding structure on $E_0/S_0$). Moreover, the action of $\underline{O_K}_S$ on the Lie algebra of $E/S$ is strict as we have identifications of $\underline{O_K}_S$-modules \[\underline{\Lie}_{E/S}=\underline{\Lie}_{E[p^\infty]/S}=\underline{\Lie}_{E[\p^\infty]/S}=\underline{\Lie}_{F/S}\] and the action of $\underline{O_K}_S$ on $F$ is strict. In particular, the functor $D_\p(S_0, S)\to \mathrm{Ell}(S)$ factors as \[D_\p(S_0, S)\to \CM(S)\] and it is easily seen to be quasi-inverse to $CM(S)\to D_\p(S_0, S)$.
\end{proof}

\begin{coro}\label{coro:cm-etale} If $S_0\to S$ is a nilpotent immersion of $\p$-adic sheaves the functor \[\CM(S)\to \CM(S_0): E/S\mto E\times_{S} S_0/S_0\] is an equivalence of categories.
\end{coro}
\begin{proof} This follows from (\ref{cm-serre-tate}) combined with (\ref{coro:lift-lubin-tate}).
\end{proof}

\begin{coro}\label{prop:lifting} Let $S$ be a $\p$-adic affine scheme and let $E/S$ be a \textup{CM} elliptic curve. Then there exists a affine scheme $\widetilde{S}$, flat over $\Spec(O_K)$, a morphism $S\to \widetilde{S}$, and a \textup{CM} elliptic curve $\widetilde{E}/\widetilde{S}$ such that $\widetilde{E}\times_{\widetilde{S}} S\isomto E$.
\end{coro}
\begin{proof} By passage to the limit we may assume that $S=\Spec(A)$ with $A$ a finite type $O_K$-algebra. We then choose a surjection $A'\to A$ where $A'$ is a flat $O_K$-algebra. Letting $I$ be the kernel of $A'\to A$ we write $\widetilde{S}=\Spec(\lim_n A'/I^{n})=\Spec(\widehat{A}')$, $\widetilde{S}_n=\Spec(A'/I^n)$ and $\widetilde{S}_\infty=\colim_{n} \Spec(A'/I^n)$. We will show that there exists a CM elliptic curve $\widetilde{E}/\widetilde{S}$ with $\widetilde{E}\times_{\widetilde{S}} S\isomto E$ which, as $\widetilde{S}=\Spec(\widehat{A}')$ is flat over $\Spec(O_K)$, will prove the claim.

For each $n\geq 0$ there is a unique $\widetilde{E}_n/\widetilde{S}_n$ equipped with an isomorphism $\widetilde{E}_n\times_{\widetilde{S}_n}S\isomto E$ by (\ref{coro:cm-etale}). Therefore there is a unique CM elliptic curve $\widetilde{E}_\infty$ over the ind-scheme $\widetilde{S}_\infty=\colim_n \widetilde{S}_n$ with $\widetilde{E}_\infty\times_{\widetilde{S}_\infty}S=E$. By (\ref{prop:formal-gaga}) there is a unique elliptic curve $\widetilde{E}$ over $\widetilde{S}=\Spec(\widehat{A}')$ with $\widetilde{E}\times_{\widetilde{S}} S=E$ and moreover it admits an action of $\underline{O_K}_S$ compatible with that on $\widetilde{E}_n$ over $\widetilde{S}_n$ which (taking $n$ large) shows that the action is strict so that $\widetilde{E}/\widetilde{S}$ is a CM elliptic curve.
\end{proof}

\section{Isogenies and the action of $\mc{CL}_{O_K}$ on $\mc{M}_\CM$} In this section we continue the study of the action of $\mc{CL}_{O_K}$ on $\mc{M}_\CM$. We show that all pairs of CM elliptic curves over a fixed base are locally isogenous (\ref{lemm:local-isogenies}) and from this deduce that the action of $\mc{CL}_{O_K}$ on $\mc{M}_\CM$ makes it a torsor (\ref{theo:mcm-is-a-torsor}).

\begin{lemm}\label{lemm:local-isogenies} Let $E, E'/S$ be a pair of \textup{CM} elliptic curves. Then the family \[(\underline{\Isom}_S^{O_K}(E, E'\otimes_{O_K}\a^{-1})\to S)_{\a\in \Id_{O_K}}\] is a finite \'etale cover of $S$.
\end{lemm}
\begin{proof} The claim is local so we may assume that $S$ is an affine scheme and by passage to the limit that $S$ is finitely presented over $\Spec(O_K)$. For every pair of CM elliptic curves $E, E'/S$ the inclusion \[\underline{\Isom}_S^{O_K}(E, E')\to \underline{\Isom}_S(E, E')\] is easily seen to be a closed immersion and so it follows from (\ref{prop:isom-are-fin-unram}) that $\underline{\Isom}_S^{O_K}(E, E')\to S$ is finite and unramified. Therefore, for each integral ideal $\a\in \Id_{O_K}$ the sheaf $\underline{\Isom}_S^{O_K}(E, E'\otimes_{O_K}\a^{-1})$ is finite and unramified over $S$. It follows from (\ref{coro:cm-etale}) that the maximal open sub-scheme of $\underline{\Isom}_S^{O_K}(E, E'\otimes_{O_K}\a^{-1})$ which is \'etale over $S$ contains all of the special fibres and is therefore equal to all of $\underline{\Isom}_S^{O_K}(E, E'\otimes_{O_K}\a^{-1})$.
As the morphisms $(\underline{\Isom}_S^{O_K}(E, E'\otimes_{O_K}\a^{-1})\to S)_{\a\in \Id_{O_K}}$ are finite \'etale this family forms a cover of $S$ if and only if it is a cover after base change to each generic point of $S$. So we way assume that $S$ is either flat over $\Spec(O_K)$ (if a generic point has characteristic 0) or that $S$ is $\p$-adic for some prime $\p$ (if a generic point has characteristic $\p$). Let us reduce the second case to the first and assume that $S$ is $\p$-adic.

Applying (\ref{prop:lifting}) to $E$ and $E'$, then taking the product of the flat $\Spec(O_K)$-schemes over which $E$ and $E'$ can be extended, we find a flat $\Spec(O_K)$-scheme $\widetilde{S}$, a morphism $S\to \widetilde{S}$ and a pair of CM elliptic curves $\widetilde{E}$ and $\widetilde{E}'$ over $\widetilde{S}$ whose pull-backs to $S$ are isomorphic to $E$ and $E'$. We now see that the claim we wish to prove is true for $E, E'/S$ if it is true for $\widetilde{E}, \widetilde{E}'/\widetilde{S}$. As the generic points of $\widetilde{S}$ are all of characteristic 0, we have reduced the second case to the first and we may assume that $S=\Spec(F)$ where $F$ is a field of characteristic 0.

By passage to the limit we may assume that $F$ has finite transcendence degree over $K$, that there is a morphism $F\to \C$ and by base change that $F=\C$. The claim now follows from (\ref{prop:cm-curves-complex}), the fact that all $O_K$-local systems over $\Spec(\C)$ are constant, and the fact that if $L$ and $L'$ are two rank one $O_K$-modules there always exists some $\a\subset O_K$ and an isomorphism $L'\isomto L\otimes_{O_K}\a^{-1}$.
\end{proof}

\begin{prop}\label{prop:tensor-eval-isom} For each pair $E, E'/S$ of \textup{CM} elliptic curves over $S$ the sheaf $\underline{\Hom}_S^{O_K}(E, E')$ is a rank one $O_K$-local system and the evaluation homomorphism \[E\otimes_{O_K}\underline{\Hom}_S^{O_K}(E, E')\to E'\] is an isomorphism.
\end{prop}
\begin{proof} The claim is local and by (\ref{lemm:local-isogenies}) the set of $S$-sheaves \[(\underline{\Isom}_S^{O_K}(E, E'\otimes_{O_K}\a^{-1})\to S)_{\a\in \Id_{O_K}}\] is a finite \'etale cover, so replacing $S$ with any one of them we may assume that $E'= E\otimes_{O_K}\a^{-1}$ for some integral ideal $\a$. Composing the evaluation homomorphism with $\id_E\otimes_{O_K}i$, where $i:\underline{\a}_S^{-1}\isomto \underline{\Hom}_S^{O_K}(E, E\otimes_{O_K}\a^{-1})$ is the isomorphism of (\ref{prop:cm-tensor-homs}), the resulting map \[E\otimes_{O_K}\a^{-1}\isomto E\otimes_{O_K}\underline{\Hom}_S^{O_K}(E, E\otimes_{O_K}\a^{-1})\to E\otimes_{O_K}\a^{-1}\] is an isomorphism (in particular it is the identity) and so the second map is an isomorphism and we are done.
\end{proof}

\begin{coro}\label{coro:isom-iso-isom} Let $E, E'/S$ be a pair of \textup{CM} elliptic curves. Then $E$ and $E'$ are isomorphic if and only if they are isogenous and locally isomorphic.
\end{coro}
\begin{proof} The only if direction is clear. Conversely, let $f: E\to E'$ be an isogeny and assume that $E$ and $E'$ are locally isomorphic. As the cover $\amalg_{\a\subset O_K} S_{(\a)}=S$ in (\ref{prop:hom-ell-curves}) is by open and closed sub-sheaves, $E$ and $E'$ are isomorphic over $S$ if and only if they are after base change to each $S_{(\a)}$. Therefore we may assume that $\ker(f)=E[\a]=\ker(i_\a)$, so that $f$ factors as $E\to E\otimes_{O_K}\a^{-1}\isomto E'$. Now as $E$ and $E'\isomto E\otimes_{O_K}\a^{-1}$ are locally isomorphic it follows that $\underline{O_K}_S$ and $\underline{\a^{-1}}_S$ are locally isomorphic. Any such and isomorphism is locally constant from which it follows that there exists an isomorphism $\a^{-1}\isomto O_K$ and we get \[E'\isomto E\otimes_{O_K}\a^{-1}\isomto E.\]
\end{proof}

\begin{theo}\label{theo:mcm-is-a-torsor} The functor \[\mc{M}_\CM\times \mc{CL}_{O_K}\to \mc{M}_\CM\times \mc{M}_\CM: (E, \mc{L})\mto (E, E\otimes_{O_K}\mc{L})\] is an equivalence of stacks and $\mc{M}_\CM$ is locally (over $\Spec(O_K)$) equivalent to $\mc{CL}_{O_K}$.
\end{theo}
\begin{proof} The functor in question is essentially surjective as, given any pair $(E/S, E'/S)\in \mc{M}_\CM(S)\times \mc{M}_\CM(S)$, by (\ref{prop:tensor-eval-isom}) we have \[(E/S, E'/S)\isomto (E/S, E\otimes_{O_K}\underline{\Hom}_S(E, E')/S).\] Full faithfulness is the bijectivity of the map \[\Isom_S^{O_K}(E, E')\times \Isom_S^{O_K}(\mc{L}, \mc{L'})\to \Isom_S^{O_K}(E, E')\times\Isom_S^{O_K}(E\otimes_{O_K}\mc{L}, E'\otimes_{O_K}\mc{L}').\]
If $\Isom_S^{O_K}(E, E')= \emptyset$ this is clear. If $\Isom_S^{O_K}(E, E')\neq \emptyset$ we may assume that $E=E'$ and instead show that \[\Isom_S^{O_K}(\mc{L}, \mc{L'})\to \Isom_S^{O_K}(E\otimes_{O_K}\mc{L}, E\otimes_{O_K}\mc{L}')\] is bijective but this follows from (\ref{prop:cm-tensor-homs}).

For the second statement, if $E/\Spec(O_K^\sep)$ is any CM elliptic curve (\ref{lemm:cm-over-separably-closed-fields}) the functor \[ \mc{CL}_{O_K}\times\Spec(O_{K^\sep})\to \mc{M}_{\CM}\times\Spec(O_{K^\sep})\] sending a rank one $O_K$-local system $\mc{L}$ over a $\Spec(O_K^\sep)$-sheaf $p: S\to \Spec(O_{K^\sep})$ to $p^*(E)\otimes_{O_K}\mc{L}$ is the desired local equivalence.
\end{proof}

\begin{rema}\label{rema:rohr-q} One might ask whether, as in the analogous situation of Lubin--Tate $O$-modules (\ref{coro:lubin-tate-moduli-trivial}), there exists a CM elliptic curve over $\mc{E}/\Spec(O_K)$ inducing a trivialisation of the $\mc{CL}_{O_K}$-torsor $\mc{M}_\CM$, i.e.\ an equivalence of stacks: \[\mc{CL}_{O_K}\isomto \mc{M}_{\CM}: \mc{L}/S\mto E_S\otimes_{O_K}\mc{L}/S\] We shall show later (see (\ref{rema:rohr})) that in general this is not the case (and for non-trivial reasons).
\end{rema}

\begin{coro}\label{coro:hom-map-reduction-epi} Let $E$ and $E'$ be a pair of \textup{CM} elliptic curves over $S$ and assume that $\a$ is invertible on $S$. Then the homomorphism \[\underline{\Hom}_S^{O_K}(E, E')\to \underline{\Hom}_S^{O_K}(E[\a], E'[\a])\] is an epimorphism.
\end{coro}
\begin{proof} We may work locally on $S$ and so assume that $E'\isomto E\otimes_{O_K} \b^{-1}$. We may find a $k\in K^\times$ such that $\b(k)$ is prime to $\a$ and using the isomorphism $k: \b\isomto \b(k)$ we may assume $\b$ is prime to $\a$. In this case, the restriction of $i_\b$ to the $\a$-torsion defines an isomorphism $E[\a]\isomto E[\a]\otimes_{O_K}\b^{-1}$ which, as $E[\a]$ and $E[\a]\otimes_{O_K}\b^{-1}$ are locally isomorphic to $\underline{O_K/\a}_S$, induces an isomorphism \[\underline{O_K/\a}_S\isomto \underline{\Hom}_S(E[\a], E[\a]\otimes_{O_K}\b^{-1}): a\mto a\cdot i_\b\] and it follows that \[\underline{\Hom}_S^{O_K}(E, E\otimes_{O_K}\b^{-1})\to \underline{\Hom}_S^{O_K}(E[\a], E[\a]\otimes_{O_K}\b^{-1})\] is an epimorphism as the image contains $i_\b$.
\end{proof}

\section{The global reciprocity map and CM elliptic curves} We now consider CM elliptic curves over $\Spec(F)$ where $F$ is a field of arbitrary characteristic. First, we define a homomorphism $[-]_F$ from the Galois group $G(F^\sep/F)$ into a certain class group (which depends only on the characteristic of the field $F$) (\ref{subsec:def-morphism-h}). We then prove a relation between this homomorphism $[-]_F$ and the character $\rho_{E/F}$ defining the action of $G(F^\sep/F)$ on the torsion of a CM elliptic curve $E/\Spec(F)$ (\ref{prop:h-homs-and-characters}). Moreover, we show the character $\rho_{E/F}$ determines $E/\Spec(F)$ upto isogeny (\ref{prop:cm-iso-iff-h-equal}) and we use the homomorphism $[-]_F$ to classify exactly which characters $\rho$ of $G(F^\sep/F)$ are of the form $\rho_{E/F}$ for some CM elliptic curve $E/\Spec(F)$ (\ref{prop:cm-exist-given-h}). In (\ref{prop:compute-the-homomorphisms-h}) we compute the homomorphisms $[-]_F$ when $F=K$, $F=K_\p$ and $F=\F_\p$ for $\p$ a prime. In particular, for $F=K$ the homomorphism $[-]_K$ takes the form \[[-]_K: G(K^\sep/K)\to \CL_{O_K, \infty}\] and we show that for $\sigma\in G(K^\sep/K)$ we have \begin{equation}\theta_K([\sigma]_K)=\sigma|_{K^\infty}\label{eqn:main-theorem}\end{equation} where $\theta_K: \CL_{O_K, \infty}\to G(K^\infty/K)$ is reciprocity map (\ref{eqn:define-theta}). This fact is, at least in spirit, equivalent to the main theorem of complex multiplication (see Theorem 5.4 of \cite{Shimura94}). The method we use to derive this fact is quite similar to (and in fact reliant on) the method used to prove (\ref{theo:main-theorem-lubin-tate}) for Lubin--Tate $O$-modules in Chapter 1. Finally, we use these results to derive a sharpening of the criterion of good reduction adapted to CM elliptic curves (\ref{coro:image-of-inertia-group}).

The results of this section are, at least when $F$ is a finite extension of $K$, probably more or less known when translated into the language of algebraic Hecke characters though the proofs we give are, to the best of the author's knowledge, original. We would like to emphasise that the rather abstract approach taken -- which eschews Hecke characters -- works for all fields $F$ simultaneously and allows for more conceptual proofs.

\subsection{}\label{subsec:def-morphism-h} Let $F$ be a field over $O_K$ with separable closure $F^\sep$, let $S=\Spec(F^\sep)$ and let $\f\in \Id_{O_K}$ be invertible on $\Spec(F)$. By (\ref{lemm:cm-over-separably-closed-fields}) there exists a CM elliptic curve $E/S$. Moreover, if $E'/S$ is another CM elliptic curve then there exists a rank one projective $O_K$-module $L$ and an isomorphism $f: E\otimes_{O_K}L\isomto E'.$ Of course, we could take $L=\Hom_{O_K}(E, E')$ and $f$ the evaluation map, but for what follows it will be more useful to work with arbitrary modules $L$ and isomorphisms $f: E\otimes_{O_K}L\isomto E'$.

In particular, for each $\sigma\in G(F^\sep/F)$ there is a rank one projective $O_K$-module $L_\sigma$ and an isomorphism \[f_\sigma: E\otimes_{O_K}L_\sigma\isomto \sigma^*(E).\] This isomorphism, restricted to the $S=\Spec(F^\sep)$-points of the $\f$-torsion, is then given by \[E[\f](S)\otimes_{O_K}L_\sigma\stackrel{E[\f](\sigma)\otimes\lambda_\sigma}{\longrightarrow} E[\f](\sigma_!(S))=\sigma^*(E)[\f](S)\] for a unique level-$\f$ structure $\lambda_\sigma: L_\sigma\to O_K/\f$ on $L_\sigma$. It is clear from this construction that the class \[[\sigma]_{F, \f}:=(L_\sigma, \lambda_\sigma)\in CL_{O_K}^{(\f)}\] is independent of the choice of $L_\sigma$ and the isomorphism $f_\sigma: E\otimes_{O_K}L_\sigma\isomto \sigma^*(E)$. Moreover, as every other CM elliptic curve $E'/S$ is of the form $E\otimes_{O_K}\a$ for some ideal $\a$, having chosen $L_\sigma$ and an isomorphism $f_\sigma: E\otimes_{O_K} L_\sigma\isomto \sigma^*(E)$, the map $f_\sigma\otimes_{O_K} \a$ defines an isomorphism \[(E\otimes_{O_K}\a)\otimes_{O_K}L_\sigma = E\otimes_{O_K}L_\sigma\otimes_{O_K}\a\stackrel{f_\sigma\otimes_{O_K} \a }\longrightarrow \sigma^*(E)\otimes_{O_K}\a=\sigma^*(E\otimes_{O_K}\a)\] which induces on the $S$-valued points of the $\f$-torsion \[(E\otimes_{O_K}\a)[\f](S)\otimes_{O_K} L_\sigma \stackrel{(E\otimes_{O_K}\a)[\f](\sigma)\otimes \lambda_\sigma}{\longrightarrow}(E\otimes_{O_K}\a)[\f](\sigma_!(S)).\] Therefore the class $[\sigma]_{F, \f}$ is also independent of the choice of $E/S$.

If $F\to F'$ is a field extension and \[\mathrm{res}:G(F'^\sep/F')\to G(F^\sep/F)\] denotes the restriction map then it follows easy from the definitions that \begin{equation} [-]_{F',\f}=[-]_{F, \f}\circ \mathrm{res}.\label{eqn:h-restriction}\end{equation} This relation implies that once we known $[-]_F$ for $F=K$, and $F=\F_\p$ for each prime ideal $\p$, we essentially know $[-]_F$ for any field $F$. The computation of the map $[-]_F$ for these values of $F$ will be given in (\ref{prop:compute-the-homomorphisms-h}).

If $\tau\in G(F^\sep/F)$ then the composition \[E\otimes_{O_K}L_\sigma\otimes_{O_K}L_\tau\stackrel{f_\sigma\otimes_{O_K}\id_{L_\tau}}{\longrightarrow}\sigma^*(E)\otimes_{O_K}L_\tau \stackrel{\sigma^*(f_\tau)}{\longrightarrow} \sigma^*(\tau^*(E))\] induces on the $S$-points of the $\f$-torsion the map \[(E[\f](\sigma^{-1}\circ \tau\circ \sigma)\otimes_{O_K}\lambda_\tau)\circ (E[\f](\sigma)\otimes_{O_K}\lambda_\sigma\otimes_{O_K}\id_{L_\tau})=E[\f](\tau\circ \sigma)\otimes_{O_K}\lambda_\sigma\otimes_{O_K}\lambda_\tau\] so that \[[\sigma\tau]_{F, \f}=(L_\tau\otimes_{O_K}L_\sigma, \lambda_\tau\otimes_{O_K}\lambda_\sigma)=(L_\tau, \lambda_\tau)(L_\sigma, \lambda_\tau)=[\sigma]_{F, \f}[\tau]_{F, \f}.\] In other words, $[-]_{F, \f}: G(F^\sep/F)\to \CL_{O_K}^{(\f)}$ is a homomorphism.

\subsection{} The homomorphism of (\ref{def:bracket-homo}) \[[-]_\f: (A\otimes_{O_K}K)^\times\to CL_{O_K}^{(\f)}\] restricted to $A_{O_K}^\times$ factors through the quotient $A_{O_K}^\times\to (O_K/\f)^\times$ and we denote the resulting map by the same symbol \[[-]_\f: (O_K/\f)^\times\to \CL_{O_K}^{(\f)}.\] Note that $\ker([-]_\f)=\im(O_K^\times\to (O_K/\f)^\times)\subset (O_K/\f)^\times$.

\subsection{}\label{subsec:definition-h-for-cm-curves} Now let $E/F$ be a CM elliptic curve and denote by \[\rho_{E/F, \f}: G(F^\sep/F)\to (O_K/\f)^\times\] the character defining the action of $G(F^\sep/F)$ on the rank one $O_K/\f$-module $E[\f](S)=E[\f](\Spec(F^\sep))$, i.e.\ \[E[\f](\sigma)=\rho_{E/F, \f}(\sigma): E[\f](S)\to E[\f](S).\] We also note for future reference that as $(O_K/\f)^\times=\Aut_{O_K}(E[\f](S))$ is abelian for each $\f$ it follows that the extension $F(E[\f])/F$ generated by the $\f$-torsion is an abelian extension of $F$. Moreover (as is obvious) the character $\rho_{E/F, \f}$ is continuous, as it vanishes on the open subgroup of $G(F^\sep/F)$ fixing $F(E[\f])$.

\begin{prop}\label{prop:h-homs-and-characters}
\begin{enumerate}[label=\textup{(\roman*)}]
\item The homomorphism $[-]_{F, \f}$ is continuous and
\item if $E/F$ is a \textup{CM} elliptic curve the diagram \[\xymatrix{G(F^\sep/F)\ar[dr]_{[-]_{F,\f}}\ar[r]^-{\rho_{E/F, \f}^{-1}} & (O_K/\f)^\times\ar[d]^{[-]_\f}\\
 & \CL_{O_K}^{(\f)}}\] commutes.
\end{enumerate}
\end{prop}
\begin{proof} (i) We shall reduce this claim to the second. By passage to the limit (applied to any CM elliptic curve over $\Spec(F^\sep)$) we may find a CM elliptic curve $E/\Spec(F')$ for some finite extension $F'/F$. If the diagram \begin{equation}\label{eqn:commutativity}\begin{gathered}\xymatrix{G(F^\sep/F')\ar[r]^-{\rho_{E/F', \f}^{-1}}\ar[d] & (O_K/\f)^\times\ar[d]^{[-]_\f}\\
G(F^\sep/F)\ar[r]^-{[-]_{F, \f}} & CL_{O_K}^{(\f)}}\end{gathered}\end{equation} commutes then, as the composition along the top and right is continuous, it follows that $[-]_{F, \f}|_{G(F^\sep/F')}$ is continuous. But $G(F^\sep/F')\subset G(F^\sep/F)$ is open and of finite index and $\CL_{O_K}^{(\f)}$ is discrete so it follows that $[-]_{F, \f}$ is continuous.

As $[-]_{F, \f}=[-]_{F, \f}|_{G(F^\sep/F')}$ (\ref{eqn:h-restriction}), the commutativity of (\ref{eqn:commutativity}) is equivalent to the commutativity of the diagram \[\xymatrix{G(F^\sep/F')\ar[dr]_{[-]_{F', \f}}\ar[r]^-{\rho_{E/F', \f}^{-1}} & (O_K/\f)^\times\ar[d]^{[-]_\f}\\
 & \CL_{O_K}^{(\f)}}\] and so we may assume that $F=F'$ and instead prove (ii).

(ii) Write $E_{F^\sep}=E\times_{\Spec(F)}\Spec(F^\sep)$ and let \[d_\sigma: E_{F^\sep}\to \sigma^*(E_{F^\sep})\] be the descent isomorphism i.e.\ (the isomorphism coming from the fact that $E=E\times_{\Spec(F)}\Spec(F^\sep)$ descends to $\Spec(F)$). Then $d_\sigma$ induces on the $S$-points of the $\f$-torsion the map \[E[\f](S)\stackrel{\rho_{E/F, \f}(\sigma)^{-1}\cdot E[\f](\sigma)}{\longrightarrow} E[\f](\sigma_!(S))=\sigma^*(E[\f])(S)\] and therefore \[[\sigma]_{F, \f}=(O_K, O_K\stackrel{\rho_{E/F, \f}(\sigma)^{-1}}{\longrightarrow}O_K/\f)=[\rho_{E/F, \f}(\sigma)^{-1}]_\f\in \CL_{O_K}^{(\f)}\]
\end{proof}
\begin{prop}\label{prop:props-of-h-for-cm-curves} Let $E/F$ be a \textup{CM} elliptic curve.
\begin{enumerate}[label=\textup{(\roman*)}]
\item Let $\mc{L}\in \mc{CL}_{O_K}(\Spec(F))$ and \[\rho_{\mc{L}/F}: G(G^\sep/F)\to \Aut_{O_K}(\mc{L}(\Spec(F^\sep)))=O_K^\times\] be the associated character. Then $\rho_{E\otimes_{O_K}\mc{L}/F}=\rho_{E/F, \f}\rho_{\mc{L}/F}$.
\item If $\tau: F \to F$ is any $O_K$-linear automorphism then \[\rho_{\tau^*(E)/F}(\sigma)=\rho_{E/F, \f}(\widetilde{\tau}^{-1}\sigma\widetilde{\tau})\] for each $\sigma\in G(F^\sep/F)$, where $\widetilde{\tau}$ denotes any extension of $\tau$ to $F^\sep$.
\end{enumerate}
\end{prop}
\begin{proof} These are immediate from the definition of $\rho_{E/F, \f}$ as the character defining the action of $G(F^\sep/F)$ on $E[\f](S)=E[\f](\Spec(F^\sep))$.
\end{proof}

\subsection{} For what follows we let $\f\in \Id_{O_K}$ vary over the integral ideals of $O_K$ which are invertible on $\Spec(F)$. We now take the limit over $\f$ to define the homomorphisms $[-]_F$, $\rho_{E/F}$ and $[-]$ by
\[[-]_F:=\lim_\f [-]_{F, \f}: G(F^\sep/F)\to \lim_\f \CL_{O_K}^{(\f)},\] \[\rho_{E/F}:=\lim_\f \rho_{E/F, \f}: G(F^\sep/F)\to \lim_{\f}(O_K/\f)^\times,\] and \[[-]:=\lim_\f[-]_\f: \lim_\f (O_K/\f)^\times\to \lim_\f \CL_{O_K}^{(\f)}.\] We find immediately from (ii) of (\ref{prop:h-homs-and-characters}) that \begin{equation}\label{eqn:rel-for-h}[-]_F=[\rho^{-1}_{E/F}].\end{equation}

\begin{prop}\label{prop:cm-iso-iff-h-equal} Let $E, E'/F$  be a pair of \textup{CM} elliptic curves. The following are equivalent:
\begin{enumerate}[label=\textup{(\roman*)}]
\item $\rho_{E/F}=\rho_{E'/F}$,
\item the character $\rho$ defining the action of $G(F^\sep/F)$ on \[\underline{\Hom}_{\Spec(F)}^{O_K}(E, E')(\Spec(F^\sep))\] is trivial,
\item the \'etale $\Spec(F)$-scheme $\underline{\Hom}_{\Spec(F)}^{O_K}(E, E')$ is constant,
\item $E$ and $E'$ are isogenous.
\end{enumerate} 
\end{prop}
\begin{proof} Let $\rho: G(F^\sep/F)\to O_K^\times$ be the character defining the action of $G(F^\sep/F)$ on the $\Spec(F^\sep)$-valued points of the rank one $O_K$-local system  $\underline{\Hom}_{\Spec(F)}^{O_K}(E, E')$.

(i) implies (ii): The isomorphism \[E\otimes_{O_K}\underline{\Hom}_{\Spec(F)^{O_K}}(E, E')\isomto E'\] combined with (i) of (\ref{prop:props-of-h-for-cm-curves}) gives $\rho_{E'/F}=\rho\cdot \rho_{E/F}$ so that $\rho$ is trivial if and only if $\rho_{E/F}=\rho_{E'/F}$.

(ii) implies (iii): This is clear from the definition of $\rho$.

(iii) implies (iv): If $\underline{\Hom}_{\Spec(F)}^{O_K}(E, E')$ is constant any non-zero $\Spec(F)$-section defines an isogeny $E\to E'$.

(iv) implies (i): If $f: E\to E'$ is an isogeny then $\ker(f)=E[\a]=\ker(i_\a)$ for some integral ideal $\a$ by (\ref{prop:hom-ell-curves}). Therefore $E\otimes_{O_K}\a^{-1}\isomto E'$ and by (i) of (\ref{prop:props-of-h-for-cm-curves}) we then get \[\rho_{E'/F}=\rho_{E\otimes_{O_K}\a^{-1}/F}=\rho_{E/F}.\] 
\end{proof}
\begin{prop}\label{prop:cm-exist-given-h} Let \[\rho=\lim_\f \rho_\f: G(F^\sep/F)\to \lim_{\f} (O_K/\f)^\times\] be a continuous homomorphism. Then there exists a \textup{CM} elliptic curve $E/F$ with $\rho_{E/F}=\rho$ if and only if the diagram \[\xymatrix{G(F^\sep/F)\ar[r]^-{\rho^{-1}}\ar[dr]_{[-]_F} & \lim\limits_{\f} (O_K/\f)^\times\ar[d]^{[-]} \\
& \lim\limits_{\f} \CL_{O_K}^{(\f)}}\] commutes.
\end{prop}
\begin{proof} The only if claim is (\ref{eqn:rel-for-h}). Conversely, by passage to the limit there exists a CM elliptic curve $E'/\Spec(F')$ where $F\subset F'$ is a finite extension. As the compositions of $\rho_{E'/F'}$ and $\rho|_{G(F^\sep/F')}$ with \[[-]:\lim_{\f}(O_K/\f)^\times\to \lim_{\f}\CL_{O_K}^{(\f)}\] coincide and $\ker([-])=O_K^\times$, the difference defines a character \begin{equation}\rho_{E'/F'}^{-1}\cdot \rho|_{G(F^\sep/F')}: G(F^\sep/F')\to O_K^\times=\ker([-])\subset \lim_{\f}(O_K/\f)^\times. \label{eqn:diff-char}\end{equation} Replacing $E'$ by the tensor product of $E'$ with any rank one $O_K$-local system with associated character (\ref{eqn:diff-char}) we may assume that $\rho_{E'/F'}=\rho|_{G(F^\sep/F')}$.

Let $\upsilon\in G(F^\sep/F')$ and $\widetilde{\sigma}\in G(F^\sep/F)$ with $\widetilde{\sigma}|_{F'}=\sigma$. We have by (ii) of (\ref{prop:props-of-h-for-cm-curves}) \[\rho_{\sigma^*(E')/F'}(\upsilon)=\rho_{E'/F'}(\sigma^{-1}\upsilon\sigma)=\rho(\sigma^{-1}\upsilon\sigma)=\rho(\upsilon)=\rho_{E'/F'}(\upsilon).\] Therefore $\rho_{\sigma^*(E')/F'}=\rho_{E'/F'}$ and by (ii) of (\ref{prop:cm-iso-iff-h-equal}) $E'$ and $\sigma^*(E')$ are isogenous.

For all $\f$ invertible on $\Spec(F)$ we have \[[\widetilde{\sigma}]_{F, \f}=[\rho_\f(\sigma)^{-1}]_\f=(O_K, O_K \stackrel{\rho_\f(\sigma)^{-1}}{\to} O_K/\f).\] This implies that, writing $E_{F^\sep}'=E\times_{\Spec(F)}\Spec(F^\sep)$, the CM elliptic curves $E_{F^\sep}'$ and $\widetilde{\sigma}^*(E_{F^\sep})$ are isomorphic. As $\widetilde{\sigma}|_{F}=\sigma$ this implies that $E'$ and $\sigma^*(E')$ are locally isomorphic. But $E$ and $E'$ are also isogenous and so by (\ref{coro:isom-iso-isom}) they are isomorphic.

Now fix an integral ideal $\f$ which separates units and is also invertible on $\Spec(F)$. As $[\widetilde{\sigma}]_{F, \f}=[\rho_\f(\widetilde{\sigma})^{-1}]$, and as $\f$ separates units, there is a unique isomorphism \[r_{\widetilde{\sigma}}: E'\isomto \sigma^*(E')\] which on $S$-valued points of the $\f$-torsion is the map \begin{equation} \label{eqn:unique-r}E'[\f](S)\stackrel{\rho_\f(\sigma)^{-1}\cdot E'[\f](\widetilde{\sigma})}{\longrightarrow} E'[\f](\sigma_!(S))=\sigma^*(E)[\f](S)\end{equation} where we view $\widetilde{\sigma}$ as a $\Spec(F')$-morphism \[\widetilde{\sigma}: \sigma_!(S)\to S.\]
 
If $\tau\in G(F'/F)$ and $\widetilde{\tau}\in G(F^\sep/F)$ satisfies $\widetilde{\tau}|_{F'}=\tau$ then the defining property (\ref{eqn:unique-r}) of the isomorphism \[r_{\widetilde{\sigma}\widetilde{\tau}}: E'\isomto (\tau\circ \sigma)^*(E')\] is also satisfied by \[\sigma^*(r_{\widetilde{\tau}})\circ r_{\widetilde{\sigma}}: E'\isomto (\tau\circ \sigma)^*(E')\] and so we get \begin{equation}\label{eqn:cocycle-condition} r_{\widetilde{\sigma}\widetilde{\tau}}=\sigma^*(r_{\widetilde{\tau}})\circ r_{\widetilde{\sigma}}.\end{equation} Moreover, if $\widetilde{\sigma}\in G(F^\sep/F')\subset G(F^\sep/F)$ then $r_{\widetilde{\sigma}}$ induces on the $S$-points of the $\f$-torsion the map \[E'[\f](S)\stackrel{\rho_\f(\widetilde{\sigma})\cdot E'[\f](\widetilde{\sigma})}{\longrightarrow} E[\f](S).\] However, by definition \[E'[\f](\widetilde{\sigma})= \rho_{E'/F', \f}(\widetilde{\sigma})\] so that $r_{\widetilde{\sigma}}$ induces on the $S$-valued points of the $\f$-torsion the map \[\rho_{E'/F', \f}(\widetilde{\sigma})^{-1}E[\f](\widetilde{\sigma})=\rho_{E'/F', \f}(\widetilde{\sigma})^{-1}\rho_{E'/F', \f}(\widetilde{\sigma})=\id_{E[\f](S)}.\] The uniqueness of $r_{\widetilde{\sigma}}$ now shows that $r_{\widetilde{\sigma}}=\id_{E'}$.

This combined with the relation (\ref{eqn:cocycle-condition}) shows that for all $\widetilde{\sigma}\in G(F^\sep/F)$ the isomorphism $r_{\widetilde{\sigma}}=r_\sigma$ depends only on $\sigma=\widetilde{\sigma}|_F$, has $r_{\id_{F'}}=\id_{E'}$ and satisfies \[r_{\sigma\tau}=\sigma^*(r_\tau)\circ r_\sigma.\] In other words, we have Galois descent data on $E'\to \Spec(F')$ relative to $\Spec(F')\to \Spec(F)$ and by construction the descended CM elliptic curve $E/\Spec(F)$ has $\rho_{E/F}=\rho$.
\end{proof}

\begin{prop}\label{prop:compute-the-homomorphisms-h}
\begin{enumerate}[label=\textup{(\roman*)}]
\item When $F=\F_\p$, in the notation of \textup{(\ref{subsec:ray-class-fields})}, we have for all $n\in \widehat{\Z}$: \[[\Fr^{N\p^n}]_{\F_\p} = \lim_{(\f, \p)=O_K}[\p]_\f^{-n}\in \lim_{(\f, \p)=O_K}CL_{O_K}^{(\f)}.\]
\item When $F=K_\p$, we have for all $\sigma\in G(K_\p^\sep/K)$: \[\sigma|_{K^\infty}=\theta_K([\sigma]_{K_\p})\] where the restriction $|_{K^\infty}$ is along any $K$-linear embedding $K^\infty\to K_\p^\sep$.
\item When $F=K$ we have for all $\sigma\in G(K^\sep/K)$: \[\sigma|_{K^\infty}=\theta_K([\sigma]_K).\]
\end{enumerate}
\end{prop}
\begin{proof} (i) As $G(\F_\p^\sep/\F_\p)$ is topologically generated by $\Fr^{N\p}$, by continuity it is enough to show that $[\Fr^{N\p}]_{\F_\p, \f}=[\p]_\f^{-1}$ for all $\f$ prime to $\p$. Write $S=\Spec(\F_\p^\sep)$, let $E/S$ be a CM elliptic curve and consider the isomorphism \[\nu_\p: E\otimes_{O_K} \p^{-1}\isomto \Fr^{N\p*}(E)\] of (\ref{coro:tensor-p-equals-frobenius}). By the definition of $\nu_\p$, the composition \[E \stackrel{i_\p}{\to} E\otimes_{O_K}\p^{-1}\isomto \Fr^{N\p^*}(E)\] is equal to the $N\p$-power relative Frobenius of $E$ which induces the map \[E[\f](S)\stackrel{E[\f](\Fr^{N\p})}{\longrightarrow} E[\f](\Fr^{N\p}_!(S))\] on the $S$-points of the $\f$-torsion. Therefore, the map $\nu_\p$ induces on the $S$-valued points of the $\f$-torsion the map \[E[\f](S)\otimes_{O_K}\p^{-1}\stackrel{E[\f](\Fr^{N\p})\otimes f}{\longrightarrow} E[\f](\Fr_{!}^{N\p}(S))\] where $f$ is level-$\f$ structure defined by (\ref{eqn:can-level-ideal}). Hence \[[\Fr^{N\p}]_{\F_\p, \f}=[\p^{-1}]_\f=[\p]_\f^{-1}.\]

(ii) Write $T=\Spec(O_{K_\p^\sep})$, $T_\p=\Spec(\F_\p^\sep)\subset T$, let $\mc{E}/T$ be a CM elliptic curve, write $\mc{E}_\p=\mc{E}\times_T T_\p$. By continuity it is enough to prove the claim for $\sigma\in W(K^\sep/K)=v_K^{-1}(\Z)$ and by multiplicativity for $\sigma$ with $v_K(\sigma)=n\geq 0$. Let $\sigma$ be such an element.

As $T=\Spec(O_{K^\sep_\p})$ admits no non-constant finite \'etale covers, the sheaf \[\underline{\Isom}_T^{O_K}(\mc{E}\otimes_{O_K}\p^{-1}, \sigma^*(\mc{E}))\] is finite and constant. Moreover, as $T$ is connected and $\sigma$ acts by $\Fr^{N\p^n}$ on the residue field $\F_\p^\sep$ of $O_{K_\p^\sep}$, there exists a unique isomorphism \[\nu_\sigma: \mc{E}\otimes_{O_K}\p^{-n}\isomto \sigma^*(\mc{E})\] lifting the isomorphism \[\nu_\p: \mc{E}_\p\otimes_{O_K}\p^{-n}\isomto \Fr^{N\p*}(\mc{E}_\p)\] of (\ref{coro:tensor-p-equals-frobenius-with-level}). We now compute the action of $\nu_\sigma$ on the $T$-points of the $\f$-torsion of $\mc{E}/T$.

First let $\f$ be prime to $\p$. As $\mc{E}[\f]$ is finite \'etale, we have $\mc{E}[\f](T)=\mc{E}[\f](T_\p)$ compatibly with the actions of $\sigma$ and $\Fr^{N\p^n}$. It now follows from (i) that the map induced by $\nu_\sigma$ on the $S$-points of the $\f$-torsion is \[\mc{E}[\f](T)\otimes_{O_K}\p^{-n} \stackrel{\mc{E}[\f](\sigma)\otimes f}{\to} \mc{E}[\f](\sigma_!(T))=\sigma^*(\mc{E})[\f](T)\] where $f$ is level-$\f$ structure defined by (\ref{subsec:ray-class-fields}).

Now let $\f=\p^r$ for some $r\geq 0$. Setting $\mc{E}_r:=\mc{E}[\p^r]$ and \[T_\infty=\Spf(O_{K_\p^\sep})=\colim_{n}\Spec(O_{K_\p^\sep}/\p^{n+1}),\] we obtain the commutative diagram
\begin{equation}\label{dia:cm-to-lubin-action}
\begin{gathered} \xymatrix{\mc{E}_r(T)\otimes_{O_K}\p^{-n}\ar[r]\ar[d] & \sigma^*(\mc{E}_r)(T)\ar[d]\\
\mc{E}_r(T_\infty)\otimes_{O_K}\p^{-n}\ar[r] & \sigma^*(\mc{E}_r)(T_\infty).}
\end{gathered}
\end{equation}
whose columns are isomorphisms as $\mc{E}_r$ is finite locally free over $S$. Now consider $F=\mc{E}[\p^\infty]\times_T T_\infty$ and set $F_r=F[\p^r]$. Then $F$ is a Lubin--Tate $O_{K_\p}$-module over the $\p$-adic sheaf $T_\infty$ by (\ref{prop:p-torsion-classification}) and the map on the bottom row of (\ref{dia:cm-to-lubin-action}) is equal to \[\nu_\sigma: F_r(T_\infty)\otimes_{O_{K_\p}}\p^{-n}\isomto \sigma^*(F_r)(T_\infty)\] where $\nu_\sigma$ is the unique map in (i) of (\ref{theo:main-theorem-lubin-tate}). It follows from (iii) of (\ref{theo:main-theorem-lubin-tate}) that this map is given by \[F_r(\sigma)\otimes_{O_{K_\p} }\chi_{K_\p}(\sigma):F_r(T_\infty)\otimes_{O_{K_\p}}\p^{-n}\isomto F[\p^{r}](\sigma_!(T_\infty))=\sigma^*(F[\p^r])(T_\infty)\] where $\sigma|_{K_\p^\ab}=(\chi_{K_\p}(\sigma), K_\p^\ab/K_\p)$. Therefore the map in the top row of (\ref{dia:cm-to-lubin-action}) is given by $\mc{E}_r(\sigma)\otimes_{O_{K}}\chi_{K_\p}(\sigma)$.

Writing $S=\Spec(K^\sep_\p)$ and considering the generic fibre $E:=\mc{E}\times_{T}S$, and arbitrary $\f$, the computations above show that the map induced by $\nu_\sigma$ \[E[\f](S)\otimes_{O_K}\p^{-n}\isomto E[\f](\sigma_!(S))=\sigma^*(E[\p^r])(S)\] is given by $E(\sigma)\otimes \chi_{K_\p}(\sigma)$ where we view $\chi_{K_\p}(\sigma)$ as \[\p^{-n} \stackrel{\chi_{K_\p}(\sigma)}{\longrightarrow} A_{O_K}\to O_K/\f.\] Therefore \[[\sigma]_{K_\p, \f}=(\p^{-n},\p^{-n}\stackrel{\chi_{K_\p}(\sigma)}{\to}A_{O_K}\to O_K/\f)=[\chi_{K_\p}(\sigma)]_\f.\] Taking limits we find $[\sigma]_{K_\p}=[\chi_{K_\p}(\sigma)]$ and by (\ref{eqn:local-global-class-compatibility}) we get \[\sigma|_{K^\infty}=\theta_K([\sigma]_{K_\p})\] where the restriction is along any $K$-embedding $K^\infty\to K_\p^\sep$.

(iii) We see from (ii) that for all primes $\p$ and all embeddings $K^\sep\to K_\p^\sep$ and all $\sigma\in G(K_\p^\sep/K)$, writing $\tau=\sigma|_{K^\sep}$ \[\tau|_{K^\infty}=(\sigma|_{K^\sep})|_{K^\infty}=\sigma|_{K^\infty}=\theta_K([\sigma]_{K_\p})=\theta_K([\sigma|_{K^\sep}]_K)=\theta_K([\tau]_K).\] However, the sub-group of $G(K^\sep/K)$ generated by elements of the form $\tau=\sigma|_{K^\sep}$ (for varying primes $\p$ and embeddings $K^\sep\to K_\p^\sep$) is dense. It follows, by continuity and multiplicativity, that for all $\tau\in G(K^\sep/K)$ we have \[\tau|_{K^\infty}=\theta_K([\tau]_K).\]
\end{proof}

\subsection{}\label{subsec:reinterpret-character-idelic} The homomorphisms $[-]_F$ can be (trivially) reinterpreted id\`elically using the isomorphism (\ref{eqn:define-theta}) \[\CL_{O_K, \infty}\stackrel{h_K^{-1}}{\to} (A\otimes_{O_K} K)^\times/K^\times\] and we do so only to make clear the relationship with the map (\ref{eqn:local-rec}). Let us do so here for $h_K$ and so define $\chi_K=[-]_K\circ h_K^{-1}$.
\begin{enumerate}[label=(\roman*)]
\item For all $\sigma\in G(K^\sep/K)$ we have $\sigma|_{K^\infty}=(\chi_K(\sigma), K^\infty/K)$ (cf.\ (iv) of (\ref{theo:main-theorem-lubin-tate})).
\item For each prime $\p$, each $\sigma\in W(K_\p^\sep/K_\p)\subset G(K_\p^\sep/K)$ and each $K$-linear $K^\sep\to K_\p^\sep$ we have \[\chi_K(\sigma|_{K^\sep})=\chi_{K_\p}(\sigma)\in K_\p^\times\subset (A_{O_K}\otimes_{O_K}K)^\times/K^\times\] where $\chi_{K_\p}$ is the homomorphism of (\ref{eqn:local-rec}).
\item If $K\subset L\subset K^\sep$ is a finite extension of $K$ and $\rho: G(K^\sep/L)\to A_{O_K}^\times$ is a continuous character then there exists a CM elliptic curve $E/\Spec(L)$ with $\rho_{E/L}=\rho$ if and only if the following diagram commutes \[\xymatrix{G(K^\sep/L)\ar[r]^-{\rho_{E/L}^{-1}}\ar[dr]_{\chi_{K}|_{G(K^\sep/L)}}& A_{O_K}^\times\ar[d]\\
& (A_{O_K}\otimes_{O_K}K)^\times/K^\times}\] where the right vertical arrow is the obvious map.
\end{enumerate}

\begin{rema}\label{rema:hecke-and-h} Let $L/K$ be a finite Galois extension and let $E/L$ be a CM elliptic curve. We now give a simple description of the algebraic Hecke character of $E/L$, which is a certain continuous homomorphism \[\psi_{E/L}: I_L\to K^\times\] satisfying $\psi_{E/L}|_{L^\times}=N_{L/K}$ (for the definition of $\psi_{E/L}$ see \S 7 of \cite{SerreTate68}).

View the homomorphism \[\rho_{E/L}: G(K^\sep/L)\to A_{O_K}^\times\] as a homomorphism \[\rho_{E/L} : G(L^\ab/L)=G(L^\sep/L)^\ab\to A_{O_K}^\times\] and for $s\in I_L$ write $s_{\mathrm{fin}}$ for the element of $(A_{O_K}\otimes_{O_K}L)^\times$ obtained by forgetting the components of $s$ at the places of $L$ lying over $\infty$. Then we claim that algebraic Hecke character $\psi_{E/L}$ associated to $E/L$ is given by \begin{equation}\psi_{E/L}: I_L\to (A_{O_K}\otimes_{O_K}K)^\times: s\mto \rho_{E/L}((s^{-1}, L^\ab/L))\cdot N_{L/K}(s_{\mathrm{fin}}).\label{def:hecke-char}\end{equation} We will not prove this but let us show that the map $\psi_{E/L}$ defined by (\ref{def:hecke-char}) satisfies $\psi_{E/L}|_{L^\times}=N_{L/K}$ and $\psi_{E/L}(I_L)\subset K^\times\subset (A_{O_K}\otimes_{O_K}K)^\times$.

For $a\in L^\times\subset I_L$ we have \[\rho_{E/L}(a)=\rho_{E/L}((a^{-1}, L^\ab/L))\cdot N_{L/K}(a_\mathrm{fin})=1\cdot N_{L/K}(a)\] so that $\rho_{E/L}|_{L^\times}=N_{L/K}$. Now computing the composition \[I_L\stackrel{\rho_{E/L}}{\to} (A\otimes_{O_K}K)^\times\stackrel{[-]}{\to} \CL_{O_K, \infty}\] we find: \begin{eqnarray*} [\rho_{E/L}((s^{-1}, L^\ab/L))\cdot N_{L/K}(s_\mathrm{fin})] &=&[\rho_{E/L}((s^{-1}, L^\ab/L))]\cdot [N_{L/K}(s_\mathrm{fin})]\\
&=& [(s^{-1}, L^\ab/L)]_L\cdot [N_{L/K}(s_{\mathrm{fin}})]\\
&=&[(N_{L/K}(s^{-1}), K^\ab/K)]_K\cdot [N_{L/K}(s_{\mathrm{fin}})]\\
&=&[(N_{L/K}(s^{-1}_\mathrm{fin}), K^\infty/K)]_K\cdot [N_{L/K}(s_{\mathrm{fin}})]\\
&=&[N_{L/K}(s_\mathrm{\fin}^{-1})]_K\cdot [N_{L/K}(s_\mathrm{fin})]\\
&=&1.
\end{eqnarray*}
Therefore, $\rho_{E/L}(I_L)\subset \ker([-]: (A\otimes_{O_K}K)^\times\to \CL_{O_K, \infty})=K^\times$.
\end{rema}

\begin{prop}\label{coro:image-of-inertia-group} Let $L/K$ be a finite Galois extension, $E/L$ a \textup{CM} elliptic curve, $v$ be a non-archimedian place of $K^\sep$ lying over the primes $\P$ of $O_L$ and $\p$ of $O_K$, and let $I_v\subset G(K^\sep/L)$ be the inertia group at $v$. Then $\rho_{E/L}(I_v)\subset O_K^\times \cdot O_{K_\p}^\times\subset A_{O_K}^\times$ and $E/L$ has good reduction at $\P$ if and only if $\rho_{E/L}(I_v)\subset O_{K_\p}^\times$.
\end{prop}
\begin{proof} The homomorphisms $\rho_{E/L}$ and $[-]_L$ both factor through $G(L^\ab/L)$ and we will denote the resulting homomorphisms by the same symbol. The image $[I_v]_L$ of the inertia group at $v$ is equal to $[(O_{L_\P}^\times, L^\ab/L)]_L$ where $(-, L^\ab/L): C_L\to G(L^\ab/L)$ is the global reciprocity map (\ref{prop:global-reciprocity}) and where $O_{L_\P}^\times\subset I_{L}/L^\times=C_L$.

As $[-]_L=[-]_K|_{G(K^\sep/L)}$, the compatibility of the reciprocity maps with the norm (\ref{subsec:global-reciprocity-properties}) shows that \[[(O_{L_\P}^\times, L^\ab/L)]_L=[(N_{L/K}(O_{L_\P}^\times), K^\ab/K)]_K\subset [(O_{K_\p}^\times, K^\ab/K)]_K.\] By (iii) of (\ref{prop:compute-the-homomorphisms-h}) we have \[[(O_{K_\p}^\times, K^\ab/K)]_K=[(O_{K_\p}^\times, K^\infty/K)]_K = [O_{K_\p}^\times]\subset \CL_{O_K, \infty}.\] This combined with the relation $[\rho_{E/L}^{-1}]=[-]_L$ and the fact that \[\ker([-]:A_{O_K}^\times\to \CL_{O_K, \infty})=O_K^\times\subset A_{O_K}^\times\] shows that $\rho_{E/L}(I_v)\subset O_K^\times\cdot O_{K_\p}^\times$.

Now let $\ell$ be a rational prime such that $\ell\cdot O_K$ is prime to $\p$. Then the action of $I_v$ on $E[\ell^\infty](\Spec(L^\sep))$ factors through the image of $\rho_{E/L}(I_v)$ under the projection $A_{O_K}^\times\to (O_K\otimes_{\Z}\Z_\ell)^\times$ and is trivial if and only if this image is trivial. As $\rho_{E/L}(I_v)\subset O_K^\times\cdot O_{K_\p}^\times$ the image of this homomorphism is equal to $\rho(I_v)\cap O_K^\times\subset (O_K\otimes_{\Z}\Z_\ell)^\times$ which is trivial if and only if $\rho_{E/L}(I_v)\subset O_{K_\p}^\times$. We may now apply (\ref{theo:criterion-of-good-reduction}) to see that $E/L$ has good reduction at $v$ if and only if $\rho_{E/L}(I_v)\subset O_{K_\p}^\times$.
\end{proof}

\begin{prop} Let $L$ be a finite Galois extension of $K$ and let $E/L$ a \textup{CM} elliptic curve. If $E[\f]$ is a constant scheme over $\Spec(L)$ for some $\f$ which separates units then $E/L$ has good reduction everywhere.
\end{prop}
\begin{proof} If $E[\f]$ is constant then the action of $G(L^\sep/L)$ on $E[\f](\Spec(L^\sep))$ is trivial and therefore $\rho_{E/L}$ takes values in $A_{O_K}^{\times, \f}\subset A_{O_K}^\times$. By (\ref{coro:image-of-inertia-group}), for each finite place $v$ of $L^\sep$ lying over the prime $\p$ of $O_K$, we have $\rho_{E/L}(I_v)\subset O_K^\times\cdot O_{K_\p^\times}\subset A_{O_K}^\times$. Combining these, we have \[\rho_{E/L}(I_v)\subset (O_K^\times\cdot O_{K_\p^\times})\cap A_{O_K}^{\times, \f}\subset O_{K}^{\times, \f}\cdot O_{K_\p}^\times.\] As $\f$ separates units $O_K^{\times, \f}=\{1\}$ and so $\rho_{E/L}(I_v)\subset O_{K_\p}^\times$. It now follows from (\ref{coro:image-of-inertia-group}) that $E/L$ has good reduction at $v$.
\end{proof}

\begin{exem} Let $\p$ be a prime of $O_K$ and let $\F_{\p^n}$ be the unique extension of $\F_\p$ of degree $n$. As $G(\F_{\p}^\sep/\F_{\p^n})\isomto \widehat{\Z}$ is topologically generated by the $N\p^n$-power Frobenius map, we see from (i) of (\ref{prop:compute-the-homomorphisms-h}) to give a continuous homomorphism \[\rho: G(\F_{\p}^\sep/\F_{\p^n})\to \lim_{(\p, \f)=O_K}(O_K/\f)^\times\] such that $[-]_{\F_{\p^n}}=[\rho^{-1}]$ is the same as giving a generator $\rho(\Fr^{N\p^n})=\pi_n\in O_K$ of the ideal $\p^n$. The corresponding isogeny class of CM elliptic curves $E/\Spec(\F_{\p^n})$ are those with the property that the endomorphism \[\pi_n: E\to E\] is equal to the $N\p^n$-power Frobenius.
\end{exem}

\begin{exem} Let $\f$ be an ideal that separates units and recall the isomorphism (\ref{ray-class-isom}): \[A_{O_K}^{\times, \f}\isomto G(K^\infty/K(\f)).\] Now define \[\rho^{-1}: G(K^\sep/K(\f))\to A_{O_K}^\times\] to be the composition \[G(K^\sep/K(\f))\to G(K^\infty/K(\f))\isomto A_{O_K}^{\times, \f}\to A_{O_K}^\times.\] Then $[\rho^{-1}(-)]=[-]_{K(\f)}$ and $\rho$ corresponds to an isogeny class of CM elliptic curves $E/\Spec(K(\f))$. These curves are distinguished by the fact that their $\f$-torsion $E[\f]$ is constant and their study is the topic of the next section.
\end{exem}

\section{Moduli and level structures} In this section we define level-$\f$ structures for CM elliptic curves (for each integral ideal $\f$) and we show that the corresponding moduli stack $\mc{M}_\CM^{(\f)}$ admits a natural action of $\mc{CL}_{O_K}^{(\f)}$ under which it becomes a torsor (\ref{prop:level-torsors}). This induces an action of $CL_{O_K}^{(\f)}$ on the coarse sheaf $M^{(\f)}_\CM$ of $\mc{M}_\CM^{(\f)}$ and using this we prove that $M_\CM^{(\f)}$ is isomorphic to $\Spec(O_{K(\f)}[\f^{-1}])$ compatibly with the isomorphism \[\theta_{K, \f}: \CL_{O_K}^{(\f)}\isomto G(K(\f)/K)\] of (\ref{eqn:define-theta-finite}).

Most of the results contained in this section are probably more or less known, however they do not seem to have appeared in the literature and so we are happy to present a detailed account.

\subsection{}\label{subsec:level-f-on-cm-ell-curves} Let $E/S$ be a CM elliptic curve and let $\f$ be an integral ideal of $O_K$. A level-$\f$ structure on $E/S$ is an isomorphism of $\underline{O_K}_{S}$-modules $\beta: E[\f]\isomto \underline{O_K/\f}_S$. An $\f$-isomorphism $(E/S, \beta)\to (E'/S, \beta')$ between CM elliptic curves with level-$\f$ structures is an isomorphism $f: E\to E'$ such that $\beta'\circ f|_{E[\f]}=\beta$. We write $\mc{M}_{\CM}^{(\f)}$ for the moduli stack over $\Sh_{O_K}$ whose fibre over an sheaf $S$ is the category of CM elliptic curves with level-$\f$ structures together with their $\f$-isomorphisms and if $\f=O_K$ we identify $\mc{M}_\CM$ with $\mc{M}_\CM^{(O_K)}$. If $(E/S, \beta)$ is a CM elliptic curve with level-$\f$ structure we shall often just denote it by $E/S$ when the level-$\f$ structure is clear (or at least does not need to be explicitly mentioned). We list the following (usual) constructions and properties of CM elliptic curves equipped with level-$\f$ structures.

\begin{rema}\label{rema:props-of-level-f-cm}
\begin{enumerate}[label=\textup{(\roman*)}]
\item If $(E, \beta)$ and $(E', \beta')$ are a pair of CM elliptic curves equipped with level-$\f$ structures then we equip the rank one $O_K$-local system $\underline{\Hom}_{S}^{O_K}(E, E')$ with the level-$\f$ structure \[\underline{\Hom}_{S}^{O_K}(E, E')\to \underline{\Hom}_S^{O_K}(E[\f], E'[\f])\isomto \underline{\Hom}_S^{O_K}(\underline{O_K/\f}_S, \underline{O_K/\f}_S)=\underline{O_K/\f}_S\] where the central isomorphism is \[f\mto \alpha\circ f\circ \alpha'^{-1}.\]
\item If $(E, \beta)$ and $(\mc{L}, \alpha)$ are a CM elliptic curve and rank one $O_K$-local system over $S$ with level-$\f$ structures then we equip $E\otimes_{O_K}\mc{L}$ with the level-$\f$ structure $\beta\otimes_{O_K}\alpha$.
\item The sheaf of $\f$-automorphisms $\underline{\Aut}_S^{(\f)}(E)$ of a CM elliptic curve with level-$\f$ structure $(E/S, \beta)$ is equal  to $\underline{O_K^{\times, \f}}$. In particular, it is trivial if $\f$ separates units.
\item The sheaf of $\f$-isomorphisms $\underline{\Isom}_S^{(\f)}(E, E')$ between two CM elliptic curves over $S$ equipped with level-$\f$ structures is finite and \'etale over $S$ and $E$ and $E'$ are locally isomorphic if and only if $\underline{\Isom}_S^{(\f)}(E, E')$ is an $\underline{O_K^{\times, \f}}_S$-torsor.
\item Given a CM elliptic curve $E/S$ with level-$\f$ structure $(E/S, \beta)\in \mc{M}_\CM^{(\f)}(S)$ the existence of the isomorphism $\beta: E[\f]\isomto \underline{O_K/\f}_S$ implies that $E[\f]$ is \'etale over $S$. It follows from (\ref{prop:p-torsion-classification}) that $\f$ is invertible on $S$ and that morphism $S\to \Spec(O_K)$ factors through $\Spec(O_K[\f^{-1}])$. In other words, the structure map $\mc{M}_\CM^{(\f)}\to \Spec(O_K)$ factors through $\Spec(O_K[\f^{-1}])\to \Spec(O_K)$.
\item Let $\mc{E}\to \Spec(O_{K^\sep})$ be a CM elliptic curve (such a curve exists by (\ref{lemm:cm-over-separably-closed-fields})). Then \[E[\f]\times_{\Spec(O_{K^\sep})}\Spec(O_{K^\sep}[\f^{-1}])\] is \'etale and constant over $\Spec(O_{K^\sep}[\f^{-1}])$ and so admits a level-$\f$ structure. Choosing such a structure one obtains a map $\Spec(O_{K^\sep}[\f^{-1}])\to \mc{M}_\CM^{(\f)}$. As $\Spec(O_{K^\sep}[\f^{-1}])\to \Spec(O_K[\f^{-1}])$ is a cover, this shows that the structure map from (the coarse sheaf of) $\mc{M}_{\CM}^{(\f)}$ to $\Spec(O_{K}[\f^{-1}])$ is an epimorphism.
\item If $\a$ is a fractional ideal prime to $\f$ then $\a$ has a natural level-$\f$ structure, coming from $\a\to \a\otimes_{O_K}O_K/\f=O_K/\f$, and we just write $-\otimes_{O_K}\a$ for the corresponding auto-equivalence of $\mc{M}_\CM^{(\f)}$.
\end{enumerate}
\end{rema}

\subsection{}\label{subsec:level-f-structure-homs} Using (ii) of (\ref{rema:props-of-level-f-cm}) we can define a functor \[\mc{M}_\CM^{(\f)}\times \mc{CL}_{O_K}^{(\f)}\to \mc{M}_\CM^{(\f)}: (E/S, \mc{L}/S)\to E\otimes_{O_K}\mc{L}/S\] (the level-$\f$ structures are understood).

\begin{prop}\label{prop:eval-compatible-with-level} Let $E, E'$ be a pair of \textup{CM} elliptic curves with level-$\f$ structures over $S$. Then the evaluation map \[E\otimes_{O_K}\underline{\Hom}_{S}^{O_K}(E, E')\isomto E'\] is an $\f$-isomorphism.
\end{prop}
\begin{proof} This is immediate from the definitions.
\end{proof}

\begin{prop}\label{coro:tensor-p-equals-frobenius-with-level} For each prime ideal $\p$ prime to $\f$ and each $n\geq 0$ there is a unique isomorphism of $\mc{M}_\CM^{(\f)}\times\Spec(\F_\p)$ auto-equivalences $\nu_{\p^n}: -\otimes_{O_K}\p^{-n}\isomto\Fr^{N\p^n*}(-)$ such that for all \textup{CM} elliptic curves $E$ over $\Spec(\F_\p)$-sheaves $S$ the diagram \[\xymatrix{&E\ar[dr]^-{\Fr_{E/S}^{N\p^n}}\ar[dl]_-{i_{\p^n}}&\\
E\otimes_{O_K}\p^{-1}\ar[rr]^{\nu_{\p^n}}_\sim && \Fr^{N\p^n*}(E)}\] commutes.
\end{prop}
\begin{proof} We need only verify that the isomorphism $\nu_{\p^n}(E/S): E\otimes_{O_K}\p^{-n}\isomto \Fr^{N\p^n*}(S)$ of (\ref{coro:tensor-p-equals-frobenius}) is an $\f$-isomorphism. However, the morphisms $i_{\p^n}$ and $\Fr_{E/S}^{N\p^n}$ induce isomorphisms on the $\f$-torsion which are compatible with the level-$\f$ structures on $E$, $E\otimes_{O_K}\p^{-n}$ and $\Fr^{N\p^n*}(E)$ so that as $\Fr_{E/S}^{N\p^n}=\nu_{\p^n}\circ i_{\p^n}$ it follows that $\nu_{\p^n}(E/S)$ is an $\f$-isomorphism.
\end{proof}

\begin{prop}\label{prop:level-torsors} The functor \[\mc{M}_{\CM}^{(\f)}\times \mc{CL}_{O_K}^{(\f)}\to \mc{M}_{\CM}^{(\f)}\times \mc{M}_{\CM}^{(\f)}: (E, \mc{L})\mto (E, E\otimes_{O_K}\mc{L})\] is an equivalence of stacks and $\mc{M}_{\CM}^{(\f)}$ is locally equivalent to $\mc{CL}^{(\f)}_{O_K}\times\Spec(O_K[\f^{-1}])$.
\end{prop}
\begin{proof} By (\ref{prop:eval-compatible-with-level}) this functor is essentially surjective. Moreover, for all $S$ the morphism \[\underline{\Isom}_S^{(\f)}(E, E')\times_S \underline{\Isom}_S^{(\f)}(\mc{L}, \mc{L'})\to \underline{\Isom}_S^{(\f)}(E, E')\times_S \underline{\Isom}_S^{(\f)}(E\otimes_{O_K}\mc{L}, E'\otimes_{O_K}\mc{L}')\] is an isomorphism of sheaves over $S$, as this can be checked on $S$ sections. Indeed, if $\underline{\Isom}_S^{(\f)}(E, E')(S)=\emptyset$ then it is clear and if $\underline{\Isom}_S^{(\f)}(E, E')(S)\neq\emptyset$, we may assume that $E=E'$ and show instead that the map \begin{equation}\Isom_S^{(\f)}(\mc{L}, \mc{L}')\to \Isom^{(\f)}_S(E\otimes_{O_K}\mc{L}, E\otimes_{O_K}\mc{L}'): h\mto \id_E\otimes h\end{equation} is bijective. The bijectivity of this map follows from (\ref{prop:cm-tensor-homs}) combined with the fact that an isomorphism $h: \mc{L}\to \mc{L}'$ induces an $\f$-isomorphism $\id_E\otimes_{O_K} h: E\otimes_{O_K}\mc{L}\isomto E\otimes_{O_K}\mc{L}'$ if and only if $h$ is an $\f$-isomorphism.
\end{proof}

\subsection{}\label{rema:coarse-space-level} We now wish to compute the coarse sheaves $M_\CM^{(\f)}:=C(\mc{M}_\CM^{(\f)})$ of the stacks $\mc{M}_\CM^{(\f)}$. First let us define an action of \[\underline{\CL_{O_K}^{(\f)}}[\f^{-1}]=\underline{CL_{O_K}^{(\f)}}\times\Spec(O_K[\f^{-1}])\] on $M_\CM^{(\f)}$. As this group is constant it is enough to define an action of $CL_{O_K}^{(\f)}$ and then, using the isomorphism \[\Id_{O_K}^{(\f)}/\Prin_{1\bmod \f}^{(\f)}\isomto \CL_{O_K}^{(\f)}:\a\mto [\a]_\f\] it is enough to define an action of $\Id_{O_K}^{(\f)}$ with the property that $\Prin_{1\bmod \f}^{(\f)}$ acts trivially. Each ideal $\a\in \Id_{O_K}^{(\f)}$, equipped with its standard level-$\f$ structure induces an auto-equivalence \[-\otimes_{O_K}\a: \mc{M}_\CM^{(\f)}\to \mc{M}_\CM^{(\f)}\] and by the universal property of the coarse sheaf an automorphism $[\a]_\f: M_{\CM}^{(\f)}\to M_{\CM}^{(\f)}.$ For $\a, \b\in \Id_{O_K}^{(\f)}$, the natural $\f$-isomorphisms \[(-\otimes_{O_K} \a)\otimes_{O_K}\b\isomto -\otimes_{O_K}\a\b\] show that $[\b]_\f \circ [\a]_\f=[\a\b]_\f.$ We have $[\a]_\f=[\b]_\f\in \CL_{O_K}^{(\f)}$ if and only if there exists an $\f$-isomorphism $\a\isomto \b$, which in turn induces an isomorphism auto-equivalences \[-\otimes_{O_K}\a\isomto -\otimes_{O_K}\b\] and gives $[\a]_\f=[\b]_\f: M_\CM^{(\f)}\to M_{\CM}^{(\f)}$ so that we have defined our action.

Moreover, by (\ref{coro:tensor-p-equals-frobenius-with-level}) it follows that for each prime $\p\in \Id_{O_K}^{(\f)}$ the pull-back of the automorphism \[[\p^{-1}]_\f: M_\CM^{(\f)}\isomto M_\CM^{(\f)}\] along $\Spec(\F_\p)\to \Spec(O_K[\f^{-1}])$ is equal to the $N\p$-power Frobenius \begin{equation}\label{eqn:frob-lift-level-coarse}[\p^{-1}]_\f\times\Spec(\F_\p)=\Fr^{N\p}: M_\CM^{(\f)}\times\Spec(\F_\p)\isomto M_\CM^{(\f)}\times \Spec(\F_\p)\end{equation}

\subsection{}\label{subsec:coarse-maps-level} If $(E/S, \beta)$ is a CM elliptic curve with level-$\f$ structure then we write \[c_{E/S, \f}:S\to M_\CM^{(\f)}\] for the composition \[S\stackrel{(E, \beta)}{\longrightarrow}\mc{M}_\CM^{(\f)}\stackrel{c_{\mc{M}_\CM^{(\f)}}}{\longrightarrow} M_{\CM}^{(\f)}\] (when $\f=O_K$ we drop the $\f$). It follows from the definition of $\sigma_\a$ that \begin{equation}\label{eqn:cl-action-on-coarse} c_{E\otimes_{O_K}\a/S, \f}=\sigma_\a\circ c_{E/S, \f}.\end{equation}

\begin{coro}\label{coro:coarse-spaces-of-level-moduli}
\begin{enumerate}[label=\textup{(\roman*)}]
\item The action of $\underline{CL_{O_K}^{(\f)}}[\f^{-1}]$ on $M_\CM^{(\f)}$ makes it a torsor over $\Spec(O_K[\f^{-1}])$.
\item $M_\CM^{(\f)}$ is finite and \'etale over $\Spec(O_K[\f^{-1}])$.
\item The action of $G(K^\sep/K)$ on $M_\CM^{(\f)}(\Spec(K^\sep))$ is through the homomorphism \[G(K^\sep/K) \to G(K(\f)/K)\stackrel{\theta_{K, \f}}{\to} \CL_{O_K}^{(\f)}\subset \Aut_{\Spec(O_K[\f^{-1}])}(M_\CM^{(\f)})\] hence \[\Spec(O_{K(\f)}[\f^{-1}])\isomto M_\CM^{(\f)}.\]
\end{enumerate}
\end{coro}
\begin{proof} (i) To show that the action of $\underline{CL_{O_K}^{(\f)}}[\f^{-1}]$ is free it is enough to show that for each $[\a]_\f\in \CL_{O_K}^{(\f)}$ the automorphism $[\a]_\f: M_\CM^{(\f)}\isomto M_\CM^{(\f)}$ fixes a section $S\to M_\CM^{(\f)}$ if and only if $[\a]_\f=[O_K]_\f$. So let $[\a]_\f\in \CL_{O_K}^{(\f)}$ fix a section $S\to M_\CM^{(\f)}$. By the definition of the coarse sheaf, there exists a cover $S'\to S$ and a CM elliptic curve $E/S'$ such that the induced map $S'\to S \to M_\CM^{(\f)}$ is equal to $c_{E/S', \f}$. This section $c_{E/S', \f}$ is also fixed by $[\a]_\f$ which is now the statement that the two CM elliptic curves with level-$\f$ structure $E$ and $E\otimes_{O_K}\a$ are locally $\f$-isomorphic over $S'$. After refining $S'$, we may assume that $E$ and $E\otimes_{O_K}\a$ are actually $\f$-isomorphic which by (\ref{prop:level-torsors}) implies the existence of an $\f$-isomorphism $\underline{O_K}_{S'}\isomto \underline{\a}_{S'}$. After refining $S'$ again such an isomorphism is constant and of the form $\underline{h}$ where $h: O_K\to \a$ is an $\f$-isomorphism and it follows that $[\a]_\f=[O_K]_\f$.

Similarly, to show that the action of $\underline{\CL_{O_K}^{(\f)}}[\f^{-1}]$ is transitive, it is enough to show that for each pair of sections $c_1, c_2: S\to M_\CM^{(\f)}$ there is a cover $(S_i\to S)_{i\in I}$ and elements $[\a_i]_\f\in \CL_{O_K}^{(\f)}$ such that $[\a_i]_\f\circ c_1|_{S_i}=c_2|_{S_i}$. Again by the definition of the coarse sheaf there exists a cover $S'\to S$ and CM elliptic curves $E_1$, $E_2$ over $S$ such that the compositions $S'\to M_\CM^{(\f)}$ of $c_1$ and $c_2$ with $S'\to S$ are equal to $c_{E_1/S, \f}$ and $c_{E_2/S, \f}$. By (\ref{prop:level-torsors}) there exists an $O_K$-local system $(\mc{L}, \alpha)$ with level-$\f$ structure and an $\f$-isomorphism \[E_1\isomto E_2\otimes_{O_K}\mc{L}.\] After base change to a cover $(S_i\to S)_{i\in I}$ (the corresponding class of $\mc{L}$ may not be constant which is why our cover may consist of multiple elements) we may assume that $(\mc{L}, \alpha)=(\underline{\a}_{S'}, \underline{f}_{S'})$ and it follows from (\ref{eqn:cl-action-on-coarse}) that $c_{E_2/S, \f}=[\a]_\f\circ c_{E_1/S, \f}$ and this proves the claim. 

(ii) The fact that $M_\CM^{(\f)}$ is finite and \'etale over $\Spec(O_K[\f^{-1}])$ follows by descent, as $M_\CM^{(\f)}$ is locally (over $\Spec(O_K[\f^{-1}])$) isomorphic to $\underline{CL_{O_K}^{(\f)}}[\f^{-1}]$ by (i) and $\underline{CL_{O_K}^{(\f)}}[\f^{-1}]$ is finite and \'etale over $\Spec(O_K[\f^{-1}])$.

(iii) For each prime $\p$ prime to $\f$, the automorphism $[\p^{-1}]_\f: M_\CM^{(\f)}\to M_\CM^{(\f)}$ lifts the $N\p$-power Frobenius automorphism  modulo $\p$ (\ref{eqn:frob-lift-level-coarse}). For any such automorphism (of any finite \'etale $\Spec(O_K[\f^{-1}])$-scheme) there exists an element $\sigma\in G(K^\sep/K)$ such that \[M_\CM^{(\f)}(\sigma)=[\p^{-1}]_\f(\Spec(K^\sep)): M_\CM^{(\f)}(\Spec(K^\sep))\to M_{\CM}^{(\f)}(\Spec(K^\sep)).\] However, the action of $\CL_{O_K}^{(\f)}$ on $M_{\CM}^{(\f)}(\Spec(K^\sep))$ is transitive by (i), and $\CL_{O_K}^{(\f)}$ is generated by elements of the form $[\p^{-1}]_\f$ and so it follows that the action $G(K^\sep/K)$ on $M_\CM^{(\f)}(\Spec(K^\sep))$ is transitive. Therefore $M_\CM^{(\f)}$ is connected and isomorphic to $\Spec(O_{L}[\f^{-1}])$ for some finite extension $L/K$ which is unramified away from $\f$. By construction, the isomorphism $G(L/K)\to \CL_{O_K}^{(\f)}$ sends the Frobenius element $\sigma_{L/K, \p}$ to $[\p^{-1}]_\f$ and it follows that $L\isomto K(\f)$ (cf.\ (\ref{eqn:define-theta-finite})).
\end{proof}

\begin{rema} We list the following consequences of (\ref{coro:coarse-spaces-of-level-moduli}).
\begin{enumerate}[label=\textup{(\alph*)}]
\item The coarse sheaf $M_\CM$ of $\mc{M}_\CM$ is isomorphic to $\Spec(O_H)$ where $H$ is the Hilbert class field of $K$. This recovers the fact the $j$-invariants of CM elliptic curves defined over extensions of $K$ lie in $H\subset L$.
\item Let $E\to S$ be any CM elliptic curve and recall that \[c_{E/S}: S\to M_\CM\] denotes the morphism \[S\stackrel{E}{\to} \mc{M}_\CM\stackrel{c_{\mc{M}_\CM}}{\longrightarrow} M_\CM.\] Then by definition, if $E'$ is another CM elliptic curve over $S$ then $c_{E/S}, c_{E'/S}: S\to M_\CM$ are equal if and only if $E$ and $E'$ are locally isomorphic. Combining this with (\ref{coro:isom-iso-isom}) and (\ref{prop:cm-exist-given-h}), we find that if $S=\Spec(F)$ for $F$ a field then the map \[E/F\mto (\rho_{E/F}, c_{E/F})\] defines a bijection between the set of isomorphism classes of CM elliptic curves over $E/\Spec(F)$ with the set of pairs $(\rho, c)$ where:
\begin{enumerate}[label=(\roman*)]
\item $\rho: G(F^\sep/F)\to \lim_{\f}(O_K/\f)^\times$ is a homomorphism such that $[\rho^{-1}]=[-]_F$, and
\item $c: \Spec(L)\to M_\CM$ is any map.
\end{enumerate}
\item If $\f$ separates units then $\mc{M}_\CM^{(\f)}\isomto M_\CM^{(\f)}\isomto \Spec(O_{K(\f)}[\f^{-1}])$. If $E/\Spec(K(\f))$ denotes the generic fibre of the universal CM elliptic curve with level-$\f$ structure the homomorphism \[\rho_{E/K(\f)}: G(K(\f)^\sep/K(\f))\to A_{O_K}^\times\] is equal to (cf.\ (\ref{ray-class-isom})) \[G(K(\f)^\sep/K(\f))\stackrel{-|_{K^\infty}}{\longrightarrow} G(K^\infty/K(\f)) \isomto A_{O_K}^{\times, \f} \stackrel{\textup{incl}}{\longrightarrow} A_{O_K}^\times.\]
\end{enumerate}
\end{rema}

\chapter{$\Lambda$-structures, Witt vectors and arithmetic jets} In this chapter we give a brief introduction to, and overview of, the theory $\Lambda$-structures, Witt vectors and arithmetic jet spaces following the approach of Borger \cite{Borger11}, \cite{Borger11bis}. It is by nature a technical and notationally heavy theory, but its many applications make it a worthwhile subject. The two papers mentioned are both an excellent introduction and general reference and we encourage to the reader to consult them. We give here really only the minimal set up necessary for our applications in Chapter 4 and proofs are given only where they cannot be cited or where it would be perhaps enlightening.

In the final section we prove a small new result which shows that $\Lambda$-structures on relative abelian schemes are determined by their underlying $\Psi$-structures.

\section{Plethories}

\subsection{} Fix a pair of rings $O$ and $O'$. An $O$-$O'$-biring $\Phi$ is an $O$-algebra together the structure of an $O'$-algebra on the set $\Hom_{O}(\Phi, A)$ which is functorial in the $O$-algebra $A$. This is structure is determined by certain homomorphisms
\begin{enumerate}[label=\textup{(\roman*)}]
\item coaddition and comultiplication: $\Delta_\Phi^+, \Delta_{\Phi}^\times: \Phi\to \Phi\otimes_{O}\Phi$
\item coadditive and comultiplicative units: $\epsilon_\Phi^+, \epsilon_\Phi^\times: \Phi\to O$
\item $O'$ coaction: $O'\to \End_O(\Phi): s\mto s_{\Phi}$
\end{enumerate} satisfying various identities (cocommutativity of the coaddition and comultiplication, coassociativity and so on).

We denote by $\Biring_{O, O'}$ the category whose objects are $O$-$O'$-birings and whose morphisms are those $O$-homomorphisms $\Phi\to \Phi'$ inducing functorial homomorphisms of $O'$-algebras $\Hom_O(\Phi', A)\to \Hom_O(\Phi', A)$. Of course, this is equivalent to the homomorphism $\Phi\to \Phi'$ being `compatible' with the maps $\Delta_\Phi^+$, $\Delta_{\Phi'}^+$ and so on.

\subsection{} \label{subsec:corep-functor-biring} The functor corepresented by a biring $\Phi$: \[A\mto \Hom_{O}(\Phi, A): \Alg_{O}\to \Alg_{O'}\] admits a left adjoint, which we denote by $B\mto \Phi\odot_{O'} B$, which is defined as follows: $\Phi\odot_{O'} B$ is the quotient of the free $O$-polynomial algebra generated by the symbols $\phi\odot b$ for $\phi\in \Phi$ and $b\in B$ subject to the relations \[(\phi+\phi')\odot b=\phi\odot b+\phi'\odot b, \quad \phi\phi'\odot b=(\phi\odot b)(\phi'\odot b),\quad (r\phi)\odot b=r(\phi \odot b)\] and \[\phi\odot(b+b')=\Delta_\Phi^+(\phi)(b, b')\footnote{This notation means that if $\Delta_\Phi^+(\phi)=\sum_i \phi_i\otimes \phi_i'$ then $\Delta_\Phi^+(\phi)(b, b')=\sum_i (\phi_i\odot b_i)(\phi_i'\odot b_i')$ and similarly for $\Delta_{\Phi}^\times(b, b')$.}, \quad \phi\odot bb'=\Delta^\times_\Phi(\phi)(b, b'), \quad \phi\odot sb=s_{\Phi}(\phi)\odot b\] for all $\phi, \phi'\in \Phi$, $b, b'\in B$, $r\in O$ and $s\in O'$. The $O$-algebra $\Phi\odot_{O'} B$ is called composition product of $\Phi$ with $B$.

\subsection{} If $\Phi$ is an $O$-$O'$-biring and $O\to O''$ is a homomorphism then $\Phi\otimes_{O}O''$ is an $O''$-$O'$-biring and $(\Phi\odot_{O'} B)\otimes_{O}O''=(\Phi\otimes_{O}O'')\odot_{O'} B$. The functor $\Biring_{O, O'}\times \Alg_{O'}\to \Alg_{O}: (\Phi, B)\to \Phi\odot_{O'} B$ commutes with colimits in each variable. 
We give the two simplest examples when $O=O'$.
\begin{enumerate}[label=\textup{(\roman*)}]
\item The functor $A\mto A$ is represented by the $O$-$O$-biring $O[e]$.
\item Even simpler is the functor $A\mto 0$ represented by the $O$-$O$-biring $O$ itself.
\end{enumerate}

\subsection{}\label{subsec:birings-and-monoidal-functors} We now concentrate on the case $O=O'$ and just call an $O$-$O$-biring an $O$-biring and shall drop $O$ from the notation when there is no risk of confusion. If $\Phi$ and $\Phi'$ are two $O$-birings then we may consider the composition product $\Phi\odot \Phi'$ which, using the standard properties of adjunctions, is again an $O$-biring. The composition product then defines a monoidal structure on the category of $O$-birings with identity $O[e]$. The functor \[\Phi\mto \Phi\odot -\] from $O$-birings to endofunctors on $\Alg_{O}$ is monoidal and fully faithful.

\subsection{}\label{subsection:definition-of-plethories} An $O$-plethory is an $O$-biring $\Phi$ together with the structure of a monad on the functor $A\mto \Phi\odot A$ or equivalently a comonad on the functor $A\mto \Hom_{O}(\Phi, A)$. The remarks in (\ref{subsec:birings-and-monoidal-functors}) then show that to give the functor $A\mto \Phi\odot A$ the structure of a monad is equivalent to giving
\begin{enumerate}[label=\textup{(\roman*)}]
\item a homomorphism of $O$-birings $i_{\Phi}:O[e]\to \Phi$ and
\item a homomorphism of $O$-birings $h_{\Phi}:\Phi\odot \Phi\to \Phi$
\end{enumerate} such that \[h_{\Phi}\circ(\Phi\odot i_{\Phi})=h_{\Phi}\circ(i_{\Phi}\odot \Phi)=\id_{\Phi} \quad \text{ and }\quad h_{\Phi}\circ(\Phi\odot h_{\Phi})=h_{\Phi}\circ(h_{\Phi}\odot \Phi).\]

\subsection{} If $\Phi$ is an $O$-plethory we define a $\Phi$-ring to be an $O$-algebra $A$ equipped with an action of the monad $\Phi\odot -$. Given a $\Phi$-ring $A$ we denote by \[h_A: \Phi\odot A\to A\] the map defining the $\Phi$-ring structure on $A$. We denote the category of $\Phi$-rings and compatible morphisms by $\Alg_{O}^\Phi$. If $A$ is an $O$-algebra then $\Phi\odot A$ and $\Hom_O(\Phi, A)$ are $\Phi$-rings and these two functors are the left and right adjoints of the forgetful functor $\Alg_{O}^\Phi\to \Alg$. We have the diagram of functors, each one left adjoint to the one below it:
\[\xymatrix{\Alg_{O}^\Phi\ar[rr]^{\textup{forget}} && \ar@/^2pc/[ll]^{\Hom_O(\Phi, -)}\ar@/_2pc/[ll]_{\Phi\odot_O -}\Alg_{O}.}\]

\section{Witt vectors and arithmetic jets I}\label{sec:lambda-section-one}

\subsection{} During \S\S \ref{sec:lambda-section-one}--\ref{sec:lambda-section-three} we fix a Dedekind domain $O$ with finite residue fields and $P\subset \Id_{O}$ a sub-monoid generated by some set of prime ideals of $O$ and we write $K$ for the fraction field of $O$. Note that we allow Dedekind domains of finite characteristic, e.g.\ $O=\F_p[T]$.

\subsection{} Denote by $\Psi$ the free polynomial $O$-algebra generated by the symbols $\psi^\a$ for $\a\in P$. We make $\Psi$ an $O$-biring by equipping the set \[\Hom_{O}(\Psi, A)=\Hom_{O}(O[\psi^\a:\a\in P], A) \isomto \prod_{\a\in P} A: f\mto (f(\psi^\a))_{\a\in P}\] with the product $O$-algebra structure. We then give $\Psi$ the structure of an $O$-plethory by setting \[i_{\Psi}: O[e]\to \Psi: e\mto \psi^{(1)} \quad \text{ and } \quad h_{\Psi}: \Psi\odot \Psi\to \Psi: \psi^\a\odot \psi^\b\mto \psi^{\a\b}.\]

\subsection{} For each ring $A$ we write $\Gamma(A)$ for the ring \[\Gamma(A):=\Hom_O(\Psi, A)=\prod_{\a\in P}A\] and call it the ring of ghost vectors of $A$. We define the ghost jets of a ring $A$ to be the ring \[\Psi\odot_O A.\] Let us examine a little more the ring of ghost jets $\Psi\odot_O A$ of a ring $A$ and what it means for $A$ to be a $\Psi$-ring.

Recall (\ref{subsec:corep-functor-biring}) that if $A$ is an $O$-algebra then $\Psi\odot_{O} A$ is given by the quotient of the $O$-polynomial algebra \[O[\psi^\a\odot a: \a\in P, a\in A]\] by the ideal generated by the elements of the following form \begin{eqnarray*}&\psi^\a\odot (a+b)-\psi^\a \odot a -\psi^\a \odot b\\& \psi^\a\odot (ab)-(\psi^\a\odot a)(\psi^\a\odot b)\\&\psi^\a\odot (ra)-r(\psi^\a\odot a)\end{eqnarray*} for $a, b\in A$, $r\in O$. Now to give $A$ the structure of a $\Psi$-ring is the same as giving an $O$-homomorphism \[h_A:\Psi\odot_{O} A\to A\] satisfying certain properties. In this case, the properties which must be satisfied are equivalent to the following:
\begin{enumerate}[label=\textup{(\roman*)}]
\item for each $\a\in P$ the map $\psi_A^\a: A\to A$ defined by $a\mto h_A(\psi^\a\odot a)$ is an $O$-algebra homomorphism,
\item for each $\a, \b\in P$ we have $\psi^\a_A\circ \psi^{\b}_A=\psi^{\b}_A\circ \psi^{\a}_A=\psi^{\a\b}_A$, and
\item we have $\psi^{(1)}_A=\id_A: A\to A$.
\end{enumerate} We see that if $A$ is a $\Psi$-ring then $A$ has an action of the monoid $P$ where $\a\in P$ acts by $\psi^\a_A: A\to A.$ This sets up a bijection between $\Psi$-ring structures on $A$ and actions of $P$ on $A$. Given a $\Psi$-ring we will write $\psi^\a_A: A\to A$ for the corresponding $P$-action. In particular, the $\Psi$-ring $\Psi$ itself has the $P$-action \[\psi^\a_{\Psi}: \psi^{\b}\mto \psi^{\a\b}\] for $\a, \b\in P.$ Note that the $\Psi$-ring structure on $O$ is given by \[\psi_O^\a=\id_O: O\to O.\]

\subsection{} As we are dealing with rings equipped with endomorphisms, it will be useful here to say a little about semi-linear self maps versus twisted linear maps. The example to have in mind is the following: if $A$ is a ring of characteristic $p$ and $A\to B$ is an $A$-algebra then the $p$-power Frobenius \[\Fr_B^p: B\to B: b\mto b^p\] is $\Fr_A^p$-linear where $\Fr_A^p: A\to A: a\mto a^p$ is the $p$-power Frobenius of $A$. This then induces an $A$-linear map, the relative $p$-power Frobenius, $\Fr_{B/A}^p: \Fr_A^{p*}(B)\to B$.

Now, a morphism of $\Psi$-rings $f: A\to B$ is just a homomorphism such that the diagram \[\xymatrix{A\ar[d]_f\ar[r]^{\psi^\a_A} & A\ar[d]^f \\B \ar[r]^{\psi_B^\a} & B}\] commutes for all $\a\in P$. Viewing $B$ as an $A$-algebra via $f$, the homomorphism $\psi_B^\a: B\to B$ is $\psi_A^\a$-linear and induces an $A$-linear map \[\psi_{B/A}^{\a}:\psi_A^{\a*}(B) \to B.\] That the homomorphisms $\psi_B^\a$ for $\a\in P$ commute is now expressed by the condition that for all $\a, \b\in P$ we have \begin{equation}\label{def:semi-linear-psi}\psi_{B/A}^{\b}\circ \psi_{A}^{\a*}(\psi_{B/A}^\b)=\psi_{B/A}^{\a}\circ \psi_{A}^{\b*}(\psi_{B/A}^\a)=\psi_{B/A}^{\a\b} \end{equation} or in a commutative diagram: \[\xymatrix{\psi_A^{\a\b*}(B) \ar[rr]^-{\psi^{\b*}_A(\psi_{B/A}^{\a})}\ar[d]_{\psi^{\a*}_A(\psi_{B/A}^{\b})} && \psi_A^{*\a}(B)\ar[d]^{\psi_{B/A}^{\a}}\\
\psi_A^{*\b}(B)\ar[rr]^{\psi_{B/A}^{\b}} && B.}\] Moreover, if $A$ is a $\Psi$-ring and $A\to B$ is an $A$-algebra then to give $B$ the structure of $\Psi$-ring such that $A\to B$ is a $\Psi$-homomorphism is the same as giving maps $\psi_{B/A}^{\a}: \psi^{\a*}_A(B)\to B$ for each $\a\in P$ such that $\psi_{B/A}^{(1)}=\id_B$ and which satisfy the commutativity condition (\ref{def:semi-linear-psi}). The category of $\Psi$-rings over a $\Psi$-ring $A$ is denoted $\Alg_{\Psi_A}$ and its objects called $\Psi_A$-rings.

\subsection{} We now come to the matter of interest which is lifting the Frobenius. We wish to impose on a $\Psi$-ring $A$ the condition that the homomorphisms $\psi_A^\p: A\to A$, for $\p\in P$ prime, be lifts of the $N\p$-power Frobenius endomorphism (recall that $N\p=\#\F_\p$): \begin{equation}\label{eqn:frobenius-lift-condition}\psi_A^\p(a)=a^{N\p}\bmod \p A\end{equation} This is done by enlarging the plethory $\Psi$ in such a way that the endomorphisms corresponding to the elements $\psi_\p\in \Psi$, for each prime $\p\in P$, will be forced to satisfy the relation (\ref{eqn:frobenius-lift-condition}).

So for each integer $n\geq 0$ define sub-$O$-algebras $\Lambda_n\subset \Psi\otimes_O K$ inductively by setting $\Lambda_{0}=\Psi$, and for $n\geq 0$, setting $\Lambda_{n+1}$ to be the sub-$\Lambda_{n}$-algebra of $\Psi\otimes_{O}K$ generated by the subsets \begin{equation}\label{eqn:put-in-frob-residues}\p^{-1} (f^{N\p}-\psi^\p_{\Psi}(f))\subset \Psi\otimes_{O}K\end{equation} for $\p\in P$ a prime ideal and $f\in \Lambda_{n}$. Finally, we set \[\Lambda=\cup_{n\geq 0}\Lambda_{n}\subset \Psi\otimes_{O}K.\] Then $\Lambda_{n}\subset \Psi\otimes_{O}K$ is stabilised by the endomorphisms $\psi_{\Psi}^\a$ of $\Psi\otimes_{O}K$ for each $\a\in P$, as is $\Lambda\subset \Psi\otimes_{R} K$. Thus $\Lambda_{n}$ and $\Lambda$ admit unique $\Psi$-ring structures such that $\Psi\subset \Lambda_{n}\subset \Lambda$ are $\Psi$-morphisms. Moreover, for each $n\geq 0$, and each $f\in \Lambda_n$ it follows from the definition (\ref{eqn:put-in-frob-residues}) of $\Lambda_{n+1}$ that \[\psi_\Lambda^\p(f)=f^{N\p}\bmod \p \Lambda_{n+1}\] and hence for all $f\in \cup_{n\geq 0}\Lambda_n=\Lambda$ we have \[\psi_\Lambda^\p(f)=f^{N\p}\bmod \p \Lambda.\]

\begin{prop} There is a unique $O$-plethory structure on $\Lambda$ such that the inclusion $\Psi\to \Lambda$ is a morphism of $O$-plethories.
\end{prop}

\begin{rema} Before we examine exactly what a $\Lambda$-ring is, the fact that $\Psi\to \Lambda$ is a morphism of $O$-plethories implies that every $\Lambda$-ring $A$ inherits a $\Psi$-ring structure and hence a family of endomorphisms $\psi_A^{\a}: A\to A$ for each $\a\in P$ such that $\psi_{(1)}=\id_A$ and $\psi_{A}^{\a\b}=\psi_{A}^\a\circ \psi_{A}^\b$ for all $\a, \b\in P$.
\end{rema}

\begin{prop}\label{prop:lambda-actions-and-frobenius-lifts} Let $A$ be an $O$-algebra. We have the following:
\begin{enumerate}[label=\textup{(\roman*)}]
\item If $A$ is a $\Lambda$-ring the endomorphism $\psi^\p_A: A\to A$ satisfies \[\psi^\p_A(a)=a^{N\p}\bmod \p A\] for each maximal ideal $\p\in P$.
\item If $A$ is a $\Psi$-ring and $\p$-torsion free for each prime $\p\in P$ (i.e.\ flat at $\p$ for each prime $\p\in P$) then the given $\Psi$-structure comes from a $\Lambda$-structure if and only if for each prime ideal $\p\in P$ the endomorphism $\psi^\p_A: A\to A$ lifts the $N\p$-power Frobenius \[\psi^\p_{A}(a)=a^{N\p}\bmod \p A.\] In this case the $\Lambda$-structure inducing the $\Psi$-structure is unique.
\item If each $\p\in P$ is invertible in $A$ then every $\Psi$-structure on $A$ is induced by a unique $\Lambda$-structure.
\end{enumerate}
\end{prop}
\begin{proof} (i) and (ii) are (essentially) the definition of $\Lambda$-ring given in \cite{Borger11}. In particular, see \S\S 1.6--1.19 \cite{Borger11}.

(iii) If each $P$ in $A$ is invertible any endomorphism $\psi_A^\p: A\to A$ lifts the $N\p$-power Frobenius (trivially) and so this follows from (i) and (ii).
\end{proof}

\begin{rema} We point out that if $A\to B$ is a homomorphism of $\Lambda$-rings then $\psi_A^\p=\Fr_{A_\p}^{N\p}\bmod \p A $ and the relative morphisms $\psi_{B/A}^{\p}: \psi_A^{\p*}(B)\to B$ are now lifts of the relative $N\p$-power Frobenius: \[\psi_{B/A}^{\p}=\Fr_{B_\p/A_\p}^{N\p}: \Fr_{A_\p}^{N\p*}(B_\p)\to B_\p\] where we write $A_\p=A\otimes_{O}\F_\p$ and $B_\p=B\otimes_{O}\F_\p$.
\end{rema}

\begin{exem} Using (\ref{prop:lambda-actions-and-frobenius-lifts}) we are now able to give the first examples of $\Lambda$-rings:
\begin{enumerate}[label=(\roman*)]
\item Of course $O$ is always a $\Lambda$-ring with $\psi_O^\a: O\to O$ equal to the identity. The fact that $\psi_O^{\p}$ lifts the $N\p$-power Frobenius is then Fermat's little theorem: \[\psi_O^\p(a)=a=a^{N\p}\bmod \p\] (recall that $O/\p=\F_\p$ is a finite field with $N\p$-elements).
\item The polynomial ring $O[T]$ is a $\Lambda$-ring with $\psi_{O[T]}^{\a}: O[T]\to O[T]$ given by $T\mto T^{N\a}$.
\item If $O=O_K$ is the ring of integers in a number field, $L/K$ is an abelian extension and $P\subset \Id_{O_K}$ is the sub-monoid generated by the primes which are unramified in $L/K$, then the ring of integers $O_L$ of $L$ is a $\Lambda_P$-ring with $\psi_{O_L}^\p=\sigma_{L/K, \p}: O_L\to O_L$ given by the Frobenius element $\sigma_{L/K, \p}\in G(L/K)$ (cf.\ (\ref{subsec:finite-extensions-frobenius})). Examples of this form will be very important later give the first link between $\Lambda$-structures and class field theory.
\end{enumerate}
\end{exem}

\subsection{}\label{subsec:single-principal-prime} We now give an explicit description of what it means for a non-flat ring to be have $\Lambda$-structure in the special case $P=\{\p, \p^2, \ldots\}\subset \Id_{O}$ is generated by a single prime ideal of $O$ and the prime ideal $\p=(\pi)$ is principal. In this case, a $\Lambda$-structure on an $O$-algebra can described explicitly, and this notion was discovered independently by Buium \cite{Buium95} (for some of the many interesting arithmetic applications of $\delta$-rings and $\delta$-geometry see \cite{Buium05}).

So let us describe $\Lambda$ in this situation. Define elements $\delta_n\in \Lambda$ for $n\geq 0$ inductively by $\delta_0=\psi^{(1)}$ and \[\delta_{n+1}=\pi^{-1}(\psi^\p_{\Lambda}(\delta_n)-(\delta_n)^{N\p})\in \Lambda.\] Let $\Delta_\pi\subset \Lambda$ be the sub-$O$-algebra generated by the elements $\{\delta_{n}\}_{n\geq 0}$.

\begin{prop}\label{prop:lambda-one-prime-is-locally-polynomial} The inclusion $\Delta_{\pi}\subset \Lambda$ is an equality and $\Delta_{\pi}$ is freely generated as an $O$-algebra by the elements $\{\delta_{n}\}_{n\geq 0}.$
\end{prop}
\begin{proof} See \S 1.19 of \cite{Borger11}.
\end{proof}

\begin{coro}\label{coro:delta-lambda-rings} With notation as in \textup{(\ref{subsec:single-principal-prime})} to give an $O$-algebra $A$ a $\Lambda$-structure is equivalent to defining a map \[\delta_\pi: A\to A\] such that:
\begin{enumerate}[label=\textup{(\roman*)}]
\item for $r\in O$ we have \[\delta_\pi(r)=\frac{r^{N\p}-r}{\pi},\]
\item for $a, b\in A$ we have \[\delta_\pi(ab)=a^{N\p}\delta_\pi(b)+b^{N\p}\delta_\pi(a)+\pi\delta_\pi(a)\delta_\pi(b),\] and
\item for $a, b\in A$ we have \[\delta_\pi(a+b)=\delta_\pi(a)+\delta_\pi(b)+\sum_{i=1}^{{N\p}-1}\frac{1}{\pi}{{N\p} \choose i} a^{{N\p}-i}b^i.\]
\end{enumerate}
\end{coro}
\begin{proof} Also see \S 1.19 of \cite{Borger11}.
\end{proof}

\begin{rema} Let us spell out the equivalence between $\Lambda$-rings and rings with an operator $\delta_\pi$ satisfying (\ref{prop:lambda-one-prime-is-locally-polynomial}) explicitly, in the case $A$ is flat.

If $\delta_\pi: A\to A$ is a map satisfying the conditions (i)--(iii) of (\ref{prop:lambda-one-prime-is-locally-polynomial}) then it follows that the map \[\psi_A^{\p}: A\to A: a\mto a^{N\p}+\pi \delta_\pi(a)\] is a $O$-algebra homomorphism satisfying \[\psi^{A}_\p(a)=a^{N\p}+\pi \delta_\pi(a)=a^{N\p}\bmod \p A.\]

Conversely, if $A$ is a $\p$-torsion free $\Lambda$-ring then the relation \[\psi_A(a)=a^{N\p}\bmod \p A\] implies that there is a unique $\delta_\pi(a)\in A$ such that \[\psi_A(a)=a^{N\p}+\pi \delta_\pi(a)\] and the fact that $\psi_A^\p$ is an $O$-algebra homomorphism forces the map $a\mto \delta_\pi(a)$ to satisfy conditions (i)--(iii) of (\ref{prop:lambda-one-prime-is-locally-polynomial}).
\end{rema}

\begin{rema} We can now explain why the approach with plethories was taken. To define a $\Lambda$-ring as an $O$-algebra with a Frobenius lift $\psi_A^\p: A\to A$ is a perfectly reasonable thing to do, however, there is a hidden existential quantifier in this definition: that for all $a\in A$ there exists an $a_\p\in \p A$ such that $\psi_A^\p(a)=a^{N\p}+a_\p$. This causes problems from the point of view of universal algebra and the effect of the plethystic approach is to remove this existential quantifier so that, rather than the Frobenius lift $\psi_A^\p$, is it the operator $\delta_\pi$ which determines the structure.
\end{rema}

\begin{coro}\label{coro:structure-map-lambda-ring-injective} If $A$ is a $\Lambda$-ring then the kernel of the homomorphism $O\to A$ is either prime to all $\p\in P$ or $A=0$ is the zero ring.
\end{coro}
\begin{proof} We shall prove this in the situation of (\ref{prop:lambda-one-prime-is-locally-polynomial}) remarking that it is possible to reduce to this case once more of the theory has been set-up. So let $A$ be a $\Lambda$-ring and assume that $\pi^n\in \p^n\subset \ker(O\to A).$ As the homomorphism $O\to A$ is a $\Lambda$-homomorphism the kernel is stabilised by the endomorphism $\delta_\pi: O\to O$ which is given by \[a\mto \frac{a^{N\p}-a}{\pi}.\] In particular, we see that \[\delta_\pi(\pi^n)=\frac{\pi^{N\p n}-\pi^n}{\pi}=\pi^{N\p n -1}-\pi^{n-1}=\pi^{n-1}(\pi^{N\p(n-1)}-1)\in \p^{n-1}.\] Therefore, $\p^{n-1}\subset \ker(O\to A)$ and by induction we find that $O\subset \ker(O\to A)$ and so $A$ is the zero ring.
\end{proof}

\subsection{} Let $P'\subset P$ be a sub-monoid generated by some set of prime ideals. Then the restriction of the map $h_{\Lambda_P}:\Lambda_{P}\odot \Lambda_{P}\to \Lambda_P$ along $\Lambda_{P'}\odot \Lambda_P\to \Lambda_P$ makes $\Lambda_{P}$ a $\Lambda_{P'}$-ring. Similarly, for $\Psi_P$ and $\Psi_{P'}$.

Now denote by \[P''=\{\a\in P: (\a, \b)=(1) \text{ for all } \b\in P'\}\subset P\] so that $P'\cdot P''=P$ and $P'\cap P''=\{O\}$. By the remark above the map $\Lambda_{P''}\to \Lambda_P$ extends by adjunction to a homomorphism of $\Lambda_{P'}$-rings \[\alpha_{P', P''}:\Lambda_{P'}\odot \Lambda_{P''}\to \Lambda_{P}\] which is also a homomorphism of $\Psi_{P}$-rings where $P''\subset P$ acts on the left factor.

\begin{prop}\label{prop:lambda-commute} The map $\alpha_{P', P''}: \Lambda_{P'}\odot \Lambda_{P''}\to \Lambda_{P}$ defined above is a $\Lambda_{P'}$-isomorphism.
\end{prop}
\begin{proof} This is Proposition 5.3 of \cite{Borger11}.
\end{proof}

\subsection{}\label{subsec:finite-length} We now define truncated versions of the birings $\Lambda$ and $\Psi$ as, when we come to geometrise the theory of $\Lambda$-rings, they will be more well behaved (cf.\ (\ref{prop:witt-vectors-preserve-etale-morphisms})).

We equip the $O$-algebra $\Psi\otimes_O K$ with an exhaustive filtration by sub-$O$-algebras, indexed by the elements of $P$ ordered by division, by setting $(\Psi\otimes K)_{\a}$ for $\a\in P$ to be the sub-$K$-algebra generated by the $\psi^\b$ such that $\b|\a$. The intersection of this filtration with the sub-$O$-algebras $\Lambda$ and $\Psi$ induces exhaustive filtrations by sub-$O$-algebras $\Lambda_{\a}\subset \Lambda$ and $\Psi_{\a}\subset \Psi$ and we have $\Psi_\a=O[\psi_\b: \b|\a]$.

\begin{prop}\label{prop:finite-length-good} We have the following:
\begin{enumerate}[label=\textup{(\roman*)}]
\item for each $\a\in P$ there is a unique $O$-biring structure on $\Lambda_{\a}$ making $\Psi_{\a}\to \Lambda_{\a}$ a biring homomorphism,
\item for each pair $\a, \b\in P$ the homomorphism $h_{\Lambda}: \Lambda\odot \Lambda\to \Lambda$ induces a homomorphism $\Lambda_{\a}\odot \Lambda_{\b}\to \Lambda_{\a\b}$ and this map is an isomorphism if $\a$ and $\b$ are relatively prime.
\end{enumerate} 
\end{prop}
\begin{proof}(i) This is Proposition 2.3 of \cite{Borger11}.

(ii) This is Propositions 2.3 and 5.3 of \cite{Borger11}.
\end{proof}

\subsection{} Let $A$ be an $O$-algebra. We write \[W(A)=\Hom_{O}(\Lambda, A) \quad \text{ and } \quad W_{\a}(A)=\Hom_{O}(\Lambda_{\a}, A)\] and call these the rings of Witt vectors and Witt vectors of length $\a$ of $A$. We also write \[\Gamma(A)=\Hom_{O}(\Psi, A)\isomto \prod_{\a\in P} A \quad \text{ and } \quad \Gamma_{\a}(A)=\Hom_{O}(\Psi_{\a}, A)\isomto \prod_{\b|\a} A\] and call these the rings of ghost vectors and length-$\a$ ghost vectors.

\begin{prop}\label{prop:witt-vectors-preserve-etale-morphisms} Let $\a\in P$ and let $A$ be an $O$-algebra. Then
\begin{enumerate}[label=\textup{(\roman*)}]
\item If $(A\to A_i)_{i\in I}$ is an \'etale cover of $A$ so is $(W_{\a}(A)\to W_{\a}(A_i))$, and
\item for all homomorphisms $A\to B$ and all \'etale homomorphisms $A\to A'$ the natural map \[W_{\a}(A')\otimes_{W_{\a}(A)}W_{\a}(B)\to W_{\a}(A'\otimes_A B)\] is an isomorphism.
\end{enumerate}
\end{prop}
\begin{proof} This is Theorem 9.2 and Corollary 9.3 of \cite{Borger11}.
\end{proof}

\section{Witt vectors and arithmetic jets II} \label{sec:lambda-section-two} The purpose of this section is to extend the definition of $\Lambda$-structures, Witt vectors and arithmetic jets to sheaves for the \'etale topology.

\subsection{} Recall that $\Sh_{O}^{\et}$ denotes the category of sheaves for the \'etale topology over $\Spec(O)$, or $\et$-sheaves over $\Spec(O)$. We begin with the arithmetic jets and coghosts. Let $X$ be an $\et$-sheaf and define the presheaves on $\Aff_O$ \[W_{\a*}(X):=X\circ W_{\a}: \Aff_O^\circ\to \Set\ \quad \Gamma_{\a*}(X):=X\circ \Gamma_{\a}: \Aff_O^\circ\to \Set.\]

\begin{prop}\label{prop:def-arithmetic-ghosts} Let $\a\in P$ and let be $X$ an $\et$-sheaf. Then \[W_{\a*}(X):=X\circ W_{\a}: \Aff_O^\circ\to \Set\ \quad \Gamma_{\a*}(X)(X):=X\circ \Gamma_{\a}: \Aff_O^\circ\to \Set\] are again $\et$-sheaves. Moreover, 
\begin{enumerate}[label=\textup{(\roman*)}]
\item if $X=\Spec(A)$ is affine then \[W_{\a *}(X)=\Spec(\Lambda_\a\odot A) \quad \text{ and }\quad\Gamma_{\a *}(X)=\Spec(\Psi_\a\odot A),\]
\item $W_{\a*}$ and $\Gamma_{\a*}$ commute with filtered colimits, and
\item $W_{\a*}$ sends smooth affine $\Spec(O_K)$-schemes to smooth affine $\Spec(O_K)$-schemes.
\end{enumerate}
\end{prop}
\begin{proof} It is clear that $\Gamma_{\a*}(X)$ of an $\et$-sheaf is again a sheaf, and for $W_{\a*}(X)$ it follows from (i) of (\ref{prop:witt-vectors-preserve-etale-morphisms}).

(i) This is clear for $\Gamma_{\a*}$ and Proposition 10.7 of \cite{Borger11bis}  for $W_{\a*}$.

(ii) This is clear for $\Gamma_{\a*}$ and Proposition 11.7 of \cite{Borger11bis} for $W_{\a*}$.

(iii) This is Proposition 13.3 of \cite{Borger11bis}.
\end{proof}

\subsection{} For an $\et$-sheaf $X$ we have the following simple description of $\Gamma_{\a}(X)$. The fact that $\Gamma_{\a}(A)=\Pi_{\b\in P, \b|\a} A$ for all $O$-algebras $A$ (\ref{subsec:finite-length}) shows that \[\Gamma_{\a*}(X)(A)=X(\Gamma_{\a}(A))=\prod_{\b\in P, \b| \a} X(A)=\prod_{\b\in P, \b| \a} X(A),\] that is \[\Gamma_{\a}(X)\isomto \prod_{\b\in P, \b | \a} X.\]

\subsection{} For an $\et$-sheaf $X$ we call $W_{\a*}(X)$ the ($\et$-sheaf of) length-$\a$ arithmetic jets of $X$ and $\Gamma_{\a*}(X)$ the ($\et$-sheaf of) length-$\a$ coghosts of $X$.

Standard sheaf theory supplies $W_{\a*}$ and $\Gamma_{\a *}$ with left adjoints which we denote by $W_{\a}^*$ and $\Gamma_{\a}^*$. For a sheaf $X$ we call $W_{\a}^*(X)$ the length-$\a$ Witt vectors of $X$ and $\Gamma_{\a}^*(X)$ the length $\a$-ghost vectors of $X$.

\begin{prop} If $X=\Spec(A)$ is an affine scheme then \[W_{\a}^*(\Spec(A))=\Spec(W_{\a}(A)) \quad \text{ and } \quad \Gamma_{\a}^*(\Spec(A))=\Spec(\Gamma_{\a}(A)).\]
\end{prop}
\begin{proof} This is true for all left adjoints to push forwards along a morphism of sites.
\end{proof}

\subsection{} For a sheaf $X$ we have the following simple description of $\Gamma_{\a*}(X)$. The fact that $\Gamma_{\a}(A)=\Pi_{\b\in P, \b| \a} A$ for all $O$-algebras $A$ (\ref{subsec:finite-length}) shows that: \begin{eqnarray*}\Gamma_{\a}^*(X)=\colim_{\Spec(A)\to X} \Gamma_{ \a}^*(\Spec(A))&=&\colim_{\Spec(A)\to X}\Spec(\Gamma_{\a}(A))\\
&=&\colim_{\Spec(A)\to X}\coprod_{\b\in P, \b|\a} \Spec(A)\\
&=&\coprod_{\b\in P, \b| \a} X\end{eqnarray*}

\subsection{} For a sheaf $X$, the inclusions $\Lambda_{\a}\to \Lambda_{\b}$ and $\Psi_{\a}\to \Psi_{\b}$ (\ref{subsec:finite-length}), for $\a, \b\in P$ with $\b| \a$, induce maps \[\Gamma_{\a*}(X)\to \Gamma_{\b*}(X) \quad \text{ and }\quad W_{\a*}(X)\to W_{\b*}(X)\] and taking the inverse limit in each case over the $\a\in P$ defines sheaves \[\Gamma_{*}(X):=\lim_\a \Gamma_{\a*}(X) \quad \text{ and } \quad W_{*}(X):=\lim_\a W_{\a*}(X)\] which we call the ($\et$-sheaf of) arithmetic jets of $X$ and ($\et$-sheaf of) coghosts of $X$.

Adjointly, there are also induced maps \[W_{\b}^*(X)\to W_{\a}^*(X)\quad \text{ and } \quad \Gamma_{\b}^*(X)\to \Gamma_{\a}^*(X)\] and taking the colimit we obtain $\et$-sheaves \[W^*(X):=\colim_\a W_{\a}^*(X)\quad \text{ and } \quad \Gamma^*(X):=\colim_\a \Gamma_{\a}^*(X)\] which we call the ($\et$-sheaf of) ghost vectors of $X$ and the ($\et$-sheaf of) Witt vectors of $X$. By construction the functor $\Gamma^*$ is left adjoint to $\Gamma_{*}$, and $W^*$ is left adjoint to $W_{*}$.

Generalising the descriptions of $\Gamma_{\a*}(X)$ and $\Gamma_{\a}^*(X)$ we have \[\Gamma_{*}(X)\isomto \prod_{\a\in P} X \quad \text{ and }\quad \Gamma^*(X)\isomto \coprod_{\a\in P}X.\]

\subsection{}\label{def:w-w-lambda-adjunction} The maps $\Lambda_{\a}\odot \Lambda_{\b}\to \Lambda_{\a\b}$ (\ref{prop:finite-length-good}) induce maps \[W^*_{\b}(W^*_{\a}(X))\to W^*_{\a\b}(X) \quad \text{ and } \quad W_{\a\b*}(X)\to W_{\a*}(W_{\b*}(X))\] and taking the colimit and limit respectively over all $\a, \b\in P$ we obtain maps \[\mu_{W^*(X)}: W^*(W^*(X))\to W^*(X) \quad \text{ and } \quad h_{W_*(X)}: W_{*}(X)\to W_{*}(W_{*}(X))\] which together with the natural maps \[g_{(1)}: X\to W^*(X) \quad \text{ and } \quad \gamma_{(1)}: W_{*}(X)\to X\] induced by the compatible maps $O[e]\to \Lambda_{\a}$ for all $\a\in P$ equip $W^*$ with the structure of a monad and $W_{*}$ with the structure of a comonad on $\Sh_{O}^\et$. We define the category of $\Lambda$-sheaves $\Sh_{\Lambda}^{\et}$ to be the category of $\et$-sheaves $X$ equipped with an action of $W^*$ or equivalently a coaction of $W_{*}$. The functors $W^*$ and $W_{*}$ now define functors $\Sh_{O}^{\et}\to \Sh_{\Lambda}^{\et}$ which are left and right adjoint to the forgetful functor $\Sh_{\Lambda}^{\et}\to \Sh_{O}^\et$. We have the following diagram of functors each one right adjoint to the one below:
\[\quad \xymatrix{\Sh_{\Lambda}^\et\ar[rr]^{\textup{forget}} && \ar@/^2pc/[ll]^{W^*}\ar@/_2pc/[ll]_{W_{*}}\Sh_{O}^{\et}.}\]

In particular, the forgetful functor in the middle admits both a left and right adjoint and so commutes with both limits and colimits. That is, limits and colimits in $\Sh_{\Lambda}^\et$ may be computed in $\Sh_{O}^\et$.

\subsection{} The same construction above applied to the maps $\Psi_{\a}\odot \Psi_{\b}\to \Psi_{\a\b}$ equips $\Gamma^*$ with the structure of a monad and $\Gamma_{*}$ with the structure of a comonad on $\Sh_{O}^\et$. In this case, it is easy to see that the map obtained \[\mu_{\Gamma^*(X)}: \Gamma^*(\Gamma^*(X))\to \Gamma^*(X)\] is identified with the map \[\coprod_{\b\in P}\coprod_{\a\in P} X=\coprod_{\a, \b\in P}X\to \coprod_{\c\in P} X\] whose restriction to the summand at $\b, \a\in P$ is the inclusion onto the summand at $\c=\a\b\in P$. Similarly, the map obtained \[h_{\Gamma_*(X)}:\Gamma_{*}(X)\to \Gamma_{*}(\Gamma_{*}(X))\] is identified with the map \[\prod_{\c\in P} X \to \prod_{\a\in P}\prod_{\b\in P} X\] whose composition with the projection onto the factor at $\a, \b\in P$ is the projection onto the factor at $\c=\a\b$.

We define the category of $\Psi$-sheaves $\Sh_{\Psi}^{\et}$ to be the category of sheaves $X$ equipped with an action of $\Gamma^*$ or equivalently with a coaction of $\Gamma_{*}$. Then $\Gamma^*$ and $\Gamma_{*}$ define functors $\Sh_{O}^{\et}\to \Sh_{\Psi}^{\et}$ which are left and right adjoint to the forgetful functor $\Sh_{\Psi}^{\et}\to \Sh_{O}^\et$. We have the following diagram of functors each one right adjoint to the one below:
\[\quad \xymatrix{\Sh_{\Psi}^\et\ar[rr]^{\textup{forget}} && \ar@/^2pc/[ll]^{\Gamma^*}\ar@/_2pc/[ll]_{\Gamma_{*}}\Sh_{O}^{\et}.}\]

\subsection{}\label{subsec:psi-action-sheaves} From the descriptions of the monadic and comonadic structures on $\Gamma^*$ and $\Gamma_*$ it is easy to see that to give a sheaf $X$ the structure of $\Psi$ sheaf is equivalent to equipping $X$ with an action of the monoid $P$. Indeed, denoting the action of $\a\in P$ by $\psi^\a_X: X\to X$ the map $\Gamma^*(X)\to X$ defining the corresponding action of $\Gamma^*$ is just \[\coprod_{\a\in P} \psi^\a_X: \coprod_{\a\in P} X=\Gamma^*(X)\to X.\]

As with $\Psi$-rings if $X\to S$ is a morphism of $\Psi$-sheaves the $\psi_S^\a$-linear endomorphisms $\psi_X^\a: X\to X$ for $\a\in P$ induce $S$-linear endomorphisms \[\psi_{X/S}^{\a}: X\to \psi_{S}^{\a*}(X)\] satisfying the commutativity condition \begin{equation}\psi_S^{\b*}(\psi_{X/S}^{\a})\circ \psi_{X/S}^\b=\psi_{X/S}^{\a\b}=\psi_S^{\a*}(\psi_{X/S}^{\b})\circ \psi_{X/S}^\a\label{eqn:semi-linear-comm-sheaves}\end{equation} for all $\a, \b\in P$. Moreover, to give an $\et$-sheaf $X\to S$ over $S$ a $\Psi$-structure such that the morphism $X\to S$ is a $\Psi$-morphism is the same as giving maps $\psi_{X/S}^{\a}: X\to \psi_{S}^{\a*}(X)$  all $\a\in P$ with $\psi^{(1)}_{X/S}=\id_X$ and satisfying the commutativity condition (\ref{eqn:semi-linear-comm-sheaves}).

\subsection{} The inclusions $\Lambda_{\a}\to \Psi_{\a}$ induce functorial maps for each sheaf $X$ \[g_{\leq \a}: \Gamma_{\a}^*(X)\to W_{\a}^*(X)\quad \text{ and } \quad\gamma_{\leq \a}: W_{ \a*}(X)\to \Gamma_{\a*}(X)\] which we call the length-$\a$ ghost and coghost maps.

The inclusions \[O[\psi_\a]\subset O[\psi_\b:\b\in P, \b|\a]=\Psi_\a\] induce for each $\a$ the $\a$-ghost component and $\a$-coghost components \[g_\a: X\to \Gamma^*(X)\to W^*(X) \quad \gamma_\a: W_*(X)\to \Gamma_*(X)\to X\] which we denote by $g_\a$ and $\gamma_\a$ and view as maps from $X\to W^*(X)$ or $X\to \Gamma^*(X)$ and similarly for the $\a$-coghost maps. We then have \[g_{\leq \a}=\coprod_{\b|\a}g_\b: \Gamma^*_\a(X)=\coprod_{\b|\a} X\to \Gamma^*(X).\]

Taking the colimit and limit along the length-$\a$ ghost and coghost maps we obtain the full ghost and coghost maps \[g: \Gamma^*(X)\to W^*(X)\quad \text{ and } \quad \gamma: W_{*}(X)\to \Gamma_{*}(X).\]

Finally, the natural transformations given by the ghost and coghost maps $g: \Gamma^*\to W^*$ and $\gamma: W_{*}\to \Gamma_{*}$ are morphisms of monads and comonads respectively and every $\Lambda$-sheaf $X$ inherits the structure of a $\Psi$-sheaf. That is every $\Lambda$-sheaf $X$ admits a canonical action of the monoid $P$ which, as for $\Psi$-sheaves, we denote by $\psi^\a_X:X\to X$ for $\a\in P$ (and similarly for the relative versions (cf.\ (\ref{subsec:psi-action-sheaves})).

\begin{prop}\label{prop:lambda-structures-frobenius-lifts-sheaves} We have the following:
\begin{enumerate}[label=\textup{(\roman*)}]
\item For each $\Lambda$-sheaf $X$ and each prime $\p\in P$ the map \[\psi^\p_X\times_{\Spec(O)}\Spec(\F_\p): X\times_{\Spec(O)}\Spec(\F_\p)\to X\times_{\Spec(O)}\Spec(\F_\p)\] is equal to the $N\p$-power Frobenius endomorphism of $X\times_{\Spec(O)}\Spec(\F_\p)$.
\item A $\Psi$-structure on a scheme $X$, flat over $O$ at all primes $\p\in O$, is induced by a $\Lambda$-structure on $X$ if and only if for each prime $\p\in P$ the map \[\psi^\p_X\times_{\Spec(O)}\Spec(\F_\p): X\times_{\Spec(O)}\Spec(\F_\p)\to X\times_{\Spec(O)}\Spec(\F_\p)\] is the $N\p$-power Frobenius endomorphism. Moreover, in this case such a $\Lambda$-structure on $X$ is unique.
\item If each $\p\in P$ is invertible on $X$ then every $\Psi$-structure on $X$ is induced by a unique $\Lambda$-structure.
\end{enumerate}
\end{prop}
\begin{proof} See \cite{Borger11ter} or it is an exercise using Theorem 17.3 of \cite{Borger11bis}.
\end{proof}

\begin{prop}\label{prop:w-and-etale-base-change} We have the following:
\begin{enumerate}[label=\textup{(\roman*)}]
\item If $f: X \to Y$ an affine \'etale morphism of ind-affine-schemes then for each $\a\in P$ the morphism \[W_{\a}^*(f): W_{\a}^*(X)\to W_{\a}^*(Y)\] is affine and \'etale.
\item If $f: X\to Y$ is a affine \'etale morphism of ind-affine-schemes then for any affine morphism $Y'\to Y$ of ind-schemes the natural map \[W_{\a}^*(X\times_Y Y')\to W_{\a}^*(X)\times_{W_{\a}^*(Y)}W_{\a}^*(Y') \] is an isomorphism.
\end{enumerate}
\end{prop}
\begin{proof} The case where $X$ and $Y$ are affine follows from (\ref{prop:witt-vectors-preserve-etale-morphisms}) and we shall reduce to this case. So write $Y=\colim_{i\in I} Y_i$ as a filtered colimit of affine schemes, and for $i\in I$ write $X_i=X\times_Y Y_i$.

(i) Let $i, j\in I$ with $Y_i\to Y_j\to Y$. Then as $Y_i$ and $Y_j$ are affine and $X\to Y$ is affine we have a cartesian diagram \[\xymatrix{W^{*}_\a(X_j)\ar[r] & W^{*}_\a(Y_j)\\
W^{*}_\a(X_i)\ar[u]\ar[r] & W^{*}_\a(Y_i)\ar[u]}\] where the bottom arrow is affine and \'etale. Now taking the colimit over $j$ yields a cartesian diagram \[\xymatrix{W^{*}_\a(X)\ar[r] & W^{*}_\a(Y)\\
W^{*}_\a(X_i)\ar[u]\ar[r] & W^{*}_\a(Y_i)\ar[u]}\] where the bottom arrow is affine, so that as $(W^*_\a(Y_i)\to W^*_\a(Y))_{i\in I}$ is a cover we are done.

(ii) Writing $Y'_i=Y'\times_Y Y_i$ the morphism \[W_{\a}^*(X\times_Y Y')\to W_{\a}^*(X)\times_{W_{\a}^*(Y)}W_{\a}^*(Y')\] is the colimit over $i\in I$ of the morphisms \[W_{\a}^*(X_i\times_{Y_i} Y_i')\to W_{\a}^*(X_i)\times_{W_{\a}^*(Y_i)}W_{\a}^*(Y'_i)\] which is an isomorphism as $Y_i, Y'_i$ and $X_i$ are all affine.
\end{proof}

\section{$\Lambda$-structures}\label{sec:lambda-section-three}

\subsection{} Let $S$ be a $\Lambda$-sheaf so that $S$ is equipped with an action of the monad $W^*$ and equivalently a coaction of the monad $W_*$.

From the monadic point of view we write \[\mu_S: W^*(S)\to S\] for map defining the action of $W^*$ on $S$ and we call this the structure map. The fact that it defines an action of $W^*$ on $S$ is the expressed via the commutativity of the diagram \[\xymatrix{& W^*(S)\ar[d]^{\mu_S}\\S\ar[r]^-{\id_S}\ar[ur]^{g_{(1)}} & S}\] and the property that the two compositions \[\xymatrix{W^*(W^*(S))\ar@<-.5ex>[r]_-{\mu_{W^*(S)}}\ar@<.5ex>[r]^-{W^*(\mu_S)} & W^*(S) \ar[r]^-{\mu_{S}} & S}\] coincide.

From the comonadic point of view we write \[h_S: S\to W_*(S)\] for the map defining the coaction of $W_*$, and the fact that it defines a coaction is expressed via the commutativity of the diagram \[\xymatrix{S\ar[r]^-{h_S}\ar[dr]_{\id_S} & W_*(S)\ar[d]^{\gamma_{(1)}}\\
& S}\] and the fact that the two compositions \[\xymatrix{S \ar[r]^{h_{S}} \ar[r] & W_*(S)\ar@<-.5ex>[r]_-{h_{W_*(S)}}\ar@<.5ex>[r]^-{W_*(h_S)} & W_*(W_*(S))}\] are equal.

\subsection{} We denote by $\Sh_{\Lambda_S}^\et$ the category of $\Lambda$-sheaves equipped with a $\Lambda$-morphism $X\to S$. The functor \[W_{S*}:\Sh_{S}^{\et}\to \Sh_{\Lambda_{S}}^\et: X\mto W_{*}(X)\times_{W_{*}(S)}S\] composed with the forgetful functor is a comonad on $\Sh_S^{\et}$ and identifies the category of $S$-sheaves $X$ with a coaction of $W_{S*}$ with the category of $\Lambda$-sheaves equipped with a $\Lambda$-morphism $X\to S$.

\begin{prop}\label{prop:hom-lambda-sheaves-are-sheaves} Let $S$ be a $\Lambda$-sheaf and let $X$ and $Y$ be a pair of $\Lambda_S$-sheaves. Then the functor \[\Sh_{\Lambda_S}^{\et}\to \Set: S'\mto \Hom_{\Lambda_{S'}}(X_{S'}, Y_{S'})\] is a sheaf for the canonical topology.
\end{prop}
\begin{proof} Let $S'\to S$ be a cover of $S$ in $\Sh_{\Lambda_S}$. Then $S'\to S$ is also a cover when viewed in $\Sh^{\et}_S$ and so any $\Lambda_{S'}$-morphism $f':X_{S'}\to Y_{S'}$ with the property that the two pull-backs of $f$ along the two projections $S'\times_S S'\to S$ coincide descends to a morphism $X\to Y$. It remains to verify that it is a $\Lambda_S$-morphism, which is the commutivity of the \begin{equation}\label{eqn:lambda-morphism}\begin{gathered}\xymatrix{X\ar[r]^-{h_{X/S}}\ar[d]_f & W_{*}(X)\times_{W_{*}(S)}S\ar[d]^{W_{*}(f)\times_{W_{*}(S)}S}\\
Y\ar[r]^-{h_{Y/S}} & W_{*}(Y)\times_{W_{*}(S)}S.}\end{gathered}\end{equation}However, using the identifications \begin{eqnarray*} W_{*}(X)\times_{W_{*}(S)} S\times_S S'&=&W_{*}(X)\times_{W_{*}(S)}S'\\
&=&W_{*}(X)\times_{W_{*}(S)}W_{*}(S')\times_{W_{*}(S')} S'\\
&=&W_{*}(X\times_S S')\times_{W_{*}(S')}S'
\end{eqnarray*} and similarly for $Y$, the diagram (\ref{eqn:lambda-morphism}) pulls-back along $S'\to S$ to the diagram

\begin{equation*}\begin{gathered}\xymatrix{X_{S'}\ar[r]^-{h_{X_{S'}/S'}}\ar[d]_{f'} & W_{*}(X)\times_{W_{*}(S')}S'\ar[d]^{W_{*}(f')\times_{W_{*}(S')}S'}\\
Y_{S'}\ar[r]^-{h_{Y_{S'}/S'}} & W_{*}(Y)\times_{W_{*}(S')}S'}\end{gathered}\end{equation*}
which commutes by hypothesis. Therefore, as $S'\to S$ is an epimorphism, the diagram (\ref{eqn:lambda-morphism}) commutes.
\end{proof}

\subsection{}\label{subsec:sheaves-of-lambda-structures} Let $S$ be a $\Lambda$-sheaf and let $X\to S$ be an $S$-$\et$-sheaf. Denote by $\Lambda_{X/S}$ the functor \[\Lambda_{X/S}: \Sh_{\Lambda_{S}}^\et\to \Set: S'/S\mto \{\text{the set of }\Lambda_{S'}\text{-structures on }X\times_S S'\}.\]
\begin{prop}\label{prop:lambda-structures-are-sheaves} The functor $\Lambda_{X/S}$ is a sheaf for the canonical topology on $\Sh_{\Lambda_{S}}^\et$.
\end{prop}
\begin{proof} Let $S'\to S$ be an epimorphism and write $S''=S\times_S S'$, $X'=X\times_S S'$ and $X''=X\times_S S''$. By the definition of a $\Lambda_S$-structure (in terms of $W^*$) we have an equaliser (the Homs are of $\et$-sheaves not $\Lambda$-sheaves) \[\xymatrix{\Lambda_{X/S}(S)\ar[r] & \Hom_S(W^*(X), X)\ar@<.5ex>[r]\ar@<-.5ex>[r] & \Hom_S(X, X)\times \Hom_S(W^*(W^*(X)), X)}\] where the first map sends a $\Lambda_S$-structure on $S$ to the map $\mu_X: W^*(X)\to X$ and the two parallel arrows send a map $\mu_X$ to \[(\mu_X\circ g_{(1)}, \mu_{W^*(X)}\circ \mu_X) \quad \text{ and } \quad (\id_X, W^*(\mu_X)\circ \mu_X).\]

Now consider corresponding commutative diagram: \[\xymatrix{\Lambda_{X/S}(S)\ar[r]\ar[d] & \Hom_S(W^*(X), X)\ar@<.5ex>[r]\ar@<-.5ex>[r]\ar[d] & \Hom_S(X, X)\times \Hom_S(W^*(W^*(X)), X)\ar[d]\\
\Lambda_{X/S}(S')\ar[r]\ar@<.5ex>[d]\ar@<-.5ex>[d] & \Hom_S(W^*(X'), X)\ar@<.5ex>[r]\ar@<-.5ex>[r]\ar@<.5ex>[d]\ar@<-.5ex>[d] & \Hom_S(X', X)\times \Hom_S(W^*(W^*(X')), X)\ar@<.5ex>[d]\ar@<-.5ex>[d]\\\Lambda_{X/S}(S'')\ar[r] & \Hom_S(W^*(X''), X)\ar@<.5ex>[r]\ar@<-.5ex>[r] & \Hom_S(X'', X)\times \Hom_S(W^*(W^*(X'')), X).
}\]

The two right columns are equalisers as $X'=X\times_S S'\to X$ is a cover and $W^*$ preserves push-outs and all three rows are equalisers. Therefore it follows that the first column is an equaliser and we are done.
\end{proof}

\begin{prop}\label{prop:lambda-structures-etale-schemes-number-fields} Let $O$ be the ring of integers in a number field $K$, $\f\in \Id_O$ an ideal and write $P=\Id_{O}^{(\f)}$. Then a finite \'etale $S=\Spec(O[\f^{-1}])$-scheme $X$ admits a $\Lambda_{P, S}$-structure if and only if $X\times_{S}\Spec(K)=\amalg_{1\leq i\leq n}\Spec(L_i)$ where each $L_i$ is a finite abelian extension of $K$. Moreover, in this case the $\Lambda_{P, S}$-structure is unique and any morphism of finite \'etale $\Lambda_{P, S}$-schemes is a $\Lambda_{P, S}$-morphism.
\end{prop}
\begin{proof} Let $X\to S$ be a finite \'etale morphism. Then $X\times_S \Spec(K)=\amalg_{i\in I} \Spec(L_i)$ with $L_i/K$ a finite extension and $X\isomto \amalg_{i} X_i$ where $X_i=\Spec(O_{L_i}[\f^{-1}])$.
If $X\to S$ has a $\Lambda_{P, S}$-structure then for each prime $\p\in P=\Id_{O_K}^{(\f)}$, the Frobenius lift $\psi^\p_X: X\to X$ fixes the fibre of each $X_i$ over $\Spec(\F_\p)$ so that $\psi_X^\p$ maps each $X_i$ to itself. It follows that the $\Lambda_{P, S}$-structure on $X$ is induced by unique $\Lambda_{P, S}$-structures on each $X_i$ and so we may assume that $X=X_i$ is connected. In this case, write $X=\Spec(O_{L}[\f^{-1}])$ with $L/K$ a finite extension. If $\P$ is a prime of $L$ laying over the prime $\p$ then there is a unique automorphism \[\sigma_{\P, L/K}: X\to X\] whose restriction to fibre over $\Spec(O_L/\P)$ is equal to the $N\p$-power Frobenius automorphism. But $\psi^\p_X$ also has this property and so $\sigma_{\P, L/K}=\psi^\p_X$. As the maps $\psi^\p_X=\sigma_{\P, L/K}$ commute and generate the group $G(L/K)$ it follows that $G(L/K)$ is abelian. Moreover, the uniqueness of $\sigma_{\P, L/K}=\sigma_{\p, L/K}=\psi_{X}^\p$ shows that the $\Lambda_{P}$-structure on $X$ is unique, is also a $\Lambda_{P, S}$-structure and that any morphism of finite \'etale $\Lambda_{P, S}$-schemes is a $\Lambda_{P, S}$-morphism.

Conversely, if each $L_i/K$ is abelian then there is a unique automorphism \[\sigma_{\p, L_i/K}: X_i\to X_i\] lifting the $N\p$-power Frobenius automorphism of the fibre over $\Spec(\F_\p)$. It follows that setting $\psi^\p_X=\amalg_{i\in I}\sigma_{L_i/K, \p}$ defines a Frobenius lift on $X$. Moreover, as $G(L_i/K)$ is abelian these Frobenius lifts commute and so define a $\Lambda_{P, S}$-structure on $X$.
\end{proof}

\begin{rema} In the notation of (\ref{prop:lambda-structures-etale-schemes-number-fields}) if $X\to S$ is a finite \'etale $\Lambda_P$-scheme then its Frobenius lifts $\psi_{X}^\a$ for $\a\in P$ will be denoted by $\sigma_{S, \a}$, or just $\sigma_\a$. This agrees with (or extends) the conventions set up in (\ref{subsec:finite-extensions-frobenius}).
\end{rema}

\section{Ghosts and coghosts}
\begin{prop}\label{prop:ghost-map-surj-geom-points} Let $X$ be a sheaf. For each $\a\in P$ the length-$\a$ ghost map \[g_{\a}:\Gamma_{\a}^*(X)\to W_{\a}^*(X)\] is surjective on geometric points and so is the full ghost map \[g: \Gamma^*(X)\to W^*(X).\]
\end{prop}
\begin{proof} Writing $X=\colim \Spec(A)$ as a colimit of affine schemes, and $W^{*}(X)=\colim_\a W_{\a}^*(X)$ and $\Gamma^{*}(X)=\colim_\a \Gamma_{\a}^*(X)$ it is enough to show this for \[\Spec(\Gamma_\a(A))=\Gamma_{\a}^*(\Spec(A))\to \Spec(W_\a(A))=W^*_\a(\Spec(A)).\] But the map $W_\a(A)\to \Gamma_\a(A)$ is integral with nilpotent kernel (Proposition 8.1 of \cite{Borger11}) and therefore \[\Spec(\Gamma_\a(A))\to \Spec(W_\a(A))\] is surjective on geometric points.
\end{proof}

\begin{prop}\label{prop:ghost-covers-etale} If $S$ is an ind-affine scheme, and $T$ an affine \'etale $W^*(S)$-sheaf then the sequence \[\xymatrix{T(W^*(S))\ar[r]^{T(g)} & T(\Gamma^*(S))\ar@<.5ex>[r]\ar@<-.5ex>[r] & T(\Gamma^*(S)\times_{W^*(S)}\Gamma^*(S))}\] is an equaliser.
\end{prop}
\begin{proof} As $W^*(S)=\colim_{\a\in P}W^*_\a(S)$ and $\Gamma^*(S)=\colim_{\a\in P}\Gamma^*_\a(S)$ and filtered colimits are exact, we may replace $W^*(S)$ and $\Gamma^*(S)$ and $T$ with $W^*_\a(S)$, $\Gamma^*_\a(S)$ and $T\times_{W^*(S)}W^*_\a(S)$ respectively. Now writing $S$ as a filtered colimit of affine schemes we may assume that $S$ is affine in which case the claim follows as $\Gamma^*_\a(S)\to W_{\a}^*(S)$ is integral and surjective, and therefore an effective descent map for the category of affine \'etale schemes.
\end{proof}

\begin{rema}\label{rema:ghost-fibre-product} In order for (\ref{prop:ghost-covers-etale}) to be useful in applications we should say something about the fibre product $W^*(S)\times_{\Gamma^*(S)}W^*(S)$. Let $\p$ be a prime ideal, $\a$ any ideal, $n\geq 1$ an integer and write $S_\p=S\times_{\Spec(O)}\Spec(\F_\p)$ and consider the two maps (where the inclusions are the obvious ones) \[r_{\a, \p^n}^{(1)}:S_\p\subset S \stackrel{g_{\a\p^n}}{\longrightarrow} \Gamma^*(S) \quad \text{and} \quad r_{\a, \p^n}^{(2)}:S_\p \stackrel{\Fr_{S_\p}^{N\p^n}}{\longrightarrow} S_\p\subset S \stackrel{g_\a}{\to} \Gamma^*(S).\] Then for $i, j\in \{1, 2\}$ the two compositions

\[\xymatrix{S_\p\ar[rr]^-{(r_{\a, \p^n}^{(i)}, r_{\a, \p^n}^{(j)})} && \Gamma^*(S)\times_{W^*(S)}\Gamma^*(S)\ar@<.5ex>[r]\ar@<-.5ex>[r] & W^*(S)}\] are equal and the morphism \[\coprod_{i\neq j\in \{1, 2\}}\coprod_{\a, \p^n} S_\p \stackrel{(r_{\a, \p^n}^{(i)}, r_{\a, \p^n}^{(j)})}{\longrightarrow}\Gamma^*(S)\times_{W^*(S)}\Gamma^*(S)\] defines a nilpotent immersion onto the complement of the diagonal in $\Gamma^*(S)\times_{W^*(S)}\Gamma^*(S)$ (this follows by an iterated application of 17.1 \cite{Borger11bis}). Therefore, in the notation of (\ref{prop:ghost-covers-etale}) an element of $T(\Gamma^*(S))$ is in the image of $T(W^*(S))\to T(\Gamma^*(S))$ if and only if for all prime ideals $\p$, ideals $\a$, integers $n\geq 0$ and pairs $i\neq j\in \{1, 2\}$, it is equalised by the maps \[\xymatrix{T(\Gamma^*(S))\ar@<.7ex>[rr]^-{T(r_{\a, \p^n}^{(i)})}\ar@<-.7ex>[rr]_-{T(r_{\a, \p^n}^{(j)})} && T(S_\p).}\]

\end{rema}
\begin{lemm}\label{lemm:coghost-is-affine-finite-point-property} If $X$ is a scheme with the property that every finite set of points of $X$ is contained in an open affine sub-scheme of $X$ then for each $\a\in P$ the length-$\a$ coghost map \[\gamma_{\leq \a}: W_{\a*}(X)\to \Gamma_{\a*}(X)\] is affine. 
\end{lemm}
\begin{proof} The property satisfied by $X$ implies that there is an open affine cover $(X_i)_{i\in I}$ of $X$ such that $(\Gamma_{\a*}(X_i))_{i\in I}$ is an open cover of $\Gamma_{\a*}(X)$ (recall that $\Gamma_{\a*}$ of a sheaf $X$ is just a finite product of copies of $X$). However, the diagram \begin{equation}\label{dia:coghost-affine}
\begin{gathered}\xymatrix{W_{\a*}(X_i)\ar[r]^{\gamma_{\leq \a}}\ar[d] & \Gamma_{\a*}(X_i)\ar[d]\\
W_{\a*}(X)\ar[r]^{\gamma_{\leq \a}} & \Gamma_{\a*}(X)}\end{gathered}\end{equation} is cartesian by Proposition 12.2 of \cite{Borger11bis} (\cite{Borger11bis} also assumes that the open immersions $X_i\to X$ are closed but the proof that the diagram (\ref{dia:coghost-affine}) is cartesian does not use this assumption) and the top morphisms is affine. As $(\Gamma_{\a*}(X_i))_{i\in I}$ is a cover of $\Gamma_{\a*}(X)$ it follows that the bottom row of (\ref{dia:coghost-affine}) is affine.
\end{proof}
\subsection{} Let $S$ be a $\Lambda$-sheaf and $X$ an $S$-group sheaf. Then to equip $X$ with a $\Lambda_{S}$-structure making it a $\Lambda_{S}$-sheaf in groups over $S$ is equivalent to equipping it with a $\Lambda_{S}$-structure such that the defining structure map \[h_{X/S}: X\to W_{*}(X)\times_{W_{*}(S)}S\] is a homomorphism of $S$-groups (recall the $W_*$ being a right adjoint preserves limits so that $W_*(X)$ is a $\Lambda_{W_*(S)}$-sheaf of groups over $W_*(S)$).

\begin{lemm}\label{lemm:coghost-affine-relative-coghost-affine} Let $S=\colim_{i\in I} S_i$ be a $\Lambda$-ind-affine scheme, $f: X\to S$ be an $\et$-sheaf over $S$ and $\a\in P$. If for each $i\in I$, setting $X_i=X\times_S S_i$, the length-$\a$ coghost map \[\gamma_{\leq \a}: W_{\a*}(X_i)\to \Gamma_{\a*}(X_i)\] is affine then the length-$\a$ relative coghost map \[\gamma_{X/S, \leq \a}: W_{\a*}(X)\times_{W_{\a*}(S)}S\to \Gamma_{\a*}(X)\times_{\Gamma_{\a*}(S)}S\] is affine.
\end{lemm}
\begin{proof} The morphism $\gamma_{X/S, \leq \a}$ if affine if and only if the morphisms \[\gamma_{X/S, \a}\times_S S_i\] are affine for each $i\in I$. Fixing such an $i$, as $S_i$ is affine (in particular, quasi-compact) and \[\colim_i W^*_\a(S_i)=W_{\a*}(\colim_i S_i)\] (by (iv) of (\ref{prop:def-arithmetic-ghosts})) there is some $j\in I$ such that $S_i\to W_{\a*}(S)$ factors through $W_{\a*}(S_j)\to W_{\a*}(S)$. Therefore, we can factor $\gamma_{X/S, \a}\times_S S_i$ as the composition \[W_{\a*}(X_j)\times_{W_{\a*}(S_j)} S_i\to W_{\a*}(X_j)\times_{\Gamma_{\a*}(S_j)} S_i\to \Gamma_{\a*}(X_j)\times_{\Gamma_{\a*}(S_j)} S_i\] where the first map is induced by $W_{\a*}(S_j)\to \Gamma_{\a *}(S_j)$, which is affine as $W_{\a*}(S_j)\to \Gamma_{\a *}(S_j)$ is affine, and the second is $\gamma_{\leq \a}\times_{\Gamma_{\a*}(S_j)} S_i$ which is affine as $\gamma_{X_j, \a}$ is affine by hypothesis. Therefore, $\gamma_{X/S, \a}\times_S S_i$ is affine and we are done.
\end{proof}

\begin{prop}\label{prop:lambda-and-psi-homomorphisms-of-abelian-varieties} Let $S$ be an $\Lambda$-ind-affine scheme, let $A$ and $A'$ be a pair of abelian schemes over $S$ and let $f: A\to A'$ be an $S$-homomorphism. If $f$ is a $\Psi_S$-morphism then it is a $\Lambda_S$-morphism.
\end{prop}
\begin{proof} Write $S=\colim_{i\in I} S_i$ as a filtered colimit of affine schemes. By Theorem 1.9 of Chapter I of \cite{FaltChai1990}, for each $i\in I$ and $\a\in P$, the $S_i$-scheme $A'\times_S S_i$ satisfies the hypotheses of (\ref{lemm:coghost-is-affine-finite-point-property}) so that we may apply (\ref{lemm:coghost-affine-relative-coghost-affine}) to deduce that the relative coghost homomorphism of length $\a$ \[\gamma_{A'/S, \a}: W_{\a*}(A')\times_{W_{\a*}(S)}S\to \Gamma_{\a*}(A')\times_{\Gamma_{\a*}(S)}S\] is affine. Taking the limit over $\a$ we see that \[\gamma_{A'/S}: W_{*}(A')\times_{W_{*}(S)}S\to \Gamma_{*}(A')\times_{\Gamma_{*}(S)}S\] is affine.
Now let the $\Lambda_{S}$-structures on $A$ and $A'$ be given by the $S$-homomorphisms \[h_{A/S}: A\to W_{*}(A)\times_{W_{*}(S)} S \quad \text{ and } \quad h_{A'/S}: A'\to W_{*}(A')\times_{W_{*}(S)} S\] and let $f: A\to A'$ be a $\Psi_{S}$-homomorphism. By hypothesis, the difference \[h_{A'/S}\circ f-(W_{*}(f)\times_{W_{*}(S)}S)\circ h_{A/S}: A\to W_{*}(A')\times_{W_*(S)}S\] factors through the kernel of the relative coghost homomorphism $\gamma_{A/S}$. However, the kernel of $\gamma_{A/S}$ is affine over $S$ and any homomorphism from an abelian $S$-scheme to an $S$-affine scheme is trivial. Therefore, \[h_{A'/S}\circ f-(W_{*}(f)\times_{W_{*}(S)}S)\circ h_{A/S}=0\] and $f$ is a $\Lambda_{S}$-homomorphism.
\end{proof}
\begin{rema}\label{rema:lambda-structures-on-abelian-varieties-are-determined-by-psi-structures} It follows from (\ref{prop:lambda-and-psi-homomorphisms-of-abelian-varieties}) above that if $A$ is an abelian variety over an $\Lambda$-ind-affine scheme $S$ then any two $\Lambda_S$-structures on $A$, compatible with the group law, coincide if and only if the underlying $\Psi_S$-structures coincide.
\end{rema}

\chapter{CM elliptic curves and $\Lambda$-structures} In this chapter we explain the connection between CM elliptic curves and $\Lambda$-structures. The first and main result being, essentially, that the moduli stack of CM elliptic curves $\mc{M}_\CM$ admits a $\Lambda$-structure. The observant reader will have noticed that we have not defined what it means for a stack to have a $\Lambda$-structure and we do not propose to do so here (not because we cannot but only because to do so would involve various 2-categorical issues that would in this instance serve only to make matters more complicated than they need be). However, we shall explain what we mean to prove.

We continue with the set-up in (\ref{subsec:imag-quad-set-up}). So that all sheaves considered are always over $\Spec(O_K)$ and we will use the theory of Chapter 3 with $\Lambda$-structures relative to the full monoid of ideals $\Id_{O_K}$ (later we will also consider sub-monoids).

Recall that to give a $\et$-sheaf $X$ a $\Lambda$-structure is to give for each $\et$-sheaf $S$, a map \[h_{X}(S): X(S)\to X(W^*(S))\] which is functorial in $S$ and such that:
\begin{enumerate}[label=\textup{(\roman*)}]
\item the composition \[X(S)\stackrel{h_X(S)}{\to} X(W^*(S))\stackrel{X(g_{(1)})}{\to} X(S)\] is the identity and
\item the two compositions \[\xymatrix{X(S)\ar[r]^-{h_X(S)} & X(W^*(S)) \ar@<.5ex>[r]^-{X(\mu_{W^*(S)})}\ar@<-.5ex>[r]_-{h_{X}(W^*(S))} & X(W^*(W^*(S)))}\] are equal.
\end{enumerate}

To this end we show that given an ind-affine scheme $S$ and an element of $\mc{M}_\CM(S)$, i.e.\ a CM elliptic curve $E/S$, there is a functorially associated element of $\mc{M}_\CM(W^*(S))$, i.e.\ a CM elliptic curve $W_\CM^*(E)/W^*(S)$, satisfying the following properties:
\begin{enumerate}[label=\textup{(\roman*)}]
\item the pull-back of $W_\CM^*(E)$ along the ghost component at $(1)$ is canonically isomorphic to $E$: \[E\isomto g_{(1)}^*(W_\CM^*(E))=W_{CM}^*(E)\times_{W^*(S)} S,\]
\item the pull-back of $W_\CM^*(E)$ along $\mu_S: W^*(W^*(S))\to W^*(S)$ is canonically isomorphic to $W_\CM^*(W_\CM^*(E))$: \[\mu_S^*(W_\CM^*(E))=W_{CM}^*(E)\times_{W^*(S)} W^*(W^*(S)) \isomto W^*_{CM}(W^*_{CM}(E)).\]
\end{enumerate}

We call $W_\CM^*(E)/W^*(S)$ the canonical lift of $E/S$. In addition to (i) and (ii) above, we also show that the CM elliptic curve $W_\CM^*(E)$ admits a canonical $\Lambda_{W^*(S)}$-structure. It is worth pointing out that these `canonical lifts' are both global and big -- the base is $S$ is an arbitrary ind-affine scheme over $\Spec(O_K)$ and the Witt vectors over which we lift the CM elliptic curves have Frobenius lifts at all primes of $O_K$.

We now give a brief overview of each of the sections. The construction of $W^*_\CM(E)/W^*(S)$ and the verification of its properties is the content of \S 1. We also use this to define what it means for a CM elliptic curve over $\Lambda$-ind-affine scheme $E/S$ to have a canonical $\Lambda$-structure (what we really do is define what it means for a morphism \[S\stackrel{E}{\to} \mc{M}_\CM\] corresponding to a CM elliptic curve $E/S$ to be a $\Lambda$-morphism).

In \S 2 we consider a certain special class of CM elliptic curves (those of `Shimura type') and show that they are exactly those admitting $\Lambda$-structures. These curves were first defined and studied by Shimura (\cite{Shimura94}), and subsequently by several other authors with particular reference to their L-functions (see \cite{CoatesWiles1977}, \cite{Deninger1989} and \cite{Rubin1981}). We finish \S 2 by showing that many CM elliptic curves of Shimura type admit global minimal models. This gives a broad generalisation of a result of Gross \cite{Gross82}. The proof we give is quite different to that in \cite{Gross82} and relies ons a certain strengthening (\ref{prop:tannaka-result-general}) of an old principal ideal theorem (valid for arbitrary number fields).

In \S 3 we show how to (intelligently) construct the quotient of the universal CM elliptic curve by its group of automorphisms and we show that this curve descends to a smooth projective curve $X\to M_\CM$ over the coarse sheaf $M_\CM$. We also show that this descended curved admits a $\Lambda_{M_\CM}$-structure and that this $\Lambda_{M_\CM}$-structure can be used to construct the maximal abelian extension of $K$ in a natural way. This gives a canonical, integral and $\Lambda$-theoretic version of the explicit generation of the ray class fields of $K$ using Weber functions of CM elliptic curves over fields. We end this section by showing that the possibly mysterious curve $X$ is non-other than $\mathbf{P}^1_{M_\CM}$.

In \S 4 we use the results of \S 1 and \S 2 to construct a flat affine formally smooth presentation $M_\CM^W\to \mc{M}_\CM$ of $\mc{M}_\CM$. The flat affine formally smooth scheme $M_\CM^W$ has a natural moduli theoretic interpretation, and is also a torsor under a certain affine flat affine group scheme $\CL_{O_K}^W$. Finally, we show that $M_\CM^W$ admits a natural $\Lambda$-structure compatible with that on $\mc{M}_\CM$.

In \S 5 we exhibit a rather interesting relationship between a (variant of) the canonical lift of an CM elliptic curve (over an arbitrary base) and its Tate module. This gives an analogue for CM elliptic curves of a certain construction in $p$-adic Hodge theory involving Lubin--Tate $O$-modules and we end by sketching a certain analytic analogue.

\section{Canonical lifts of CM elliptic curves} We continue with the set-up of Chapter 2, so that we work over the base scheme $\Spec(O_K)$ where $O_K$ is the ring of integers of an imaginary quadratic field. In order to apply the theory of Chapter 3, we will no longer be working with arbitrary sheaves $S\in \Sh_{O_K}$ but only with ind-affine schemes $S\in \IndAff_{O_K}\subset \Sh_{O_K}$.

By a $\Lambda$-structure is meant one relative to the Dedekind domain $O_K$ and to the full set of ideals $P=\Id_{O_K}$.

We note that if $S$ is an ind-affine scheme then $W^*(S)$ is again an ind-affine scheme and therefore a sheaf for the fpqc topology. Moreover, $W_*(S)=\lim_{\a\in \Id_{O_K}}W_{\a*}(S)$ is an inverse limit of ind-affine schemes (as $W_{\a*}$ commutes with filtered colimits and sends affine schemes to affine schemes) and is also a sheaf for the fpqc topology.

For technical reasons we also work with the affine \'etale topology on $\IndAff_{O_K}$ whose covers are given by families of affine \'etale morphisms $(S_i\to S)_{i\in I}$ which are covers when viewed in $\Sh_{O_K}$ (or equivalently $\Sh_{O_K}^\et$).

We shall also continue to work with the fibred category $\mc{M}_{\CM}$ but from here on view it as a fibred category over the category of ind-affine schemes $\IndAff_{O_K}\subset \Sh_{O_K}$, rather than all of $\Sh_{O_K}$. 

Unless otherwise noted $S$ will denote an arbitrary ind-affine scheme.

\subsection{} Denote by $\mc{L}_\CM$ the rank one $O_K$-local system over $\Gamma^*(\Spec(O_K))$ whose fibre over the ghost component at $\a\in \Id_{O_K}$ is the constant sheaf $\underline{\a}^{-1}$, i.e.\ \[\mc{L}_\CM=\coprod_{\a\in \Id_{O}} \underline{\a}^{-1}\to \coprod_{\a\in \Id_{O}}\Spec(O_K)=\Gamma^*(\Spec(O_K)).\] For an ind-affine scheme $S$ we shall abuse notation and write $\mc{L}_\CM$ for the rank one $O_K$-local system over $\Gamma^*(S)$ obtained by pulling back $\mc{L}_{\CM}$ along $\Gamma^*(S)\to \Gamma^*(\Spec(O_K))$.

\subsection{}\label{subsec:def-can-ghost} Let $E\to S$ a CM elliptic curve . Writing $p_S: \Gamma^*(S)\to S$ for the map \[\coprod_{\a\in \Id_{O_K}}\id_S: \Gamma^*(S)=\coprod_{\a\in \Id_{O_K}}S\to S,\] we define a new CM elliptic curve over $\Gamma^*(S)$ by \[\Gamma^*_\CM(E)=p_S^*(E)\otimes_{O_K}\mc{L}_\CM.\] In other words, we have \[\Gamma^*_\CM(E)=\coprod_{\a\in \Id_{O_K}}E\otimes_{O_K}\a^{-1}\to \coprod_{\a\in \Id_{O_K}} S=\Gamma^*(S).\] If $f: E\to E'$ is a homomorphism of CM elliptic curves over $S$ then we write \[\Gamma^*_\CM(f)=p^*_S(f)\otimes_{O_K}\mc{L}_\CM:\Gamma^*_\CM(E)\to \Gamma^*_\CM(E')\] so that $\Gamma^*_{\CM}$ defines a functor from the category of CM elliptic curves over $S$ to the category of CM elliptic curves over $\Gamma^*(S)$.

By construction, the rank one $O_K$-local system $\mc{L}_\CM$ over $\Gamma^*(S)$ satisfies $\mc{L}_\CM\otimes_{O_K}\a^{-1}=\psi^{\a*}(\mc{L}_\CM)$ for each ideal $\a$ and this induces isomorphisms \[\Gamma^*_\CM(E)\otimes_{O_K}\a^{-1}\isomto \psi^{\a*}(\Gamma^*_\CM(E)).\] We equip $\Gamma^*_{\CM}(E)$ with the $\Psi_{\Gamma^*(S)}$-structure with the relative endomorphisms given for each ideal $\a$ by the composition \[\Gamma_{\CM}^*(E)\stackrel{i_\a}{\to}\Gamma_{\CM}^*(E)\otimes_{O_K}\a^{-1}\isomto \psi^{\a*}(\Gamma_\CM^*(E)).\] In order to avoid overly cumbersome notation, we will denote this map by $\varphi^\a_{E/S}$ (instead of by the usual by monstrous $\psi^{\a*}_{\Gamma^*_\CM(E)/\Gamma^*(S)}$). Note that $\ker(\varphi^\a_{E/S})=\Gamma^*_\CM(E)[\a]$.

The sheaf of rings $\underline{O_K}_{\Gamma^*(S)}=\Gamma^*(\underline{O_K}_S)$ is naturally a sheaf of $\Psi_{\Gamma^*(S)}$-rings and the $\underline{O_K}_{\Gamma^*(S)}$-module structure \[\underline{O_K}_{\Gamma^*(S)}\times \Gamma_\CM^*(E)\to \Gamma_\CM^*(E)\] is compatible with the $\Psi_{\Gamma^*(S)}$-structures.

If $S'\to S$ is a morphism there is a obvious isomorphism of CM elliptic curves over $\Gamma^*(S')$ equipped with $\Psi_{\Gamma^*(S')}$-structures \[\Gamma_\CM^*(E)\times_{\Gamma^*(S)}\Gamma^*(S')\isomto \Gamma_\CM^*(E\times_S S').\]

\subsection{} Write $\Gamma_*(\mc{M}_\CM)_{\Psi}$ for the fibred category over $\IndAff_{O_K}$ whose fibre over $S$ is the essential image of the functor $E/S\mto \Gamma^*_\CM(E)/\Gamma^*(S)$ from CM elliptic curves over $S$ to CM elliptic curves over $\Gamma^*(S)$ equipped with $\Psi_{\Gamma^*(S)}$-structures compatible with their $\underline{O_K}_{\Gamma^*(S)}$-module structure. The pull-back maps for $S'\to S$ are given by \[E/\Gamma^*(S)\mto E\times_{\Gamma^*(S)}\Gamma^*(S')/\Gamma^*(S').\]

\begin{rema} The symbol $\Gamma_*(\mc{M}_\CM)_{\Psi}$ is only notation but it supposed to inspire the following interpretation. Assuming $\mc{M}_\CM$ did admit a $\Lambda$-structure, and so a fortiori a $\Psi$-structure, then we should have equivalences \[\mc{M}_\CM(S)=\Hom(S, \mc{M}_\CM)\isomto \Hom_{\Spec(O_K)}^{\Psi}(\Gamma^*(S), \mc{M}_\CM)=\Gamma_*(\mc{M}_\CM)_\Psi(S).\] Of course, we do not define a $\Psi$-structure on $\mc{M}_\CM$. Instead we opt to define which morphisms (equivalently CM elliptic curves $E/\Gamma^*(S)$) \[\Gamma^*(S)\stackrel{E}{\to} \mc{M}_\CM\] should be though of as (coming from) $\Psi$-morphisms. The following gives some justification to this.
\end{rema}

\begin{lemm}\label{lemm:canonical-ghost-equivalence} The functor \[\mc{M}_\CM\to \Gamma_*(\mc{M}_\CM)_{\Psi}: E/S\mto \Gamma_\CM^*(E)/\Gamma^*(S)\] is an equivalence of stacks with quasi-inverse given by pull-back along the ghost component at $(1)$ \[E/\Gamma^*(S)\mto g_{(1)}^*(E)/S.\]
\end{lemm}
\begin{proof} The functor in question is clearly essentially surjective and faithful as composing it with \[E/\Gamma^*(S)\mto g_{(1)}^*(E)/S\] yields a functor isomorphic to the identity on $\mc{M}_\CM$. Now to see that it is an equivalence, and that $E/\Gamma^*(S)\mto g_{(1)}^*(E)/S$ is a quasi-inverse, we need only show that it is full.

So let $f: \Gamma^*_\CM(E)\to \Gamma^*_\CM(E')$ be a $\Psi_{\Gamma^*(S)}$-isomorphism and write $f$ as the sum of its ghost components \[f=\coprod_{\a\in \Id_{O_K}}g_\a^*(f).\] As $f$ is a $\Psi_{\Gamma^*(S)}$-homomorphism the diagram \[\xymatrix{E\ar[r]^{g_{(1)}^*(f)}\ar[d]_{i_\a} & E'\ar[d]^{i_\a}\\
E\otimes_{O_K}\a^{-1}\ar[r]^{g_{\a}^*(f)} & E'\otimes_{O_K}\a^{-1}}\] commutes for each $\a\in \Id_{O_K}$ by the definition of $\Gamma_\CM^*(E)$ and $\Gamma_\CM^*(E')$ and their $\Psi_{\Gamma^*(S)}$-structures. However, as $i_\a$ is epimorphism the only map $g_{\a}^*(f)$ for which this is possible is $g_{(1)}^*(f)\otimes_{O_K}\a^{-1}$. It follows that \[f=\coprod_{\a\in \Id_{O_K}} g_{(1)}^*(f)\otimes_{O_K}\a^{-1}=\Gamma^*_\CM(g_{(1)}^*(f))\] and we are done.
\end{proof}

\subsection{}\label{subsec:canonical-lambda-structures-witt-vectors} Let $E/W^*(S)$ be a CM elliptic curve equipped with a $\Lambda_{W^*(S)}$-structure compatible with its $\underline{O_K}_{W^*(S)}=W^*(\underline{O_K}_S)$-module structure. We say that the $\Lambda_{W^*(S)}$-structure on $E/W^*(S)$ is canonical if there exists a $\Psi_{\Gamma^*(S)}$-isomorphism \[E\times_{W^*(S)}\Gamma^*(S)\isomto \Gamma_{\CM}^*(g_{(1)}^*(E))\] inducing the identity on $g_{(1)}^*(E)$ after pull-back along $g_{(1)}: S\to \Gamma^*(S)$. Such an isomorphism is unique by (\ref{lemm:canonical-ghost-equivalence}) and so, when it exists, we shall denote it by $\rho_{E/W^*(S)}$.
\begin{lemm}\label{lemm:properties-of-lambda-cm-elliptic-curves-over-W} Let $E/W^*(S)$ be a \textup{CM} elliptic curve.
\begin{enumerate}[label=\textup{(\roman*)}]
\item $E$ admits at most one canonical $\Lambda_{W^*(S)}$-structure.
\item If $S'\to S$ is a morphism of ind-affine schemes then, writing $E'=E\times_{W^*(S)}W^*(S')$, the $\Lambda_{W^*(S')}$-structure on $E'$ is canonical and \[\rho_{E/W^*(S)}\times_{\Gamma^*(S)}\Gamma^*(S') = \rho_{E'/W^*(S')}.\]
\item Let $(S_i\to S)_{i\in I}$ be an affine \'etale cover in $\IndAff_{O_K}$. If $E\times_{W^*(S)}W^*(S_i)$ admits a canonical $\Lambda_{W^*(S_i)}$-structure for each $i\in I$ then $E/W^*(S)$ admits a canonical $\Lambda_{W^*(S)}$-structure.
\end{enumerate}
\end{lemm}
\begin{proof} (i) Let $E/W^*(S)$ have a pair of canonical $\Lambda_{W^*(S)}$-structures and write $\rho_{E/W^*(S), 1}$ and $\rho_{E/W^*(S), 2}$ for the corresponding unique isomorphisms \[E\times_{W^*(S)}\Gamma^*(S)\isomto \Gamma_\CM^*(g_{(1)}^*(E))\] inducing the identity on $g_{(1)}^*(E)$ after pull-back along $g_{(1)}: S\to \Gamma^*(S)$. The composition \[\rho_{E/W^*(S), 1}^{-1}\circ\rho_{E/W^*(S), 2}\] defines an $\Psi_{\Gamma^*(S)}$-automorphism of $\Gamma_\CM^*(g_{(1)}^*(E))$ which is the identity on $g_{(1)}^*(E)$ after pull-back along $g_{(1)}$ and so is the identity itself by (\ref{lemm:canonical-ghost-equivalence}).

It follows that the two $\Psi_{\Gamma^*(S)}$-structures on each $E\times_{W^*(S)}\Gamma^*(S)$, induced by the two $\Lambda_{W^*(S)}$-structures on $E$ and $E'$, are equal. Writing \[\psi_{E/W^*(S), 1}^{\p}, \psi_{E/W^*(S), 2}^{\p}: E\to \psi^{\p*}(E)\] for the relative Frobenius lifts at $\p$ corresponding to the two $\Lambda_{W^*(S)}$-structures, we have shown that \[\psi^{\p}_{E/W^*(S), 1}\times_{W^*(S)}\Gamma^*(S)=\psi^{\p}_{E/W^*(S), 2}\times_{W^*(S)}\Gamma^*(S).\] As $\Gamma^*(S)\to W^*(S)$ is surjective on geometric points (\ref{prop:ghost-map-surj-geom-points}), it follows by rigidity that \[\psi_{E/W^*(S), 1}^{\p}=\psi_{E/W^*(S), 2}^{\p}.\] Therefore, the two $\Psi_{W^*(S)}$-structures on $E/W^*(S)$ induced by the two $\Lambda_{W^*(S)}$-structures are equal and by (\ref{rema:lambda-structures-on-abelian-varieties-are-determined-by-psi-structures}) it follows that the two $\Lambda_{W^*(S)}$-structures themselves are equal.

(ii) The CM elliptic curve $E':=E\times_{W^*(S)} W^*(S')$ has a natural $\Lambda_{W^*(S')}$-structure and the pull-back of $\rho_{E/W^*(S)}$ along $\Gamma^*(S')\to \Gamma^*(S)$ defines a $\Psi_{\Gamma^*(S')}$-isomorphism \[E'\times_{W^*(S')}\Gamma^*(S')=E\times_{W^*(S)} \Gamma^*(S') \stackrel{\rho_{E/W^*(S)\times_{\Gamma^*(S)}\Gamma^*(S')}}{\longrightarrow} \Gamma_\CM^*(g_{(1)}^*(E))\times_{\Gamma^*(S)}\Gamma^*(S')=\Gamma_\CM^*(g_{(1)}^*(E'))\] inducing the identity after pull-back along $g_{(1)}$. It follows by uniqueness of such an isomorphism that \[\rho_{E/W^*(S)}\times_{\Gamma^*(S)}\Gamma^*(S') = \rho_{E'/W^*(S')}.\]

(iii) We will first show that $E/W^*(S)$ admits a unique $\Lambda_{W^*(S)}$-structure inducing the canonical $\Lambda_{W^*(S_i)}$-structures on $E\times_{W^*(S)}W^*(S_i)$ and then show that this $\Lambda_{W^*(S)}$-structure is canonical.
The family $(W^*(S_i)\to W^*(S))_{i\in I}$ is a cover in $\Sh_{\Lambda}^\et$ and we have \[W^*(S_{ij}):=W^*(S_i\times_S S_j)\isomto W^*(S_i)\times_{W^*(S)} W^*(S_j)\] by (\ref{prop:w-and-etale-base-change}). By (ii) above, the two $\Lambda_{W^*(S_{ij})}$-structures on \[E\times_{W^*(S)}W^*(S_{ij})\] induced by pull-back are canonical and by (i) they are equal. By (\ref{subsec:sheaves-of-lambda-structures}), this defines an element of the equaliser \[\xymatrix{\Lambda_{E/W^*(S)}(W^*(S))\ar[r] & \Lambda_{E/W^*(S)}(W^*(S_i))\ar@<.5ex>[r]\ar@<-.5ex>[r] &\prod\limits_{i, j\in I}\Lambda_{E/W^*(S)}(W^*(S_{ij}))}\] or in other words, there is a unique $\Lambda_{W^*(S)}$-structure on $E/W^*(S)$ inducing the canonical $\Lambda_{W^*(S_i)}$-structures on each $E\times_{W^*(S)}W^*(S_i)$.
The uniqueness of the isomorphisms $\rho_{E\times_{W^*(S)}W^*(S_i)/W^*(S_i)}$ and their compatibility with pull-backs (this is (ii) above) show that they descend to an $\Psi_{\Gamma^*(S)}$-isomorphism \[\rho_{E/W^*(S)}: E\times_{W^*(S)}\Gamma^*(S)\isomto \Gamma^*_\CM(g_{(1)}^*(E))\] inducing the identity after pull-back along $g_{(1)}$ so that the $\Lambda_{W^*(S)}$-structure on $E/W^*(S)$ is canonical.
\end{proof}
\subsection{} Using (\ref{lemm:properties-of-lambda-cm-elliptic-curves-over-W}) we may define a fibred category $W_*(\mc{M}_{\CM})_\Lambda$ over $\IndAff_{O_K}$ by setting the fibre over $S$ to be the category of CM elliptic curves $E/W^*(S)$ equipped with a canonical $\Lambda_{W^*(S)}$-structure and whose pull-back maps for $S'\to S$ are given by \[E/W^*(S)\mto E\times_{W^*(S)}W^*(S').\]

\begin{rema} As with $\Gamma_*(\mc{M}_\CM)_{\Psi}$ the symbol $W_*(\mc{M}_{\CM})_\Lambda$ is supposed to inspire in the reader the idea that $\mc{M}_\CM$ admits some $\Lambda$-structure and that we have \[\mc{M}_\CM(S)=\Hom_{\Spec(O_K)}(S, \mc{M}_\CM)\isomto\Hom_{\Spec(O_K)}^{\Lambda}(W^*(S), \mc{M}_\CM)=W_*(\mc{M}_\CM)_\Lambda(S).\]
\end{rema}

\begin{lemm} The fibred category $W_*(\mc{M}_{\CM})_\Lambda$ over $\IndAff_{O_K}$ is a stack for the affine \'etale topology.
\end{lemm}
\begin{proof} Let $(S_i\to S)_{i\in I}$ be an affine \'etale cover and let $(E_i/W^*(S_i))_{i\in I}$ be a collection of objects of $W_*(\mc{M}_{CM})_\Lambda$ equipped with descent data relative to the cover $(S_i\to S)_{i\in I}$. By (ii) of (\ref{prop:w-and-etale-base-change}) we have \[W^*(S_{ij}):=W^*(S_i\times_S S_j)\isomto W^*(S_i)\times_{W^*(S)} W^*(S_j)\] an so we may view the objects $(E_i/W^*(S_i))_{i\in I}$ of $W_*(\mc{M}_\CM)_\Lambda$ equipped with their descent data relative to the affine \'etale cover $(S_i\to S)_{i\in I}$ as objects of $\mc{M}_{\CM}$ equipped with descent data relative to the fpqc cover $(W^*(S_i)\to W^*(S))_{i\in I}$. As $\mc{M}_\CM$ is a stack over $\Sh_{O_K}$ (\ref{prop:moduli-ell-stack}), the family $(E_i/W^*(S_i))_{i\in I}$ descends to a CM elliptic curve $E/W^*(S)$ unique upto compatible isomorphisms $E\times_{W^*(S)}W^*(S_i)\isomto E_i$. It remains to see that $E/W^*(S)$ admits a canonical $\Lambda_{W^*(S)}$-structure and this follows from (iii) of (\ref{lemm:properties-of-lambda-cm-elliptic-curves-over-W}).

In much the same way, using (\ref{prop:w-and-etale-base-change}) and (\ref{prop:hom-lambda-sheaves-are-sheaves}), $\Lambda$-isomorphisms of CM elliptic curves with canonical $\Lambda$-structures also satisfy descent for the affine \'etale topology and we find that $W_*(\mc{M}_{\CM})_\Lambda$ is a stack over $\IndAff_{O_K}$.
\end{proof}

\begin{theo} The functor induced by base change along the ghost component at $(1)$ \[W_*(\mc{M}_{\CM})_\Lambda\to \mc{M}_{\CM}: E/W^*(S)\mto g_{(1)}^*(E)/S\] is an equivalence of stacks over $\IndAff_{O_K}$ for the affine \'etale topology.
\end{theo}
\begin{proof} The functor in question factors as \[W_*(\mc{M}_{\CM})_\Lambda\to \Gamma_*(\mc{M}_{CM})_\Psi \to \mc{M}_{\CM}\] where the first functor is $E/W^*(S)\mto E\times_{W^*(S)}\Gamma^*(S)/\Gamma^*(S)$ and the second is pull-back along $g_{(1)}: S\to \Gamma^*(S)$. By (\ref{lemm:canonical-ghost-equivalence}) the second functor is an equivalence and, as $\Gamma^*(S)\to W^*(S)$ is surjective on geometric points, the first functor is faithful by rigidity. Therefore, \[W_*(\mc{M}_{\CM})_\Lambda\to \mc{M}_{\CM}\] is faithful.

Now fix a pair $E, E'/W^*(S)$ of CM elliptic curves equipped with canonical $\Lambda_{W^*(S)}$-structures and let $f: E\times_{W^*(S)}\Gamma^*(S)\to E'\times_{W^*(S)}\Gamma^*(S)$ be a $\Psi_{\Gamma^*(S)}$-isomorphism. Let $\p$ be a prime ideal, $\a$ any ideal and $n\geq 0$ an integer. Consider the diagram:
\begin{equation}\label{dia:lambda-frob-psi}
\begin{gathered}
\xymatrix{\overline{g_\a^*(E)}\ar[rr]^{\overline{g_\a^*(f)}}\ar[d]_-{\overline{g_\a^*(\psi_E^{\p^n})}} && \overline{g_\a^*(E')}\ar[d]^-{\overline{(g_\a^*(\psi^{\p^n}_{E'})}}\\
\overline{g_{\a\p^n}^*(E)}\ar[rr]^{\overline{g_{\a\p^n}^*(f)}}\ar[d]_-{\wr} && \overline{g_{\a\p^n}^*(E')}\ar[d]^-{\wr}\\
\Fr_{\overline{S}}^{N\p^n*}(\overline{g_\a^*(E)})\ar[rr]^{\Fr_{\overline{S}}^{N\p^n*}(\overline{g_\a^*(f)})} && \Fr_{\overline{S}}^{N\p^n*}(\overline{g_\a^*(E')}),}
\end{gathered}\end{equation}
where the bar denotes pull-back along $\overline{S}=S\times_{\Spec(O_K)}\Spec(\F_\p)\to S$ and the bottom vertical isomorphisms are the unique such making the vertical compositions equal to the $N\p^n$-power relative Frobenius morphisms of $\overline{g_\a^*(E)}$ and $\overline{g_\a^*(E')}$ respectively (such isomorphisms exist as $\psi_E^{\p^n}: E\to E$ and $\psi_{E'}^{\p^n}: E'\to E'$ lift the $N\p^n$-power Frobenius). As $f$ is a $\Psi_{\Gamma^*(S)}$-morphism, the top square of (\ref{dia:lambda-frob-psi}) commutes and, by functoriality of the $N\p^n$-power relative Frobenius, the outer square commutes. As $\overline{g_{\a}^*(\psi_E^{\p^n})}$ is an epimorphism, it follows that the bottom square of (\ref{dia:lambda-frob-psi}) also commutes. We will show that this implies that the functor \[W_*(\mc{M}_{\CM})_\Lambda\to \mc{M}_{\CM}\] is full.

For the functor to be full, it is enough to show that the (injective) map \[\Isom_{W^*(S)}^{O_K, \Lambda}(E, E')\to \Isom_{\Gamma^*(S)}^{O_K, \Psi}(E_{\Gamma^*(S)}, E'_{\Gamma^*(S)})\] is surjective (where the super script $\Lambda$ and $\Psi$ denote $\Lambda$- and $\Psi$-morphisms). Now to show this, it is enough to show that each $\Psi_{\Gamma^*(S)}$-isomorphism \[f: E_{\Gamma^*(S)}\to E'_{\Gamma^*(S)}\] is obtained via pull-back from an isomorphism $E\to E'$ over $W^*(S)$ as by (\ref{prop:lambda-and-psi-homomorphisms-of-abelian-varieties}) it will follow that any such isomorphism $E\to E'$ is also a $\Lambda_{W^*(S)}$-isomorphism.

Applying (\ref{prop:ghost-covers-etale}) and (\ref{rema:ghost-fibre-product}) to the finite \'etale $W^*(S)$-sheaf \[\underline{\Isom}_{W^*(S)}^{O_K}(E, E'),\] we see that $f$ comes from a morphism $E\to E'$ over $W^*(S)$ if and only if, for each prime $\p$, each ideal $\a$ and each $n\geq 0$, the two pull-backs of $f$ to $\overline{S}:=S\times_{\Spec(O_K)}\Spec(\F_\p)$ along the maps \[\overline{S}\subset S \stackrel{g_{\a\p^n}}{\longrightarrow} \Gamma^*(S) \quad \text{and} \quad \overline{S} \stackrel{\Fr_{\overline{S}}^{N\p^n}}{\longrightarrow} \overline{S}\subset S \stackrel{g_\a}{\to} \Gamma^*(S)\] of (\ref{rema:ghost-fibre-product}) are equal. This is precisely the commutativity of the bottom square of the diagram (\ref{dia:lambda-frob-psi}) and therefore \[W_*(\mc{M}_{\CM})_\Lambda\to \mc{M}_{\CM}\] is full.

Finally, for the functor in question to be an equivalence it is enough (as both categories are stacks and we have already shown that it is fully faithful) to show that for each ind-affine scheme $S$ and each CM elliptic curve $E/S$ there is an affine \'etale cover $(S_i\to S)_{i\in I}$ such that $E\times_{S} S_i$ is in the essential image of \[W_*(\mc{M}_{\CM})_\Lambda\to \mc{M}_{\CM}.\] By (\ref{coro:level-isom-cover}) the family of $S$-sheaves \[(\underline{\Isom}_S^{O_K}(E[\f], \underline{O_K/\f}_S)\to S)_{\f}\] indexed by integral ideals $\f$ which separate units forms an affine \'etale cover of $S$. We may then base change to any element of this cover and assume that $E/S$ admits a level-$\f$ structure for some integral ideal $\f$ which separates units, and then we may assume that $E/S=E^{(\f)}/M_{\CM}^{(\f)}$. This case is (\ref{lemm:universal-cm-have-lambda}) below.
\end{proof}

\begin{lemm}\label{lemm:universal-cm-have-lambda} Let $\f\in \Id_{O_K}$ separate units. The universal \textup{CM} elliptic curve with level $\f$-structure $E^{(\f)}\to M_\CM^{(\f)}$ is in the essential image of the functor \[W_*(\mc{M}_{\CM})_\Lambda\to \mc{M}_\CM.\]
\end{lemm}
\begin{proof} Write $M=M_\CM^{(\f)}$, $E=E^{(\f)}$ and $P=\Id_{O_K}^{(\f)}$ and let $P'\subset \Id_{O_K}$ be the sub-monoid generated by the prime ideals dividing $\f$ so that $P\cap P'=\{O_K\}$ and $\Id_{O_K}=P\cdot P'$.

Then the finite \'etale $\Spec(O_K[\f^{-1}])$-scheme $M$ admits a unique $\Lambda_{P}$-structure by (\ref{prop:lambda-structures-etale-schemes-number-fields}). By the definition of the automorphisms $\sigma_\p=[\p^{-1}]_\f:M\to M$ there exists a unique $\f$-isomorphism \[E\otimes_{O_K}\p^{-1}\isomto \sigma_\p^*(E)\] whose pull-back along $M\times \Spec(\F_\p)\to M$ is the isomorphism $\nu_{\p}$ of (\ref{coro:tensor-p-equals-frobenius-with-level}). It follows that setting $\psi_{E/M}^{\p}: E\to \sigma_\p^*(E)$ to be the composition \[E\stackrel{i_\p}{\to} E\otimes_{O_K}\p^{-1} \isomto \sigma_\p^*(E)\] defines a lift of the relative $N\p$-power Frobenius on $E\to M$. Note that $\ker(\psi^\p_{E/M})=E[\p]$. Let $\l\in P$ be another prime ideal and consider the diagram
\begin{equation}\label{dia:lambda-on-level}
\begin{gathered}\xymatrix{E\ar[r]^{\psi_{E/M}^{\p}}\ar[d]_{\psi_{E/M}^{\l}} & \sigma_\p^*(E)\ar[d]^{\sigma_{\p}^*(\psi_{E/M}^{\l})}\\
\sigma_{\l}^*(E)\ar[r]^{\sigma_\l^*(\psi_{E/M}^{\p})} & \sigma_{\p\l}^*(E).}\end{gathered}\end{equation} The kernel of the compositions along the top and right, and left and bottom of (\ref{dia:lambda-on-level}) are both equal to $E[\p\l]$ so that, as $M$ is connected, these compositions differ by scaling by some $\epsilon\in O_K^\times$. As the $\f$-torsion $E[\f]$ is constant over $M$, and $M\isomto \Spec(O_{K(\f)}[\f^{-1}])$, it admits a unique $\Lambda_{P, M}$-structure by (\ref{prop:lambda-structures-etale-schemes-number-fields}). The uniqueness of the $\Lambda_{P, M}$-structure on $E[\f]$ now implies that the diagram (\ref{dia:lambda-on-level}) commutes when restricted to the $\f$-torsion and therefore $\epsilon\in O_{K}^{\times, \f}$. However, $\f$ separates units, so that $\epsilon\in O_{K}^{\times, \f}=\{1\}$ and (\ref{dia:lambda-on-level}) commutes. As $E$ is flat over $\Spec(O_K)$ this defines a $\Lambda_{P, M}$-structure on $E$ by (\ref{prop:lambda-structures-frobenius-lifts-sheaves}).

Pulling-back $E\to M$ along the map $\Gamma^*_{P}(M)\to M$ corresponding to the $\Psi_P$-structure on $M$, we obtain an isomorphism of CM elliptic curves \[E\times_M \Gamma^*_{P}(M)=\coprod_{\a\in P} \sigma_\a^*(E)\isomto \coprod_{\a\in P} E\otimes_{O_K}\a^{-1} \isomto \Gamma^*_\CM(E)\times_{\Gamma^*(M)}\Gamma^*_{P}(M)\] compatible with the $\Psi_{P, \Gamma^*_{P}(M)}$-structures. We then get a $\Psi_{\Gamma^*(M)}$-isomorphism \[\Gamma^*_\CM(E)\isomto \coprod_{\b\in P'} (E\times_{M}\Gamma^*_{P}(M))\otimes_{O_K}\b^{-1}\] inducing the identity after pull-back along $g_{(1)}$ and where on the right hand side the relative Frobenius lifts for $\p\in P'$ are defined in the obvious way.

Now set $\widetilde{E}_P:=E\times_{M}W_{P, M}^*(M)$ where we have base changed along the map $\mu_{M}: W_{P}^*(M)\to M$ defining the $\Lambda_P$-structure on $M$. Then $\widetilde{E}_P$ has a $\Lambda_{P, W^*_P(S)}$-structure by virtue of the facts that $E/M$ has a $\Lambda_{P, M}$-structure and $W^*_P(M)\to M$ is a $\Lambda_P$-morphism. We also have $\Psi_{\Gamma^*_P(S)}$-isomorphisms \begin{equation}\label{eqn:level-is-canonical}\widetilde{E}_P\times_{W_{P}^*(S)}\Gamma_{P}^*(S)\isomto E\times_M \Gamma_P^*(S)\isomto \Gamma^*_\CM(E)\times_{\Gamma^*(M)}\Gamma_P^*(M)\end{equation} inducing the identity after pull-back along the first ghost component.

Finally, setting \[\widetilde{E}:=\coprod_{\b\in P'}\widetilde{E}_P\otimes_{O_K}\b^{-1}\to \coprod_{\b\in P'}W^*_{P}(M)=W^*(M)\] and defining relative Frobenius lifts on $\widetilde{E}$ for the primes $\p\in P'$ in the obvious way equips $\widetilde{E}$ with a $\Lambda_{W^*(S)}$-structure. Again by construction we have a $\Psi_{\Gamma^*(M)}$-isomorphism of CM elliptic curves \[\widetilde{E}\times_{W^*(M)}\Gamma^*(M)\isomto \Gamma^*_\CM(E)\] inducing the identity after pull-back along the ghost component at $(1)$. Therefore the $\Lambda_{W^*(M)}$-structure so defined on $\widetilde{E}/W^*(M)$ is canonical and $E\to M$ is in the essential image of the functor \[W_*(\mc{M}_{\CM})_\Lambda\to \mc{M}_\CM.\]
\end{proof}

\subsection{} We now fix for all time an inverse equivalence \[\mc{M}_\CM\to W_*(\mc{M}_{\CM})_{\Lambda}: E/S\mto W_\CM^*(E)/W^*(S)\] to the functor $W_*(\mc{M}_{\CM,\Lambda})\to \mc{M}_\CM$ and call $W_\CM^*(E)/W^*(S)$, equipped with its canonical $\Lambda_{W^*(S)}$-structure, the canonical lift of $E/S$.

\begin{rema}\label{rema:level-structures-canonical} We note for future reference that the proof of (\ref{lemm:universal-cm-have-lambda}) shows that the CM elliptic curve $E^{(\f)}/M_\CM^{(\f)}$ admits a $\Lambda_{P, M_\CM^{(\f)}}$-structure and that this $\Lambda_{P, M_\CM^{(\f)}}$-structure has the property that there exists a $\Lambda_{P, W^*_P(M_\CM^{(\f)})}$-isomorphism \[W^*_{\CM}(E^{(\f)})\times_{W^*(M_\CM^{(\f)})}W^*_P(M_\CM^{(\f)})\isomto E^{(\f)}\times_{M_\CM^{(\f)}} W_P^*(M_\CM^{(\f)})=\mu_{M_\CM^{(\f)}}^*(E)\] inducing the identity on $E$ after pull-back to the ghost component at $(1)$ (this is (\ref{eqn:level-is-canonical})).
\end{rema}

\subsection{} By virtue of the definition of $W_\CM^*(E)/W^*(S)$ we have a canonical isomorphism \[E\isomto g_{(1)}^*(W_\CM^*(E)).\] That is, the composition \[\mc{M}_\CM(S)\stackrel{W^*_\CM}{\to}\mc{M}_\CM(W^*(S))\stackrel{g^*_{(1)}}{\to}\mc{M}_\CM(S)\] is canonically isomorphic to the identity. We will now show that the two compositions \[\xymatrix{\mc{M}_\CM(S)\ar[r]^-{W_\CM^*} & \mc{M}_\CM(W^*(S))\ar@<.5ex>[r]^-{W_\CM^*}\ar@<-.5ex>[r]_-{\mu_{W^*(S)}^*} & \mc{M}_\CM(W^*(W^*(S)))}\] are also canonically isomorphic.

\begin{prop}\label{prop:canonical-lifts-have-canonical-lambda-structures} Let $S$ be an ind-affine scheme and let $E/S$ be a \textup{CM} elliptic curve. Then there is a unique isomorphism, compatible with the $\Lambda_{W^*(W^*(S))}$-structures, \[\mu_{W^*(S)}^*(W^*_\CM(E))\isomto W^*_{CM}(W^*_\CM(E))\] inducing the identity after pull-back along $g_{(1)}: W^*(S)\to W^*(W^*(S))$.
\end{prop}
\begin{proof} Both $\mu^*{W^*(S)}(W^*_\CM(E))$ and $W^*_\CM(W^*_\CM(E))$ are equipped with natural $\Lambda_{W^*(W^*(S))}$-structures and the pull-backs of each along the ghost component at $(1)$ are equal to $W^*_\CM(E)$. Therefore, it is enough by (\ref{lemm:universal-cm-have-lambda}) to show that the $\Lambda_{W^*(W^*(S))}$-structures on $\mu_{W^*(S)}^*(W^*_\CM(E))$ are $W^*_{CM}(W^*_\CM(E))$ are canonical.

For $W^*_\CM(W^*_\CM(E))$ this is true by definition. For $\mu_{W^*(S)}^*(W^*_\CM(E))$ we have the sequence of $\Psi_{\Gamma^*(W^*(S))}$-isomorphisms \begin{eqnarray*} \mu_{W^*(S)}^*(W_\CM^*(E))\times_{W^*(W^*(S))}\Gamma^*(W^*(S)) & = & (W^*_{\CM}(E)\times_{W^*(S)}W^*(W^*(S)))\times_{W^*(W^*(S))}\Gamma^*(W^*(S))\\
& = & W^*_{\CM}(E)\times_{W^*(S)}\Gamma^*(W^*(S))\\
&\isomto& \coprod_{\a}\psi^{\a*}(W_\CM^*(E))\\&\isomto& \coprod_{\a} W_\CM^*(E)\otimes_{O_K}\a^{-1}\\&\isomto & \Gamma_\CM^*(W_\CM^*(E))\\
&=&\Gamma_\CM^*(g_{(1)}^*(\mu_{W^*(S)}^*(W_\CM^*(E)))).
\end{eqnarray*} The resulting $\Psi_{\Gamma^*(W^*(S))}$-isomorphism \[\mu_{W^*(S)}^*(W_\CM^*(E))\times_{W^*(W^*(S))}\Gamma^*(W^*(S))\isomto \Gamma^*_\CM(g_{(1)}^*(\mu_S^*(W_\CM^*(E))))\] induces the identity after pull-back along the ghost component at $(1)$ and this is precisely the definition of a canonical $\Lambda_{W^*(W^*(S))}$-structure.
\end{proof}

\subsection{}\label{subsec:p-canonical-lambda-structure} We now define what it means for a CM elliptic curve over an arbitrary $\Lambda$-ind-affine scheme (not just those of the form $W^*(S)$) to have a canonical $\Lambda$-structure.

It will be convenient later to allow ourselves the flexibility of working with $\Lambda$-structures relative to a sub-monoid $P\subset \Id_{O_K}$ generated by some set of prime ideals and so we fix such a $P$.

Let $S$ be an $\Lambda_P$-ind-affine scheme and let $E/S$ be a CM elliptic curve equipped with a $\Lambda_{P, S}$-structure (again compatible with its $\underline{O_K}_S$-module structure). We say that the $\Lambda_P$-structure on $E/S$ is canonical if there is a $\Lambda_{P, W^*_P(S)}$-isomorphism \[\lambda_{E/S}: E\times_{S}W^*_P(S) \isomto W_\CM^*(E)\times_{W^*(S)}W_P^*(S)\] inducing the identity on the ghost components at $(1)$. There is at most one such isomorphism $\lambda_{E/S}$ satisfying this condition and so we are safe to label it. Indeed, any two differ by a $\Lambda_{P, W_P^*(S)}$-automorphism of $W_\CM^*(E)\times_{W^*(E)}W_P^*(E)$ inducing the identity on the ghost component at $(1)$ and all such automorphisms are the identity as this can be checked after pull-back along $\Gamma_P^*(S)\to W_P^*(S)$, and arguing as in the proof of (\ref{lemm:canonical-ghost-equivalence}) one sees that a $\Psi_{P, \Gamma_P^*(S)}$-automorphism of $\Gamma_\CM^*(E)\times_{\Gamma^*(S)}\Gamma_P^*(S)$ is equal to the identity if and only if i
t is after pull-back to the ghost component at (1).

\begin{rema} When $P=\Id_{O_K}$ and $E/W^*(S)$ is a CM elliptic curve equipped with a $\Lambda_{W^*(S)}$-structure then it follows from (\ref{prop:canonical-lifts-have-canonical-lambda-structures}) that this $\Lambda_{W^*(S)}$-structure is canonical in the sense of (\ref{subsec:p-canonical-lambda-structure}) if and only if there exists a $\Lambda_{W^*(S)}$-isomorphism \[f: E\isomto W_\CM^*(g_{(1)}^*(E))\] inducing the identity after pull-back to the ghost component at $(1)$, i.e.\ $E/W^*(S)$ is canonical in the sense of (\ref{subsec:canonical-lambda-structures-witt-vectors}).

Indeed, if there exists such an isomorphism $f: E\isomto W_\CM^*(g_{(1)}^*(E))$ then (\ref{prop:canonical-lifts-have-canonical-lambda-structures}) gives an isomorphism $\lambda_{E/W^*(S)}$ via the compositions \begin{eqnarray*} E\times_{W^*(S)}W^*(W^*(S))&\stackrel{f\times_{W^*(S)}W^*(W^*(S))}{\longrightarrow} & W_\CM^*(g_{(1)}^*(E))\times_{W^*(S)}W^*(W^*(S))\\
&\stackrel{\textup(\ref{prop:canonical-lifts-have-canonical-lambda-structures})}{\longrightarrow} & W_\CM^*(W_\CM^*(g_{(1)}^*(E)))\\
&\stackrel{W_\CM^*(f^{-1})}{\longrightarrow} & W_\CM^*(E).\end{eqnarray*} Conversely, the pull-back of an isomorphism \[\lambda_{E/W^*(S)}: E\times_{W^*(S)}W^*(W^*(S))\isomto W_\CM^*(E)\] along $\widetilde{g}=W^*(g_{(1)}): W^*(S)\to W^*(W^*(S))$ yields a $\Lambda_{W^*(W^*(S))}$-isomorphism $f: E\isomto W_\CM^*(g_{(1)}^*(E))$ via \[E\isomto \widetilde{g}^*(E\times_{W^*(S)}W^*(W^*(S)))\stackrel{\widetilde{g}^*(\lambda_{E/W^*(S)})}{\to}\widetilde{g}^*(W_\CM^*(E)) \isomto W_\CM^*(g_{(1)}^*(E)).\] Hence the two notions of canonical $\Lambda_{W^*(S)}$-structure on a CM elliptic curve $E/W^*(S)$ defined in (\ref{subsec:canonical-lambda-structures-witt-vectors}) and (\ref{subsec:p-canonical-lambda-structure}) coincide.
\end{rema}

\begin{prop}\label{prop:canonical-lambda-p-structures} Let $S$ be an $\Lambda_{P}$-ind-affine-scheme and $E/S$ a \textup{CM} elliptic curve. Then:
\begin{enumerate}[label=\textup{(\roman*)}]
\item $E/S$ admits at most one canonical $\Lambda_{P, S}$-structure.
\item If $S'\to S$ is a morphism of $\Lambda_P$-ind-affine schemes, the $\Lambda_{P, S'}$-structure on $E\times_S S'$ is canonical and \[\lambda_{E/S}\times_{W^*_P(S)}W^*_P(S') = \lambda_{E\times_S S'/S'}.\]
\item Let $(S_i\to S)_{i\in I}$ be a $\et$-cover of $\Lambda_{P}$-ind-affine schemes. If $E\times_S S'$ admits a canonical $\Lambda_{P, S_i}$-structure for each $i\in I$ then $E/S$ admits a canonical $\Lambda_{P, S}$-structure.
\end{enumerate}
\end{prop}
\begin{proof} (i) Let $E/S$ admit two canonical $\Lambda_{P, S}$-structures with corresponding isomorphisms \[\lambda_{E/S}, \lambda_{E/S}': E\times_{S}W_P^*(S)\isomto W_\CM^*(E)\times_{W^*(S)} W_P^*(S).\] The difference \[\lambda_{E/S}'\circ\lambda_{E/S}^{-1}: W_\CM^*(E)\times_{W^*(S)} W_P^*(S)\isomto W_\CM^*(E)\times_{W^*(S)} W_P^*(S)\] defines $\Lambda_{P, W^*_P(S)}$-automorphism of $W_\CM^*(E)\times_{W^*(S)} W_P^*(S)$ which is the identity on the first ghost component. Base changing along $\Gamma_P^*(S)\to W_P^*(S)$ and arguing again as in the proof of (\ref{lemm:canonical-ghost-equivalence}) we find that such an automorphism must be the identity itself and so $\lambda_{E/S}=\lambda_{E/S}'$.

It follows now that the two $\Lambda_{P, W^*_P(S)}$-structures on $E\times_{S}W_P^*(S)$ coincide so that as $W_P^*(S)\to S$ is an epimorphism and $\Lambda_{P}$-structures descend
 (\ref{prop:lambda-structures-are-sheaves}) the two $\Lambda_{P, S}$-structures on $E/S$ coincide.

(ii) This follows from the uniqueness of the isomorphisms $\lambda_{E/S}$.

(iii) As canonical $\Lambda_{P}$-structures are unique and compatible with pull-back (this is (ii) above) if $E\times_S S_i$ admits a canonical $\Lambda_{P, S_i}$-structure for each $i\in I$ it follows that $E$ admits a $\Lambda_{P, S}$-structure. As the isomorphisms making a $\Lambda_{P}$-structure canonical are unique and compatible with pull back (again this is (ii)) the isomorphisms $\lambda_{E\times_S S_i/S_i}$ descend to an isomorphism \[\lambda_{E/S}: E\times_S W_P^*(S)\isomto W_\CM^*(E)\times_{W^*(S)}W_P^*(S)\] inducing the identity after base change long the ghost component at $(1)$ and so the $\Lambda_{P, S}$-structure on $E/S$ canonical.
\end{proof}

\begin{prop}\label{prop:canonical-lambda-structures-and-level-structures} Let $\f\in \Id_{O_K}$ be an ideal which separates units. Then the unique $\Lambda_{P, M_{\CM}^{(\f)}}$-structure on the universal \textup{CM} elliptic curve with level-$\f$ structure $E^{(\f)}\to M^{(\f)}_\CM$ is canonical.
\end{prop}
\begin{proof} This is the content of (\ref{rema:level-structures-canonical}).
\end{proof}
\section{CM elliptic curves of Shimura type}\label{sec:shimura-curves}

The purpose of this section is to explain the relationship between $\Lambda$-structures and CM elliptic curves of Shimura type. A CM elliptic curve of Shimura type is a CM elliptic curve $E/\Spec(L)$ where $L$ is an abelian extension of $K$ with the property that the extension $L(E[\tors])/K$ is an abelian extension of $K$ (it is always an abelian extension of $L$).

This class of CM elliptic curves was introduced by Shimura (see Theorem 7.44 of \cite{Shimura94}). By virtue of their definition, CM elliptic curves of Shimura type have much simpler arithmetic than arbitrary CM elliptic curves (and of course elliptic curves in general). In the special case where $K$ has class number one, every CM elliptic curve $E/\Spec(K)$ is of Shimura type, and the first (partial) verifications of the Birch-Swinnerton-Dyer conjecture made by Coates and Wiles (\cite{CoatesWiles1977}) concerned these curves. Along these lines let us also mention the paper of Rubin (\cite{Rubin1981}) which considers the Birch-Swinnerton-Dyer conjecture for general CM elliptic curves of Shimura type.

Let us now give a brief outline of what follows. We first recall a result of Shimura (\ref{prop:shimura-lambda-curves-over-hilbert}) stating the existence of infinitely many CM elliptic curves of Shimura type over the Hilbert class field $H$ of $K$. We then show that such CM elliptic curves cannot have good reduction everywhere (\ref{prop:no-abelian-torsion-good-reduction-hilbert}) and answer the question raised in (\ref{rema:rohr-q}) regarding the triviality of the $\mc{CL}_{O_K}$-torsor $\mc{M}_\CM$.

We then prove the main result (\ref{theo:shimura-is-lambda}) which is that a CM elliptic curve $E/\Spec(L)$ is of Shimura type if and only if it admits a canonical $\Lambda$-structure. There are some minor technicalities in that one must avoid the ramified primes in $L/K$ and the primes of bad reduction for $E$ but we get around this naturally enough.

The second part of this section concerns itself with the tangent spaces of N\'eron models of CM elliptic curves of Shimura type. The $\Lambda$-structure on $E/\Spec(L)$ induces a rather rigid structure on Lie algebra of its N\'eron model. We then show (\ref{coro:shimura-curves-minimal-model}) that if $E/\Spec(K(\f))$ is a CM elliptic curve of Shimura type over the ray class field of conductor $\f$, and the $\f$-torsion of $E$ is constant, then the Lie algebra of its N\'eron model is free away from $\f$, or in other words that $E/\Spec(K(\f))$ admits global minimal model away from $\f$. In particular, if one takes $\f=O_K$ so that $K(\f)=H$, this shows that every CM elliptic curve of Shimura type $E/\Spec(H)$ admits a global minimal model \textit{everywhere}. This extends a result of Gross (\cite{Gross82}) and will be used crucially in (\ref{sec:lambda-cover-of-m}) in our construction of a flat affine $\Lambda$-presentation of the stack $\mc{M}_\CM$.

\subsection{}\label{subsec:cm-over-a-field-shimura} Let $L$ be an abelian extension of $K$ and $E/\Spec(L)$ a CM elliptic curve. One says that $E/\Spec(L)$ is of Shimura type if the extension $L(E[\tors])/K$ is abelian.

\begin{prop}[Shimura]\label{prop:shimura-lambda-curves-over-hilbert} There exist infinitely many prime ideals $\p$ of $K$ for which there is a \textup{CM} elliptic curve $E/\Spec(H)$ of Shimura type with good reduction away from $\p$.
\end{prop}
\begin{proof} By Proposition 7, \S 5 of \cite{Shimura71} there exists infinitely many primes $\p$ of $K$ with $N\p=p$ a rational prime, $N\p=1\bmod w$ and $(N\p-1)/w$ prime to $w$, where $w=\# O_K^\times$. Given such a prime $\p$ it follows that the reduction map \[O_K^\times\to (O_K/\p)^\times\] is the inclusion of a direct factor. Therefore, we may define a map $\alpha: A_{O_K}^\times\to O_K^\times$ satisfying $\alpha|_{O_K^\times}=\id_{O_K^\times}$ by \[A_{O_K}^\times\to (O_K/\p)^\times\to O_K^\times\] where the first map is the quotient map and the second is a retraction of $O_K^\times\to (O_K/\p)^\times$. Finally, setting \[g: A_{O_K}^\times/O_K^\times\to A_{O_K}^\times: s\bmod O_K^\times\mto s \alpha(s)^{-1}\] we define $\rho^{-1}: G(H^\sep/H)\to A_{O_K}^\times$ by \[G(H^\sep/H)\stackrel{-|_{K^\infty}}{\longrightarrow} G(K^\ab/H)\stackrel{\theta_{K}^{-1}|_{G(K^\infty/H)}}{\longrightarrow} A_{O_K}^\times/O_K^\times\stackrel{g}{\to} A_{O_K}^\times.\]

We will now show that $\rho$ satisfies the conditions of (\ref{prop:cm-exist-given-h}) to construct a CM elliptic curve $E/\Spec(H)$ with $\rho_{E/H}=\rho$. For each $\sigma\in G(H^\sep/H)$ there is an $s\in A_{O_K}^\times/O_K^\times$ such that $\sigma|_{K^\infty}=\theta_K(s)\in G(K^\infty/H)$. Unwinding the definition of $\rho^{-1}$ we find \[[\rho(\sigma)^{-1}]=[g(s \bmod O_K^\times)]=[s\alpha(s)^{-1}]=[s]\] so that by (iii) of (\ref{prop:compute-the-homomorphisms-h}) we have \[[\sigma]_H=[\sigma|_{K^\sep}]_K=[s]=[\rho(\sigma)^{-1}].\] Therefore, the diagram \[\xymatrix{G(H^\sep/H)\ar[dr]_-{[-]_H}\ar[r]^-{\rho^{-1}} & A_{O_K}^\times\ar[d]^{[-]}\\
& CL_{O_K, \infty}}\] commutes and by  (\ref{prop:cm-exist-given-h}) there exists a CM elliptic curve $E/\Spec(H)$ with $\rho_{E/H}=\rho$. As $\rho_{E/H}: G(H^\sep/H)\to A^\times_{O_K}$ factors through $G(K^\ab/H)\to A_{O_K}^\times$, it follows that $H(E[\tors])/K$ is abelian over $K$ and $E/\Spec(H)$ is of Shimura type. Moreover, $E/\Spec(H)$ has good reduction away from $\p$ by (\ref{coro:image-of-inertia-group}) as by construction the composition \[O_{K_\l}^\times\to A_{O_K}^\times/O_K^\times\isomto G(K^\ab/H)\to A_{O_K}^\times\] is equal to the natural inclusion $O_{K_\l}^\times\to A_{O_K}^\times$ for all primes $\l\neq \p$.
\end{proof}

\begin{rema} It follows from (\ref{prop:shimura-lambda-curves-over-hilbert}) above that the map $c_{\mc{E}/\mc{M}}: \mc{M}_\CM\to M_\CM$ from $\mc{M}_\CM$ to its coarse sheaf admits sections Zariski locally over $\Spec(O_K)$. Indeed, by (\ref{prop:shimura-lambda-curves-over-hilbert}) there are infinitely many primes $\p$ of $O_K$ and CM elliptic curves $E/M_\CM[\p^{-1}]$. Each map $c_{E/M_\CM[\p^{-1}]}: M_\CM[\p^{-1}]\to M_\CM$ is then equal to the inclusion $M[\p^{-1}]\to M$ followed by some automorphism $\sigma_{\a}$ of $M_\CM$. By (\ref{subsec:coarse-maps-level}), replacing $E$ with $E\otimes_{O_K}\a$, we may assume that $c_{E/M_\CM[\p^{-1}]}: M_\CM[\p^{-1}]\to M_\CM$ is the inclusion.

We note for future reference that if $\p_1$ and $\p_2$ are two such primes with $E_1/M_\CM[\p_1^{-1}]$ and $E_2/M_\CM[\p_2^{-1}]$ as above, then the fact that $c_{E_1/M_\CM[\p_1^{-1}]}$ and $c_{E_2/M_\CM[\p_2^{-1}]}$ are equal to the natural inclusions implies that $E_1$ and $E_2$ are locally isomorphic on the over lap $M_\CM[\p^{-1}_1]\cap M_\CM[\p_2^{-1}]=M_\CM[(\p_1\p_2)^{-1}]\subset M_\CM$.
\end{rema}

\begin{prop}\label{prop:no-abelian-torsion-good-reduction-hilbert} There does not exist a \textup{CM} elliptic curve $E/H$ of Shimura type with good reduction everywhere.
\end{prop}
\begin{proof} Let $E/H$ be a CM elliptic curve and consider the character (\ref{subsec:definition-h-for-cm-curves}) \[\rho_{E/H}: G(K^\sep/H)\to A_{O_K}^\times.\] If $H(E[\tors])/K$ is abelian then $\rho_{E/H}$ factors as \[\rho_{E/H}: G(K^\infty/H)=G(K^\ab/H)\to A_{O_K}^\times\] (note that $K^\ab=K^\infty$). Composing the reciprocal $\rho_{E/H}^{-1}$ with the isomorphism $\theta_K: A_{O_K}^\times/O_K^\times\isomto G(K^\infty/H)=G(K^\ab/H)$ of (\ref{ray-class-isom}) we obtain a homomorphism \[f:A_{O_K}/O_K^\times\to A_{O_K}^\times.\]

The fundamental relation $[\rho_{E/H}^{-1}]=[-]_H$, the fact that $[-]_H=[-]_K|_{G(K^\sep/H)}$, the fact that $\theta_K\circ [-]_K=|_{K^\infty}$ and the fact that the inertia group $I_\p(K^\infty/H)\subset G(K^\infty/K)$ corresponds to $O_{K_\p}^\times\subset A_{O_K}^\times/O_K^\times \subset (A_{O_K}\otimes_{O_K}K)^\times/O_K^\times$ under $\theta_K$ combined with (\ref{coro:image-of-inertia-group}) shows that for $E/H$ to have good reduction at all places of $H$ lying above a prime $\p$ of $O_K$ is equivalent to the composition \[O_{K_\p}^\times\subset A_{O_K}^\times/O_K^\times\stackrel{f}{\to} A_{O_K}^\times\] being equal to the inclusion $O_{K_\p}^\times\subset A_{O_K}^\times$. If this is true for all prime ideals $\p$, the composition of $f$ with the quotient map \begin{equation}\label{eqn:f-comp-quot}A_{O_K}^\times\to A_{O_K}^\times/O_K^\times \stackrel{f}{\to} A_{O_K}^\times\end{equation} is equal to the identity on the sub-group of $A_{O_K}^\times$ generated by the sub-groups $O_{K_\p}^\times$ for all primes $\p$ of $O_K$. But this sub-group is dense and so it follows that the composition (\ref{eqn:f-comp-quot}) is equal to the identity and this is clearly impossible.
\end{proof}

\begin{rema}\label{rema:rohr} We can now answer the question raised in (\ref{rema:rohr-q}) asking for a trivialisation of the $\mc{CL}_{O_K}$-torsor $\mc{M}_{O_K}$. In light of the fact that the coarse sheaf $M_\CM$ of $\mc{M}_\CM$ is isomorphic to $\Spec(O_H)$ it more natural to ask the following: does there exists a CM elliptic curve $\mc{E}/\Spec(O_H)$ inducing a trivialisation of the $\mc{CL}_{O_K}$-torsor $\mc{M}_\CM$, i.e.\ an equivalence of stacks \[\mc{CL}_{O_H}\times \Spec(O_H)\isomto \mc{M}_\CM\times \Spec(O_H): \mc{L}/S\mto E_S\otimes_{O_K}\mc{L}\ ?\] While we have shown (\ref{prop:no-abelian-torsion-good-reduction-hilbert}) that there do not exist CM elliptic curves $E/H$ of Shimura type with good reduction everywhere, it is possible for there to exist CM elliptic curves $\mc{E}/\Spec(O_H)$ (of course they will not be of Shimura type). Indeed, Rohrlich \cite{Rohrlich1982} has shown that if the discriminant of $K$ over $\Q$ is divisible by at least two primes congruent to $3\bmod 4$ then there does exist a CM elliptic curve $\mc{E}/\Spec(O_H)$. The answer in general to the question above is however negative. Indeed, if $K$ has class number, so that $K=H$, then every CM elliptic curve $E/\Spec(K)$ is of Shimura type and so cannot have good reduction everywhere, i.e.\ there does not exist a CM elliptic curve $\mc{E}/\Spec(O_K)$.
\end{rema}

\subsection{}\label{subsec:cm-over-a-field-shimura-integral} We continue with the notation of (\ref{subsec:cm-over-a-field-shimura}) so that $L/K$ is an abelian extension and $E/\Spec(L)$ is a CM elliptic curve. We also fix an integral ideal $\g\in \Id_{O_K}$ such that $\Spec(O_L[\g^{-1}])$ is unramified over $\Spec(O_K)$ and such that $E$ has good reduction over $\Spec(O_L[\g^{-1}])$. We then set $S=\Spec(O_L[\g^{-1}])$ and $P=\Id_{O_K}^{(\g)}$ so that $S$ admits a unique $\Lambda_{P}$-structure (\ref{prop:lambda-structures-etale-schemes-number-fields}) whose Frobenius lifts we denote by $\sigma_\a=\sigma_{\a, S}: S\to S$ for $\a\in P$. We write $\mc{E}\to S$ for the N\'eron model of $E$ relative to $\Spec(L)\to S$ so that $\mc{E}\to S$ is a CM elliptic curve. For $\p\in P$ a prime ideal let us write $S_\p=S\times_{\Spec(O_K)}\Spec(\F_\p)$ and $\mc{E}_\p=S\times_S S_\p$.

\begin{lemm}\label{prop:frob-lifts-unique} For each $\p\in P$, there is at most one homomorphism \[\psi_{\mc{E}/S}^{\p}:\mc{E}\to \sigma_\p^*(\mc{E})\] lifting the $N\p$-power relative Frobenius map of $\mc{E}_\p$.
\end{lemm}
\begin{proof} By rigidity the difference of two such homomorphisms is equal to the zero map on some open and closed sub-scheme of $S$, the only choices of which are $S$ and $\emptyset$. Therefore, as any two such isomorphisms must agree on $S_\p\subset S$ which is non-empty, they agree everywhere.
\end{proof}

\begin{theo}\label{theo:shimura-is-lambda} In the notation of \textup{(\ref{subsec:cm-over-a-field-shimura-integral})}, the following are equivalent:
\begin{enumerate}[label=\textup{(\roman*)}]
\item The \textup{CM} elliptic curve $\mc{E}\to S$ admits a canonical $\Lambda_{P, S}$-structure.
\item The \textup{CM} elliptic curve $\mc{E}\to S$ admits a $\Lambda_{P, S}$-structure.
\item The extension $L(E[\tors])/K$ is abelian, i.e.\ $E/\Spec(L)$ is a \textup{CM} elliptic curve of Shimura type.
\item The homomorphism $\rho_{E/L}: G(L^\sep/L)\to A_{O_K}^\times$ factors throught $G(L^\sep/L) \to G(K^\ab/L)$.
\end{enumerate}
\end{theo}
\begin{proof} (i) implies (ii): This is clear.

(ii) implies (iii): For each ideal $\a$ the sub-schemes $\mc{E}[\a]=\ker(\psi_{\mc{E}/S}^\a)\subset \mc{E}$ are finite and locally free $\Lambda_{P, S}$-schemes. After inverting $\a$ and forgetting about the Frobenius lifts for primes dividing $\a$, we may apply (\ref{prop:lambda-structures-etale-schemes-number-fields}) to see that $K(E[\a])=L(E[\a])$ is abelian over $K$.

(iii) is equivalent to (iv): This is immediate from the definition of $\rho_{E/L}$ (\ref{subsec:definition-h-for-cm-curves}).

(iv) implies (i): If $\g=O_K$ then $L/K$ is unramified and so $L\subset H$. However, $M_\CM\isomto \Spec(O_H)$ and so we must have $L=H$. However, if $L=H$ and $\g=O_K$ then $E/\Spec(H)$ admits good reduction everywhere which, as $E/\Spec(H)$ is of Shimura type by hypothesis, is impossible (\ref{prop:no-abelian-torsion-good-reduction-hilbert}). Therefore, $\g\neq O_K$ and it follows that replacing $\g$ with $\g^{n}$ for some $n\geq 0$ we may assume that $\g$ separates units (this changes neither $\Spec(O_L[\g^{-1}])$ nor $P=\Id_{O_K}^{(\g)}$).

Write $L'=L(E[\g])$ and $S'=\Spec(O_{L'}[\g^{-1}])$. The extension $L'/K$ is abelian (by hypothesis $L(E[\tors])$ is abelian) and unramified away from $\g$ (as $\mc{E}[\g]$ is \'etale over $S$) so that $S'$ admits a unique $\Lambda_{P}$-structure. By construction, the CM elliptic curve $\mc{E}\times_S S'$ admits a level-$\g$ structure and, choosing one, we obtain a map $S\to M_\CM^{(\g)}$ and an isomorphism $\mc{E}\times_S S'\isomto E^{(\g)}\times_{M_\CM^{(\g)}} S'$. As the morphism $S'\to M_\CM^{(\g)}$ is a $\Lambda_P$-morphism and $E^{(\g)}\to M_{\CM}^{(\g)}$ admits a canonical $\Lambda_{P, M_\CM^{(\g)}}$-structure it follows that \[\mc{E}\times_S S'\isomto E^{(\g)}\times_{M_\CM^{(\g)}} S'\] admits a canonical $\Lambda_{P, S'}$-structure and by (iii) of (\ref{prop:canonical-lambda-p-structures}) that $\mc{E}\to S$ admits a canonical $\Lambda_{P, S}$-structure.
\end{proof}

\begin{rema}\label{rema:shimura-curves-locally-lambda-isomorphic} Now let $\mc{E}, \mc{E}'$ be a pair of CM elliptic curves of Shimura type over $S$ (we keep $\g$ the same). If $\mc{E}$ and $\mc{E}'$ are locally isomorphic (note that there is always some ideal $\a$ such that $\mc{E}\otimes_{O_K}\a$ and $\mc{E}'$ are locally isomorphic) then they are actually $\Lambda_{P, S}$-locally isomorphic. Indeed, generically, they become isomorphic over the extension $L'/L$ corresponding to the character \[\rho:=\rho_{E/L}\rho_{E'/L}^{-1}: G(K^\ab/L)\to O_K^\times.\] It is clear that the extension $L'/K$ is abelian and it also is unramified away from $\g$ (the characters $\rho_{E/L}$ and $\rho_{E'/L}$ agree on the inertia sub-groups for all $\p\nmid \g$ --- this is the good reduction of $\mc{E}$ and $\mc{E}'$ away from $\g$). Therefore, $S'=\Spec(O_L[\g^{-1}])$ is finite and \'etale over $S$ and admits a unique $\Lambda_{P}$-structure. Moreover, there exists an isomorphism \[f: \mc{E}\times_{S} S\isomto \mc{E}'\times_S S'\] and it remains to observe that all such isomorphisms are $\Lambda_{P, S'}$-isomorphisms. Indeed, $f$ is a $\Lambda_{P, S'}$-isomorphism if and only if the $\Lambda_{P, S'}$-structure on $\mc{E}$ induced by transport of structure along $f$ is equal to the given $\Lambda_{P, S'}$-structure on $\mc{E}'$ and this follows from (\ref{prop:frob-lifts-unique}).
\end{rema}

\subsection{} We now wish to study the Lie algebras of the N\'eron models of CM elliptic curves of Shimura type. So we continue with the notation of (\ref{subsec:cm-over-a-field-shimura-integral}) but will also assume that the CM elliptic curve $\mc{E}\to S$ admits a canonical $\Lambda_{P, S}$-structure, i.e.\ that $E/\Spec(L)$ is a CM elliptic curve of Shimura type. The field $L$ must contain the Hilbert class field $H\subset L$. The following well known property enjoyed by the Hilbert class field $H$ will be crucial:

\begin{prop}[Hauptidealsatz]\label{prop:hauptidealsatz} Every rank one $O_K$-module becomes free after base change to $O_H$.
\end{prop}

\subsection{}\label{subsec:lambda-module-for-cm-curves} We shall abuse notation and write, for each prime ideal $\p\in P$, \[\nu_{\p}: \mc{E}\otimes_{O_K}\p^{-1}\isomto \sigma_\p^*(\mc{E})\] for the unique isomorphism lifting the isomorphism $\nu_{\p}(\mc{E}_\p/S_\p)$ of (\ref{coro:tensor-p-equals-frobenius}) or equivalently the unique isomorphism such that the composition \[\mc{E}\stackrel{i_\p}{\to} \mc{E}\otimes_{O_K}\p^{-1}\stackrel{\nu_\p}{\to}\sigma_\p^*(\mc{E})\] is equal to the relative Frobenius lift $\psi_{\mc{E}/S}^\p$.

For a pair of primes $\p$, $\l$ the commutativity of the (relative) Frobenius lifts on $\mc{E}\to S$ amounts to the equalities \begin{equation}\sigma_{\l}^*(\nu_\p)\circ (\nu_{\l}\otimes_{O_K}\p^{-1})=\sigma_{\p}^*(\nu_{\l})\circ (\nu_{\p}\otimes_{O_K}\l^{-1})\label{eqn:frob-commute-prime}\end{equation} so that for any ideal $\a\in P$, choosing a prime factorisation of $\a$, we may define isomorphisms \[\nu_\a: \mc{E}\otimes_{O_K}\a^{-1}\isomto \sigma_\a^*(\mc{E})\] by composing the various $\nu_\p$ for $\p|\a$, with the resulting isomorphism being independent of any choices involved by virtue of (\ref{eqn:frob-commute-prime}). These isomorphisms now satisfy \begin{equation}\sigma_{\a}^*(\nu_\b)\circ (\nu_{\a}\otimes_{O_K}\b^{-1})=\sigma_{\b}^*(\nu_{\a})\circ (\nu_{\b}\otimes_{O_K}\a^{-1})\label{eqn:frob-commute-all}\end{equation} for each $\a, \b\in P$.

We can actually say more about the $\nu_\a$ when $\sigma_\a=\id_{L}.$ In this case, we get an isomorphism $\nu_\a: \mc{E}\otimes_{O_K}\a^{-1}\to \sigma_\a^*(\mc{E})=\mc{E}$ which must be of the form $1\otimes l(\a)$ for some $l(\a)\in O_K$ which generates $\a$. If, moreover, there is an ideal $\f$ such that $E[\f]$, and hence $\mc{E}[\f]$, are constant then composition \[\mc{E}[\f]\to \mc{E}[\f]\otimes_{O_K}\a^{-1}\stackrel{\nu_\a[\f]}{\to} \mc{E}[\f]\] is equal to the unique relative Frobenius lift \[l(\a)=\psi_{\mc{E}[\f]/S}^\a=\id_{\mc{E}[\f]}:\mc{E}[\f]\to \mc{E}[\f]\] so that $l(\a)=1\bmod \f$.

By the N\'eron mapping property the isomorphisms $\nu_\a$ extend to isomorphisms on the full N\'eron model \[\nu_\a: \Ner_{O_L}(E)\otimes_{O_K}\a^{-1}\isomto \sigma_\a^*(\Ner_{O_L}(E))\] satisfying the same commutativity condition (note that $\Ner_{O_L}(E)$ is a smooth group scheme of relative dimension one, not necessarily proper). Denoting by \[T=\underline{\Lie}_{\Ner_{O_L}(E)/\Spec(O_L)}\] the Lie algebra of the N\'eron model, which is a projective rank one $O_L$-module, the isomorphisms $\nu_\a$ induce $O_L$-isomorphisms (which we denote by the same letter) \[\nu_\a: T\otimes_{O_K}\a^{-1}\isomto \sigma_\a^*(T)\] for each $\a\in P$, satisfying the same commutativity condition (\ref{eqn:frob-commute-all}) as the $\nu_\a$ on $\Ner_{O_L}(E)$.

\begin{coro}\label{coro:shimura-curves-minimal-model} In the notation of \textup{(\ref{subsec:lambda-module-for-cm-curves})}, if $L=K(\f)$ is a ray class field and $E[\f]$ is constant then the Lie algebra of the N\'eron model $\Ner_{O_{K(\f)}}(E)$ becomes free after inverting $\f$. In other words, $E/\Spec(K(\f))$ admits a global minimal model away from $\f$.
\end{coro}
\begin{proof} We have the isomorphisms \[\nu_\a: T\otimes_{O_K}\a^{-1}\isomto \sigma_\a^*(T)\] for each ideal $\a\in P=\Id_{O_K}^{(\g)}$ and when $\sigma_\a=\id_{O_{K(\f)}}$ (this is the case if and only if $\a=(a)$ is principal with $a=1\bmod \f$) the isomorphism \[\nu_\a: T\otimes_{O_K}\a^{-1}\isomto \sigma_\a^*(T)=T\] takes the form $1\otimes l(\a)$ where $l(\a)\in O_{K}$ is a generator of $\a$ such that $l(\a)=1\bmod \f$. We now apply (\ref{prop:tannaka-result-special}) to extend $l$ to a map $l:\Id_{O_K}^{(\g)}\to O_{K(\f)}$ satisfying \begin{equation} l(\a)\cdot O_{K(\f)}=\a\cdot O_{K(\f)} \quad \text{ and } \quad l(\a\b)=l(\a)\sigma_\a(l(\b))\label{eqn:l-com-condition}\end{equation} for all $\a, \b\in \Id_{O_K}^{(\g)}$. Define $t_\a: T\to \sigma_\a^*(T)$ be the composition \[T\stackrel{1\otimes l(\a)^{-1}}{\longrightarrow} T\otimes_{O_K}\a^{-1}\stackrel{\nu_\a}{\longrightarrow} \sigma^*_a(T).\] Then $t_\a$ is an isomorphism by (\ref{eqn:l-com-condition}), and if $\a\in \Prin_{1\bmod \f}^{(\g)}$ then $t_\a=l(\a)\otimes l(\a)^{-1}=\id_T$ which, combined with the commutativity conditions (\ref{eqn:frob-commute-all}) on the $\nu_\a$ and (\ref{eqn:l-com-condition}) on the $l(\a)$, show that $t_\a$ depends only on the class $\sigma_\a\in G(K(\f)/K)$, that $t_{\id_{K(\f)}}=\id_T$ and that $t_{\sigma\tau}=t_\sigma\circ \sigma^*(t_\tau)$ for $\sigma, \tau\in G(K(\f)/K)$.

In other words, the isomorphisms $t_\sigma$ (or perhaps what is more standard, their inverses) define Galois descent data on $T$ relative to $O_{K}\to O_{K(\f)}$. As $O_K\to O_{K(\f)}$ becomes finite and \'etale after inverting $\f$, there exists an $O_K[\f^{-1}]$-module $T_0$ such that \[T_0\otimes_{O_K[\f^{-1}]} O_{K(\f)}[\f^{-1}]\isomto T\otimes_{O_{K(\f)}}O_{K(\f)}[\f^{-1}].\] However, as $K(\f)$ contains the Hilbert class field $H$, all rank one projective $O_{K}[\f^{-1}]$-modules become free after base change to $O_{K(\f)}[\f^{-1}]$ (all rank one projective $O_K$-modules do by the Hauptidealsatz (\ref{prop:hauptidealsatz}) and every rank one projective $O_K[\f^{-1}]$-module is a localisation of a rank one projective $O_K$-module). It follows that \[T_0\otimes_{O_K[\f^{-1}]}O_{K(\f)}[\f^{-1}]\isomto T\otimes_{O_{K(\f)}}O_{K(\f)}[\f^{-1}]=\underline{\Lie}_{\Ner_{O_L}(E)/\Spec(O_L)}\otimes_{O_L}O_L[\f^{-1}]\] is free.
\end{proof}

\begin{rema} We note that when $\f=O_K$ (\ref{coro:shimura-curves-minimal-model}) the condition on the $\f$-torsion becomes trivial so that \textit{any} CM elliptic $E/H$ of Shimura type admits a \textit{global} minimal model. This generalises the main result (see Corollary 4.4) of \cite{Gross82} where it is shown that a CM elliptic curve $E/H$ admits a global minimal model whenever the conductor of $K$ over $\Q$ is prime and the homomorphism $\rho_{E/H}$ satisfies a certain invariance condition. One can show that these assumptions imply $E/H$ is a CM elliptic curve of Shimura type so that (\ref{coro:shimura-curves-minimal-model}) is indeed a generalisation of \cite{Gross82}. To say a little, the invariance condition on $\rho_{E/H}$ is equivalent to $E$ being isogenous to $\sigma^*(E)$ for each $\sigma\in G(H/K)$ and the primality of the discriminant of $K$ over $\Q$ implies that the order of $G(H/K)$ is prime to the order of $O_K^\times$ and together these properties allow one to show that $E/H$ is a CM elliptic curve of Shimura type (for a result along these lines see Proposition 2 of \cite{Gilles1985}).
\end{rema}

\section{Weber functions} The purpose of this section is to define, for each CM elliptic curve $E$ over an arbitrary base $S$, a certain quotient $E\to X_{E/S}$ and then to study its resulting properties. Informally, $X_{E/S}$ will be the quotient of $E$ by its group of automorphisms $\underline{O_K^\times}_S$. However, $\underline{O_K^\times}_S$ does not act freely on $E$ and so the orbits of are not well behaved. This makes it difficult to construct a quotient (in the naive sense) with any useful properties. We get around this problem by using Cartier divisors to define intelligent orbits for the action of $\underline{O_K^\times}_S$ under which it behaves as though it were free. Taking the quotient by the resulting equivalence relation, we get a smooth, proper curve $X_{E/S}$ together with a $\underline{O_K^\times}_S$-invariant finite locally free map of degree $w=\# O_K^\times$ \[p_{E/S}: E\to X_{E/S}.\]

The construction we give is functorial in $E/S$ and so we may run it for the universal CM elliptic curve $\mc{E}\to \mc{M}_\CM$ to obtain a smooth, proper curve $X_{\mc{E}/\mc{M}_\CM}\to \mc{M}_\CM$. Almost by definition, this curve descends to a smooth proper curve $f : X\to M_\CM$ over the coarse sheaf.

The remainder of the section is devoted to the study of $X\to M_\CM$. We first show that it has the following properties:

\begin{enumerate}[label=(\roman*)]
\item $f: X\to M_\CM$ admits a natural $\Lambda_{M_\CM}$-structure and a $\Lambda_{M_\CM}$-point $0_{X}: M_\CM\to X$,
\item $f: X\to M_\CM$ has genus zero. Thus, if $\mc{I}_X\subset \mc{O}_X$ denotes the ideal sheaf defining the closed point $0_X: M_\CM\to X$ and $\mc{W}=f_*(\mc{I}_X^{-1})$ then $\mc{W}$ is locally free of rank two over $\mc{O}_{M_\CM}$, the map $f^*(\mc{W})\to \mc{I}^{-1}_X$ is an epimorphism and the resulting map \[X\isomto \mathbf{P}_{M_\CM}(\mc{W})\] is an isomorphism,
\item setting $X[\a]=\psi_{X/M_\CM}^{\a*}(0_X)\subset X$, the scheme $X[\a]$ is a finite locally free $\Lambda_{M_\CM}$-scheme of degree $N\a$ and $K(X[\a])=K(\a)$.
\end{enumerate}

Thus the curve $X\to M_\CM$ together with its $\Lambda_{M_\CM}$-structure and its $\Lambda_{M_\CM}$-point $0_X: M_\CM\to X$ allow one to construct the ray class fields of $K$ in an integral and coherent, choice free manner. Of course, this is just a more streamlined and abstract approach to the classical construction of the ray class fields of $K$ using Weber functions --- this approach being to choose a CM elliptic curve $E/H$ and to consider the image of $E[\a]\subset E$ under a `Weber map' \[E\to \mathbf{P}^1_H\] which is a certain $O_K^\times$-invariant map of degree $w$ (see Theorem 5.6 of \cite{Silverman94}).

The only defect of our approach is that the curve $X\to M_\CM$ is not particularly explicit. However, we end the section by showing (by the same method we used to show that CM elliptic curves of Shimura type admit global minimal models) that there exists an isomorphism $X\isomto \mathbf{P}^1_{M_\CM}$.

\subsection{} We begin by recalling some basic facts regarding Cartier divisors on curves (see \S\S 1.1-1.2 Chapter I of \cite{KatzMazur85}). Let $S$ be a scheme and let $X\to S$ be a smooth proper $S$-curve, i.e.\ $X\to S$ is smooth of relative dimension one and proper. An $S$-relative Cartier divisor on $X$, or just a Cartier divisor, is a closed sub-scheme $D\subset X$ which is finite locally free over $S$. Equivalently, a closed sub-scheme $D\subset X$ is a Cartier divisor if and only if $D\to S$ is flat and the ideal sheaf $\mc{I}_D\subset \mc{O}_E$ defining $D$ is a locally free rank one $\mc{O}_X$-module. Given two Cartier divisors $D, D'\subset X$, their sum $D+D'\subset X$ is defined to be the closed sub-scheme corresponding to the ideal sheaf $\mc{I}_{D+D'}:=\mc{I}_D\otimes_{\mc{O}_X}\mc{I}_{D'}\subset \mc{O}_X$, which is again a Cartier divisor.

The degree $\deg(D)$ of a Cartier divisor $D\subset S$ is defined to be the degree of the finite locally free $S$-scheme $D$. We have $\deg(D+D')=\deg(D)+\deg(D)$. The structure map $D\to S$ of a Cartier divisor on $X$ is an isomorphism if and only if $\deg(D)=1$ and the set of degree one Cartier divisors on $E$ is equal to the set of $S$-points $S\to X$. If $s\in X(S)$ is an $S$-point then we will denote the corresponding Cartier divisor by $s$.

If $f: X'\to X$ is a finite locally free map of smooth, proper $S$-curves and if $D\subset X$ is a Cartier divisor then $f^*(D)\subset X'$ is a Cartier divisor and $\deg(f^*(D))=\deg(f)\deg(D)$.

Given a pair of Cartier divisors $D, D'\subset X$ such that $\deg(D)\leq \deg(D')$ (resp. $\deg(D)=\deg(D')$) we can form the inclusion (resp. equality) $S$-sheaf of $D$ and $D'$: \[\In_{X/S}(D, D')\subset S \quad \text{ (resp. } \quad \Eq_{X/S}(D, D')\subset S)\] defined by the property that $T\to S$ factors through $\In_{X/S}(D, D')\to S$ (resp. $\Eq_{X/S}(D, D')\to S$) if and only we have an inclusion (resp. equality) of Cartier divisors \[D\times_S T\subset D'\times_S T\subset X\times_{S} T \quad \text{ (resp. } \quad D\times_S T= D'\times_S T\subset X\times_{S} T).\] By Key Lemma 1.3.4 and Corollary 1.3.5 of \cite{KatzMazur85}, the sub-sheaves $\mathrm{In}_{X/S}(D, D')$ and $\mathrm{Eq}_{X/S}(D, D')$ of $S$ are finitely presented closed sub-schemes of $S$. Finally, if $S'\to S$ is a morphism then we have natural isomorphisms \[\mathrm{In}_{X_{S'}/S'}(D_{S'}, D'_{S'})\isomto \mathrm{In}_{X/S}(D, D')_{S'} \quad \text{ and } \quad \mathrm{Eq}_{X_{S'}/S'}(D_{S'}, D'_{S'})\isomto \mathrm{Eq}_{X/S}(D, D')_{S'}.\]

\subsection{} Now let $E/S$ be a CM elliptic curve. We would like to take the quotient of $E$ by the action of its group of automorphisms $\underline{O_K}_S^\times$ but, as noted in the introduction, this action is not free and in particular the map \[\coprod_{\epsilon\in O_K^\times} (\epsilon, \id_E): \coprod_{\epsilon\in O_K^\times}E=\underline{O_K^\times}_S\times_S E\to E\times_S E\] is not injective so that the orbits of points under this action are not well behaved. Of course, one could just take the image of this map and obtain an equivalence relation but then one would have little control over the quotient.

We get around this as follows. If $s: S\to E$ is an $S$-point, we define its `orbit' $[O_K^\times](s)\subset E$ to be the Cartier divisor \[[O_K^\times](s)=\sum_{\epsilon\in O_K^\times} \epsilon^*(s)\] (note the sum is of Cartier divisors and has nothing to do with the group law on $E$). Then $[O_K^\times](s)\subset E$ contains the Cartier divisor $s$, is stable under the action of $\underline{O_K^\times}_S$ and is finite locally free of degree $w$ over $S$ (as the usual orbit would be if the action of $G$ were free). Equality of these `orbits' defines an equivalence relation on $E$, which we denote by \[\Eq_{E/S}^{O_K^\times}\subset E\times_S E,\] so that an $S$-morphism $T\to E\times_S E$ factors through $\Eq_{E/S}^{O_K^\times}$ if and only if, writing $(t_1, t_2): T\to (E\times_S E)\times_S T=E_T\times_T E_T$ for the resulting map, we have an equality of Cartier divisors \[[O_K^\times](t_1)=[O_K^\times](t_2)\subset E_T=E\times_S T.\]

For this equivalence relation to behave as though it really does come from a group action, it should be the case that, given two points $s_1, s_2: S\to E$ then having $s_1\in [O_K^\times](s_2)$, i.e.\ $s_1: S\to E$ factoring through $[O_K^\times](s_2)$, should imply the equality of `orbits' \[[O_K^\times](s_1)=[O_K^\times](s_2).\]

With this in mind, we define the sub-sheaf \[\In_{E/S}^{O_K^\times}\subset E\times_S E\] by the property that an $S$-morphism $T\to E\times_S E$ factors through $\In_{E/S}^{O_K^\times}\subset E\times_S E$ if and only if, writing $(t_1, t_2): T\to (E\times_S E)\times_S T=E_T\times_T E_T$ for the corresponding map, the Cartier divisor $t_1: E\to E_T$ factors through the Cartier divisor $[O_K^\times](t_2)\subset E_T$, i.e.\ $t_1\in [O_K^\times](t_2)$. There is a natural inclusion \[\Eq_{E/S}^{O_K^\times}\subset \In_{E/S}^{O_K^\times}\subset E\times_S E\] and our claim that the equivalence relation $\Eq_{E/S}^{O_K^\times}$ behaves as though it were coming from a free action of a group is that we have an equality \[\Eq_{E/S}^{O_K^\times} = \In_{E/S}^{O_K^\times}\subset E\times_S E.\]

Before we prove this, let us make two observations. First, the sub-sheaves $\Eq_{E/S}^{O_K^\times}\subset \In_{E/S}^{O_K^\times}\subset E\times_S E$ are in fact closed sub-schemes. View $E\times_S (E\times_S E)\to E\times_S E$ as a CM elliptic curve over $E\times_S E$ via projection onto the second two factors and consider the Cartier divisors \[u_{E/S, 1}: E\times_S E\to E\times_S (E\times_S E): (e_1, e_2)\mto (e_1, e_1, e_2)\] and \[u_{E/S, 2}: E\times_S E\to E\times_S (E\times_S E): (e_1, e_2)\mto (e_2, e_1, e_2).\] Then it is an easy exercise to check that we have equalities of sub-sheaves of $E\times_S E$ \[\Eq_{E/S}^{O_K^\times}=\Eq_{E\times_S (E\times_S E)/E\times_S E}([O_K^\times](u_{E/S, 1}), [O_K^\times](u_{E/S, 2}))\subset E\times_S E\] and \[\In_{E/S}^{O_K^\times}=\In_{E\times_S (E\times_S E)/E\times_S E}(u_{E/S,2}, [O_K^\times](u_{E/S,2}))\subset E\times_S E\] so that by the representability of equality and inclusion sub-sheaves of Cartier divisors we find that $\Eq_{E/S}^{O_K^\times}$ and $\In_{E/S}^{O_K^\times}$ are closed sub-schemes of $E\times_S E$.

Our second observation is that, viewing $E\times_S E\to E$ as a CM elliptic curve over $E$ via projection onto the second factor, we have an equality of sub-schemes \[\In_{E/S}^{O_K^\times}=[O_K^\times](\Delta_{E/S})\subset E\times_S E.\] Indeed, fixing an affine scheme $T$ over $S$ and a morphism $T\to E$, which we identify with a morphism $t_2: T\to E\times_S T=E_T$, the pull-back of the $E$-morphism $[O_K^\times](\Delta_{E/S})\subset E\times_S E$ along $T\to E$ is given by \[[O_K^\times](t_2)\subset E_T=E\times_S T\] (as the pull-back of $\Delta_{E/S}: E\to E\times_S E$ is given by $t_2: T\to E_T=E\times_S T$). Therefore, for a second $S$-morphism $T\to E$, which we identify with a morphism $t_1: T\to E_T=E\times_S T$, to have the property that the induced map $T\to E\times_S E$ factors through $[O_K^\times](\Delta_{E/S})$, is equivalent to the morphism $t_1: T\to E_T$ factoring through $[O_K^\times](t_2)\subset E_T$. All said and done, a morphism $T\to E\times_S E$ factors through $[O_K^\times](\Delta_{E/S})\subset E\times_S E$ if and only if, in the notation above, the morphism $t_1: T\to E_T$ factors through $[O_K^\times](t_2)\subset E_T$ which is to say that $T\to E\times_S E$ factors through $\In_{E/S}^{O_K^\times}\subset E\times_S E$. We are now ready to prove our claim.

\begin{prop}\label{prop:trace-equivalence-relation} Let $E/S$ be a \textup{CM} elliptic curve. Then we have equalities of closed sub-schemes of $E\times_S E$ \[\Eq_{E/S}^{O_K^\times} = \In_{E/S}^{O_K^\times} = [O_K^\times](\Delta_{E/S})\subset E\times_S E.\] In particular, $\Eq^{O_K^\times}_{E/S}$ is a finite locally free equivalence relation of degree $w$.
\end{prop}
\begin{proof} The only thing we need to show is that the inclusion \[\Eq_{E/S}^{O_K^\times} \subset \In_{E/S}^{O_K^\times}\] is an equality. The first thing we note is that it is bijective on geometric points. That is, if $S=\Spec(F)$ is the spectrum of an algebraically closed field $F$, and $(s_1, s_2) \in E(S)\times E(S)$ satisfy $s_1\in [O_K^\times](s_2)$ then $[O_K^\times](s_1) = [O_K^\times](s_2)$. But this is clear, as $E$ is then a Dedekind scheme and unique factorisation of Cartier divisors (i.e.\ of the corresponding ideals) shows that there exists an $\epsilon\in O_K^\times$ (not necessarily unique!) such that $s_1=\epsilon s_2$ which gives \[[O_K^\times](s_1)=[O_K^\times](\epsilon s_1)=[O_K^\times](s_2).\] It follows that the inclusion \[\Eq_{E/S}^{O_K^\times} \subset \In_{E/S}^{O_K^\times},\] which is a closed immersion, is a nilpotent thickening.

Now, our claim is local on $S$ and so we may assume that $S$ admits a level-$\f$ structure for some $\f$ which separates units. It follows that way assume there exists a morphism $S\to M_\CM^{(\f)}$ and an isomorphism \[E \isomto E^{(\f)}\times_{M_\CM^{(\f)}} S\] and again by the compatibility of inclusion and equality schemes of Cartier divisors with base change, we may assume that $E/S=E^{(\f)}/M_\CM^{(\f)}$, or what is important, that $S$ is integral.

We will now show that the nilpotent immersion \[\Eq_{E/S}^{O_K^\times}\to \In_{E/S}^{O_K^\times}\] is an isomorphism by showing that $\In_{E/S}^{O_K^\times}$ is reduced. Since $S$ is an integral scheme, it follows that $E$ is also an integral scheme. Therefore, the finite locally free $E$-scheme $\In_{E/S}^{O_K^\times}\subset E\times_S E$ is reduced if and only if for some non-empty open sub-scheme $U\subset E$ the pull-back $\In_{E/S}^{O_K^\times}\times_E U$ is reduced.

So let \[U=\bigcap_{1\neq \epsilon\in O_K^\times} (E-E[1-\epsilon]).\] Then over $U$, the Cartier divisors \[\epsilon^*(\Delta_{E/S}\times_E U): U\to E\times_S U\] for $\epsilon\in O_K^\times$ are all disjoint so that \[\In_{E/S}^{O_K^\times}\times_E U=[O_K^\times](\Delta_{E/S})\times_E U=\coprod_{\epsilon\in O_K^\times} \epsilon^*(\Delta_{E/S}\times_E U)\isomto \coprod_{O_K^\times} E\times_S U.\] It follows that $\In_{E/S}^{O_K^\times}\times_E U$ and therefore $\In_{E/S}^{O_K^\times}$ are reduced and the nilpotent immersion \[\Eq_{E/S}^{O_K^\times} \subset \In_{E/S}^{O_K^\times}\] is an isomorphism.
\end{proof}

\subsection{} We are now ready to construct our quotients. So let $E/S$ be a CM elliptic curve. Then the Cartier divisor of degree $w$ \[\Eq_{E/S}^{O_K^\times}=[O_K^\times](\Delta_{E/S})\subset E\times_S E\] defines the equivalence relation on $E/S$ where $s_1, s_2\in E(S)$ are equivalent if and only if \[[O_K^\times](s_1)=[O_K^\times](s_2).\] We write \[p_{E/S}: E\to X_{E/S}\] for the resulting quotient (in the category of fpqc sheaves). As our equivalence relation is finite locally free (and $E$ is projective over $S$) it follows from Corollaire 7.1 Expos\'e VII of \cite{SGA1} that $X_{E/S}$ is representable by a scheme over $S$. We have thus constructed a finite locally free $\underline{O_K^\times}_S$-invariant map of degree $w$ \[p_{E/S}: E\to X_{E/S}.\]

\begin{prop} The $S$-scheme $X_{E/S}\to S$ is smooth of relative dimension one, proper and geometrically connected. 
\end{prop}
\begin{proof} As $E\to S$ is proper, flat, and geometrically connected and $p_{E/S}: X\to X_{E/S}$ is finite locally free of degree $w$ (in particular, proper, flat and surjective), it follows that $X_{E/S}\to S$ is proper, flat, and geometrically connected. As $X_{E/S}\to S$ is flat, it is smooth if and only if its fibres are smooth, but when $S$ is the spectrum of a field $E$ is a regular scheme of dimension one and $E\to X_{E/S}$ is finite locally free so that $X_{E/S}$ is also regular of dimension one and therefore smooth over $S$.
\end{proof}

\begin{rema} The method used above to construct the quotients $X_{E/S}$ applies more generally. Indeed, if $X\to S$ is a smooth (not necessarily proper) curve, $S$ is an integral scheme and $G$ is a finite group acting generically freely on $X$ by $S$-automorphisms then there exists a smooth curve $X/[G]$ and a finite locally free $G$-invariant map $p: X\to X/[G]$ of degree $\#G$, i.e.\ what might be called a `quotient' of $X$ by $G$. It would be interesting to know whether this method could be extended to construct `quotients' of curves over more general bases $S$, or under more general (non-constant) group actions.
\end{rema}

\subsection{} Let us now consider the functorial properties of the association $E/S\mto X_{E/S}$. First, if $f: E\to E'$ is a morphism of CM elliptic curves over $S$, then $(f\times_S f)|_{\Eq^{O_K^\times}_{E/S}}\subset \Eq^{O_K^\times}_{E'/S}\subset E\times_S E$ and so there is an induced morphism $X_f: X_{E/S}\to X_{E'/S}$ and a commutative diagram \[\xymatrix{E\ar[d]_{p_{E/S}}\ar[r]^{f}&E'\ar[d]^{p_{E'/S}}\\
X_{E/S}\ar[r]^{X_f} &  X_{E'/S}.}\] The invariance of $p_{E/S}$ and $p_{E'/S}$ under $\underline{O_K^\times}_S$ shows that $f\mto X_f$ is also invariant under $\underline{O_K^\times}_S$. In symbols, the map \[\underline{\Hom}_S^{O_K}(E, E')\to \underline{\Hom}_S(X_{E/S}, X_{E'/S'}): f\mto X_{f}\] factors through the quotient sheaf \[\underline{\Hom}_{S}^{O_K}(E, E')/\underline{O_K^\times}_S \to \underline{\Hom}_S(X_{E/S}, X_{E'/S'}).\] In particular, if $E$ and $E'$ are locally isomorphic then $\underline{\Isom}_{S}^{O_K}(E, E')$ is an $\underline{O_K^\times}_S$-torsor so that $\underline{\Isom}_S^{O_K}(E, E')/\underline{O_K^\times}_S\isomto S$ and we obtain a canonical map $S\to \underline{\Isom}_S(X_{E/S}, X_{E'/S})$ or what is the same an isomorphism \[X_{E/S}\isomto X_{E'/S}.\]

\subsection{} Now let $S$ be an $M_\CM$-sheaf, i.e.\ $S\to M_\CM$. By the definition of the coarse sheaf $M_\CM$, this implies that there exists a cover $(S_i\to S)_{i\in I}$ of $S$ and CM elliptic curves $(E_i/S_i)_{i\in I}$, such that for $i, j\in I$, writing $S_{ij}=S_i\times_S S_j$, the CM elliptic curves $E_i\times_{S_i} S_{ij}$ and $E_j\times_{S_j}S_{ij}$ are locally isomorphic. Therefore, writing $X_i=X_{E_i/S_i}$ for $i\in I$ we have for all $i, j\in I$ canonical isomorphisms \begin{equation} X_i\times_{S_i} S_{ij}\isomto X_j\times_{S_j} S_{ij}.\label{eqn:descent}\end{equation} The independence of these isomorphisms from the choice of isomorphism $E_i\times_{S_i} S_{ij}\isomto E_j\times_{S_j} S_{ij}$ (and similarly on the triple products) show that the isomorphisms (\ref{eqn:descent}) equip the family of curves $(X_i/S_i)_{i\in I}$ with descent data relative to the cover $(S_i\to S)_{i\in I}$, which furnishes us with (a priori) a sheaf $X_S\to S$. Similar observations show that the sheaf $X_S\to S$ is independent (upto canonical isomorphism) of the cover $(S_i\to S)_{i\in I}$ and the CM elliptic curves $(E_i/S_i)_{i\in I}$ and that if $S'\to S$ is a morphism of $M_\CM$-sheaves then we have a canonical isomorphism \[X_{S'}\isomto X_{S}\times_S S'.\] In particular, applying this to $\id_{M_\CM}: M_\CM\to M_\CM$ we obtain a sheaf $X\to M_\CM$ and isomorphisms \[X_{S}\isomto X\times_{M_\CM} S.\]

\subsection{}\label{subsec:props-of-x} Before we consider the geometric properties of the sheaf $X\to M_\CM$, let us first make the following observations.

\begin{enumerate}[label=(\roman*)]
\item There is a unique morphism $0_X: M_\CM\to X$ with the property that if $E/S$ is a CM elliptic curve and $c_{E/S}: S\to M_\CM$ is the coarse map then $c_{E/S}^*(0_X)=0_E: S\to X_{E/S}=c_{E/S}^*(X)$.
\item For each $\a\in \Id_{O_K}$ there is a unique morphism \[\psi_{X/M_\CM}^\a: X\to \sigma_\a^*(X)\] such that for all CM elliptic curves $E/S$, using the identifications \[c_{E\otimes_{O_K}\a^{-1}/S}=\sigma_\a\circ c_{E/S} \quad \text{ and then }\quad c_{E\otimes_{O_K}\a^{-1}/S}^*(X)=c_{E/S}^*(\sigma_\a^*(X))\] the diagram \begin{equation}\label{eqn:def-frob-lattes}\xymatrix{E\ar[rr]^-{i_\a}\ar[d] && E\otimes_{O_K}\a^{-1}\ar[d]\\
c_{E/S}^*(X)\ar[rr]^{c_{E/S}^*(\psi_{X/M_\CM}^\a)} && c_{E/S}^*(\sigma_\a^*(X))}\end{equation} commutes.
\item The morphisms $\psi_{X/M_\CM}^{\p}: X\to \sigma_\p^*(X)$ for $\p$ prime lift the $N\p$-power relative Frobeniuses and for $\a,\b$ any two ideals we have the commutativity condition \[\psi_{X/M_\CM}^{\a\b}=\sigma_\a^*(\psi_{X/M_\CM}^\b)\circ \psi_{X/M_\CM}^\a\] which equips $X_{E/S}$ with the structure of a $\Psi_{M_\CM}$-sheaf.
\item The morphism $0_X: M_\CM\to X$ is a $\Psi_{M_\CM}$-morphism.
\end{enumerate}

\begin{prop} The sheaf $f: X\to M_\CM$ is a smooth, projective curve of genus zero. In particular, if $\mc{I}_{X}\subset \mc{O}_X$ denotes the ideal sheaf defining the closed point $0_X: M_\CM\to X$ then $\mc{W}:=f_{X*}(\mc{I}^{-1}_{X})$ is a locally free rank two $\mc{O}_{M_\CM}$-module, the morphism $f_X^*(\mc{W})\to \mc{I}_{X}^{-1}$ is an epimorphism, and the induced map $X\to \mathbf{P}_{M_\CM}(\mc{W})$ is an isomorphism.
\end{prop}
\begin{proof} The fact that $X\to M_\CM$ is a curve is immediate from the fact that $M_\CM$ admits an open cover $(M_i\to M_\CM)_{i\in I}$ with CM elliptic curves $E_i/M_i$. Indeed, the defining property of $X\to M_\CM$ then gives \[X\times_{M_\CM} M_i\isomto X_{E_i/M_i}\] and we know that $X_{E_i/M_i}$ is a smooth of relative dimension one, proper and geometrically connected.

It remains to compute the genus. The genus is constant along the fibres of a smooth, proper geometrically connected curve, and so to compute it we may do so after base change along any morphism $\Spec(K^\sep)\to M$. Fixing a CM elliptic curve $E/\Spec(K^\sep)$ and considering the degree $w$ finite locally free morphism $p=p_{E/\Spec(K^\sep)}: E\to X\times_{M}\Spec(K^\sep)$ the Riemann-Hurwitz formula gives: \[2-2g_E=w (2-2 g_X)-\sum_{x\in E(K^\sep)} (e_x-1)\] where $g_E$ is the genus of $E$, $g_X$ is the genus of $X\times_{M_\CM}\Spec(K^\sep)$ and where $e_x$ is the ramification degree of $p$ at $x$. We have $g_E=1$, and for each $x\in E(K^\sep)$ the ramification degree $e_x$ is equal to $\#\mathrm{Stab}(x)-1$ where $\mathrm{Stab}(x)\subset O_K^\times$ is the stabiliser of $x$. It is now just a matter computation, depending on whether $O_K^\times=\mu_2$, $ \mu_4$ or $\mu_6$, to verify that the equality \[0=(2-2g_E)=w (2-2 g_X)-\sum_{x\in E(K^\sep)} (e_x-1)\] implies $g_X=0$. We do it for $O_K^\times=\mu_6$. The only point with stabiliser $\mu_6$ is $0\in E(K^\sep)$, the points with stabiliser $\mu_2\subset \mu_6$ are the three points of $E[2]-0$ and the points with stabiliser $\mu_3\subset \mu_6$ are the two points of $E[1-\zeta_3]-0$ (for $\zeta_3\in \mu_3$ a generator). Therefore, we find \[0=6\cdot (2-2g_X)-1\cdot (6-1) - 3\cdot (2-1) - 2\cdot (3-1)=-12 g_X\] and hence $g_X=0.$

The other claims now follow using standard arguments from the theory of curves.
\end{proof}

\begin{coro} For each $\a\in \Id_{O_K}$ the morphism $\psi_{X/M_\CM}^\a: X\to \sigma_\a^*(X)$ is finite locally free of degree $N\a$ and together they equip $X$ with the structure of a $\Lambda_{M_\CM}$-scheme and the morphism $0_X: M_\CM\to X$ is a $\Lambda_{M_\CM}$-morphism.
\end{coro}
\begin{proof} As $X$ is flat over $\Spec(O_K)$ it follows by (\ref{prop:lambda-structures-frobenius-lifts-sheaves}) and (iv) of (\ref{subsec:props-of-x}) that $X$ admits a $\Lambda_{M_\CM}$-structure and that the morphism $0_X: M_\CM\to X$ is a $\Lambda_{M_\CM}$-morphism. The only thing we need to verify is that $\psi_{X/M_\CM}^{\a}$ is finite locally free of degree $N\a$ but this follows from the diagram (\ref{eqn:def-frob-lattes}) defining $\psi_{X/M_\CM}^\a$ as both columns are finite locally free of degree $w$ and the top map is finite locally free of degree $N\a$ and hence the bottom map must also be finite locally free of degree $N\a$.
\end{proof}

\subsection{} Write $X[O_K]\subset X$ for the (image of) the morphism $0_X: M_\CM\to X$, and for each $\a\in \Id_{O_K}$ define $X[\a]:=\psi_{X}^{\a*}(X[O_K])\subset X$. Then $X[\a]\subset X$ is a finite locally free $\Lambda_{M_\CM}$-scheme of degree $N\a$.

\begin{prop} For each ideal $\a\in \Id_{O_K}$ the extension $K(X[\a])$ of $K$ generated by the coordinates of $X[\a]$ is equal to ray class field $K(\a)$.
\end{prop}
\begin{proof} As $M_\CM\isomto \Spec(O_H)$ it is enough to show that the action of $G(K^\sep/H)$ on $X[\a](\Spec(K^\sep))$ factors faithfully through the quotient $G(K^\sep/H)\to G(K(\a)/K)$. To do this we may choose a CM elliptic curve $E/\Spec(H)$ with character $\rho_{E/H}: G(K^\sep/H)\to A_{O_K}^\times$ and consider the map \[p=p_{E/\Spec(H)}: E\to X_{E/\Spec(H)}=X\times_{M_\CM}\Spec(H).\] Then $p$ induces a surjective map of $G(K^\sep/H)$-sets \[E[\a](\Spec(K^\sep))\to X[\a](\Spec(K^\sep))\] and moreover factors through an isomorphism \[E[\a](\Spec(K^\sep))\to E[\a](\Spec(K^\sep))/O_K^\times\isomto X[\a](\Spec(K^\sep)).\] Therefore, an element $\sigma\in G(K^\sep/H)$ acts trivially on $X[\a](\Spec(K^\sep))$ if and only if $\rho_{E/H, \a}(\sigma)\in (O_K/\a)^\times$ is contained in the image of $O_K^\times\to (O_K/\a)^\times$, i.e.\ in the notation of (\ref{subsec:def-morphism-h}) if and only if $[\rho_{E/H, \a}(\sigma)]_\a=1$. But \[[\rho_{E/H, \a}(\sigma)^{-1}]_\a=[\sigma]_{H, \a}=[\sigma]_{K, \a}\] and the kernel of $[-]_{K, \a}$ is precisely $G(K^\sep/K(\a))$ and so we are done.
\end{proof}

\subsection{} We now wish to study the Lie algebra $T:=\underline{\Lie}_{X/M_\CM}$ at the closed point $0_{M_\CM}: M_\CM\to X$.

\begin{prop} For each $\a\in \Id_{O_K}$ the map \[D_\a: T\to \sigma^*_\a(T)\] induced by $\psi_{X/M_\CM}^{\a*}: X\to \sigma^*_\a(X)$ factors as \[T \isomto \sigma_\a^*(T)\otimes_{O_K}\a^{w}\to \sigma^*_\a(T)\] where the second map is multiplication.
\end{prop}
\begin{proof} It is enough to show that for each $\a$, Zariski locally on $M_\CM$, the image of \[D_\a: T\to \sigma^*_\a(T)\] is equal to $\a^w\otimes_{\mc{O}_{M_\CM}} \sigma_\a^*(T)$.

So let $x\in M_\CM$ be a point and let $S=\Spec(\mc{O}_{M_\CM, x})\to M_\CM$ be the inclusion of the local ring at $x$, and let $E/S$ be a CM elliptic curve. The inverse image of $0: S\to X_{E/S}=X\times_{M_\CM}S$ along the map $p_{E/S}: E\to X_{E/S}$ is equal to $[O_K^\times](0_E)$ and as $0: S\to E$ is invariant under $O_K^\times$ we have $[O_K^\times](0_E)=\Inf_{0_E}^{w-1}(E)$ is the $(w-1)$st infinitesimal neighbourhood of $0_E: S\to E$. It follows that, writing $\widehat{X}=\Inf_{0_X}(X)$ that $p_{E/S}^*(\widehat{X}_S)=\Inf_{0_E}(E)=\widehat{E}$ and that the map \[\widehat{E}\to \widehat{X}\] is finite locally free of degree $w$ and $O_K^\times$ invariant.

It now follows from Proposition 7.5.2 that, choosing an isomorphism $\widehat{E}\isomto \widehat{\mathbf{A}}^1_S$ there is a unique isomorphism $\widehat{X}_S\isomto \widehat{\mathbf{A}}^1_S$ such that the diagram \[\xymatrix{\widehat{E}\ar[r]\ar[d]_-\wr & \widehat{X}_S\ar[d]^-\wr\\
\widehat{\mathbf{A}}^1_S\ar[r]^{N_{O_K^\times}(T)} &\widehat{\mathbf{A}}^1_S}\] commutes where the bottom map is \[a\mto N_{O_K^\times}(a)=\prod_{\epsilon\in O_K^\times}[\epsilon](a)\] and $[\epsilon](T)$ is the power series on $\widehat{\mathbf{A}}^1_S\isomto \widehat{E}$ representing the automorphism $\epsilon: \widehat{E}\to \widehat{E}$. But as the action of $O_K$ is strict we known that $[\epsilon](T)=\epsilon T+\cdots$ and so $N_{O_K^\times}(T)=-T^{w}+\cdots$. From this and the fact that the induced map on Lie algebras of $\psi_{E/S}^{\p}: E\to \sigma^*_\p(E)$ factors as \[\underline{\Lie}_{E/S}\isomto \p\otimes_{\mc{O}_{M_\CM}} \sigma_\p^*(\underline{\Lie}_{E/S})\to \sigma_\p^*(\underline{\Lie}_{E/S})\] the claim follows.
\end{proof}

\begin{coro}\label{prop:tangent-lattes-free} The rank one $\mc{O}_{M_\CM}$-module $T$ is free and there exists an isomorphism \[X\isomto \mathbf{P}^1_{M_\CM}.\]
\end{coro}
\begin{proof} We have \[X\isomto \mathbf{P}_{M_\CM}(\mc{W})\] where $\mc{W}=f_*(\mc{I}_X^{-1})$ where $\mc{I}_X\subset \mc{O}_X$ is the ideal sheaf defining the closed point $0_X: M_\CM\to X$. Moreover, we have an exact sequence \[0\to \mc{O}_{M_\CM}\to \mc{W}\to T\to 0.\] As $T$ is projective this exact sequence splits. Then by the same method as in the proof of (\ref{coro:shimura-curves-minimal-model}) one shows that the isomorphisms \[D_\a: T\isomto\a^{w}\otimes_{\mc{O}_{M_\CM}}\sigma_\a^*(T)\] can be turned into descent data from $M_\CM$ to $\Spec(O_K)$ and so by the Hauptidealsatz $T$ is free. Therefore, $\mc{W}\isomto \mc{O}_{M_\CM}^2$ and \[X\isomto \mathbf{P}_{M_\CM}(\mc{W})\isomto \mathbf{P}_{M_\CM}(\mc{O}_{M_\CM}^2)=\mathbf{P}_{M_\CM}^1.\]
\end{proof}

\begin{rema} We end this section with a remark regarding possible applications to monogeneity of rings of integers. Define Cartier divisors \[\Theta_\a\subset X[\a]\subset X\] inductively by setting $\Theta_{O_K}=X[O_K]$ and, having defined $\Theta_\a\subset X[\a]$, we define $\Theta_{\a\p}$ to be \[\Theta_{\a\p}:=\psi_\p^*(\sigma_\a^*(\Theta_\a))-\Theta_\a\]  if $\p\nmid\a$ (an analysis of the $\Lambda_{M_\CM}$-structure on $X[\a]$ shows that this is possible) or to be \[\Theta_{\a\p}:=\psi_{X/M_\CM}^{\p*}(\sigma_\p^*(\Theta_\a))\] if $\p|\a$. Then one can show that $\Theta_{\a}\subset X$ is irreducible (but in general non-reduced), that $K(\Theta_\a)=K(\a)$ and that \[X[\a]=\sum_{\d|\a}\Theta_\d.\] The identity above should be viewed as analogous to the factorisation of the polynomials $X^n-1$ in terms of cyclotomic polynomials.

Moreover, when $\a$ is composite $\Theta_\a$ and $X[O_K]=\Theta_{O_K}$ are disjoint so that one finds a closed immersion \[\Theta_\a\subset X-X[O_K]=\mathbf{P}_{M_\CM}-X[O_K](\mc{W})\isomto \mathbf{V}_{M_\CM}(T)\isomto \mathbf{A}^1_{M_\CM}.\] With a more detailed analysis of the $\Lambda_{M_\CM}$-structure on $X_{M_\CM}\isomto \mathbf{P}^1_{O_H}$ it may be possible to show that the divisors $(\Theta_\a)_\red$ are regular which would imply isomorphisms \[(\Theta_\a)_\red\isomto \Spec(O_{K(\a)}).\] This would then give closed immersion \[\Spec(O_{K(\a)})\to \mathbf{A}^1_{O_H}\isomto \mathbf{A}^1_{M_\CM}\] or in other words $O_{K(\a)}$ would be monogenic over the ring of integers in the Hilbert class field $O_H$. It is worth noting that this method would only apply to conductors $\a$ which are composite, and that when $\a$ is not composite counterexamples to the monogeneity of $O_{K(\a)}$ over $O_{H}$ are known to exist (see \cite{CouFleck89}).
\end{rema}

\section{A $\Lambda$-equivariant cover of $\mc{M}_\CM$}\label{sec:lambda-cover-of-m} In this section we show how one can use the existence of canonical lifts of CM elliptic curves to define a flat, affine and formally smooth cover of $\mc{M}_\CM$ admitting a $\Lambda$-structure compatible with that on $\mc{M}_\CM$. Indeed, we rigidify $\mc{M}_\CM$ by equipping a CM elliptic curve $E/S$ with a trivialisation of the Lie algebra of its canonical lift. The results of (\ref{sec:shimura-curves}) allow us to show that this does indeed define a cover of $\mc{M}_\CM$ with the desired properties.

\subsection{} Let $S$ be an ind-affine scheme. The category $\mc{CL}_{O_K}(S)$ acts on the category $\mc{M}_\CM^*(S)$. However, we have shown that $\mc{M}_\CM(S)$ is equivalent to the category $W_*(\mc{M}_{\CM})_\Lambda(S)$ so that $\mc{CL}_{O_K}$ should also act on $W_*(\mc{M}_{\CM})_\Lambda$ and now explain how.

As $W^*$ commutes with \'etale fibre products and $\underline{O_K}_S$ is \'etale over $S$, the sheaf $W^*(\underline{O_K}_S)$ is naturally a $\Lambda_{W^*(S)}$-sheaf of rings over $W^*(S)$. Moreover, $W^*$ also commutes with disjoint unions which gives an identification \[W^*(\underline{O_K}_S)\isomto \underline{O_K}_{W^*(S)}\] compatible with the ring structures.
In much the same way, if $\mc{L}$ is a rank one $O_K$-local system over $S$, then $W^*(\mc{L})$ is a rank one $O_K$-local system over $W^*(S)$ equipped with a $\Lambda_{W^*(S)}$-structure which is compatible with its $\underline{O_K}_{W^*(S)}=W^*(\underline{O_K}_S)$-module structure, i.e.\ the map \[\underline{O_K}_{W^*(S)}\times_{W^*(S)}W^*(mc{L})\to W^*(\mc{L})\] defining the action of $\underline{O_K}_{W^*(S)}$ on $W^*(\mc{L})$ is a $\Lambda_{W^*(S)}$-morphism.

Now if $E/S$ is a CM elliptic curve then, as both $W_\CM^*(E)$ and $W^*(\mc{L})$ have $\Lambda_{W^*(S)}$-structures compatible with their $\underline{O_K}_S$-module structures, and as the forgetful functor $\Sh_{\Lambda_{W^*(S)}}\to \Sh_{W^*(S)}$ commutes with all limits and colimits (\ref{def:w-w-lambda-adjunction}), the CM elliptic curve \[W_\CM^*(E)\otimes_{O_K}W^*(\mc{L})\] is also equipped with a $\Lambda_{W^*(S)}$-structure compatible with its $\underline{O_K}_{W^*(S)}$-module structure.

\begin{prop}\label{prop:action-of-cl-on-can-lifts} The $\Lambda_{W^*(S)}$ structure on $W_\CM^*(E)\otimes_{O_K}W^*(\mc{L})$ is canonical and there is a unique $\Lambda_{W^*(S)}$-isomorphism \[W_\CM^*(E\otimes_{O_K}\mc{L})\isomto W_\CM^*(E)\otimes_{O_K} W^*(\mc{L})\] inducing the identity on the ghost components at $(1)$.
\end{prop}
\begin{proof} The $\Psi_{\Gamma^*(S)}$-structure on $W^*(\mc{L})\times_{W^*(S)}\Gamma^*(S)$ is compatible with the isomorphisms \[W^*(\mc{L})\times_{W^*(S)}\Gamma^*(S)\isomto \Gamma^*(\mc{L})\isomto \mc{L}\times_{S}\Gamma^*(S)\] where the last fibre product is over the sum of the identity maps \[\Gamma^*(S)=\coprod_{\a\in \Id_{O_K}} S\to S\] and where the $\Psi_{\Gamma^*(S)}$-structure on $\mc{L}\times_{S}\Gamma^*(S)$ is induced by the $\Psi$-structure on $\Gamma^*(S)$. Therefore, we obtain $\Psi_{\Gamma^*(S)}$-isomorphisms \[(W_\CM^*(E)\otimes_{O_K}W^*(\mc{L}))\times_{W^*(S)}\Gamma^*(S)\isomto \Gamma_{\CM}^*(E)\otimes_{O_K}\Gamma^*(\mc{L})\isomto \Gamma_{\CM}^*(E\otimes_{O_K}\mc{L})\] which induce the identity after pull-back to the ghost component at $(1)$. This is precisely the definition of a canonical $\Lambda_{W^*(S)}$-structure and this proves the first statement. As the ghost components at $(1)$ of $W_\CM^*(E)\otimes_{O_K}W^*(\mc{L})$ and $W_\CM^*(E\otimes_{O_K}\mc{L})$ are equal to $E\otimes_{O_K}\mc{L}$ we get a unique $\Lambda_{W^*(S)}$-isomorphism \[W_\CM^*(E)\otimes_{O_K}W^*(\mc{L})\isomto W_\CM^*(E\otimes_{O_K}\mc{L})\] inducing the identity after pull-back along the ghost component at $(1)$.
\end{proof}

\begin{rema}\label{rema:structure-map-injective} Consider the sheaf \[W_*(\mathbf{A}^1)=W_*(\Spec(O_K[T]))=\Spec(\Lambda\odot O_K[T])=\Spec(\Lambda)\] of arithmetic jets of $\mathbf{A}^1$ over $\Spec(O_K)$. It is a ring scheme over $\Spec(O_K)$ and its sections over an affine scheme $S=\Spec(A)$ are given by \[W_*(\mathbf{A}^1)(\Spec(A))=W(A).\] The structure map \[O_K\to W(A)=W_*(\mathbf{A}^1)(S)\] for varying affine schemes $S=\Spec(A)$ is injective (\ref{coro:structure-map-lambda-ring-injective}) and induces a \textit{monomorphism} of sheaves of rings \begin{equation} \label{def:structure-injective} i_{W}:\underline{O_K}\to W_*(\mathbf{A}^1).\end{equation} This fact will be crucial for what follows.
\end{rema}

\subsection{} Let $\mc{L}/S$ be a rank one $O_K$-local system. A level-$W$ structure on $\mc{L}$ is an isomorphism of $\mc{O}_{W^*(S)}$-modules \[\lambda: W^*(\mc{L})\otimes_{\underline{O_K}_{W^*(S)}} \mc{O}_{W^*(S)}\isomto \mc{O}_{W^*(S)}.\] A $W$-isomorphism $(\mc{L}/S, \lambda)\isomto (\mc{L}'/S, \lambda')$ of rank one $O_K$-local systems with level-$W$ structure is an $\underline{O_K}_S$-isomorphism $h: \mc{L}\isomto \mc{L}'$ such that \[\lambda=\lambda'\circ W^*(h).\]

The tensor product of two rank one $O_K$-local systems with level-$W$ structure $(\mc{L}, \lambda)$ and $(\mc{L}', \lambda')$ is defined to be \[(\mc{L}\otimes_{O_K}\mc{L}', \lambda\otimes_{\mc{O}_{W^*(S)}}\lambda')\] where we view $\lambda\otimes_{\mc{O}_{W^*(S)}}\lambda'$ as a level-$W$ structure on $\mc{L}\otimes_{O_K}\mc{L}'$ using the identification \[W^*(\mc{L}\otimes_{O_K}\mc{L}')\isomto W^*(\mc{L})\otimes_{O_K}W^*(\mc{L}').\]

We write $\mc{CL}_{O_K}^W$ for the fibred category over $\IndAff_{O_K}$ with fibre over $S$ given by the category rank one $O_K$-local systems over $S$ equipped with a level-$W$ structure $(\mc{L}/S, \lambda)$ together with their $W$-isomorphisms.

\begin{prop}\label{prop:cw-stack-sheaf} $\mc{CL}_{O_K}^W$ is a stack over $\IndAff_{O_K}$ for the \'etale topology and is equivalent to its coarse sheaf for the \'etale topology.
\end{prop}
\begin{proof} That $\mc{CL}_{O_K}^W$ is a stack for the affine \'etale topology on $\IndAff_{O_K}$ is easy to see using the fact that if $\mc{L}/S$ is a rank one $O_K$-local system and $S'\to S$ is any morphism then \[W^*(\mc{L})\times_{W^*(S)}W^*(S')\isomto W^*(\mc{L}\times_S S')\] and that for any \'etale morphism $T\to S$ we have \[W^*(S'\times_S T)\isomto W^*(S')\times_{W^*(S)}W^*(T).\]

For the second statement it is enough to show that if $(\mc{L}, \lambda)$ is a rank one $O_K$-local system with level-$W$ structure then every $W$-automorphism of $\mc{L}$ is trivial. But if $\epsilon\in \underline{O_K^\times}_S(S)$ defines a $W$-automorphism of $\mc{L}$, we have an equality \begin{equation} \id_{W^*(\mc{L})}\otimes_{\underline{O_K}_{W^*(S)}} i_W(\epsilon) = W^*(\epsilon)\otimes_{\underline{O_K}_{W^*(S)}}\id_{\mc{O}_{W^*(S)}}\label{eqn:w-autom}\end{equation} of automorphisms of $W^*(\mc{L})\otimes_{\underline{O_K}_{W^*(S)}}\mc{O}_{W^*(S)}$. As $i_W$ is a monomorphism and the $\rho$ is invariant under the automorphism (\ref{eqn:w-autom}) we must have $\epsilon=1.$
\end{proof}

\subsection{} We write $CL_{O_K}^W$ for the coarse sheaf of $\mc{CL}_{O_K}^W$ with which we identify it by (\ref{prop:cw-stack-sheaf}). The tensor product of rank one $O_K$-local systems with level-$W$ structure equips $CL_{O_K}^W$ with the structure of a sheaf of groups over $\Spec(O_K)$.

We now describe a short exact sequence relating $CL_{O_K}^W$ to the $\Spec(O_K)$-group scheme of arithmetic jets $W_*(\mathbf{G}_\mathrm{m})$ of $\mathbf{G}_{\mathrm{m}}$.

Let $S=\Spec(A)$ be an affine $\Spec(O_K)$-scheme. The sections of $W_*(\mathbf{G}_\mathrm{m})$ over $S$ are given by \[W_*(\mathbf{G}_\mathrm{m})(\Spec(A))=W(A)^\times\] and the monomorphism (\ref{def:structure-injective}) restricts to a monomorphism again denoted \[i_W: \underline{O_K^\times}\to W_*(\mathbf{G}_{\mathrm{m}}).\]

For each \[a\in \mathbf{G}_{\mathrm{m}}(W^*(S))=\Aut_{W^*(S)}(\mc{O}_{W^*(S)})\] we define an element $[a]_W\in CL_{O_K}^W(S)$ by $[a]_W:=(\underline{O_K}_S, a)$ where we view $a$ as the level-$W$ structure on $\underline{O_K}_S$: \[\underline{O_K}_{W^*(S)}\otimes_{\underline{O_K}_{W^*(S)}}\mc{O}_{W^*(S)}=\mc{O}_{W^*(S)}\stackrel{a}{\to}\mc{O}_{W^*(S)}.\] This defines a homomorphism \[[-]_W:W_*(\mathbf{G}_\mathrm{m})\to CL_{O_K}^W\]
Finally, composing the forgetful map \[CL_{O_K}^W\to \mc{CL}_{O_K}: (\mc{L}/S, \lambda)\mto \mc{L}/S\] with the map $\mc{CL}_{O_K}\to\underline{CL_{O_K}}$ of (\ref{prop:class-stacks-triv-admis}) we obtain a homomorphism \[f_W: CL_{O_K}^W\to \underline{CL_{O_K}}.\]

\begin{prop} The sequence of sheaves \[0\to \underline{O_K^\times}\stackrel{i_W}{\longrightarrow} W_*(\mathbf{G}_\mathrm{m})\stackrel{[-]_W}{\longrightarrow} CL_{O_K}^W\stackrel{f_W}{\longrightarrow} \underline{\CL_{O_K}}\to 0\] is exact for the \'etale topology and $CL_{O_K}^W$ is representable by a flat, affine formally smooth group scheme over $\Spec(O_K)$.
\end{prop}
\begin{proof} We first show that $f_W: CL_{O_K}^W\to \underline{CL_{O_K}}$ is an epimorphism for the \'etale topology. This is equivalent to showing that for each ideal $\a$ (or at least one in each ideal class of $\CL_{O_K}$) there is an \'etale cover $S\to \Spec(O_K)$ with the property that the $O_K$-local system $\underline{\a}_{S}$ admits a level-$W$ structure.

Let $S=\Spec(O_H)\to \Spec(O_K)$. Then we have an isomorphism of $\mc{O}_{W^*(S)}$-modules \[\underline{\a}_{W^*(S)}\otimes_{O_K}\mc{O}_{W^*(S)}\isomto \a\cdot \mc{O}_{W^*(S)}\] where $\a\cdot \mc{O}_{W^*(S)}$ is the ideal sheaf defining the closed immersion $W^*(S)\times_{\Spec(O_K)}\Spec(O_K/\a)\to W^*(S)$. However, this map is obtained by pulling-back the map $S\times_{\Spec(O_K)}\Spec(O_K/\a)\to S$ along the morphism $\mu_S: W^*(S)\to S$ defining the $\Lambda$-structure on $S=\Spec(O_H)$. Therefore, it is enough to show that the ideal sheaf defining the closed immersion $S\times_{\Spec(O_K)}\Spec(O_K/\a)\to S$ is free, but this sheaf is $\a\otimes_{O_K} O_H$ which is free by the Hauptidealsatz (\ref{prop:hauptidealsatz}).

We now show that the map $[-]_W: W_*(\mathbf{G}_\mathrm{m})\to CL_{O_K}^W$ defines an epimorphism onto the kernel of $f_W$. It is clear that $[-]_W$ maps to $\ker(f_W)$ as the rank one $O_K$-local system underling $[-]_W$ is the trivial one. Now let $S$ be an affine scheme and let $(\mc{L}/S, \lambda)\in \ker(f_W)$. We will show that there exists a cover $(S_i\to S)_{i\in I}$ and elements $a_i\in W_
*(\mathbf{G}_\mathrm{m})(S_i)$ with $[a_i]=(\mc{L}_{S_i}, \lambda_{S_i})$.

Since $(\mc{L}, \lambda)\in \ker(f_W)$ it follows that $\mc{L}$ is \'etale locally isomorphic $\underline{O_K}_S$. Therefore, we may assume that $(\mc{L}, \lambda)=(\underline{O_K}_S, \lambda)$ but then the isomorphism \[\lambda: \underline{O_K}_S\otimes_{\underline{O_K}_S}\mc{O}_{W^*(S)}=\mc{O}_{W^*(S)}\isomto \mc{O}_{W^*(S)}\] is given by some $a\in W_*(\mathbf{G}_{\mathrm{m}})(S)=\Aut_{W^*(S)}(\mc{O}_{W^*(S)})$ and we have \[(\underline{O_K}_S, \lambda)=(\underline{O_K}_S, a)=[a]_W.\]

Finally, we compute the kernel of $[-]_W$. So let $a\in W_*(\mathbf{G}_\mathrm{m})(S)$ and assume that $[a]_W=(\underline{O_K}_S, a)=(\underline{O_K}_S, \id_{\mc{O}_{W^*(S)}})$. By definition, this implies the existence of an isomorphism \[\epsilon\in \Isom_S^{O_K}(\underline{O_K}_S, \underline{O_K}_S)=\underline{O_K^\times}_S(S)\] such that $W^*(\epsilon)=a$. But $W^*(\epsilon)$ viewed as an element of \[\underline{\Aut}_{\mc{O}_{W^*(S)}}(\mc{O}_{W^*(S)})=W_*(\mathbf{G}_\mathrm{m})(S),\] is precisely $i_W(\epsilon)$. Therefore $a=i_W(\epsilon)$ and $\ker([-]_W)=\underline{O_K}^\times\subset W_*(\mathbf{G}_\mathrm{m}).$

It follows from the above that the kernel of $f_W$ is equal to the sheaf of groups \[W_*(\mathbf{G}_{\mathrm{m}})/\underline{O_K^\times}.\] However, $\underline{O_K^\times}$ is finite \'etale and $W_*(\mathbf{G}_{\mathrm{m}})$ is flat and affine and so it follows that the quotient sheaf \[W_*(\mathbf{G}_{\mathrm{m}})/\underline{O_K^\times}\] is also flat and affine. As $\underline{CL_{O_K}}$ is affine, $f_W$ is an epimorphism and \[CL_{O_K}^W\times_{\underline{CL}_{O_K}} CL_{O_K}^W\isomto CL_{O_K}^W\times W_*(\mathbf{G}_{\mathrm{m}})/\underline{O_K^\times}\] is affine and flat it follows that $CL_{O_K}^W$ is affine and flat over $\Spec(O_K)$. Similarly, as $W_*(\mathbf{G}_{\mathrm{m}})$ is formally smooth (being an inverse limit of smooth affine schemes) and $W_*(\mathbf{G}_{\mathrm{m}})\to W_*(\mathbf{G}_{\mathrm{m}})/\underline{O_K}^\times$ is \'etale it follows that $W_*(\mathbf{G}_{\mathrm{m}})/\underline{O_K}^\times$ is formally smooth and by descent that $CL_{O_K}^W$ is formally smooth.
\end{proof}

\subsection{} Let $S$ be an affine scheme and let $E/S$ be a CM elliptic curve. A level-$W$ structure on $E/S$ is an isomorphism of $\mc{O}_{W^*(S)}$-modules \[\rho: \underline{\Lie}_{W_{\CM}^*(E)/W^*(S)}\isomto \mc{O}_{W^*(S)}.\] A $W$-isomorphism $(E/S, \rho)\isomto (E'/S, \rho')$ of CM elliptic curves with level-$W$ structures is an isomorphism $f: E\isomto E'$ of CM elliptic curves such that \[\rho = \rho'\circ \underline{\Lie}_{W_{\CM}^*(E')/W^*(S)}(W_{\CM}^*(f)).\]

We denote by $\mc{M}_\CM^W$ the fibred category over $\IndAff_{O_K}$ whose fibre over an ind-affine scheme $S$ is given by the category of CM elliptic curves with level-$W$ structures together with their $W$-isomorphisms.

Just as with $CL_{O_K}^W$, the objects of $\mc{M}_\CM^W$ admit no non-trivial automorphisms and so the stack $\mc{M}_\CM^W$ is equivalent to its coarse sheaf which we denote by $M_\CM^W$.

There is an action of $CL_{O_K}^W$ on $M_\CM^W$ given by \[M_\CM^W\times CL_{O_K}^W\to M_\CM^W: ((\mc{L}/S, \lambda), (E/S, \rho))\mto (E\otimes_{O_K}\mc{L}, \rho\otimes_{\mc{O}_{W^*(S)}}\lambda)\] where we use the identification (\ref{prop:action-of-cl-on-can-lifts}) \[W^*_{\CM}(E\otimes_{O_K}\mc{L})\isomto W^*_\CM(E)\otimes_{O_K} W^*(\mc{L}).\]

\begin{theo} We have the following:
\begin{enumerate}[label=\textup{(\roman*)}]
\item The forgetful map \[M_{CM}^W\to \mc{M}_\CM\] is affine, faithfully flat and formally smooth.
\item $M_{CM}^W$ is an $CL_{O_K}^W$-torsor over $\Spec(O_K)$ and is therefore flat, affine and formally smooth over $\Spec(O_K)$.
\item The map \[M_{CM}^W\times_{\Spec(O_K)} W_*(\mathbf{G}_\mathrm{m})\to M_{CM}^W\times_{\mc{M}_\CM}M_{CM}^W: ((E/S, \rho), a)\mto ((E/S, \rho), (E/S, a\rho))\] is an isomorphism.
\end{enumerate} 
\end{theo}
\begin{proof} (i) Let $E/S$ be CM elliptic curve and write \[T=\underline{\Lie}_{W^*_{\CM}(E)/W^*(S)}.\] Then the fibre of the map $\mc{M}_{CM}^W\to \mc{M}_\CM$ along $S\stackrel{E}{\to} \mc{M}_\CM$ is given by \[t_{E/S}: W_*(\underline{\Isom}_{\mc{O}_{W^*(S)}}(T, \mc{O}_{W^*(S)}))\times_{W_*(W^*(S))} S\to S\] and so to the prove the claim we need only show that $t_{E/S}$ is affine, faithfully flat and formally smooth.

We now show that $t_{E/S}$ is affine, faithfully flat and formally smooth whenever $E/S$ admits a level-$W$ structure. Indeed, if $(E/S, \rho)$ is a level-$W$ structure on $E/S$ then we obtain an isomorphism \[\mathbf{G}_{\mathrm{m}}\times W^*(S)\isomto \underline{\Isom}_{\mc{O}_{W^*(S)}}(T, \mc{O}_{W^*(S)}): a\mto a\cdot \rho\] and so an isomorphism \[W_*(\mathbf{G}_{\mathrm{m}}\times W^*(S))\times_{W_*(W^*(S))} S \isomto W_*(\underline{\Isom}_{\mc{O}_{W^*(S)}}(T, \mc{O}_{W^*(S)}))\times_{W_*(W^*(S))} S.\] But \[W_*(\mathbf{G}_\mathrm{m}\times W^*(S))\times_{W_*(W^*(S))} S = W_*(\mathbf{G}_{\mathrm{m}})\times S\to S\] is faithfully flat, affine and formally smooth as \[W_*(\mathbf{G}_{\mathrm{m}})=\lim_{\a\in \Id_{O_K}} W_{\a*}(\mathbf{G}_{\mathrm{m}})\] is an inverse limit of faithfully flat, affine and smooth $\Spec(O_K)$-schemes.

As the morphism $t_{E/S}$ is compatible with base change in order to finish the proof of our claim, we may localise $S$ and in particular assume that $E$ admits a level-$\f$ structure for some $\f$ which separates units. We may now assume that $S=M_\CM^{(\f)}$ and that $E=E^{(\f)}$ and, by the previous arguments, to show that $t_{E^{(\f)}/M_\CM^{(\f)}}$ is faithfully flat, affine and formally smooth it is enough to show that $E^{(\f)}/M_\CM^{(\f)}$ admits a level-$W$ structure. To ease notation, let us write $M=M_\CM^{(\f)}$, $E=E^{(\f)}$, $P=\Id_{O_K}^{(\f)}$ and $P'\subset \Id_{O_K}$ for the sub-monoid generated by the prime divisors of $\f$ (so that $P\cap P'=\{O_K\}$ and $P\cdot P'=\Id_{O_K}$).

Then $W^*(M)=W_P^*(W_{P'}^*(M))$ and as $\f$ is invertible on $M$ we have $\Gamma^*_{P'}(M)=W^*_{P'}(M)$ so that  \[W^*(M)=\coprod_{\a\in P'}W_P^*(M).\] Using this and the fact that the $\Lambda_{P, M}$-structure on $E$ is canonical (\ref{prop:canonical-lambda-structures-and-level-structures}) we find \[W^*_{CM}(E)\isomto \coprod_{\a\in P'} \mu_M^*(E)\otimes_{O_K} \a^{-1}= \coprod_{\a} \mu_M^*(E\otimes_{O_K} \a^{-1})\] where \[\mu_M: W^*_{P}(M)\to M\] defines the (unique) $\Lambda_{P}$-structure on $M$.

Now to show that $E/M$ admits level-$W$ structure it is enough to show that \[\underline{\Lie}_{E\otimes_{O_K}\a^{-1}/M}=\underline{\Lie}_{E/M}\otimes_{O_K}\a^{-1}=\underline{\Lie}_{E/M}\] (the last equality is because $\a\in P'$ is invertible on $M$) is free for each $\a\in P'$ but this is (\ref{coro:shimura-curves-minimal-model}).

(ii) Let $(E/S, \rho)$ and $(E'/S, \rho')$ be a pair of CM elliptic curves with level-$W$ structures. Then $E'\isomto E\otimes_{O_K}\mc{L}$ for some $\mc{L}\in \mc{CL}_{O_K}(S)$ and therefore \[W_\CM^*(E)\otimes_{O_K} W^*(\mc{L})\isomto W^*_\CM(E\otimes_{O_K}\mc{L})=W^*_\CM(E').\] We then have \begin{eqnarray*}\underline{\Lie}_{W^*_\CM(E')/W^*(S)}&=&\underline{\Lie}_{W^*_\CM(E\otimes_{O_K} \mc{L})/W^*(S)}\\&=&\underline{\Lie}_{W^*_\CM(E)\otimes_{O_K} W^*(\mc{L})/W^*(S)}\\&=&\underline{\Lie}_{W^*_\CM(E)/W^*(S)}\otimes_{O_K} W^*(\mc{L})\end{eqnarray*} so that the isomorphism \[\rho': \underline{\Lie}_{W^*_\CM(E')/W^*(S)}=\underline{\Lie}_{W^*_\CM(E)/W^*(S)}\otimes_{O_K} W^*(\mc{L})\isomto \mc{O}_{W^*(S)}\] must be of the form $\rho\otimes_{\mc{O}_{W^*(S)}}\lambda$ where $\lambda: W^*(\mc{L})\otimes_{O_K} \mc{O}_{W^*(S)}\to \mc{O}_{W^*(S)}$ is a level-$W$ structure on $\mc{L}$. Therefore \[(E'/S, \rho')=(\mc{L}/S, \lambda) \cdot (E/S, \rho)=(E\otimes_{O_K}\mc{L}, \rho\otimes_{\mc{O}_{W^*(S)}}\lambda)\] and the action of $CL_{O_K}^W$ on $M_\CM^W$ is transitive.

Let us now see that it is free which is equivalent to the claim that if $(E/S, \rho)$ and $(\mc{L}/S, \lambda)$ are a CM elliptic curve and rank one $O_K$-local system with level-$W$ structure then there exists an $W$-isomorphism \[f: (E/S, \rho)\isomto (E\otimes_{O_K}\mc{L}, \rho\otimes \lambda)\] only if there exists a $W$-isomorphism $(\underline{O_K}_S, 1)\isomto (\mc{L}, \lambda)$.

So let $f$ be such an isomorphism. Then $f: E\isomto E\otimes_{O_K}\mc{L}$ is of the form $\id_{E}\otimes_{O_K} h$ for a unique isomorphism $ h : \underline{O_K}_S\isomto \mc{L}$ (\ref{theo:mcm-is-a-torsor}) and as $f$ is compatible with the level-$W$ structures $\rho$ and $\lambda\otimes \rho$, the isomorphism $h$ defines a $W$-isomorphism \[(\underline{O_K}_S, \id_{\mc{O}_{W^*(S)}})\isomto (\mc{L}, \lambda)\] and our claim follows. 

Therefore, the action of $CL_{O_K}^W$ on $M_\CM^W$ is free and transitive so that to show $M_\CM^W$ is a torsor it is enough to show that the structure map $M_\CM^W\to \Spec(O_K)$ is an epimorphism, but this follows from the work done in (ii) showing that any CM elliptic curve $E/S$ \'etale locally admits a level-$W$ structure.

(iii) It is clear that any pair of level-$W$ structures on a CM elliptic curve $E/S$ differ by scaling by an element of $W_*(\mathbf{G}_\mathrm{m})(S)$ which is the claim.
\end{proof}

\subsection{} We now equip $M_{\CM}^W$ with a $\Lambda$-structure. As $M_\CM^W$ is flat and affine to give $M_\CM^{W}$ a $\Lambda$-structure it is enough to define a commuting family of Frobenius lifts $\psi_{M_\CM^W}^\p$ for each prime ideal $\p$. We now set $\psi_{M_\CM^W}^\p$ to be the map defined on $S$-sections \[(E/S, \rho)\mto (E\otimes_{O_K}\p^{-1}/S, \psi^{\p*}(\rho))\] where $\psi^{\p*}(\rho)$ can be viewed as a level-$W$ structure on $E\otimes_{O_K}\p^{-1}$ via the identifications \[\psi^{\p*}(W_\CM^*(E))\isomto W_\CM^*(E)\otimes_{O_K}\p^{-1}\isomto W_\CM^*(E\otimes_{O_K}\p^{-1}).\] It is clear that these maps commute and that they lift the $N\p$-power Frobenius endomorphisms modulo $\p$. Finally, this equips the stack $\mc{M}_\CM$ with a flat affine presentation \[\xymatrix{CL_{O_K}^W\times M_\CM^W \ar@<.5ex>[r]\ar@<-.5ex>[r] & M_\CM^W\ar[r] & \mc{M}_\CM}\] and where the two parallel arrows are morphisms of $\Lambda$-schemes, again expressing the `fact' that $\mc{M}_\CM$ is a $\Lambda$-stack.

\section{Perfect $\Lambda$-schemes and Tate modules} In this final section we exhibit a rather interesting relationship the Tate module of a canonical CM elliptic curve over an $\Lambda$-ind-affine-scheme and a certain deformation of it to the perfection of $S$, which is the universal $\Lambda$-ind-affine-scheme under $S$ on which the Frobenius lifts are isomorphisms. We then show that the canonical lift $E/S$ can be deformed to a canonical CM elliptic curve $E_\per/S^\per$. We end the section by making some remarks about the relationship of this exact sequence with periods, both $p$-adic and analytic.

\subsection{} We first define the Tate module of an arbitrary CM elliptic curve. So let $S$ be an ind-affine scheme $E/S$ a CM elliptic curve. The Tate module of $E/S$ is defined to be the pro-finite locally free $S$-group scheme \[T(E):=\lim_\a E[\a]\otimes_{O_K}\a\] where the transition maps are induced multiplication \[E[\a\b]\otimes_{O_K}\a\b\to E[\a]\otimes_{O_K}\a.\] We also define the universal cover of $E/S$ to be the sheaf \[\widetilde{E}:=\lim_\a E\otimes_{O_K}\a\] where again the transition maps are induced by multiplication. The inclusion $T(E)\to \widetilde{E}$ identifies $T(E)$ with the kernel of the projection $\widetilde{E}\to E$ so that we have an exact sequence of sheaves (for the fpqc topology) \begin{equation}\label{eqn:univ-cover-tate}0\to T(E)\to \widetilde{E}\to E\to 0.\end{equation}

\subsection{} Now let $S$ be a $\Lambda$-ind-affine scheme. The perfection $S^\per$ of $S$ is the $\Lambda$-sheaf defined by \[S^\per=\colim_{\a\in\Id_{O_K}, \psi^\a} S \label{eqn:colim-def-perfect}\] where the transition maps are the Frobenius lifts. The action of the monoid of ideals $\Id_{O_K}$ on $S^\per$ is now by automorphisms so that it extends to an action of the group of fractional ideals $\Id_K$. For $\a\in \Id_K$ (now a fractional ideal) we write $\psi^{\a}_{S_\per}: S^\per\to S^\per$ for the corresponding automorphism. It follows immediately from the definition that $S^\per$ is universal among $\Lambda$-sheaves lying under $S$ on which the Frobenius lifts are isomorphisms. When $S=W^*(T)$ is the Witt vectors of an ind-affine scheme we write \[\widehat{W}^*(T)=W^*(T)^\per.\]

Viewing $S$ as a $\Lambda_{S^\per}$-ind-affine scheme via the element of the colimit (\ref{eqn:colim-def-perfect}) corresponding to $O_K\in \Id_{O_K}$ we see that the structure map of the element of the colimit (\ref{eqn:colim-def-perfect}) corresponding $\a$ is \[S\to S^\per \stackrel{\psi^{\a^{-1}}_{S^\per}}{\to}S^\per\] and the colimit (\ref{eqn:colim-def-perfect}) can be rewritten as the colimit of $S^\per$-sheaves \begin{equation}\label{w-perf-int-colim}S^\per=\colim_{\a\in \Id_{O_K}, \psi_\a}\psi^{\a^{-1}}_{S^\per!}(S).\end{equation} 

\subsection{} If $E/S$ is a CM elliptic curve equipped with a canonical $\Lambda$-structure then we have a natural identification $E\otimes_{O_K}\a^{-1}\isomto \psi^{\a*}_S(E)$. Hence for $\a, \b\in \Id_{O_K}$ we have isomorphisms \[E\otimes \b\isomto \psi^{\a*}_S(E\otimes \a\b)\] and this gives for all $\a$ and $\b$ a Cartesian diagram \begin{equation}\label{dia:cartesian-perfect}\begin{gathered}\xymatrix{E\otimes_{O_K}\b\ar[r]\ar[d] & E\otimes_{O_K}\a\b\ar[d]\\
\psi_{S^\per!}^{\b^{-1}}(S)\ar[r]^{\psi^\a} & \psi_{S^\per !}^{(\a\b)^{-1}}(S).}\end{gathered}\end{equation} We then define a $\Lambda$-sheaf over $S^\per$ by \[E_\per=\colim_{\a\in \Id_{O_K}} E\otimes_{O_K} \a\to S^\per=\colim_{\a\in \Id_{O_K}}\psi^{\a^{-1}}_{S^\per!}(S)=S^\per.\]
\begin{prop} $E_\per$ is a \textup{CM} elliptic curve over $S^\per$ equipped with a canonical $\Lambda_{S^\per}$-structure.
\end{prop}
\begin{proof} First, $E_\per$ admits a $\Lambda_{S^\per}$-structure as it is a colimit of $\Lambda_{S^\per}$-sheaves. Secondly, as the colimit is filtered and the diagrams (\ref{dia:cartesian-perfect}) are all cartesian it follows that for all $\a\in \Id_{O_K}$ we have $\Lambda_{S^\per}$-isomorphisms \[E_\per\times_{S^\per}\psi_{S^\per!}^{\a^{-1}}(S)\isomto E\otimes_{O_K}\a.\] This shows that $E_\per$ is a CM elliptic curve. It now also follows from (iii) of (\ref{prop:canonical-lambda-p-structures}) that its $\Lambda_{S^\per}$-structure is canonical, as the $\Lambda_{S}$-structures of $E\otimes_{O_K}\a$ are canonical and $(\psi_{S^\per!}^{\a^{-1}}(S)\to S_\per)_{\a\in \Id_{O_K}}$ is a cover.
\end{proof}

\subsection{} We now relate $E_\per/S^\per$ to $\widetilde{E}/S$. This is nothing more than an application of certain adjunctions and the definitions. Indeed, for each ind-affine scheme $T\to S$, the morphism $T\to S^\per$ induces a morphism $\widehat{W}^*(T)\to S_\per$ and viewing $\widehat{W}^*(T)$ as a $\Lambda_{S^\per}$-ind-affine scheme we have \begin{eqnarray*}\Hom_{S^\per}^{\Lambda}(\widehat{W}^*(T), E_\per)&\stackrel{}{=}&\lim_{\a}\Hom_{S^\per}^{\Lambda}(\psi_{S^\per!}^{\a^{-1}}(W^*(T)), E_\per)\\
&\stackrel{}{=}& \lim_\a \Hom_{S^\per}^\Lambda(W^*(T), \psi^{\a^{-1}*}_{S^\per}(E_\per))\\
&=& \lim_\a \Hom_{S^\per}^\Lambda(W^*(T), E_\per\otimes_{O_K}\a))\\
&\stackrel{}{=}& \lim_\a \Hom_{S}^\Lambda(W^*(T), E\otimes_{O_K}\a)\\
&\stackrel{}{=}& \lim_\a \Hom_S(T, E\otimes_{O_K}\a)\\
&=& \Hom_S(T, \widetilde{E}).\end{eqnarray*} Therefore, we have a natural isomorphism of functors on ind-affine $S$-schemes \[\Hom_{S^\per}^{\Lambda}(\widehat{W}^*(-), E_\per)\isomto \Hom_S(-, \widetilde{E}).\] If we denote the left hand side by \[\widehat{W}_*(E_\per)_\Lambda: T/S\mto \Hom_{S^\per}^\Lambda(\widehat{W}^*(T), E_\per)\] then we may then rewrite the exact sequence of (\ref{eqn:univ-cover-tate}) to obtain the following:

\begin{theo}\label{coro:tate-teichmuller} The is an exact sequence of \textup{fpqc} sheaves \[0\to T(E)\to \widehat{W}_*(E_\per)_\Lambda\to E\to 0.\]
\end{theo}

\subsection{} If now $S$ is an ind-affine scheme (not a $\Lambda$-scheme) and $E/S$ is any CM elliptic curve then we may perform the above construction for $W_\CM^*(E)/W^*(S)$ to obtain a CM elliptic curve \[\widehat{W}_\CM^*(E):=W_\CM^*(E)_\per\to \widehat{W}^*(S)=W^*(S)_\per\] and an exact sequence of sheaves for the fpqc topology over $W^*(S)$: \begin{equation}\label{eqn:per-can-lift} 0\to T(W_\CM^*(E))\to \widehat{W}_*(\widehat{W}^*_\CM(E))_\Lambda\to \widehat{W}^*_\CM(E)\to 0.\end{equation}

If we pull-back this exact sequence along the first ghost component $g_{(1)}: S\to W^*(S)$, and just write \[\widehat{W}_*(\widehat{W}^*_\CM(E))_\Lambda|_{S}=g_{(1)}^*(\widehat{W}_*(\widehat{W}^*_\CM(E))_\Lambda)\] then (\ref{eqn:per-can-lift}) becomes the exact sequence over $S$: \[0\to T(E)\to \widehat{W}_*(\widehat{W}^*_\CM(E))_\Lambda|_{S}\to E\to 0\] (where we use the fact that $g_{(1)}^*(W_\CM^*(E)))=E$).

\begin{coro}\label{coro:tate-teichmuller-two} For any ind-affine scheme $S$ and any \textup{CM} elliptic curve $E/S$ there exists an exact sequence of \textup{fpqc} sheaves over $S$ \[0\to T(E)\to \widehat{W}_*(\widehat{W}^*_\CM(E))_\Lambda|_{S}\to E\to 0.\]
\end{coro}

\begin{rema} Although we have not discussed it, there is a completely analogous theory of canonical lifts for Lubin--Tate $O$-modules for which the above constructions can also be made. For $O=\Z_p$ and $F=\mu_{p^\infty}$, the analogue of the exact sequence (\ref{coro:tate-teichmuller-two}) evaluated on $\Spf(\overline{\Z}_p)$ gives \[0\to T(\mu_{p^\infty})(\overline{\Z}_p)\to \mu_{p^\infty}(A_{\mathrm{inf}})^{\varphi_p=p}\to \mu_{p^\infty}(\overline{\Z}_p)\to 0\] where $A^{\mathrm{inf}}$ with its Frobenius lift $\varphi_p$ is Fontaine's ring, and we have used the (non-obvious) fact that \[\Spf(A^{\mathrm{inf}})\isomto \widehat{W}^*(\Spf(\overline{\Z}_p)).\] The image in $\mu_{p^\infty}(A^{\mathrm{inf}})$ of a generator $\epsilon\in T(\mu_{p^\infty})(\overline{\Z}_p)$ is Fontaine's element $[\epsilon]$, the logarithm of which is the $p$-adic period $t$.
\end{rema}

\begin{rema} While there is (as yet) no theory of analytic $\Lambda$-structures one can fudge a theory of analytic canonical lifts. Here let us  sketch the construction of an analytic analogue of the short exact sequence (\ref{coro:tate-teichmuller-two}) relating the period lattice of $E^\an/S^\an$ to a certain analytic canonical lift of $E/S$.

So let $S\to \Spec(\C)$ be a complex scheme of finite type. We note that as $K\subset \C$, we have \[W^*(S)=\Gamma^*(S)=\coprod_{\a\in \Id_{O_K}} S =\underline{\Id_{O_K}}\times S\] and \[\widehat{W}^*(S) = \underline{\Id_{K}}\times S.\]

There is a natural analytic analogue of $\Id_K$ which we call $\Id_\C$ and is given by $(\Id_K\times \C^\times)/K^\times$ where $a\in K^\times$ acts on $\Id_K\times \C$ by $(\a, s)\mto ((a)\a, a s)$. Note there is a short exact sequence \[0\to O_K^\times\to \C^\times\to \Id_\C\to \CL_{O_K}\to 0.\]

It is natural, in light of the above, to define the analytic Witt vectors of the analytification $S^\an$ of $S$ to be the analytic space \[W^{\an*}(S):=\Id_\C\times S^\an\] together with its action of $\Id_\C$.

We can now analytically mimic our construction of $W_\CM^*$ over $W^{\an*}(S)$. So define the rank one $O_K$-local system $\mc{L}_\an$ to have fibre over $(\a, s)\times S^\an\subset S^\an\times \Id_\C=W^{\an*}(S)$ the constant $O_K$-local system associated to the rank one $O_K$-module $s\cdot \a^{-1}\subset s\cdot K \subset \C$ (note this depends only on the class of $(\a, s)\in (\Id_K\times\C^\times)/K^\times = \Id_\C$). If $E/S$ is a CM elliptic curve then we define $W^{\an*}_\CM(E)$ to be \[p_S^*(E^\an)\otimes_{O_K}\mc{L}_\an\] where $p_S: W^{\an*}(S)=\Id_{\C}\times S^\an\to S^\an$ is the projection (cf.\ (\ref{subsec:def-can-ghost})).

The CM elliptic curve $W^{\an*}_\CM(E)\to W^{\an*}(S)$ inherits a natural action of $\Id_{O_K}$ (not $\Id_\C$!) which is compatible with that on $W^{\an*}(S)$. Finally, setting $W^{\an*}(X)=\Id_\C\times X$ with its $\Id_{\C}$ action for any analytic space, we define a sheaf on the big analytic site of $S$ by \[W^\an_*(W_\CM^{\an*}(E))_\Lambda|_S: X/S\mto \Hom_{W^{\an*}(S)}^{\Id_{O_K}}(W^{\an*}(X), W^{\an*}_\CM(E)).\] Let us spell out here that the right hand side here denotes the $\Id_{O_K}$-equivariant \textit{analytic} $W^{\an*}(S)$-maps \[W^{\an*}(X)\to W^{\an*}_\CM(E).\] With this definition one can show (almost as formally as in the algebraic situation) that there exists an isomorphism of sheaves on the big analytic site of $S^\an$ \[W^\an_*(W_\CM^{\an*}(E))_\Lambda|_S\isomto \underline{\Lie}_{E^\an/S^\an}\] and so the exponential sequence \[0\to T_{O_K}(E^\an)\to \underline{\Lie}_{E^\an/S^\an}\to E^\an\to 0\] of $E^\an/S^\an$ can be rewritten as \[0\to T_{O_K}(E^\an)\to W^\an_*(W_\CM^{\an*}(E))_\Lambda|_S\to E^\an\to 0.\]
\end{rema}

\appendix

\chapter{Odds and ends}

\section{Formal groups} The purpose of this section is to give an intrinsic definition of a (smooth) formal group. We do this using Messing's definition of infinitesimal neighbourhoods given in \cite{Messing72}, and the general theory of tangent spaces given in \cite{SGA3}.

\subsection{}\label{subsec:infinitesimal} First let us recall the construction of infinitesimal neighbourhoods from Chapter II of \cite{Messing72} and some its properties. So fix a monomorphism of sheaves $Z\to X$. The $k$th infinitesimal neighbourhood of $Z$ in $X$, denoted $\Inf_Z^{(k)}(X)$, is the sub-sheaf of $X$ defined by the property that an affine scheme $T$ mapping to $X$ factors through $\Inf_Z^{(k)}(X)\to X$ if and only if there exists an fpqc cover $(T_i\to T)_{i\in I}$ and closed sub-schemes $(\overline{T}_i\to T_i)_{i\in I}$ defined by ideals whose $(k+1)$st power is $(0)$, such that the composition $\overline{T}_i\to T\to X$ factors through $Z\to X$. Diagrammatically, we have \[\xymatrix{Z\ar@{^(->}[r] & \Inf_Z^{(k)}(X)\ar@{^(->}[r] & X\\
\overline{T}_i\ar[u]\ar@{^(->}[r] & T_i\ar[r]\ar[u] & T.\ar[u]\ar@{.>}[ul]}\] If $Z\to X$ is a closed immersion of \textit{schemes} defined by a quasi-coherent ideal $\mc{I}\subset \mc{O}_X$ then $\Inf_{Z}^{(k)}(X)\to X$ is the closed sub-scheme defined by the ideal $\mc{I}^{(k+1)}\subset \mc{O}_X$.

These constructions satisfy the following:

\begin{enumerate}[label=\textup{(\alph*)}]
\item For $k\leq k'$ there is an inclusion $\Inf_Z^{(k)}(X)\subset \Inf_Z^{k+1}(X)$. We write $\Inf_Z(X)=\colim_{k} \Inf_Z^{(k)}(X)$ and call this completion of $X$ along $Z$.
\item If $X'\to X$ is any morphism then we have \[\Inf_Z^{(k)}(X)\times_X X'=\Inf_{Z\times_X X'}^{(k)}(X').\]
\item If $Z\to X$ is an monomorphism of $S$-sheaves for some sheaf $S$ and $Y$ is another $S$-sheaf equipped with an $S$-monomorphism $Z\to Y$ then, as sub-sheaves of $X\times_S Y$, there are inclusions: \[\Inf_{Z}^{(k)}(Y\times_S X)\subset \Inf_{Z}^{(k)}(X)\times_S \Inf_{Z}^{(k)}(Y)\subset \Inf_{Z}^{(2k)}(X\times_S Y)\] and taking colimits we have \[\Inf_Z(X\times_S Y)=\Inf_Z(Y)\times_S \Inf_Z(X).\]
\end{enumerate}

\subsection{}\label{subsec:formal-completion} Write $\Sh_{S}^\bullet$ for the category of $S$-pointed $S$-sheaves. If $X$ is a pointed $S$-sheaf we write $\widehat{X}=\colim_k \Inf_S^{(k)}(X)$ for the formal neighbourhood of the point $S\to X$. The functor $X\mto \widehat{X}$ preserves finite products of pointed $S$-sheaves so that if $G$ is an $S$-group, viewed as a pointed $S$-sheaf via the identity $S\to G$, then $\widehat{G}$ is again a sheaf of groups over $S$ which we call the formal group of $G$.

\subsection{} Here we recall part of the rather general construction of Lie algebras given in Expos\'e II of \cite{SGA3}. Let $S$ be a sheaf and $\mc{V}$ a vector bundle. Make $\mc{O}_S\oplus\mc{V}$ a sheaf of quasi-coherent $\mc{O}_S$-algebras by declaring that $\mc{V}$ be a square zero ideal, and write $D_S(\mc{V})$ for the $S$-sheaf whose sections over an ind-affine scheme $T\to S$ are given \[D_S(\mc{V})(T)=\Hom_{\mc{O}_{T}}(\mc{O}_{T}\oplus \mc{V}_T, \mc{O}_{T}).\] This defines a contravariant functor $\mathrm{QCoh}(\mc{O}_S)\to \Sh_{S}^\bullet$. If $\mc{V}_1$ and $\mc{V}_2$ are two quasi-coherent $\mc{O}_S$-modules then the two projections $\mc{V}_1\oplus \mc{V}_2\to \mc{V}_i$ for $i=1, 2$ induce for each pointed $S$-sheaf $X$ a morphism \begin{equation}\underline{\Hom}^\bullet_S(D_S(\mc{V}_1\oplus \mc{V}_2), X)\to \underline{\Hom}_S^\bullet(D_S(\mc{V}_1), X)\times_S \underline{\Hom}_S^\bullet(D_S(\mc{V}_2), X)\label{eqn:condition-e}\end{equation} and a pointed $S$-sheaf $X$ is said to satisfy condition (E) if for all vector bundles $\mc{V}_1$, $\mc{V}_2$ the morphism (\ref{eqn:condition-e}) is an isomorphism.

\begin{prop} With notation as above:
\begin{enumerate}[label=\textup{(\roman*)}]
\item if $X$ satisfies condition $(E)$ over $S$ and $S'\to S$ is a morphism then $X_{S'}$ satisfies condition $(E)$ over $S'$,
\item $X$ satisfies condition $E$ for $S$ if and only if there is a cover $(S_i\to S)_{i\in I}$ such that for each $i\in I$ the sheaf $X_{S_i}$ satisfies condition $(E)$ over $S_i$,
\item if $S$ is a scheme and $X$ is an $S$-pointed ind-scheme then $X$ satisfies condition $(E)$.
\end{enumerate}
\end{prop}
\begin{proof} (i) and (ii) follow immediately from the definition. For (iii) we may assume that $X$ is a scheme as filtered colimits preserve fibre products. In which case the claim follows from the fact that \[\underline{\Hom}_S^\bullet(D_S(\mc{M}), X)=\mc{M}\otimes_{\mc{O}_S}\underline{\Hom}_{\mc{O}_S}(\Omega_{X/S, \bullet}, \mc{O}_S).\]
\end{proof}
\subsection{} If $X$ is a pointed $S$ sheaf satisfying condition (E) over $S$ then, writing $D^{(n)}_S=D_S(\mc{O}^n)$, the inverse of the isomorphism \[\underline{\Hom}^\bullet_S(D^{(2)}_S, X)\to \underline{\Hom}_S^\bullet(D_S^{(1)}, X)\times_S \underline{\Hom}_S^\bullet(D_S^{(1)}, X)\] composed with the map \[\underline{\Hom}^\bullet_S(D^{(2)}_S, X)\to \underline{\Hom}^\bullet_S(D^{(1)}_S, X)\] induced by the sum $\mc{O}_S^2\to \mc{O}_S$ defines a map \[\underline{\Hom}^\bullet_S(D^{(1)}_S, X)\times \underline{\Hom}^\bullet_S(D^{(1)}_S, X)\to \underline{\Hom}^\bullet_S(D^{(1)}_S, X).\] This equips \[\underline{\Lie}_{X/S}:=\underline{\Hom}^\bullet_S(D_S^{(1)}, X)\] with the structure of an abelian group over $S$. Moreover, the $S$-pointed sheaf $D_S^{(1)}$ admits an action of the sheaf of monoids $\mc{O}_S$ induced by the action of $\mc{O}_S$ on itself and this equips $\underline{\Lie}_{X/S}$ with the structure of an $\mc{O}_S$-module. We call $\underline{\Lie}_{X/S}$ the Lie algebra of the pointed $S$-sheaf $X$.

\subsection{}\label{def:formal-var} Let $S$ be a sheaf and let $X$ be a pointed $S$-sheaf. We say that $X$ is a formal variety over $S$ if the following conditions hold:
\begin{enumerate}[label=\textup{(\roman*)}]
\item the inclusion $\widehat{X}\to X$ is an isomorphism,
\item for each $k\geq 0$ the morphism $\Inf_S^{(k)}(X)\to S$ affine,
\item $X$ is formally smooth and locally finitely presented over $S$, and
\item $\underline{\Lie}_{X/S}$ is a vector bundle.\footnote{(i) and (ii) combined show that $X$ satisfies condition (E) over $S$ so that this makes sense.}
\end{enumerate}
\begin{prop}\label{prop:formal-explicit-intrinsic} Let $S$ be a sheaf and $X$ a pointed $S$-sheaf. Then the following are equivalent:
\begin{enumerate}[label=\textup{(\roman*)}]
\item $X$ is a formal variety over $S$,
\item there exists a cover $(S_i\to S)_{i\in I}$, integers $n_i\geq 0$ for $i\in I$ and pointed $S_i$-isomorphisms \[\widehat{\mathbf{A}}^{n_i}_{S_i}=\Inf_{S_i}(\mathbf{A}^{n_i}_{S_i})\isomto X_{S_i}.\]
\end{enumerate}
\end{prop}
\begin{proof} This follows from Proposition 3.1.1 of \cite{Messing72}.
\end{proof}
\subsection{}\label{subsec:dimension-formal-groups} The dimension $\dim_S(X)$ of a formal variety $X/S$ is defined to be the map $S\to \underline{\mathbf{N}}_S$ giving the rank of the locally free $\mc{O}_S$-module $\underline{\Lie}_{X/S}$.

If $S$ is a scheme and $X/S$ is a smooth scheme over $S$ equipped with a closed $S$-point $S\to X$ then $\widehat{X}$ is a formal variety over $S$. A formal group $F$ over $S$ is a formal variety $F/S$ which is also a sheaf of groups over $S$ (with identity the given point). As the functor $X\mto \widehat{X}$ commutes with products, it follows that if $S$ is a scheme and $G/S$ is a smooth separated group scheme over $S$ then $\widehat{G}/S$ is a formal group over $S$.

\section{Serre's tensor product} Here we give a (very broad) generalisation of a basic construction due to Serre (Chapter XIII \S 2 of \cite{CasselsFrohlich67}). The idea that this construction of Serre could and should be generalised, at least in the study of CM elliptic curves, is not an idea wholly original to the author (see \S 1.7.4 of \cite{CCO2014}) however it was arrived at independently.

\subsection{}\label{subsec:sheaves-of-rings} Let $S$ be a sheaf and let $\mc{A}$ be a sheaf of rings over $S$. We say that an $\mc{A}$-module $\mc{V}$ satisfies condition (P) if, locally on $S$, $\mc{V}$ is a direct factor of a free $\mc{A}$-module of finite rank. The class of $\mc{A}$-modules satisfying condition (P) is clearly closed under the operations of taking $\mc{A}$-linear Hom, tensor product, direct factors and direct sums.

\begin{prop}\label{prop:tensor-exact} Let $\mc{V}$ satisfy condition $(P)$. Then the functor \[\Mod(\mc{A})\to \Mod(\mc{A}): G\mto G\otimes_{\mc{A}}\mc{V}\] is exact.
\end{prop}
\begin{proof} We need only show that this functor preserves monomorphisms. This is local on $S$ and so we may assume that $\mc{V}$ is a direct factor of a free $\mc{A}$-module in which case it is clear as $G\otimes_{\mc{A}}\mc{V}\subset G\otimes_{\mc{A}}\mc{A}^n=G^n$ for some $n\geq 0$ and $G\mto G^n$ is exact.
\end{proof}
\begin{prop}\label{prop:tensor-properties-lie-hom} We have the following
\begin{enumerate}[label=\textup{(\roman*)}]
\item\label{item:hom-tensor-a-modules} For each pair of $\mc{A}$-modules $\mc{V}, \mc{W}$ satisfying condition $(P)$ and each pair of $\mc{A}$-modules $G, F\in \Mod(\mc{A})$ the morphism \[\underline{\Hom}_S^{\mc{A}}(F, G)\otimes_{\mc{A}}\underline{\Hom}_{S}^{\mc{A}}(\mc{V}, \mc{W})\to \underline{\Hom}_S^{\mc{A}}(F\otimes_{\mc{A}}\mc{V}, G\otimes_{\mc{A}}\mc{W})\] is an isomorphism.
\item If $G$ is an $\mc{A}$-module satisfying condition $(E)$ and $\mc{V}$ is an $\mc{A}$-module satisfying condition $(P)$ then the natural morphism $\underline{\Lie}_{G/S}\otimes_{\mc{A}}\mc{V}\to \underline{\Lie}_{G\otimes_{\mc{A}}\mc{V}/S}$ is an isomorphism.
\end{enumerate}
\end{prop}
\begin{proof} (i) The map defined is functorial and so we may, by adjunction, assume that $\mc{V}=\mc{A}$ so that we are reduced to showing that the map \[i_{\mc{W}}:\underline{\Hom}_S^{\mc{A}}(F, G)\otimes_{\mc{A}}\mc{W}\to \underline{\Hom}_S^{\mc{A}}(F, G\otimes_{\mc{A}}\mc{W})\] is an isomorphism. The claim is clearly local on $S$ so we may assume that $\mc{A}^n=\mc{W}\oplus \mc{W}'$. We then have $i_{\mc{A}^n}=i_{\mc{W}}\oplus i_{\mc{W}'}$ so that it is enough to show the claim for $\mc{A}^n$ which is clear.

(ii) This is proved in much the same way as (i).
\end{proof}
\subsection{}\label{subsec:strict-formal-modules} A formal $\mc{A}$-module over $S$ is a formal group $F/S$ equipped with an action of $\mc{A}$. If we are given a homomorphism $\mc{A}\to \mc{O}_S$ then we say the action is strict, or that $F$ is a strict formal $\mc{A}$-module, if the two actions of $\mc{A}$ on the $\mc{O}_S$-module $\underline{\Lie}_{F/S}$ coming from the action of $\mc{A}$ on $F$ and the homomorphism $\mc{A}\to \mc{O}_S$ coincide. In this case, if $\mc{V}$ is an $\mc{A}$-module satisfying condition (P) then $\mc{V}\otimes_{\mc{A}}\mc{O}_S$ is a locally free $\mc{O}_S$-module and we write $\rk(\mc{V}): S\to \underline{\N}_S$ for the rank of this $\mc{O}_S$-module.
\begin{coro}\label{coro:strict-tensor} If $F$ is a strict formal $\mc{A}$-module over $S$ and $\mc{V}$ is an $\mc{A}$-module over $S$ satisfying condition $(P)$ then $F\otimes_{\mc{A}}\mc{V}$ is a strict formal $\mc{A}$-module. Moreover, we have $\dim(F\otimes_{\mc{A}}\mc{V})=\dim(F)\cdot \rk(\mc{V})$
\end{coro}
\begin{proof} As the claim is local on $S$ we may assume that $\mc{V}$ is the kernel of some idempotent endomorphism $\mc{A}^n\to \mc{A}^n$ for some $n$. The diagram of $S$-pointed sheaves \[\xymatrix{F^n\ar[r] & F^n\\
F\otimes_{\mc{A}}\mc{V}\ar[r]\ar[u] & S\ar[u]}\] remains cartesian after applying $\Inf_S^{(k)}$ for each $k\geq 1$. This shows that $F\otimes_{\mc{A}}\mc{V}$ satisfies conditions (i) and (ii) of (\ref{def:formal-var}) and by (\ref{prop:tensor-properties}) we see that $F\otimes_{\mc{A}}\mc{V}$ also satisfies (iii). Therefore $F\otimes_{\mc{A}}\mc{V}$ satisfies condition (E) over $S$ and by (\ref{prop:tensor-properties-lie-hom}) it satisfies condition (iv) of (\ref{def:formal-var}) so that $F/S$ is a formal $\mc{A}$-module over $S$. That the action of $\mc{A}$ on $F\otimes_{\mc{A}}\mc{V}$ is strict and $\dim(F\otimes_{\mc{A}}\mc{V})=\dim(F)\cdot \rk(\mc{V})$ follows from (\ref{prop:tensor-properties-lie-hom}).
\end{proof}
\begin{prop}\label{prop:tensor-properties} Let $f: G\to F$ be a homomorphism of $\mc{A}$-modules and let $\mc{V}$ satisfy condition $(P)$. If $f$ satisfies one of the following properties then so does $f\otimes_{\mc{A}}\mc{V}: G\otimes_{\mc{A}}\mc{V}\to F\otimes_{\mc{A}}\mc{V}$:
\begin{enumerate}[label=\textup{(\roman*)}]
\item formally unramified, formally smooth or formally \'etale,
\item formally universally closed, formally separated, formally proper\footnote{Here we mean that $G/S$ satisfies the local existence, resp.\ uniqueness (resp.\ local existence and uniqueness) of the valuative criterion.}
\item locally finitely presented, quasi-compact or quasi-separated,
\item has connected geometric fibres,
\item affine, or affine and flat,
\item finite locally free.
\end{enumerate}
\end{prop}
\begin{proof} All properties of morphisms of sheaves descend under covers of $S$ so we may assume that $\mc{V}\oplus \mc{V}'=\mc{A}^n$ and hence \[(f\otimes_{\mc{A}}\mc{V})\oplus (f\otimes_{\mc{A}}\mc{V}')=f\otimes_{\mc{A}}\mc{A}^n=f^n.\] Moreover, these properties are all preserved by $f\mto f^n$ and base change. From the cartesian diagram \[\xymatrix{G^n\ar[rr]^{f^n} && F^n\\
(G\otimes_{\mc{A}} \mc{V})\oplus \ker(f\otimes_{\mc{A}}\mc{V}')\ar[rr]^-{(f\otimes_{\mc{A}}\mc{V})\oplus 0}\ar[u] && F\otimes_{\mc{A}} \mc{V}\ar[u]}\] we see that \[(f\otimes_{\mc{A}}\mc{V})\oplus 0=(f\otimes_{\mc{A}}\mc{V})\times_{S}\ker(f\otimes_{\mc{A}}\mc{V}')\] satisfies the given property. But $\ker(f\otimes_{\mc{A}}\mc{V}')\to S$ is an epimorphism hence, $f\otimes_{\mc{A}}\mc{V}$ satisfies the given property.
\end{proof}

\section{Strict finite $O$-modules} \label{appendix:faltings-strict-modules}
\subsection{} Here we give a short overview of faltings' generalisation of Cartier duality to strict finite $O$-modules \cite{Faltings02}. We will then use it to prove (\ref{prop:lubin-tate-faltings}) and (\ref{prop:lubin-tate-torsor}) as claimed (see (\ref{coro:lubin-tate-faltings})).

Let $O$ be a complete local Dedekind domain with maximal ideal $\p$, residue field $\F$ of cardinality $N\p$ and fix an affine scheme $S\to \Spf(O)$. In \cite{Faltings02} faltings' defines the notion of a strict finite $O$-module $G$ over $S$ and the notion of a strict homomorphism between strict finite $O$-modules. We will not recall the definition but only say the following. A strict finite $O$-module is a finite locally free group scheme $G$ over $S$, equipped with an action of $O$ satisfying a certain strictness condition. The strictness condition on the $O$-action means that for each $a\in O$, the endomorphism $a: G\to G$, can be lifted along a certain nilpotent thickening $G\to G^{\flat}$ in such a way that this lift acts by multiplication by $a$ on the fibre of the cotangent complex of $G/S$ at the origin. A homomorphism $f: G\to G'$ of strict finite $O$-modules is strict if it can be lifted to a map $G^\flat\to G'^\flat$ compatible with the lift of the $O$-action. We refer the reader to \S 2 of \cite{Faltings02} for the precise definitions.

In any case, one obtains the category of strict finite $O$-modules and it is a sub-category of the category of finite locally free groups schemes over $S$ equipped with an action of $O$. Moreover, if $S'\to S$ is a morphism and $G/S$ is a strict finite $O$-module so is $G\times_S S'$.

\begin{exem} Every finite locally free \'etale group scheme over $S$ equipped with an action of $O$ is strict and every $O$-linear homomorphism either to or from an \'etale strict finite $O$-module is strict. This is explained by the fact that the cotangent complex of such a scheme is trivial.

For each $\p$-adic affine scheme $S$ and each Lubin--Tate module $F\to \Spf(O)$, writing $F_{S}=F\times_{\Spf(O)}S$, the finite locally free group schemes $F_S[\p^{n}]$ equipped with their $O$-action are naturally strict finite $O$-modules over $S$.
\end{exem}

\subsection{} We now explain faltings' version of Cartier duality for strict finite $O$-modules. Let $F=F_\pi$ be the Lubin--Tate module over $\Spf(O)$ associated to the uniformiser $\pi\in O$, so that \[\pi: F\to F\] lifts the $N\p$-power Frobenius map over $\Spec(\F)\to \Spf(O)$. Given a strict finite $O$-module $G/S$ we define the sheaf of $O$-modules over $S$ \[D_\pi(G)=\colim_r \underline{\Hom}_S^{O, \str}(G, F_S[\p^r]).\] Faltings then proves the following (see Theorem 8 of \cite{Faltings02}):

\begin{theo} The functor \[G\mto D_\pi(G):=\colim_r \underline{\Hom}_S^{O, \str}(G, F_{\pi/S}[\p^r])\] defines a duality on the category strict finite $O$-modules over $S$ and is compatible with base change in $S$. Moreover, the degree of $D_\pi(G)$ is equal to the degree of $G$ and if $f: G\to G'$ is a strict homomorphism of strict finite $O$-modules over $S$ then $D_\pi(f)$ is a closed immersion (resp. faithfully flat) if and only if $f$ is faithfully flat (resp. a closed immersion).
\end{theo}

\subsection{} Given a strict finite $O$-module $G$ we call $D_\pi(G)$ the dual of $G$.

If $S$ has characteristic $\p$ then for each strict finite $O$-module $G$ the $N\p$-power relative Frobenius map \[\Fr_{G/S}^{N\p}: G\to \Fr^{N\p*}(G)\] is strict so that taking the dual of the $N\p$-power relative Frobenius map of the dual of $G$ one obtains a map \[V_{G/S}: \Fr^{N\p*}(G)\to G\] and Faltings shows that the composition \[G\stackrel{\Fr_{G/S}^{N\p}}{\longrightarrow}\Fr^{N\p*}(G)\stackrel{V_{G/S}}{\longrightarrow} G\] is equal to the endomorphism $\pi: G\to G$ (see the paragraph of \S 7 \cite{Faltings02}).

It is a formality to extend the duality $G\mto D_\pi(G)$ to a pair of functors defining inverse anti-equivalences between the categories of ind-strict finite $O$-modules and pro-strict finite $O$-modules.

The final observation we need to make is that if $F/S$ is a Lubin--Tate $O$-module then the inclusions \[F[\p^n]\to F[\p^{n+1}]\] are strict so that we may view $F=\colim_n F[\p^n]$ as an ind-strict $O$-module over $S$. Moreover, with this definition every homomorphism of Lubin--Tate $O$-modules over $S$ is a morphism of ind-strict finite $O$-modules (this is a consequence of the formal smoothness of $F$).

\begin{coro} The functor from ind-strict finite $O$-modules to pro-strict finite $O$-modules \[F/S\mto \lim_n D_\pi(F[\p^n])\] defines an anti-equivalence of categories between Lubin--Tate $O$-modules over $S$ and rank one $O$-local systems over $S$ with quasi-inverse \[\mc{L}\mto F_{\pi/S}\otimes_{O}\mc{L}^{\vee}.\]
\end{coro}
\begin{proof} It is enough to show that given an ind-strict finite $O$-module $G$, the pro-strict finite $O$-module $D_\pi(G)$ is a rank one $O$-local system if and only if $G$ is a Lubin--Tate $O$-module.
If $D_\pi(G)$ is a rank one $O$-local system, as the functor $D_\pi$ is compatible with base change and being a Lubin--Tate $O$-module is local on $S$ we may assume that $D_\pi(F)\isomto \widehat{O}_S$. We then get $G\isomto D_\pi(D_\pi(G))\isomto D_\pi(\widehat{O}_S)=F_{\pi/S}$ so that $G$ is a Lubin--Tate $O$-module.

Conversely, let $G$ be a Lubin--Tate $O$-module. We claim that $D_\pi(G[\p^n])$ is \'etale for all $n\geq 0$. As $D_\pi(G[\p^n])$ is finite locally free and $S$ is $\p$-adic, to show this we may assume that $S$ has characteristic $\p$. In this case, the composition \[G\stackrel{\Fr_{G/S}^{N\p}}{\to} \Fr^{N\p*}(G)=G\otimes_{O}\p^{-1}\stackrel{1\otimes \pi}{\longrightarrow} G\] is equal to $\pi$ from which it follows that $V_{G/S}$ is an isomorphism. This implies that $V_{G[\p^n]/S}$ is an isomorphism for all $n\geq 0$, so that $D(V_{G[\p^n]/S})=\Fr_{D_\pi(G[\p^n])/S}^{N\p}$ is an isomorphism. Therefore $D_\pi(G[\p^n])$ is \'etale and it follows that $D_\pi(G)=\lim_n D_\pi(G[\p^n])$ is a pro-finite \'etale strict $O$-module scheme.

The exact sequences \[0\to G[\p^n]\to G[\p^{n+1}]\stackrel{\pi}{\to} G[\p^{n+1}]\] now give exact sequences \begin{equation} \label{eqn:short-etale-faltings}D_{\pi}(G[\p^{n+1}])\stackrel{\pi}{\longrightarrow}D_\pi(G[\p^{n+1}])\to D_\pi(G[\p^{n}])\to 0.\end{equation} Localising, we may assume that $D_\pi(G[\p^n])$ is a strict finite constant $O$-module for all $n\geq 0$. Then the short exact sequences (\ref{eqn:short-etale-faltings}) combined with the fact that $\deg(D_{\pi}(G[\p^{n}]))=N\p^n$ for all $n\geq 0$, show inductively that there exists an isomorphism $D_{\pi}(F[\p^{n}])\isomto \underline{O/\p^n}_S$ such that the maps \[D_{\pi}(G[\p^{n+1}])\to D_{\pi}(G[\p^{n}])\] correspond to the reduction maps \[\underline{O/\p^{n+1}}_S\to \underline{O/\p^n}_S.\] Therefore, \[D_\pi(G)=\lim_n D_\pi(G[\p^n])\isomto \lim_n \underline{O/\p^n}=\widehat{O}_S\] is a rank one $O$-local system on $S$ and this shows that $D_\pi$ defines a contravariant equivalence between the category of Lubin-Tate $O$-modules and rank one $O$-local systems over $S$.

For the statement regarding the quasi-inverse, we have \begin{eqnarray*} D_\pi(\mc{L})&=&\colim_n \underline{\Hom}^{O, \str}_S(\mc{L}, F_S[\p^n])\\
&=&\colim_n \underline{\Hom}^{O}_S(\mc{L}, F_S[\p^n])\\
&=&\colim_n \underline{\Hom}_S(\widehat{O}_S, F_S[\p^n]\otimes_{O}\mc{L}^\vee)\\
& = & \colim_n F_S[\p^n]\otimes_{O}\mc{L}^\vee\\
&=&F_S\otimes_{O}\mc{L}^\vee.\end{eqnarray*}
\end{proof}

\begin{coro}\label{coro:lubin-tate-faltings} Let $S$ be a $\p$-adic sheaf.
\begin{enumerate}[label=\textup{(\roman*)}] 
\item If $F/S$ is a Lubin--Tate $O$-module the natural homomorphism \[\widehat{O}_S\to \underline{\End}_S^O(F)\] is an isomorphism.
\item If $F, F'/S$ are a pair of Lubin--Tate $O$-modules over $S$ then $\underline{\Hom}_S^{O}(F, F')$ is an $O$-local system over $S$ and the evaluation homomorphism \[F\otimes_{O}\underline{\Hom}_S^{O}(F, F')\to F'\] is an isomorphism.
\item The functor \[\mc{M}_{LT}\times \mc{CL}_{O}\to \mc{M}_{LT}\times \mc{M}_{LT}: (F,\mc{L})\mto (F, F\otimes_{O}\mc{L})\] is an equivalence of stacks.
\end{enumerate}
\end{coro}
\begin{proof} (i) For any rank one $O$-local system we have $\widehat{O}_S\isomto \underline{\End}_S^{O}(\mc{L})$ so that the composition \[\widehat{O}_S\to \underline{\End}_S^{O}(F)\isomto \underline{\End}^{O}_S(D_\pi(F))\] is an isomorphism and therefore $\widehat{O}_S\to \underline{\End}_S^{O}(F)$ is an isomorphism.

(ii) We may assume that $F'=F\otimes_{O}\mc{L}$. Then (i) combined with (\ref{prop:cm-tensor-homs-lt}) gives \[\mc{L}\isomto \underline{\Hom}_S(F, F\otimes_{O}\mc{L}).\] Moreover, using this identification the evaluation homomorphism \[F\otimes_{O}\mc{L}=F\otimes_{O}\underline{\Hom}_S^{O}(F, F\otimes_{O}\mc{L})\to F'=F\otimes_{O}\mc{L}\] becomes the identity.

(iii) The functor in question is the product of the equivalences $\id_{\mc{M}_\LT}$ and $\mc{L}\mto D_\pi(\mc{L}^\vee)$ and is therefore an equivalence.
\end{proof}
\section{A principal ideal theorem} In this section we would like to prove a strengthening of an old principal ideal theorem (see Tannaka \cite{Tannaka58}). We first state a special case for imaginary quadratic fields (\ref{prop:tannaka-result-special}) as it is possible to do so without having to make any new definitions and it is the only case we need in main the text. We shall then prove the general result (\ref{prop:tannaka-result-general}) for arbitrary number fields $K$ and explain how it strengthens the result in \cite{Tannaka58}.

The author would like to point out that while the result is new, our proof is really just a refinement of Tannaka's, essentially a combination of his proof, some very old results in class field theory, and a result of one of his contemporaries (see Terada \cite{Terada1952}). It is also interesting to point out that Tannaka was motivated to prove his result by Deuring who had conjectured it, presumably inspired during his work on CM elliptic curves.

\begin{prop}\label{prop:tannaka-result-special} Let $K$ be an imaginary quadratic field, let $K(\f)/K$ be the ray class field of conductor $\f$, let $\g$ be an ideal divisible by $\f$ and let \[l: \Prin^{(\g)}_{1\bmod \f}\to K^\times: \a\mto l(\a)\] be a homomorphism such that $l(\a)\cdot O_K=\a\subset K$ and such that $l(\a)=1\bmod \f$. Then $l$ can be extended to a map \[l: \Id^{(\g)}_{O_K}\to K(\f)^\times: \a \mto l(\a)\] such that $l(\a)\cdot O_{K(\f)}=\a\cdot O_{K(\f)}$ and such that for all $\a, \b\in \Id^{(\g)}_{O_K}$ we have \[l(\a\b)=l(\a)\sigma_{\a}(l(\b)).\]
\end{prop}

\subsection{} Let $K$ be a number field with ring of integers $O_K$. Recall by a modulus of $K$ is meant a finite formal sum \[\f=\sum_{v} \f_v v\] over the places $v$ of $K$ such that $\f_v\in \N$ for all $v$ and such that $\f_v\in \{0, 1\}$ for $v$ infinite and real and $\f_v=0$ for $v$ infinite and complex. If $\f$ and $\f'$ are moduli of $K$ then their product $\f\f'$ is defined by $(\f\f')_v=\f_v+\f_{v'}$ for all finite places $v$ and $(\f\f')_v=\max(\f_v, \f'_v)$ for all infinite primes $v$. We say $\f$ divides $\f'$ if $\f_v\leq \f'_v$ for all places $v$. If $\f_v=0$ for all infinite $v$ then we identify $\f$ with an ideal of $O_K$ in the usual way.

If $a\in K^\times$ we write $a=1\bmod \f$ to mean that $v(a-1)\geq \f_v$ for all finite places $v$ and such that $a>0$ for all infinite real places $v$ of $K$ with $\f_v\neq 0$. For a pair of fractional ideals $\a, \b$ of $K$ we write $\a=\b\bmod \f$ to mean that $\a\b^{-1}=(a)$ with $a=1\bmod \f$.

For each modulus $\f$ there is a certain extension $K(\f)/K$ called the ray class field of conductor $\f$ which is unramified at all finite places $v$ of $K$ with $\f_v=0$. This extension is characterised by the property that if $\g$ is any ideal of $K$ divisible by (the finite part of) $\f$ and $\Id_{K}^{(\g)}$ denotes the group of fractional ideals prime to $\g$ then the map \begin{equation}\Id_K^{(\g)}\to G(K(\f)/K): \a\mto \sigma_\a\label{eqn:prime-generators}\end{equation} is surjective and its kernel is equal to $\Prin_{1\bmod \f}^{(\g)}\subset \Id_K^{(\g)}$, the sub-group generated by ideals $\a$ prime to (the finite part of) $\g$ with $\a=O_K\bmod \f$.

We now recall certain moduli defined for a finite abelian extension of number fields $L/K$ (see \S 1 of \cite{Terada1952} for precise details). We write $\f_{L/K}$, $\D_{L/K}$ and $\G_{L/K}$ for conductor, different and genus ideal (`Geschlechtermodul') of the extension $L/K$ which are moduli for $K$, $L$ and $K$ respectively and we have \[\f_{L/K}=\D_{L/K}\G_{L/K}.\] The moduli $\f_{L/K}$, $\D_{L/K}$ are not necessarily ideals however $\G_{L/K}$ is always an ideal of $O_L$. Finally, if $L/K'/K$ is an intermediate extension we define $\f_{L/K'/K}=\D_{L/K'}\G_{L/K}$ which is an integral ideal of $K'$. We note that $\f_{L/K}$, $\G_{L/K}$, and $\f_{L/K'/K}$ are all invariant under $G(L/K)$. 

Finally, for what follows we will use exponential notation for the action of Galois groups on elements or fractional ideals of the respective fields.

\begin{theo}[Hasse's Norm Theorem] Let $L/K$ be a finite cyclic extension. Then $N_{L/K}(I_L)\cap K^\times=N_{L/K}(L^\times)$.
\end{theo}

\begin{theo}[Principal Genus Theorem] Let $L/K$ be a finite cyclic extension with generator $\sigma\in G(L/K)$ and let $\a$ be a fractional ideal of $L$. Then $N_{L/K}(\a)=O_K$ if and only if $\a=\b^{1-\sigma}$ for some fractional ideal $\b$ of $L$.
\end{theo}
\begin{proof} As $G(L/K)$ is cyclic and generated by $\sigma$ we have \[\frac{\{\a\in \Id_{L}: \a\a^{\sigma}\cdots \a^{\sigma^{n-1}}=O_L\}}{\Id_L^{1-\sigma}}\isomto H^1(G(L/K), \Id_{L}): \a \mto (\sigma^i\mto \a \a^{\sigma}\cdots\a^{\sigma^{i-1}})\] and by Proposition 6, \S 13, Chapter V of \cite{Lemmermeyer07} the group $H^1(G(L/K), \Id_{L})$ vanishes which is precisely the claim.
\end{proof}

\begin{theo}[Terada's Norm Theorem] Let $L/K$ be a finite cyclic extension with generator $\sigma\in G(L/K)$, let $a\in K$ and let $\m$ be an ideal of $O_L$. Then the following are equivalent:
\begin{enumerate}[label=\textup{(\roman*)}]
\item $N_{L/K}(a)=1\bmod \f_{L/K}\m$.
\item $a=b^{1-\sigma}\bmod \G_{L/K}\m$ for some $b\in L$.
\end{enumerate}
\end{theo}
\begin{proof} This is Theorem 2 of \cite{Terada1952}.
\end{proof}
\subsection{} Now let $L/K$ be the ray class field of conductor $\f_{L/K}$ and let $\p_1, \ldots, \p_r$ be prime ideals of $K$, unramified in $L/K$ and such that \[G(L/K)=\bigoplus_{i=1}^r \langle \sigma_i \rangle\] where $\sigma_i$ is the Frobenius element corresponding to $\p_i$ (we can do this by (\ref{eqn:prime-generators})). Also, for $1\leq i\leq n$ let $n_i$ be the order of $\sigma_i$ and let $K_i\subset L$ be the sub-extension fixed by the sub-group of $G(L/K)$ generated by $\{\sigma_j\}_{j\neq i}$. We note that $G(K_i/K)\isomto \langle \sigma_i \rangle$.

\begin{theo}[Tannaka] For $1\leq i\leq r$ let $\a_i$ be a fractional ideal of $K_i$ such that $\p_i=\a_i^{1-\sigma_i}\bmod \f_{L/K_i/K}$. Then $\a_1\cdots \a_r=O_L \bmod \G_{L/K}$.
\end{theo}
\begin{proof} This is Theorem 3 of \cite{Tannaka58}.
\end{proof}

\begin{theo}\label{prop:tannaka-result-general} Let $K$ be a number field, let $L/K$ be the ray class field of conductor $\f_{L/K}$, let $\g$ be an ideal divisible by (the finite part of) $\f_{L/K}$ and let \[l: \Prin^{(\g)}_{1\bmod \f_{L/K}}\to K^\times: \a\mto l(\a)\] be a homomorphism such that $l(\a)\cdot O_K=\a\subset K$ and such that $l(\a)=1\bmod \f_{L/K}$. Then $l$ can be extended to a map \[l: \Id^{(\g)}_{O_K}\to L^\times: \a \mto l(\a)\] such that:
\begin{enumerate}[label=\textup{(\roman*)}]
\item $l(\a)\cdot O_{L}=\a\cdot O_{L}$,
\item $l(\a)=1\bmod \G_{L/K}$, and
\item for all $\a, \b\in \Id^{(\g)}_{O_K}$ we have \[l(\a\b)=l(\a)\sigma_{\a}(l(\b)).\]
\end{enumerate}
\end{theo}
\begin{proof} For each $1\leq i\leq r$ let us make the following constructions. As $\p_i^{n_i}\in \Prin_{1\bmod \f_{L/K}}^{(\g)}$ we have  $l(\p_i^{n_i})=1\bmod \f_{L/K}$ and a fortiori $l(\p_i^{n_i})=1\bmod \f_{K_i/K}$ so that $l(\p_i^{n_i})\in N_{K_i/K}(I_{K_i})$. By Hasse's Norm Theorem, there is some $\pi_i\in K_i$ with $N_{K_i/K}(\pi_i)=l(\p_i^{n_i})$. By construction we have \[N_{K_i/K}(\p_i(\pi_i)^{-1})=\p_i^{n_i}\p_i^{-n_i}=O_K\] so that, by the Principal Genus Theorem, we can find an ideal $\b_i$ of $K_i$ with \[\p_i(\pi_i)^{-1}=\b_i^{1-\sigma_i}.\] We note that \[\f_{L/K_i/K}=\G_{K_i/K}(\f_{L/K}\f_{K_i/K}^{-1})\] so that \[N_{K_i/K}(\pi_i)=l(\p_i^{f_i})=1 \bmod \f_{L/K}=1\bmod \f_{K_i/K}(\f_{L/K}\f_{K_i/K}^{-1}).\] We now apply Terada's Norm Theorem (with $\m=\f_{L/K}\f_{K_i/K}^{-1}$) to find $\alpha_i\in K_i$ with $\alpha_i=1\bmod \f_{L/K_i/K}$ and $\pi_i =\alpha_i \beta_i^{1-\sigma_i}$ and \begin{equation}\label{eqn:def-alpha} N_{K_i/K}(\pi_i)=N_{K_i/K}(\alpha_i \beta_i^{1-\sigma_i})=N_{K_i/K}(\alpha_i)=l(\p_i^{n_i}).
\end{equation} Finally, we set $\a_i=(\beta_i)\b_i$ to get \[\p_i=(\alpha_i) \a^{1-\sigma_i}_i \quad \text{ and so } \quad \p_i=\a_i^{1-\sigma_i}\bmod \f_{L/K_i/K}.\] The ideals $\a_i$ for $1\leq i\leq r$ satisfy the conditions of Tannaka's Theorem and so we find an $A\in L$ with $A=1\bmod \G_{L/K}$ with \[\prod_{1\leq i\leq m} \a_i=(A).\] Finally, we set \[\Theta_i=\alpha_i A^{1-\sigma_i}\] and note that $\Theta_i\cdot O_L=\p_i$.

We now go about extending the map $l$, following rather closely the method of \S 1 of \cite{Tannaka58}. Each $\a\in \Id_{O_K}^{(\g)}$ can be written uniquely as a product of ideals \begin{equation} \a=\gamma(\a)\cdot \prod_{i=1}^r \p_i^{x_i}\label{eqn:unique-decomp}\end{equation} with $\gamma(\a)\in \Prin^{(\g)}_{1\bmod \f_{L/K}}$ and $0\leq x_i<n_i$. Before we continue let us note the following multiplicative relations for the ideals $\gamma(\a)$:
\begin{enumerate}[label=\textup(\roman*)]
\item If $\b\in \Prin_{1\bmod \f_{L/K}}^{(\g)}$ then \begin{equation}\label{eqn:mult-prin} \gamma(\a\b) = \gamma(\a)\gamma(\b).
\end{equation}
\item If $\b=\p_j$ and $x_j\neq n_j-1$ then \begin{equation}\label{eqn:mult-prime-triv} \gamma(\p_j)=O_K \quad \text{ and } \quad \gamma(\a\p_j)=\gamma(\a).
\end{equation}
\item If $\b=\p_j$ and $x_j=n_j-1$ then \begin{equation}\label{eqn:mult-prime-carry}\gamma(\a\p_j)=\gamma(\a)\gamma(\p_j^{n_j}). 
\end{equation}
\end{enumerate}
Still following \S 1 of \cite{Tannaka58} we now define $l(\a)$ by \[l(\a)=l(\gamma(\a))\prod_{i=1}^n \Theta_i^{w_i(x_i)}=l(\gamma(\a)) A^{\sigma_\a-1} \prod_{i=1}^r \alpha_i^{1+\sigma_i+\cdots + \sigma_i^{x_i-1}}\] where \[w_i(x_i)=\left(\sum_{j=1}^{x_i}\sigma_{\p_i}^j\right)\cdot\prod_{k=1}^{i-1} \sigma_{\p_k}^{x_k}.\] It is clear that $l(\a)\cdot O_{L}=\a\cdot O_{K(\a)}$ and that this map does indeed extend the given map $l$. Moreover, by construction we have the relations \[A=1\bmod \G_{L/K} \quad\quad l(\a)=1\bmod \f_{L/K} \quad\quad \alpha_i=1\bmod \f_{L/K_i/K},\] so that as $\G_{L/K}$ divides both $\f_{L/K}$ and $\f_{L/K_i/K}$, and as $\G_{L/K}$, $\f_{L/K}$ and $\f_{L/K_i/K}$ are invariant under the action of $G(L/K)$, we get the relation \[l(\a)=1\bmod \G_{L/K}.\] All that remains to be shown is that $l(\a\b)=l(\a)\sigma_\a(l(\b))$ for $\a, \b\in \Id_{O_K}^{(\g)}$.

So let $\b\in \Id_{O_K}^{(\g)}$ be another fractional ideal, and also write \[\b=\gamma(\b)\cdot \prod_{i=1}^r \p_i^{y_i} \quad \text{ and }\quad \a\b=\gamma(\a\b)\cdot \prod_{i=1}^r \p_i^{z_i}\] where $\gamma(\b), \gamma(\a\b)\in \Prin_{1\bmod \f_{L/K}}^{(\g)}$ and $0\leq y_i, z_i<n_i$. Define $\delta_i\in \{0, 1\}$ for $1\leq i\leq r$ by the equality \[z_i=x_i+y_i-\delta_i n_i.\] Then one finds (see equation (9) of \cite{Tannaka58}) \begin{equation}\label{eqn:tannaka}\frac{l(\a)\sigma_\a(l(\b))}{l(\a\b)}=\frac{l(\gamma(\a))l(\gamma(\b))}{l(\gamma(\a\b))}\cdot \prod_{i=1}^r N_{K_i/K}(\alpha_i)^{\delta_i}.\end{equation}

It remains to check that the right hand side of (\ref{eqn:tannaka}) is equal to one for all ideals $\b$. The group $\Id_{O_K}^{(\g)}$ is generated (as a monoid!) by $\Prin_{1\bmod \f_{L/K}}^{(\g)}$ and $\p_1, \ldots, \p_n$ so that by induction, it is enough to check that $l(\a\b)=l(\a)\sigma_\a(l(\b))$ when $\b\in \Prin_{1\bmod \f_{L/K}}^{(\g)}$ or when $\b=\p_j$ for $1\leq j\leq r.$
If $\b\in \Prin_{1\bmod \f_{L/K}}^{(\g)}$ then $\delta_i=0$ for $1\leq i\leq r$ and by (\ref{eqn:mult-prin}) we have $\gamma(\a)\gamma(\b)=\gamma(\a\b)$ so that \[\frac{l(\gamma(\a))l(\gamma(\b))}{l(\gamma(\a\b))}\prod_{i=1}^r N_{K_i/K}(\alpha_i)^{\delta_i}=1.\]

If $\b=\p_j$ and $x_j\neq n_j-1$ then $\delta_i=0$ for $1\leq i\leq r$ and by (\ref{eqn:mult-prime-triv}) we have $\gamma(\p_j)=O_K$ and $\gamma(\a\p_j)=\gamma(\a)$ so that \[\frac{l(\gamma(\a))l(\gamma(\b))}{l(\gamma(\a\b))}\prod_{i=1}^r N_{K_i/K}(\alpha_i)^{\delta_i}=1.\]

Finally, if $\b=\p_j$ and $x_j=n_j-1$ then $\delta_i=0$, unless $i=j$ in which case $\delta_j=1$, so that \[\prod_{i=1}^r N_{K_i/K}(\alpha_i)^{\delta_i}=N_{K_j/K}(\alpha_j)=l(\gamma(\p^{n_j}_j))\] by (\ref{eqn:def-alpha}). By (\ref{eqn:mult-prime-carry}) we have $\gamma(\a\p_j)=\gamma(\a)\gamma(\p_j^{n_j})$ so that \[\frac{l(\gamma(\a))l(\gamma(\b))}{l(\gamma(\a\b))}\prod_{i=1}^r N_{K_i/K}(\alpha_i)^{\delta_i}=l(\gamma(\p_j^{n_j}))^{-1}l(\gamma(\p_j^{n_j}))=1.\] Therefore, for all $\a, \b\in \Id_{O_K}^{(\g)}$ we have $l(\a\b)=l(\a)\sigma_\a(l(\b)).$
\end{proof}
\begin{rema} Let us now explain how (\ref{prop:tannaka-result-general}) strengthens the main result of \cite{Tannaka58}. The result in question is Theorem 1 of \cite{Tannaka58} and it states (we continue to use the notation of (\ref{prop:tannaka-result-general}))
\begin{theorem*}[Tannaka's Principal Ideal Theorem] There exist numbers $\Theta(\a)\in L$, indexed by ideals $\a\in \Id_{O_K}$, such that the following hold: \begin{enumerate}[label=\textup{(\roman*)}]
\item $\Theta(\a)\cdot O_{L}=\a\cdot O_L$,
\item $\Theta(\a)=1\bmod \G_{L/K}$, and
\item \[\frac{\Theta(\a)\sigma_\a(\Theta(\b))}{\Theta(\a\b)}\in O_K^\times.\]
\end{enumerate}
\end{theorem*} It is now an easy to replace (iii) of Tannaka's Principal Ideal Theorem with \[\frac{\Theta(\a)\sigma_\a(\Theta(\b))}{\Theta(\a\b)}=1.\]
Consider the sub-group \[\Prin_{1\bmod \f_{L/K}}^{(\f_{L/K})}\subset \Id_{K}^{(\f_{L/K})}.\] As $\Id_{K}^{(\f_{L/K})}$ is free abelian (generated by the prime ideals prime to $\f_{L/K}$) and every sub-group of a free abelian group is itself free abelian, one can define a multiplicative map \[\Prin_{1\bmod \f_{L/K}}^{(\f_{L/K})}\to K^\times: \a\mto l(\a)\] such that $l(\a)\cdot O_K=\a$ and such that $l(\a)=1\bmod \f_{L/K}$. Applying (\ref{prop:tannaka-result-general}) to the map $l$ we find a map \[l: \Id_{O_K}^{(\f_{L/K})}\to L^\times: \a\mto l(\a)\] such that $l(\a)\cdot O_{L}=\a\cdot O_L$, $l(\a)=1\bmod \G_{L/K}$ and such that \[\frac{l(\a)\sigma_\a(l(\b))}{l(\a\b)}=1.\] Therefore, putting $l(\a)=\Theta(\a)$ we can replace (iii) of Tannaka's Principal Ideal Theorem with \[\frac{\Theta(\a)\sigma_\a(\Theta(\b))}{\Theta(\a\b)}=1.\]
\end{rema}
\backmatter
%    Bibliography styles amsplain or harvard are also acceptable.
\bibliographystyle{smfplain}
\bibliography{biblio}
%\printnomenclature
\end{document}

%% file: custom.tex
\DeclareMathSymbol{A}{\mathalpha}{operators}{`A}%
\DeclareMathSymbol{B}{\mathalpha}{operators}{`B}%
\DeclareMathSymbol{C}{\mathalpha}{operators}{`C}%
\DeclareMathSymbol{D}{\mathalpha}{operators}{`D}%
\DeclareMathSymbol{E}{\mathalpha}{operators}{`E}%
\DeclareMathSymbol{F}{\mathalpha}{operators}{`F}%
\DeclareMathSymbol{G}{\mathalpha}{operators}{`G}%
\DeclareMathSymbol{H}{\mathalpha}{operators}{`H}%
\DeclareMathSymbol{I}{\mathalpha}{operators}{`I}%
\DeclareMathSymbol{J}{\mathalpha}{operators}{`J}%
\DeclareMathSymbol{K}{\mathalpha}{operators}{`K}%
\DeclareMathSymbol{L}{\mathalpha}{operators}{`L}%
\DeclareMathSymbol{M}{\mathalpha}{operators}{`M}%
\DeclareMathSymbol{N}{\mathalpha}{operators}{`N}%
\DeclareMathSymbol{O}{\mathalpha}{operators}{`O}%
\DeclareMathSymbol{P}{\mathalpha}{operators}{`P}%
\DeclareMathSymbol{Q}{\mathalpha}{operators}{`Q}%
\DeclareMathSymbol{R}{\mathalpha}{operators}{`R}%
\DeclareMathSymbol{S}{\mathalpha}{operators}{`S}%
\DeclareMathSymbol{T}{\mathalpha}{operators}{`T}%
\DeclareMathSymbol{U}{\mathalpha}{operators}{`U}%
\DeclareMathSymbol{V}{\mathalpha}{operators}{`V}%
\DeclareMathSymbol{W}{\mathalpha}{operators}{`W}%
\DeclareMathSymbol{X}{\mathalpha}{operators}{`X}%
\DeclareMathSymbol{Y}{\mathalpha}{operators}{`Y}%
\DeclareMathSymbol{Z}{\mathalpha}{operators}{`Z}%

\newcommand{\im}{\mathrm{im}}
\newcommand{\isomto}{\overset{\sim}{\longrightarrow}}
\newcommand{\mto}{\mapsto}
\DeclareMathOperator*{\colim}{colim}

\newcommand{\Aff}{\mathrm{Aff}}
\newcommand{\IndAff}{\mathrm{IndAff}}

\newcommand{\QCoh}{\mathrm{QCoh}}

\newcommand{\Sch}{\mathrm{Sch}}
\newcommand{\Mod}{\mathrm{Mod}}
\newcommand{\Set}{\mathrm{Set}}
\newcommand{\CL}{\mathrm{CL}}
\newcommand{\CM}{\mathrm{CM}}
\newcommand{\PSh}{\mathrm{PSh}}

\newcommand{\Alg}{\mathrm{Alg}}
\newcommand{\Ell}{\mathrm{Ell}}
\newcommand{\LT}{\mathrm{LT}}
\renewcommand{\L}{\mathrm{L}}

\newcommand{\Biring}{\mathrm{Biring}}

\renewcommand{\O}{\mathcal{O}}
\newcommand{\Q}{\mathbf{Q}}
\newcommand{\Sh}{\mathrm{Sh}}
\newcommand{\Z}{\mathbf{Z}}
\newcommand{\N}{\mathbf{N}}
\newcommand{\F}{\mathbf{F}}
\newcommand{\R}{\mathbf{R}}
\newcommand{\G}{\mathbf{G}}
\newcommand{\C}{\mathbf{C}}

\newcommand{\Id}{\mathrm{Id}}

\newcommand{\End}{\mathrm{End}}
\newcommand{\Isom}{\mathrm{Isom}}
\newcommand{\Hom}{\mathrm{Hom}}
\newcommand{\Aut}{\mathrm{Aut}}

\newcommand{\ideal}[1]{\mathfrak{#1}}
\renewcommand{\a}{\ideal{a}}
\renewcommand{\b}{\ideal{b}}
\renewcommand{\l}{\ideal{l}}
\renewcommand{\d}{\ideal{d}}
\renewcommand{\c}{\ideal{c}}
\newcommand{\m}{\ideal{m}}
\newcommand{\p}{\ideal{p}}
\newcommand{\f}{\ideal{f}}
\newcommand{\g}{\ideal{g}}
\renewcommand{\P}{\ideal{P}}
\newcommand{\D}{\ideal{D}}
\renewcommand{\G}{\ideal{G}}

\newcommand{\univ}{\mathrm{univ}}

\newcommand{\tors}{\mathrm{tors}}
\newcommand{\ab}{\mathrm{ab}}

\newcommand{\sep}{\mathrm{sep}}
\newcommand{\ur}{\mathrm{ur}}
\newcommand{\arch}{\mathrm{arch}}

\newcommand{\fin}{\mathrm{fin}}
\newcommand{\red}{\mathrm{red}}
\newcommand{\an}{\mathrm{an}}

\newcommand{\str}{\mathrm{str}}
\newcommand{\et}{\textup{\'et}}
\newcommand{\per}{\mathrm{per}}

\newcommand{\id}{\mathrm{id}}
\newcommand{\Fr}{\mathrm{Fr}}

\newcommand{\rk}{\mathrm{rk}}

\newcommand{\Inf}{\mathrm{Inf}}
\newcommand{\Spec}{\mathrm{Spec}}

\newcommand{\Spf}{\mathrm{Spf}}
\newcommand{\Lie}{\mathrm{Lie}}

\newcommand{\Prin}{\mathrm{Prin}}
\newcommand{\Ner}{\mathrm{N\acute{e}r}}

\newcommand{\Eq}{\mathrm{Eq}}
\newcommand{\In}{\mathrm{In}}